\documentclass[10pt, reqno]{amsart}
\usepackage{amssymb}
\usepackage{yfonts}
\usepackage{hyperref}
\usepackage{mathrsfs}
\usepackage {latexsym}
\usepackage{graphics}
\usepackage{txfonts}
\usepackage{breqn}
\usepackage{cite}
\usepackage{color}
\usepackage[shortlabels]{enumitem}
\usepackage{wasysym}
\usepackage{tikz}
\usetikzlibrary{snakes}
\usetikzlibrary{arrows.meta}
\usetikzlibrary{decorations.pathmorphing}
\usetikzlibrary{er, positioning} 
\usepackage{mathtools}
\usepackage{comment}

\newcommand{\circlesign}[1]{ 
\mathbin{
\mathchoice
{\buildcirclesign{\displaystyle}{#1}}
{\buildcirclesign{\textstyle}{#1}}
{\buildcirclesign{\scriptstyle}{#1}}
{\buildcirclesign{\scriptscriptstyle}{#1}}
} 
}

\newcommand\buildcirclesign[2]{%
\begin{tikzpicture}[baseline=(X.base), inner sep=0, outer sep=0]
\node[draw,circle] (X)  {\ensuremath{#1 #2}};
\end{tikzpicture}%
}

\newif\ifshow
\showtrue

\ifshow

\else
\excludecomment{details}
\fi

\newtheorem{theorem}{Theorem}[section]
\newtheorem{remark}{Remark}[section]
\newtheorem{lemma}[theorem]{Lemma}
\newtheorem{assume}[theorem]{Assumption}
\newtheorem{proposition}[theorem]{Proposition}
\newtheorem{corollary}[theorem]{Corollary}
\newtheorem{define}{Definition}[section]

\newtheorem*{proposition*}{Proposition}

\newcommand{\joinR}{\hspace{-.1em}}
\newcommand{\RomanI}{I}
\newcommand{\RomanII}{\mbox{\RomanI\joinR\RomanI}}
\newcommand{\RomanIII}{\mbox{\RomanI\joinR\RomanII}}
\newcommand{\RomanV}{V}
\newcommand{\RomanIV}{\mbox{\RomanI\joinR\RomanV}}

\DeclareMathOperator*{\Tr}{Tr}
\DeclareMathOperator*{\Id}{Id}

\DeclareMathOperator*{\supp}{supp}
\DeclareMathOperator*{\divergence}{div}
\DeclareMathSymbol{:}{\mathord}{operators}{"3A}

\allowdisplaybreaks

\begin{document}
\title[Singular Navier-Stokes equations with Lions' exponent]{Remarks on the three-dimensional Navier-Stokes equations with Lions' exponent forced by space-time white noise}

 
\author[Kazuo Yamazaki]{Kazuo Yamazaki}  
\address{Department of Mathematics, University of Nebraska, Lincoln, 243 Avery Hall, PO Box 880130, Lincoln, NE 68588-0130, U.S.A.; Phone: 402-473-3731; Fax: 402-472-8466}
\email{kyamazaki2@nebraska.edu}
\date{}
\thanks{This work was supported by the Simons Foundation MPS-TSM-00962572.}

\begin{abstract}
We study the three-dimensional Navier-Stokes equations forced by space-time white noise and diffused via the fractional Laplacian with Lions' exponent so that it is precisely the energy-critical case. We prove its global solution theory following the approach of Hairer and Rosati (2024, Annals of PDE, \textbf{10}, pp. 1--46). 
\end{abstract}

\maketitle

\section{Introduction}\label{Section 1}  

\subsection{Motivation from physics and real-world applications}\label{Subsection 1.1}
\hfill\\
Ever since the pioneering work \cite{LL57} on hydrodynamic fluctuations by Landau and Lifshitz, an abundance of physics literature has been devoted to partial differential equations (PDEs) forced by random noise, which we call stochastic PDEs (SPDEs), and a special type of such a force called the space-time white noise (STWN) (see \eqref{STWN}) has caught an exceptional amount of attention. Examples include, but are not limited to,  ferromagnet model \cite{MM75}; Kardar-Parisi-Zhang (KPZ) equation \cite{KPZ86}; magnetohydrodynamics (MHD) system \cite{CT92}; Navier-Stokes equations \cite{FNS77, YO86}; $\Phi^{4}$ model \cite{PW81, GJ87}; and the Rayleigh-B$\acute{\mathrm{e}}$nard equation \cite{ACHS81, GP75, HS92, SH77, ZS71}. As one example, postponing details of notation, we present the special case of \cite[Equations (3), (6), and (8)]{SH77}, which is the three-dimensional (3D) Navier-Stokes equations with a force in the form of a matrix $\xi = (\xi_{ij})_{1\leq i,j \leq 3}$ that represents the random effect of molecular noise:  
\begin{subequations}\label{SH77}
\begin{align}
&\partial_{t} u_{i} + (u\cdot\nabla) u_{i} + \partial_{i} \pi - \nu \Delta u_{i} = (\divergence\xi)_{i}, \hspace{10mm} \nabla\cdot u = 0,  \\
&\mathbb{E} [\xi_{ij} (x, t) \xi_{lm} (x', t') ] = \delta(x - x') \delta(t-t') (\delta_{il}\delta_{jm} + \delta_{im} \delta_{jl}), 
\end{align}
\end{subequations} 
for all $i,j,l,m \in \{1,2,3\}$ where $u(t,x)$, $\pi(t,x)$ respectively denote the velocity and pressure fields, $\nu \geq 0$ the viscosity coefficient, and $\delta$ the Dirac delta at the origin.  A general consensus in the community of researchers on SPDEs is that such divergence of a STWN makes rigorous analysis of \eqref{SH77} currently out of reach, even in the two-dimensional (2D) case. If $\{ \xi_{ij}\}_{ \{1\leq i, j \leq 3 \}}$ is replaced by a vector $\xi = (\xi_{1} \hspace{1mm} \xi_{2} \hspace{1mm}\xi_{3})^{T}$ and there is no divergence on the noise, \eqref{SH77} in the 2D case is known to admit a unique global solution thanks to Da Prato and Debussche \cite{DD02} and Hairer and Rosati \cite{HR24}, which is consistent with the deterministic case (see \eqref{Energy-critical}). In the 3D deterministic case, the global solution theory of \eqref{SH77} is a famous outstanding open problem. On the other hand, the global solution theory is available due to Lions \cite{L69} if the diffusion $-\Delta$ is replaced by $\Lambda^{m}$ with the Lions' exponent $m = 1+ \frac{d}{2}$ (see \eqref{Energy-critical}), where $\Lambda^{\gamma} \triangleq (-\Delta)^{\frac{\gamma}{2}}$ is a fractional Laplacian of order $\frac{\gamma}{2} \in \mathbb{R}$, specifically a Fourier operator with a symbol $\lvert k \rvert^{\gamma}$ so that $\widehat{ \Lambda^{\gamma} f} (k) = \lvert k \rvert^{\gamma} \hat{f} (k)$ (see Section \ref{Subsection 1.3} for further discussions). The purpose of this manuscript is to prove global solution theory for the 3D Navier-Stokes equations forced by STWN with $\Lambda^{\frac{5}{2}}$ on a torus via the approach of \cite{HR24}, i.e., extend \cite{L69} under the forcing by STWN. 
 
 We briefly highlight four major difficulties in this endeavor. 
\begin{enumerate}[label=(\roman*)]
\item Commutator: The diffusion in \cite{HR24} took the form of $-\Delta$, and the ``commutator term'' in \cite[Equation (4.15)]{HR24} consisted of ``$2\Tr [ ( \nabla w \circlesign{\prec} (\nabla Q^{\mathcal{H}})]$.'' The origin of this product can be described heuristically as a result of paracontrolled ansatz (see \eqref{Define w sharp}) that leads to 
\begin{equation}\label{HR24 commutator} 
\Delta (fg) - (\Delta f) g - f (\Delta g) = 2 (\nabla f) \cdot (\nabla g)
\end{equation} 
where we merely relied on Leibniz rule. In the field of Analysis and PDEs, it is well known that the extreme cases of distributions of derivatives such as $(\Delta f) g$ or $f (\Delta g)$ are the most difficult because it requires both $f$ and $g$ to be twice differentiable. In this perspective, $(\nabla f)\cdot (\nabla g)$ in \eqref{HR24 commutator} is informally as convenient as it gets. An estimate of $\lVert 2\Tr [ ( \nabla w \circlesign{\prec} (\nabla Q^{\mathcal{H}})]\rVert_{H^{\frac{1}{2}+ \kappa}} \lesssim \lVert \nabla w \rVert_{H^{-\frac{1}{2} + 2 \kappa}} \lVert \nabla Q^{\mathcal{H}} \rVert_{\mathscr{C}^{1-\kappa}}$ via \eqref{Sobolev products d} can be seen in \cite[p. 22]{HR24}; no commutator estimate from harmonic analysis is needed. In the current manuscript, we must face estimates of  
\begin{equation}\label{our commutator}
\Lambda^{\frac{5}{2}} (fg) - \Lambda^{\frac{5}{2}} f g - f \Lambda^{\frac{5}{2}} g 
\end{equation} 
in \eqref{est 87} multiple times, for which Leibniz rule is of no help (see also \eqref{Burgers' Define B}). We will elaborate on this difficulty furthermore in Remark \ref{Remark 2.2}. 

\item Criticality in the proof of uniqueness of weak solution: Due to higher dimension, the proof of uniqueness of weak solution in Theorem \ref{Theorem 2.3} faced difficulties that we had to overcome by finding an appropriate Besov interpolation. We will elaborate on this difficulty further in Remark \ref{Remark 2.3}. 

\item Anderson Hamiltonian: We needed to extend the well-known facts on the Anderson Hamiltonian to the 3D case with $\Lambda^{\frac{5}{2}}$. While \cite{HR24} relied on the previous work \cite{AC15} by Allez and Chouk that provided the necessary results on the 2D Andderson Hamiltonian using the theory of paracontrolled distributions from \cite{GIP15} by Gubinelli, Imkeller, and Perkowski, at first sight, the first-order paracontrolled calculus did not seem readily applicable in our case due to such a high power of the fractional Laplacian, similarly to the difficulty (i) aforementioned. To overcome this approach, we initially attempted relying on the theory of regularity structures (RS) following \cite{L19a} by Labb$\mathrm{\acute{e}}$; yet, we faced difficulty with that too due to the necessary behavior of asymptotic behavior of the Green's function to fractional Laplacian (see \eqref{Asymptotic of Bessel}). We overcame this obstacle by finding a room within the approach of \cite{AC15}. Upon doing so, it turns out that the same commutator issue in (i) arises in the approach of \cite{AC15} as well (see \eqref{Burgers' Define B}). We will elaborate on this difficulty furthermore in Remark \ref{Remark 2.4}. 

\item Renormalization: The Leray projection operator in the 2D case can be written in a very concise form (see \eqref{2D projection}), heuristically because curl of a 2D vector gives a scalar. We can extend the same formulation of the Leray projection to our 3D case (see \eqref{alternative projection}), but we found that it becomes very cumbersome to compute involving multiple curl operators. We ultimately chose the matrix form of the projection for simplicity (see \eqref{projection}). However, its price that we must pay was that the necessary computations upon renormalization increased exponentially. We will elaborate on this difficulty furthermore in Remark \ref{Remark 2.5}.  
\end{enumerate} 

\subsection{Introduction of the main equations}\label{Subsection 1.2}
\hfill\\ We set up a minimum amount of notations and preliminaries, postponing the rest to Section \ref{Section 3}. We define $\mathbb{N} \triangleq \{ 1, 2, \hdots \}$ and $\mathbb{N}_{0} \triangleq \mathbb{N} \cup \{0\}$, and work with a spatial variable $x \in \mathbb{T}^{d} = (\mathbb{R} \setminus \mathbb{Z})^{d}$ for $d \in \mathbb{N}$. We abbreviate by denoting $\partial_{t} \triangleq  \frac{\partial}{\partial t}, \partial_{i}  \triangleq \frac{\partial}{\partial x_{i}}$ for $i \in \{1, \hdots, d\}$, and define $\mathbb{P}_{\neq 0} f \triangleq f - \fint_{\mathbb{T}^{d}} f(x) dx$. We write $A \lesssim_{\alpha, \beta} B$ whenever there exists a constant $C = C(\alpha,\beta) \geq 0$ such that $A \leq CB$ and $A \approx_{\alpha,\beta} B$ in case $A \lesssim_{\alpha,\beta}B$ and $A \gtrsim_{\alpha,\beta}B$. We write $A \overset{( \cdot)}{\lesssim} B$ whenever $A\lesssim B$ due to $(\cdot)$.  We follow the typical convention that a universal constants in a series of inequalities continue to be denoted by ``$C$.'' We denote the Lebesgue, homogeneous and inhomogeneous Sobolev spaces by $L^{p}, \dot{H}^{s}$, and $H^{s}$ for $p\in [1,\infty], s\in \mathbb{R}$ with corresponding norms of $\lVert \cdot\rVert_{L^{p}}, \lVert \cdot \rVert_{\dot{H}^{s}}$, and $\lVert \cdot \rVert_{H^{s}}$, respectively. Finally, we denote the Schwartz space and its dual by $\mathcal{S}$ and $\mathcal{S}'$, and Fourier transform of $f$ by $\mathcal{F}(f) = \hat{f}$. We also recall the H$\ddot{\mathrm{o}}$lder-Besov spaces $\mathscr{C}^{\gamma} \triangleq B_{\infty,\infty}^{\gamma}$ for $\gamma \in \mathbb{R}$ which are equivalent to the classical H$\ddot{\mathrm{o}}$lder spaces $C^{\lfloor \alpha\rfloor, \alpha - \lfloor \alpha \rfloor}$ whenever $\alpha \in (0,\infty) \setminus \mathbb{N}$ although $C^{\alpha} \subsetneq \mathscr{C}^{\alpha}$ for all $\alpha \in \mathbb{N}$ (see \cite[p. 99]{BCD11}); we defer detailed definitions of Besov spaces to Section \ref{Subsection 3.2}. 

We define the 3D Leray projection onto the space of divergence-free vector fields as 
\begin{equation}\label{projection}
\mathbb{P}_{L} f(x) \triangleq \sum_{k \in\mathbb{Z}^{3} \setminus \{0\}} e^{i 2 \pi k \cdot x} \left( \hat{f}(k) - \frac{ k (k \cdot \hat{f} (k))}{\lvert k \rvert^{2}} \right) = \sum_{k\in\mathbb{Z}^{3} \setminus \{0\}} e^{i 2\pi k \cdot x} \left( \Id - \frac{k\otimes k}{\lvert k \rvert^{2}} \right) \hat{f}(k),  
\end{equation} 
for all $f \in \mathcal{S}' ( \mathbb{T}^{3}; \mathbb{R}^{3})$ that is mean-zero, where $\Id$ represents the $3 \times 3$ identity matrix. This projection in the 2D case can take a significantly simpler form (see \cite[p. 6]{HR24}):
\begin{align}\label{2D projection}
\mathbb{P}_{L} f(x) = \sum_{k\in\mathbb{Z}^{2} \setminus \{0\}} e^{i 2 \pi k \cdot x} k^{\bot}\left( \hat{f} (k) \cdot k^{\bot} \right) \frac{1}{\lvert k^{\bot} \rvert^{2}}, 
\end{align}
where $k^{\bot} \triangleq (k_{2}, -k_{1})$ for $k= (k_{1}, k_{2}) \in \mathbb{Z}^{2}$. We can extend this to the 3D case as 
\begin{equation}\label{alternative projection}
\mathbb{P}_{L} f(x) = \sum_{k\in\mathbb{Z}^{3} \setminus \{0\}} e^{i 2 \pi k \cdot x}  k \times  (\hat{f}(k) \times k)   \frac{1}{\lvert k \rvert^{2}}; 
\end{equation} 
however, unfortunately, preliminary computations using \eqref{alternative projection} turned out to be too cumbersome. For example, in the 2D case, $\hat{f}(k) \cdot k^{\bot}$ is a scalar and thus $k^{\bot} (\hat{f}(k) \cdot k^{\bot})$ in \eqref{2D projection} is just a vector multiplied by a scalar. In contrast, in the 3D case, $k\times (\hat{f}(k) \times k)$ in \eqref{alternative projection} turned out to be strenuous and hence we chose to work directly with the second expression in  \eqref{projection} in the matrix form. 

We now fix a probability space $(\Omega, \mathcal{F}, \mathbb{P})$ so that the STWN $\xi = (\xi_{1}\hspace{1mm} \xi_{2} \hspace{1mm} \xi_{3})$ can be introduced as a distribution-valued Gaussian field with a  covariance of 
\begin{equation}\label{STWN} 
\mathbb{E} [ \xi_{i}(t,x) \xi_{j}(s,y) ] = 1_{\{ i=j\}} \delta(t-s) \prod_{l=1}^{3} \delta(x_{l} - y_{l}); 
\end{equation} 
i.e.,
\begin{align}\label{Define STWN}
\mathbb{E} [ \xi_{i} (\phi) \xi_{j} (\psi) ] = 1_{\{ i = j \}} \int_{\mathbb{R} \times \mathbb{T}^{3}} \phi(t,x) \psi(t,x) dx dt \hspace{3mm} \forall \hspace{1mm} \phi, \psi \in \mathcal{S} ( \mathbb{R} \times \mathbb{T}^{3}). 
\end{align}
We also introduce the stochastic generalized Navier-Stokes equations of our main interest:
\begin{equation}\label{NS} 
\partial_{t} u + \divergence (u\otimes u) + \nabla \pi + \nu \Lambda^{m} u = \mathbb{P}_{\neq 0} \xi,\hspace{3mm} \nabla\cdot u = 0 
\end{equation} 
for $m \geq 0$, or equivalently
\begin{equation}\label{NS after PL}
\partial_{t} u + \mathbb{P}_{L} \divergence (u \otimes u) + \nu \Lambda^{m} u =  \mathbb{P}_{L} \mathbb{P}_{\neq 0} \xi. 
\end{equation} 
We refer to the case $\xi \equiv 0$ as the generalized Navier-Stokes equations, and additionally the case $m = 2$ the Navier-Stokes equations. The Navier-Stokes equations with $\nu = 0$ reduces to the Euler equations.

\subsection{Review of relevant results}\label{Subsection 1.3}
\hfill\\ The mathematical analysis of the Navier-Stokes equations dates back to the pioneering work of Leray \cite{L34}, who together with Hopf \cite{H51}, proved the global existence of a Leray-Hopf weak solution to the 3D Navier-Stokes equations. While the uniqueness of the Leray-Hopf weak solution to the 3D Navier-Stokes equations remains open, Lions introduced the generalized Navier-Stokes equations in \cite{L59} and proved in \cite{L69} (see also \cite{W03}) the uniqueness of its Leray-Hopf weak solution as long as 
\begin{equation}\label{Energy-critical}
m \geq 1+ \frac{d}{2}. 
\end{equation} 
This threshold can be understood through rescaling; i.e., if $(u,\pi)$ solves the generalized Navier-Stokes equations, then so does $(u_{\lambda},\pi_{\lambda}) (t,x) \triangleq (\lambda^{2m-1} u, \lambda^{4m-2} \pi) (\lambda^{2m} t, \lambda x)$. We call the cases $m > 1 + \frac{d}{2}, m = 1 + \frac{d}{2}$, and $m < 1 + \frac{d}{2}$, respectively $L^{2}(\mathbb{T}^{d})$-subcritical, $L^{2}(\mathbb{T}^{d})$-critical, and $L^{2}(\mathbb{T}^{d})$-supercritical. 

In relevance to our subsequent discussions in Section \ref{Subsection 1.4}, we recall that this energy-critical threshold can be overcome \emph{logarithmically}. Specifically, Tao \cite{T09} observed that the traditionally desired \emph{a priori} bounds of $\partial_{t} \lVert u(t) \rVert_{H^{s}}^{2} \leq a(t) \lVert u(t) \rVert_{H^{s}}^{2}$ for any $a \in L_{\text{loc}}^{1}(0,T)$ is sufficient but unnecessary because a logarithmically worse bound of  
\begin{equation}\label{logarithmically supercritical}
\partial_{t} \lVert u(t) \rVert_{H^{s}}^{2} \leq a(t) \lVert u(t) \rVert_{H^{s}}^{2} \ln(e+ \lVert u (t) \rVert_{H^{s}}^{2})
\end{equation} 
suffices as it results in the finite bound of  $\lVert u(t) \rVert_{H^{s}}^{2} \leq \lVert u(0) \rVert_{H^{s}}^{2} \exp (\exp (\int_{0}^{T} a(s) ds))$. Tao was able to show that the Navier-Stokes equations with diffusion $\mathcal{L} u$ such that 
\begin{equation*}
\widehat{\mathcal{L}u}(\xi) = \frac{  \lvert \xi \rvert^{1+ \frac{d}{2}}}{ \sqrt{\ln ( 2 + \lvert \xi \rvert^{2} )}} \hat{u}(\xi), 
\end{equation*} 
and hence in the \emph{logarithmically supercritical} case, still admits a global unique solution from every sufficiently smooth initial data by employing delicate harmonic analysis techniques. Specifically, he split the dyadic decomposition of Littlewood-Paley theory by a time-dependent cutoff so that 
\begin{align}\label{frequency decomposition}
\sum_{j\geq -1} = \sum_{j: 2^{j} \leq e + \lVert u(t) \rVert_{H^{s}}^{2} } + \sum_{j: 2^{j} > e + \lVert u(t) \rVert_{H^{s}}^{2}}, 
\end{align}
of which the lower frequency part already gives 
\begin{align*}
\sum_{j: 2^{j} \leq e+ \lVert u(t) \rVert_{H^{s}}^{2}} 1 \approx \ln (e+ \lVert u (t) \rVert_{H^{s}}^{2}), 
\end{align*}
indicating the heuristic of how the $\ln(e+ \lVert u (t) \rVert_{H^{s}}^{2})$ in \eqref{logarithmically supercritical} comes about (see also \cite{BMR14} and \cite{W11, Y18} in the case of the MHD system). 

If the STWN $\xi$ in \eqref{NS after PL} is replaced by a sufficiently smooth noise, then the global solution theory of  \cite{L69} can be readily extended in the $L^{2}(\mathbb{T}^{d})$-critical and subcritical cases (see e.g. \cite{BF17}). However, the situation becomes drastically different in the presence of the STWN, primarily due to its irregularity. Specifically, following \cite[Lemma 10.2]{H14} (or more precisely \cite[Lemma 4.1]{BK17} in the case of fractional Laplacian), we know that 
\begin{equation}\label{Regularity of xi}
\xi \in \mathscr{C}^{\beta} (\mathbb{T}^{d}) \hspace{1mm} \text{ for }   \hspace{1mm} \beta < - \frac{d+m}{2} \hspace{3mm} \mathbb{P}\text{-a.s.}
\end{equation}
The irregularity of the force can directly limit the smoothness of the solution and the product within the nonlinear term can become ill-defined considering Lemma \ref{Lemma 3.1}; we call SPDEs in such a case singular. 

Let us review the developments in the past few decades concerning the singular SPDEs. First, Bertini and Giacomin \cite{BG97} studied the 1D KPZ equation using Cole-Hopf transform. Hairer \cite{H13} was the first to provide a local solution theory to the 1D KPZ equation using rough path theory due to Lyons \cite{L98} (see also \cite{HQ18}). Then Hairer \cite{H14} and Gubinelli, Imkeller, and Perkowski \cite{GIP15} invented the novel theories of RS and paracontrolled distributions, respectively. Informally stated, these powerful tools provide systematic approaches to prove local-in-time solution theory for singular SPDEs that are locally subcritical, i.e., the homogeneity of nonlinear term is strictly larger than that of the force (see \cite[Assumption 8.3]{H14} for a  precise definition). Using these powerful new tools, the research area on singular SPDEs flourished, e.g. local-in-time solution theory \cite{CC18, Y23a, ZZ15} and strong Feller property \cite{HM18a, Y21a, ZZ20}, just to name a few. Initially in \cite{H14}, for every fixed locally subcritical SPDE, it was necessary to create an RS and verify issues such as the well-posedness of the abstract formulation of the system, identification of the renormalized solutions, and convergence of the renormalized models to a limiting model. However, a series of works \cite{BCCH21, BHZ19, CH18} unified these aspects and ultimately provided a more general theorem that leads to local-in-time solution theory and stability for a wide class of locally subcritical singular SPDEs (e.g. ``Metatheorem in \cite[Theorem 2.18]{BCCH21}''). 

While the local solution theory for locally subcritical SPDEs now has a clear pathway, extending such a local solution for all time is currently a very important research direction. Da Prato and Debussche  \cite{DD02} exploited the explicit knowledge of the invariant measure $\mu$ of the 2D Navier-Stokes equations forced by STWN due to the identity of 
\begin{equation}\label{Special identity}
\int_{\mathbb{T}^{2}} (u\cdot\nabla) u \cdot \Delta u dx = 0, 
\end{equation} 
which is a consequence of the solution's divergence-free property, and proved the global existence of a unique solution starting from $\mu$-almost every (a.e.) initial data (We will discuss \cite{HR24} by Hairer and Rosati on the 2D Navier-Stokes equations forced by STWN in Section \ref{Subsection 1.4}). Exploiting the damping effect from the nonlinear term of the $\Phi^{4}$ model, they obtained analogous result in \cite[Theorem 4.2]{DD03b}; subsequently, it was extended to the whole plane by Mourrat and Weber \cite{MW17} (see also \cite{HM18}).  Finally, Gubinelli and Perkowski \cite{GP17} observed that the solution to the 1D KPZ equation constructed by Hairer \cite{H13} is global-in-time (see \cite[p. 170 and Corollary 7.5]{GP17}). We also mention that the 2D and 3D stochastic Yang-Mills equation forced by STWN has seen significant developments in \cite{CCHS22a, CCHS22b} recently; however, its complex nonlinearity does not seem to admit favorable properties such as the explicit knowledge of its invariant measure, damping effect, or transformation. 

When extending local well-posedness to global well-posedness seems difficult, or even a construction of a global-in-time solution that is not necessarily unique, presents significant challenges, the technique of convex integration has shown signs of potentials. This technique that originated in differential geometry  was adapted by De Lellis and Sz$\acute{\mathrm{e}}$kelyhidi Jr. to the Euler equations in a groundbreaking work \cite{DS09} that constructed a bounded solution with compact support in space-time. Further extensions (e.g. \cite{DS13}) led Isett \cite{I18} to prove the Onsager's conjecture and Buckmaster and Vicol \cite{BV19a} to prove non-uniqueness of weak solutions to the 3D Navier-Stokes equations. Of relevance to our manuscript, we highlight that Buckmaster, Colombo, and Vicol \cite{BCV22} and Luo and Titi \cite{LT20} proved sharp non-uniqueness of weak solutions to the 3D generalized Navier-Stokes equations with $\Lambda^{m}$ for all $m < \frac{5}{2}$ (recall \eqref{Energy-critical}). We refer to \cite{BV19b} for further references of convex integration in the deterministic case. 

Concerning the convex integration technique applied on SPDEs, Breit, Feireisl, and Hofmanov$\acute{\mathrm{a}}$ \cite{BFH20} and Chiodaroli, Feireisl, and Flandoli \cite{CFF19} proved path-wise non-uniqueness of certain Euler equations forced by white-in-time noise in dimensions two and three; it was pointed out there already that the convex integration solutions to SPDEs are remarkably probabilistically strong, the construction of which was an open problem (\hspace{1sp}\cite{F08}). Subsequently, Hofmanov$\acute{\mathrm{a}}$, Zhu, and Zhu \cite{HZZ19} proved non-uniqueness in law of the 3D Navier-Stokes equations forced by white-in-time noise, and \cite{Y22a} extended it to the generalized case with $\Lambda^{m}$ for all $m < \frac{5}{2}$ (also \cite{Y22c} and \cite{CLZ24, Y24a} for extension respectively to the Boussinesq and MHD systems). Very recently, Bru$\acute{\mathrm{e}}$, Jin, Li, and Zhang \cite{BJLZ24} extended \cite{Y22a} to the case of the Leray-Hopf weak solution that satisfies the energy inequality, although under an additional non-zero force. Finally, Hofmanov$\acute{\mathrm{a}}$, Zhu, and Zhu \cite{HZZ21b} constructed, from a given initial data, infinitely many global-in-time solutions to the 3D Navier-Stokes equations forced by STWN (see \cite{LZ23} in the 2D case); prior to \cite{HZZ21b}, it was an open problem to extend \cite{ZZ15} to global-in-time. We emphasize that the solution to the 3D Navier-Stokes equations does not satisfy $\int_{\mathbb{T}^{3}} (u\cdot\nabla) u \cdot \Delta u dx = 0$, in sharp contrast to \eqref{Special identity}. Additionally, Hofmanov$\acute{\mathrm{a}}$, Zhu, and Zhu \cite{HZZ22a} constructed, from a given initial data, infinitely many solutions to the 2D surface quasi-geostrophic equations forced by spatial derivatives of white-in-space noise, which covered the locally critical and even supercritical cases; subsequently, the same authors, together with Luo, extended \cite{HZZ22a} to the case of STWN in \cite{HLZZ23}.  

\subsection{Program of \cite{HR24}}\label{Subsection 1.4}
\hfill\\ The main inspiration of this manuscript comes from \cite{HR24} which provided the second proof of the global solution theory for the 2D Navier-Stokes equations forced by STWN. Its significance is at least two-fold. First, \cite{HR24} did so without relying on the explicit knowledge of the invariant measure in contrast to \cite{DD02}; to emphasize this capability, \cite{HR24} added an extra force $\zeta \in \mathscr{C}^{-2+\kappa}$ for $\kappa > 0$ small to clearly disallow the approach of \cite{DD02}: 
\begin{equation}\label{HR24 main equation}
\partial_{t} u + \divergence ( u \otimes u) + \nabla \pi = \nu\Delta u + \mathbb{P}_{\neq 0} (\xi + \zeta), \hspace{3mm} \nabla\cdot u = 0.  
\end{equation} 
Second, the approach of \cite{HR24} seems more estimates-oriented rather than probabilistic, which we briefly outline. First, given an initial data $u^{\text{in}}$,  following \cite{DD02}, the authors of \cite{HR24} consider
\begin{equation}\label{HR24 Equation of X}
\partial_{t} X = \nu\Delta X + \mathbb{P}_{L} \mathbb{P}_{\neq 0} \xi, \hspace{3mm} X(0, \cdot) = 0, 
\end{equation} 
and then $v \triangleq u - X$ so that 
\begin{equation}\label{HR24 Equation of v}
\partial_{t} v + \mathbb{P}_{L} \divergence (v+X)^{\otimes 2} = \nu\Delta v + \mathbb{P}_{L} \mathbb{P}_{\neq 0} \zeta, \hspace{3mm} v (0,\cdot) = u^{\text{in}}(\cdot), 
\end{equation} 
where we denoted $A^{\otimes 2} \triangleq A \otimes A$. Then $\xi \in \mathscr{C}^{\beta}(\mathbb{T}^{2})$ for $\beta < - 2$ so that $X \in \mathscr{C}^{\beta}(\mathbb{T}^{2})$ for $\beta < 0$ $\mathbb{P}$-a.s. To exclude the ill-defined products, they consider again 
\begin{equation}
\partial_{t} Y + \mathbb{P}_{L} \divergence ( X \otimes Y + Y \otimes X + X^{\otimes 2}) = \nu\Delta Y + \mathbb{P}_{L} \mathbb{P}_{\neq 0} \zeta, \hspace{3mm} Y (0,\cdot ) = 0, 
\end{equation} 
so that $w \triangleq v - Y$ solves 
\begin{equation}\label{HR24 w}
\partial_{t} w + \mathbb{P}_{L} \divergence \left(  w^{\otimes 2} + (X+Y) \otimes w + w \otimes (X+Y) + Y^{\otimes 2} \right) = \nu\Delta w, \hspace{3mm} w(0,\cdot) = u^{\text{in}}(\cdot).
\end{equation} 
Here, because $X \in \mathscr{C}^{\beta}(\mathbb{T}^{2})$ for all $\beta < 0$ is more singular than $Y$ that has a strictly positive regularity, $w \in \mathscr{C}^{\beta}(\mathbb{T}^{2})$ for $\beta < 1$ due to the most singular term $\divergence (X \otimes w + w \otimes X)$ in \eqref{HR24 w}. Now, because the 2D Navier-Stokes equations is energy-critical (recall \eqref{Energy-critical}), informally we may assume for now that it suffices to perform the $L^{2}(\mathbb{T}^{2})$-estimate of $w^{\mathcal{L}}$ which is defined by $w$ (see \eqref{Define QH, wH, and wL}). Upon doing so, one is faced with 
\begin{equation}\label{HR24 estimate}
\partial_{t} \lVert w^{\mathcal{L}} \rVert_{L^{2}}^{2} = -  \nu\lVert w^{\mathcal{L}} \rVert_{\dot{H}^{1}}^{2}  + \int_{\mathbb{T}^{2}} w^{\mathcal{L}} \cdot \left( \frac{\nu\Delta}{2} \Id  -  2\nabla_{\text{sym}} X  \right)w^{\mathcal{L}}dx + \text{l.s.t.} 
\end{equation}  
where we denoted the less singular terms by ``l.s.t.'' and $\nabla_{\text{sym}}$ is a symmetric gradient matrix (see \eqref{symmetric gradient}).  Because $w^{\mathcal{L}} \in \mathscr{C}^{1-\kappa}(\mathbb{T}^{2})$ while $\nabla_{\text{sym}} X \in \mathscr{C}^{-1-\kappa}$ $\mathbb{P}$-a.s., the product $(w^{\mathcal{L}}\cdot\nabla_{\text{sym}}) X$ is ill-defined. Then Hairer and Rosati define 
\begin{equation}\label{HR24 At}
\mathcal{A}_{t}^{\lambda} \triangleq \frac{\nu\Delta}{2} \Id  -  2\nabla_{\text{sym}} X  - r_{\lambda}(t) \Id 
\end{equation} 
where $r_{\lambda}(t)$ is a renormalization constant to compensate for the ill-defined product (see \eqref{Define P lambda and r lambda}). Applying \eqref{HR24 At} to \eqref{HR24 estimate} gives us 
\begin{equation}\label{heuristic estimate} 
\partial_{t} \lVert w^{\mathcal{L}} \rVert_{\dot{H}^{1}}^{2} = - \nu\lVert w^{\mathcal{L}} \rVert_{\dot{H}^{1}}^{2} + \int_{\mathbb{T}^{2}} w^{\mathcal{L}} \cdot \mathcal{A}_{t}^{\lambda} w^{\mathcal{L}} dx + r_{\lambda}(t) \lVert w^{\mathcal{L}} \rVert_{L^{2}}^{2} + \text{l.s.t.}
\end{equation} 
Now the rest of the proof consists of three key steps, at least in the case of a mild solution (see Theorem \ref{Theorem 2.2}): 
\begin{enumerate}[label=(\alph*)]
\item Renormalization:  To close the estimate, considering \eqref{logarithmically supercritical} we realize that the renormalization constant $r_{\lambda}(t)$ must satisfy 
\begin{equation}\label{needed logarithmic renormalization}
r_{\lambda}(t) \lesssim \ln ( e+ \lVert w^{\mathcal{L}}(t) \rVert_{L^{2}}^{2}). 
\end{equation} 
\item Estimates: To obtain \eqref{needed logarithmic renormalization}, identically to the case of Tao, one can split the Fourier frequency to the lower and higher parts with a well-suited cutoff, similarly to \eqref{frequency decomposition}. Additionally, a commutator type term will arise (recall \eqref{our commutator}).   
\item Anderson Hamiltonian: To close the estimate \eqref{heuristic estimate}, we need not only \eqref{needed logarithmic renormalization} but also 
\begin{equation}\label{est 62}
\left\lvert \int_{\mathbb{T}^{2}} w^{\mathcal{L}} \mathcal{A}_{t}^{\lambda} w^{\mathcal{L}} dx \right\rvert  \lesssim \lVert w^{\mathcal{L}} \rVert_{L^{2}}^{2}. 
\end{equation} 
For this purpose, \cite{HR24} was able to directly rely on such a result from \cite{AC15}, which focused on the 2D parabolic Anderson model with white-in-space noise; informally, a 2D white-in-space noise is of regularity $\mathscr{C}^{-1-\kappa}$ with probability one, which is same as $\nabla_{\text{sym}} X$ in \eqref{HR24 estimate}. 
\end{enumerate} 
In the case of a weak solution (see Theorem \ref{Theorem 2.3}), there is an extra step of proving its uniqueness that presents a difficulty (recall (ii) in Section \ref{Subsection 1.1} and see Remark \ref{Remark 2.3}). 

The author in \cite{Y23c, Y25d} extended \cite{HR24} to the 2D MHD system forced by STWN, for which an analogue of \eqref{Special identity} does not exist, and 1D Burgers' equation forced by $\Lambda^{\frac{1}{2}} \xi$ where $\xi$ is a STWN. The purpose of this manuscript is to extend further to the 3D case with $\Lambda^{\frac{5}{2}}$, specifically \eqref{NS after PL} with $d = 3, m = \frac{5}{2}$, the energy-critical case according to \eqref{Energy-critical}. It is well known that even if both are at energy-critical level, the energy estimates in the 3D case is significantly more difficult than the 2D case due to \eqref{Special identity} that allows even the 2D Euler equations, which is $L^{2}(\mathbb{T}^{d})$-supercritical, to admit a unique global-in-time solution. Additional major difficulties were previewed in Section \ref{Subsection 1.1} (i)-(iv) and will be elaborated in Remarks \ref{Remark 2.1}-\ref{Remark 2.5}. Our estimates (e.g. \eqref{est 116}) clearly indicate that any  further improvement of Theorems \ref{Theorem 2.2}-\ref{Theorem 2.3} will require significantly new ideas. 

\section{Main results and ideas of their proofs }\label{Section 2}
\subsection{Statement of main results}
\hfill\\ In order to treat \eqref{NS after PL} in case $d=3$ and $m = \frac{5}{2}$, analogously to \eqref{HR24 Equation of X}-\eqref{HR24 Equation of v}, we consider 
\begin{equation}\label{Equation of X}
\partial_{t}X + \nu \Lambda^{\frac{5}{2}} X = \mathbb{P}_{L} \mathbb{P}_{\neq 0} \xi, \hspace{3mm} X(0,\cdot) = 0, 
\end{equation} 
so that $v \triangleq u-X$ satisfies 
\begin{equation}\label{Equation of v}
\partial_{t} v + \mathbb{P}_{L} \divergence ( v + X)^{\otimes 2} + \nu \Lambda^{\frac{5}{2}} v = 0, \hspace{3mm} v(0,\cdot) = u^{\text{in}} (\cdot). 
\end{equation} 
We note that we can expect 
\begin{equation}\label{Regularity of X} 
X \in \mathscr{C}^{-\frac{1}{4} - \kappa}(\mathbb{T}^{3})
\end{equation} 
and consequently $v \in \mathscr{C}^{1- 3 \kappa}(\mathbb{T}^{3})$ for all $\kappa > 0$ $\mathbb{P}$-a.s. 

Let us note that our case is much better than the 3D Navier-Stokes equations forced by STWN, as expected because we have $\Lambda^{\frac{5}{2}}$ instead of $-\Delta$ as diffusion. E.g., we can consider in comparison \cite[p. 4444]{ZZ15} where the authors state ``in the three-dimensional case, the trick in the two-dimensional case breaks down since $v$ and $z$ in (1.2) are so singular that not only $z \otimes z$ is not well-defined but also $v \otimes z$ and $v \otimes v$ have no meaning.'' (The ``$z$'' in \cite{ZZ15} plays the role of our $X$ in \eqref{Equation of X}.) In contrast, in our case, $v\otimes X$ and $v \otimes v$ are both well-defined as $X \in \mathscr{C}^{-\frac{1}{4} - \kappa}(\mathbb{T}^{3})$ and $v \in \mathscr{C}^{1- 3 \kappa}(\mathbb{T}^{3})$ $\mathbb{P}$-a.s.

\begin{proposition}\label{Proposition 2.1} 
Recall the $(\Omega, \mathcal{F}, \mathbb{P})$ that we fixed upon introducing \eqref{STWN}-\eqref{Define STWN}. There exists a null set $\mathcal{N} \subset \Omega$ such that for any $\omega \in \Omega \setminus \mathcal{N}$ and $\kappa > 0$, the following holds. For any $u^{\text{in}} \in \mathscr{C}^{-\frac{5}{4} + \kappa}(\mathbb{T}^{3})$ that is divergence-free and mean-zero, there exists $T^{\max} (\omega, u^{\text{in}}) \in (0, \infty]$ and a unique maximal mild solution $v(\omega)$ to \eqref{Equation of v} on $[0, T^{\max} (\omega, u^{\text{in}}))$ such that $v(\omega, 0, x) = u^{\text{in}}(x)$. 
\end{proposition}  

Our first main result is stated as follows. 
\begin{theorem}\label{Theorem 2.2}
Recall the $(\Omega, \mathcal{F}, \mathbb{P})$ that we fixed upon introducing \eqref{STWN}-\eqref{Define STWN}. There exists a null set $\mathcal{N}' \subset \Omega$ such that the following holds. For any $\kappa > 0$, $u^{\text{in}} \in \mathscr{C}^{-1+ \kappa}(\mathbb{T}^{3})$ that is divergence-free and mean-zero, $T^{\max} (\omega, u^{\text{in}})$ from Proposition \ref{Proposition 2.1} satisfies $T^{\max} (\omega, u^{\text{in}}) = \infty$ for all $\omega \in \Omega \setminus \mathcal{N}'$.  
\end{theorem}

Next, we define $L_{\sigma}^{2}\triangleq \{ f \in L^{2}(\mathbb{T}^{3}): \hspace{1mm} \fint_{\mathbb{T}^{3}} f dx= 0, \nabla\cdot f = 0 \}$. Let us postpone the precise definition of the high-low (HL) weak solution of \eqref{Equation of v} to Definition \ref{Definition 5.1} and state our second main result. 

\begin{theorem}\label{Theorem 2.3}
Recall the $(\Omega, \mathcal{F}, \mathbb{P})$ that we fixed upon introducing \eqref{STWN}-\eqref{Define STWN}. Consider the same null set $\mathcal{N}' \subset \Omega$ from Theorem \ref{Theorem 2.2}. For every $\omega \in \Omega \setminus \mathcal{N}'$, every $\kappa > 0$, and $u^{\text{in}} \in L_{\sigma}^{2}$, there exists a unique HL weak solution $v$ to \eqref{Equation of v} on $[0,\infty)$. 
\end{theorem}

\subsection{Difficulties and new ideas}
\hfill\\ We elaborate on the difficulties of the proofs of Theorems \ref{Theorem 2.2}-\ref{Theorem 2.3} in the following Remarks \ref{Remark 2.1}-\ref{Remark 2.5}. 
\begin{remark}[Extending local solution theory]\label{Remark 2.1} 
Our Proposition \ref{Proposition 2.1} only required $\mathscr{C}^{-\frac{5}{4} + \kappa}(\mathbb{T}^{3})$; however, we restricted ourselves in Theorem \ref{Theorem 2.2} to initial data in $\mathscr{C}^{-1+\kappa}(\mathbb{T}^{3})$ in consideration of Proposition \ref{Proposition 4.2}. This has the following consequence. Informally stated, to extend our solution theory from local to global in time, we will need the embedding of $H^{\epsilon}(\mathbb{T}^{3}) \hookrightarrow \mathscr{C}^{-1+ 2\kappa}(\mathbb{T}^{3})$, which then requires $\epsilon > \frac{1}{2} + 2\kappa$ according to Sobolev embedding (see \eqref{needed embedding}). In contrast, the analogous \cite[Lemma 5.3]{HR24} in the 2D case was able to achieve such $\dot{H}^{\epsilon}(\mathbb{T}^{2})$-estimate only for ``$\epsilon \in (0,\kappa)$'' for $\kappa > 0$ arbitrarily small and this sufficed to extend their local solution to become global-in-time. The difference between such small $\epsilon$ in \cite{HR24} and our $\epsilon > \frac{1}{2}$ is clearly due to the difference of spatial dimensions. In fact, in Proposition \ref{Proposition 4.11} we will accomplish $\dot{H}^{\epsilon}(\mathbb{T}^{3})$-estimate for all $\epsilon < \frac{9}{10}$, which may not even be sharp because some dedicated computations suggest potential improvement up to $\epsilon < 1$. A close inspection of the proof of \cite[Lemma 5.3]{HR24} reveals that its restriction of ``$\epsilon \in (0,\kappa)$'' comes from the low regularity of $Y$ in the case of \cite{HR24}, which was $\mathscr{C}^{2\kappa}(\mathbb{T}^{2})$ (see ``$Y \in \mathscr{C}^{2\kappa}$'' on \cite[p. 8]{HR24}). Because extending the approach of \cite{DD02} that relies on an explicit knowledge of invariant measure seems inapplicable for the 3D Navier-Stokes equations anyway, we chose to not include the ``$\zeta$'' in \eqref{HR24 main equation} which allows our $Y$ that solves \eqref{Equation of Y} to be more regular, specifically $\mathscr{C}^{1-3\kappa}(\mathbb{T}^{3})$ (see \eqref{Regularity of Y}). This is the primary heuristic why improvement up to $\epsilon < 1$ seems hopeful; although multiple technical difficulties remain, which will be discussed in Remarks \ref{Remark 2.2} and \ref{Remark 4.3}. 
\end{remark}

\begin{remark}[Commutator estimate]\label{Remark 2.2}
As previewed in Section \ref{Subsection 1.1}, upon $L^{2}(\mathbb{T}^{3})$- and $\dot{H}^{\epsilon} (\mathbb{T}^{3})$-estimates, we will need to handle the non-trivial commutator in \eqref{our commutator} that does not cancel out via Leibniz rule, in sharp contrast to \eqref{HR24 commutator} (see \eqref{Define C4} and \eqref{Define tilde C4}). 
\begin{enumerate}[label=(\alph*)]
\item First idea: It turns out that an estimate for a commutator of type \eqref{our commutator} exists in the literature, but they are not applicable to our case. E.g., \cite[Remark 1.3]{L19b} by Li has an estimate of 
\begin{equation}\label{Li and D'Ancona}
\lVert \Lambda^{s}(fg) - f \Lambda^{s} g - g \Lambda^{s} f \rVert_{L^{p}} \lesssim_{s, s_{1}, s_{2}, p_{1}, p_{2}} \lVert \Lambda^{s_{1}} f \rVert_{L^{p_{1}}}  \lVert \Lambda^{s_{2}} g \rVert_{L^{p_{2}}}
\end{equation}
for all $p, p_{1}, p_{2} \in (1,\infty)$ such that $\frac{1}{p} = \frac{1}{p_{1}} + \frac{1}{p_{2}}, s = s_{1} + s_{2}$  but only for $s, s_{1}, s_{2} \in [0,1)$. Its extended version can be found in \cite[Equation (1.5)]{A19} by D'Ancona where $s \in (0,2)$ is allowed as long as $\frac{2d}{d+ 2s_{j}} < p_{j} < \infty$; of course, the upper endpoint case $s= 2$ is the trivial situation that we described in \eqref{HR24 commutator}. Anyway, this does not apply to our case where we would need $s = \frac{5}{2}$. 
\item Second idea: Another idea may be to make use of our strong diffusion $\Lambda^{\frac{5}{2}}u$. E.g., in the case of $L^{2}(\mathbb{T}^{3})$-estimate, we will face
\begin{align*}
\int_{\mathbb{T}^{3}} w^{\mathcal{L}} \cdot \divergence \left( \Lambda^{\frac{5}{2}} (w \circlesign{\prec}_{s} Q^{\mathcal{H}} ) - (\Lambda^{\frac{5}{2}} w) \circlesign{\prec}_{s} Q^{\mathcal{H}} - w \circlesign{\prec}_{s} \Lambda^{\frac{5}{2}} Q^{\mathcal{H}} \right)dx  
\end{align*}
in \eqref{Define C4} that if we estimate by  
\begin{align*}
C \lVert w^{\mathcal{L}} \rVert_{\dot{H}^{1}} \lVert  \Lambda^{\frac{5}{2}} (w \circlesign{\prec}_{s} Q^{\mathcal{H}} ) - (\Lambda^{\frac{5}{2}} w) \circlesign{\prec}_{s} Q^{\mathcal{H}} - w \circlesign{\prec}_{s} \Lambda^{\frac{5}{2}} Q^{\mathcal{H}}  \rVert_{L^{2}},
\end{align*}
then we have no hope to apply \eqref{Li and D'Ancona} or \cite[Equation (1.5)]{A19}. However, we can make use of diffusion and try instead 
\begin{align*}
C \lVert w^{\mathcal{L}} \rVert_{\dot{H}^{\frac{5}{4}}} \lVert  \Lambda^{\frac{5}{2}} (w \circlesign{\prec}_{s} Q^{\mathcal{H}} ) - (\Lambda^{\frac{5}{2}} w) \circlesign{\prec}_{s} Q^{\mathcal{H}} - w \circlesign{\prec}_{s} \Lambda^{\frac{5}{2}} Q^{\mathcal{H}}  \rVert_{\dot{H}^{-\frac{1}{4}}},
\end{align*}
in hope to be able to apply \cite[Equation (1.5)]{A19}. But this is highly unlikely because $\frac{5}{2} - \frac{1}{4} = \frac{9}{4} > 2$. Even if we could do so in the $L^{2}(\mathbb{T}^{3})$-estimate, this strategy is doomed for the necessary $\dot{H}^{\epsilon}(\mathbb{T}^{3})$-estimate, for which we need $\epsilon > \frac{1}{2}$ as described in Remark \ref{Remark 2.1} (see \eqref{Define tilde C4}). In the $\dot{H}^{\epsilon}(\mathbb{T}^{3})$-estimate, in order to not face an $\dot{H}^{s}(\mathbb{T}^{3})$-estimate of a commutator for $s \neq 0$, we have no choice but to bound 
\begin{align*}
&\int_{\mathbb{T}^{3}} \Lambda^{2\epsilon} w \cdot \divergence \left( \Lambda^{\frac{5}{2}} (w \circlesign{\prec}_{s} Q^{\mathcal{H}} ) - (\Lambda^{\frac{5}{2}} w) \circlesign{\prec}_{s} Q^{\mathcal{H}} - w \circlesign{\prec}_{s} \Lambda^{\frac{5}{2}} Q^{\mathcal{H}} \right)dx \\
\lesssim& \lVert w \rVert_{\dot{H}^{1+ 2\epsilon}} \lVert  \Lambda^{\frac{5}{2}} (w \circlesign{\prec}_{s} Q^{\mathcal{H}} ) - (\Lambda^{\frac{5}{2}} w) \circlesign{\prec}_{s} Q^{\mathcal{H}} - w \circlesign{\prec}_{s} \Lambda^{\frac{5}{2}} Q^{\mathcal{H}}   \rVert_{L^{2}}. 
\end{align*}
But considering that the diffusive term in the $\dot{H}^{\epsilon}(\mathbb{T}^{3})$-estimate gives us only $\nu \lVert w \rVert_{\dot{H}^{\frac{5}{4} + \epsilon}}^{2}$, this then requires $1+ 2 \epsilon \leq \frac{5}{4} + \epsilon$ and equivalently $\epsilon \leq \frac{1}{4}$, which is not good enough for us because we need $\epsilon > \frac{1}{2}$ according to Remark \ref{Remark 2.1}.  
\item Third idea: The following due to \cite{W25} may be of independent interest in harmonic analysis and PDEs in general. First, let us estimate 
\begin{align*}
&\int_{\mathbb{T}^{3}} \Lambda^{2\epsilon} w \cdot \divergence \left( \Lambda^{\frac{5}{2}} (w \circlesign{\prec}_{s} Q^{\mathcal{H}} ) - (\Lambda^{\frac{5}{2}} w) \circlesign{\prec}_{s} Q^{\mathcal{H}} - w \circlesign{\prec}_{s} \Lambda^{\frac{5}{2}} Q^{\mathcal{H}} \right)dx \\
\lesssim& \lVert w \rVert_{\dot{H}^{\frac{5}{4}+ \epsilon}} \lVert  \Lambda^{\frac{5}{2}} (w \circlesign{\prec}_{s} Q^{\mathcal{H}} ) - (\Lambda^{\frac{5}{2}} w) \circlesign{\prec}_{s} Q^{\mathcal{H}} - w \circlesign{\prec}_{s} \Lambda^{\frac{5}{2}} Q^{\mathcal{H}}   \rVert_{\dot{H}^{\epsilon - \frac{1}{4}}}. 
\end{align*}
Then, denoting $\alpha \triangleq \epsilon -\frac{1}{4}$ for simplicity, we see that we need to estimate for $\alpha > \frac{1}{4}$, 
\begin{align*}
\left\lVert \Lambda^{\alpha}\left( \Lambda^{\frac{5}{2}} (fg) - (\Lambda^{\frac{5}{2}} f) g - f( \Lambda^{\frac{5}{2}} g) \right)  \right\rVert_{L^{2}}
\end{align*}
involving $\lVert g \rVert_{\mathscr{C}^{\beta}}$ for $\beta < \frac{9}{4}$ (see \eqref{Regularity of Q}). We can rewrite 
\begin{align*}
 \Lambda^{\alpha} \left( \Lambda^{\frac{5}{2}} ( f g) - (\Lambda^{\frac{5}{2}} f) g - f (\Lambda^{\frac{5}{2}} g) \right) = \sum_{i=1}^{3} J_{i}
\end{align*}
where 
\begin{align*}
J_{1} \triangleq  \Lambda^{\alpha + \frac{5}{2}} (f g) - f( \Lambda^{\alpha + \frac{5}{2}} g ), \hspace{2mm} J_{2} \triangleq  f ( \Lambda^{\alpha+ \frac{5}{2}} g) - \Lambda^{\alpha} \left( f ( \Lambda^{\frac{5}{2}} g) \right),  \hspace{2mm} J_{3} \triangleq  \Lambda^{\alpha} \left( ( \Lambda^{\frac{5}{2}} f) g \right). 
\end{align*}
We can estimate $\lVert J_{1} \rVert_{L^{2}}$ by the classical commutator estimate stated in Lemma \ref{Lemma 3.3} and $\lVert J_{3} \rVert_{L^{2}}$ by the classical product estimate (e.g. \cite[Lemma A.2]{K90}) so that $g$ gets no norm worse than that of $\mathscr{C}^{\alpha + \frac{3}{2}}(\mathbb{T}^{3})$. For the difficult $J_{2}$ term, we use the fact that $\Lambda^{2} = -\sum_{j=1}^{3} \partial_{j}^{2}$ and rewrite it as
\begin{align*}
J_{2} = \sum_{i=1}^{2} J_{2i} 
\end{align*} 
where 
\begin{align*}
J_{21} \triangleq - \sum_{j=1}^{3} \left[f ( \Lambda^{\alpha + \frac{1}{2}} \partial_{j}^{2} g) - \Lambda^{\alpha} \partial_{j} \left( f ( \Lambda^{\frac{1}{2}} \partial_{j} g ) \right) \right], \hspace{3mm} J_{22}\triangleq - \sum_{j=1}^{3}\Lambda^{\alpha} \left(( \partial_{j} f) (\Lambda^{\frac{1}{2}} \partial_{j} g ) \right).  
\end{align*}
Again, the classical product estimate (e.g. \cite[Lemma A.2]{K90}) allows one to estimate $J_{22}$ so that at worst $\mathscr{C}^{\frac{3}{2} + \alpha}$-norm is placed on $g$. This boils everything down to how we can estimate $J_{21}$, which resembles the structure of Lemma \ref{Lemma 3.3} and hence incites high hope. The best way to estimate such a combination of $\Lambda^{\alpha}$ and $\partial_{j}$ seems to be through an application of the Plancherel theorem, and writing the Fourier symbols as 
\begin{align*}
\int_{0}^{1} \frac{\partial}{\partial \tau} \left( \lvert \tau \xi + (1-\tau) ( \xi - \eta ) \rvert^{\alpha} [ \tau \xi + (1-\tau) (\xi - \eta)_{j} ]_{j} \right) d \tau \lesssim \lvert \eta \rvert \lvert \xi - \eta \rvert^{\alpha} + \lvert \eta \rvert^{1+ \alpha}. 
\end{align*}
However, this seems to lead to a bound of $\lVert \widehat{\Lambda^{\alpha} g} \rVert_{L^{1}}$ and ultimately 
\begin{align*}
\lVert \Lambda^{\alpha} \partial_{j} \left( f ( \Lambda^{\frac{1}{2}} \partial_{j} g ) \right) -  f ( \Lambda^{\alpha + \frac{1}{2}} \partial_{j}^{2} g) \rVert_{L^{2}}  \lesssim \lVert \nabla f \rVert_{L^{2}}  \lVert g \rVert_{H^{3 + \delta + \alpha}}+ \lVert \Lambda^{1+ \alpha} f  \rVert_{L^{2}}\lVert g \rVert_{H^{3 + \delta}}
\end{align*}
for any $\delta > 0$ (see also \cite[Proposition 2.1]{CCCGW12}). This is not good enough for us because $3+ \delta + \alpha$ for $\alpha > \frac{1}{4}$ is much bigger than $\frac{9}{4}$. 
\end{enumerate} 
We will overcome these difficulties in \eqref{Split C4}-\eqref{Estimate on C4}, \eqref{est 126}-\eqref{est 131}, and \eqref{est 273}-\eqref{est 276}. 
\end{remark}

\begin{remark}[Criticality in the proof of uniqueness of HL weak solution]\label{Remark 2.3}
To explain this difficulty, let us first describe the analogue of the proof of \cite[Theorem 2.3]{Y23c} concerning the uniqueness of HL weak solutions in the case of the 2D Navier-Stokes equations forced by STWN; to emphasize the difference in dimensions, we denote $\mathbb{T}^{2}$  and $\mathbb{T}^{3}$ in each norm appropriately. Denoting the difference between two solutions $w = w^{\mathcal{L},\lambda} + w^{\mathcal{H},\lambda}$ and $\bar{w} = \bar{w}^{\mathcal{L},\lambda} + \bar{w}^{\mathcal{H},\lambda}$ and their components by $z, z^{\mathcal{L},\lambda}$ and $z^{\mathcal{H},\lambda}$ (see \eqref{est 149}), we are faced with the estimate of 
\begin{align*}
\frac{1}{2} \partial_{t} \lVert z^{\mathcal{L},\lambda} \rVert_{L^{2} (\mathbb{T}^{2}) }^{2} + \nu \lVert z^{\mathcal{L},\lambda} \rVert_{\dot{H}^{1}(\mathbb{T}^{2}) }^{2} = - \langle z^{\mathcal{L},\lambda}, \divergence (z^{\mathcal{H},\lambda} \otimes w^{\mathcal{L},\lambda} ) \rangle + \hdots 
\end{align*}
(see \cite[Equations (170) and (183)]{Y23c}). We can bound this term by 
\begin{subequations} 
\begin{align}
- \langle z^{\mathcal{L},\lambda},  \divergence (z^{\mathcal{H},\lambda}& \otimes w^{\mathcal{L},\lambda} ) \rangle \lesssim \lVert z^{\mathcal{L},\lambda} \rVert_{\dot{H}^{1}(\mathbb{T}^{2}) }\lVert z^{\mathcal{H},\lambda} \rVert_{L^{4}(\mathbb{T}^{2}) } \lVert w^{\mathcal{L},\lambda} \rVert_{L^{4}(\mathbb{T}^{2}) }  \label{est 58a}\\
\lesssim& \lVert z^{\mathcal{L},\lambda} \rVert_{\dot{H}^{1}(\mathbb{T}^{2}) } \lVert z^{\mathcal{H},\lambda} \rVert_{L^{2}(\mathbb{T}^{2}) }^{\frac{1}{2}} \lVert z^{\mathcal{H},\lambda} \rVert_{B_{\infty,2}^{0}(\mathbb{T}^{2}) }^{\frac{1}{2}} \lVert w^{\mathcal{L},\lambda} \rVert_{L^{2}(\mathbb{T}^{2}) }^{\frac{1}{2}} \lVert w^{\mathcal{L},\lambda} \rVert_{\dot{H}^{1}(\mathbb{T}^{2}) }^{\frac{1}{2}} \label{est 58b} 
\end{align}
\end{subequations}
(see \cite[Equation (187)]{Y23c}). Then \cite[Equation (188)]{Y23c} shows that there exist universal constants $C_{1}, C_{2} \geq 0$ such that for any $\kappa \in (0, \frac{1}{2})$ sufficiently small, 
\begin{equation}\label{est 59}
\lVert z \rVert_{B_{\infty,2}^{0}(\mathbb{T}^{2}) } \leq C_{1} \lVert z^{\mathcal{L},\lambda} \rVert_{\dot{H}^{1}(\mathbb{T}^{2}) } + C_{2} \lambda^{-\frac{\kappa}{2}} \lVert z \rVert_{B_{\infty,2}^{0}(\mathbb{T}^{2}) }. 
\end{equation} 
Thus, for $\lambda \gg 1$, we have $\lVert z^{\mathcal{H},\lambda} \rVert_{B_{\infty,2}^{0}(\mathbb{T}^{2}) } \lesssim  \lVert z^{\mathcal{L},\lambda} \rVert_{\dot{H}^{1}(\mathbb{T}^{2}) }$. By using another estimate of $\lVert z \rVert_{H^{s}(\mathbb{T}^{2}) } \lesssim \lVert z^{\mathcal{L},\lambda} \rVert_{H^{s}(\mathbb{T}^{2}) }$ for all $s \in [0, 1 - 2 \kappa)$ (see \eqref{est 156}), we were able to deduce in \cite[Equations (192)-(193)]{Y23c} 
\begin{align*}
- \langle z^{\mathcal{L},\lambda}, \mathbb{P}_{L} \divergence (z^{\mathcal{H},\lambda} \otimes w^{\mathcal{L},\lambda} ) \rangle \leq \frac{\nu}{2} \lVert z^{\mathcal{L},\lambda} \rVert_{\dot{H}^{1}(\mathbb{T}^{2}) }^{2} + C \lVert z^{\mathcal{L},\lambda} \rVert_{L^{2}(\mathbb{T}^{2}) }^{2} \lVert w^{\mathcal{L},\lambda} \rVert_{L^{2}(\mathbb{T}^{2}) }^{2} \lVert w^{\mathcal{L},\lambda}\rVert_{\dot{H}^{1}(\mathbb{T}^{2}) }^{2}.
\end{align*}
In \eqref{est 58a}, we assigned to $z^{\mathcal{L},\lambda}$ the $\dot{H}^{1}(\mathbb{T}^{2})$-norm, the regularity worth full diffusion; thus, we can similarly estimate 
\begin{align*}
\frac{1}{2} \partial_{t} \lVert z^{\mathcal{L},\lambda} \rVert_{L^{2} (\mathbb{T}^{3}) }^{2} + \nu \lVert z^{\mathcal{L},\lambda} \rVert_{\dot{H}^{\frac{5}{4}}(\mathbb{T}^{3}) }^{2} = - \langle z^{\mathcal{L},\lambda}, \divergence (z^{\mathcal{H},\lambda} \otimes w^{\mathcal{L},\lambda} ) \rangle + \hdots 
\end{align*}
where 
\begin{align}
&- \langle z^{\mathcal{L},\lambda},   \divergence (z^{\mathcal{H},\lambda} \otimes w^{\mathcal{L},\lambda}) \rangle  \lesssim \lVert z^{\mathcal{L},\lambda} \rVert_{\dot{H}^{\frac{5}{4}} (\mathbb{T}^{3})  } \lVert z^{\mathcal{H},\lambda} \otimes w^{\mathcal{L},\lambda} \rVert_{\dot{H}^{-\frac{1}{4}}(\mathbb{T}^{3})}  \nonumber \\
\lesssim& \lVert z^{\mathcal{L},\lambda} \rVert_{\dot{H}^{\frac{5}{4}}(\mathbb{T}^{3})} \lVert z^{\mathcal{H},\lambda} \otimes w^{\mathcal{L},\lambda} \rVert_{L^{\frac{12}{7}}(\mathbb{T}^{3})}  \lesssim \lVert z^{\mathcal{L},\lambda} \rVert_{\dot{H}^{\frac{5}{4}}(\mathbb{T}^{3})} \lVert z^{\mathcal{H},\lambda} \rVert_{L^{12}}  \lVert w^{\mathcal{L},\lambda} \rVert_{L^{2}(\mathbb{T}^{3})}.  \label{est 60}
\end{align}  
It turns out that in the 3D case, $L^{12}(\mathbb{T}^{3})$-norm somewhat plays the same role of $B_{\infty,2}^{0}(\mathbb{T}^{2})$-norm in \eqref{est 59}; i.e., there exist constants $C_{1}, C_{2} \geq 0$ such that 
\begin{equation}
\lVert z \rVert_{L^{12}(\mathbb{T}^{3})} \leq C_{1} \lVert z^{\mathcal{L},\lambda} \rVert_{\dot{H}^{\frac{5}{4}}(\mathbb{T}^{3})} + C_{2} \lambda^{-\frac{\kappa}{2}} \ \lVert z \rVert_{L^{12}(\mathbb{T}^{3})}
\end{equation} 
so that $\lVert z^{\mathcal{H},\lambda} \rVert_{L^{12}(\mathbb{T}^{3})} \lesssim \lVert z^{\mathcal{L},\lambda} \rVert_{\dot{H}^{\frac{5}{4}}(\mathbb{T}^{3})}$ (see \eqref{est 166}-\eqref{est 168}). But if we continue from \eqref{est 60} by such estimates to reach  
\begin{align*}
- \langle z^{\mathcal{L},\lambda}, \mathbb{P}_{L} \divergence (z^{\mathcal{H},\lambda} \otimes w^{\mathcal{L},\lambda}) \rangle  \lesssim \lVert z^{\mathcal{L},\lambda} \rVert_{\dot{H}^{\frac{5}{4}}(\mathbb{T}^{3})}^{2}  \lVert   w^{\mathcal{L},\lambda} \rVert_{L^{2}(\mathbb{T}^{3})}, 
\end{align*}
then this will not allow us to prove uniqueness. Hence, we need to employ estimates differently from \cite{HR24, Y23c, Y25d}. We refer to Remark \ref{Remark 5.3} for more details.
\end{remark}

\begin{remark}[Anderson Hamiltonian and the theory of RS]\label{Remark 2.4}
According to step (c) described in Section \ref{Subsection 1.4}, we need help from the theory of Anderson Hamiltonian to obtain \eqref{est 62}. The previous works of \cite{HR24, Y23c} relied on \cite{AC15} that utilized the theory of paracontrolled distributions to obtain the necessary results in the 2D case with $-\Delta$ and white-in-space noise, and its proof is flexible so that the author \cite{Y25d} was able to extend it to the case of 1D Burgers' equation with $-\Delta$ and $\Lambda^{\frac{1}{2}} \xi$ with $\xi$ being the STWN. The 3D case has been covered in \cite{GUZ20} by Gubinelli, Ugurcan, and Zachhuber using the theory of paracontrolled distributions too. The difference is that while \cite{HR24, Y23c, Y25d} all had `$\nabla_{\text{sym}} \mathcal{L}_{\lambda_{t}} X w^{\mathcal{L}}$ as the ill-defined product of $\mathscr{C}^{-1-\kappa}(\mathbb{T}^{2})$ and $\mathscr{C}^{1-\kappa}(\mathbb{T}^{2})$ objects, we will face the ill-defined product of $\mathscr{C}^{-\frac{5}{4} -\kappa}(\mathbb{T}^{3})$ and $\mathscr{C}^{\frac{5}{4} - \kappa}(\mathbb{T}^{3})$ objects (see $\nabla_{\text{sym}} \mathcal{L}_{\lambda_{t}}X$ multiplied by $w^{\mathcal{L}}$ in \eqref{est 179}). The crux of the application of the theory of paracontrolled distributions in \cite{AC15} relies on the following result (see Lemma \ref{Burgers' Lemma A.3}), which is the Sobolev space analogue of \cite[Lemma 2.4]{GIP15}: If $\alpha \in (0,1), \beta, \gamma \in \mathbb{R}$ satisfy $\beta + \gamma < 0$, and $\alpha + \beta + \gamma > 0$, then $\mathcal{R}(f,g,h) \triangleq (f \circ g) \circ h - f (g \circ h)$ satisfies 
\begin{equation}\label{est 180}
\lVert \mathcal{R}(f,g,h) \rVert_{H^{\alpha + \beta + \gamma - \delta}} \lesssim \lVert f \rVert_{H^{\alpha}} \lVert g \rVert_{\mathscr{C}^{\beta}} \lVert h \rVert_{\mathscr{C}^{\gamma}} \hspace{3mm} \forall \hspace{1mm} f \in H^{\alpha}, g \in \mathscr{C}^{\beta}, \text{ and }h \in \mathscr{C}^{\gamma} \text{ and } \delta > 0. 
\end{equation}
The difficulty is that we will want $w^{\mathcal{L}}$ to play the role of ``$f$'' in \eqref{est 180} and have the regularity of $H^{\frac{5}{4} - \kappa}$; however, it violates the hypothesis of $\alpha \in (0,1)$.   

Fortunately, Labb$\acute{e}$ \cite{L19a} extended \cite{AC15} to the 3D case using the theory of RS. The theory of RS has seen generalization to fractional Laplacian in the past (e.g. \cite{BK17}) and diffusion of high order (e.g. $\Delta^{2}$ the Cahn-Hilliard equation in \cite[Equation (4b)]{H14a}); hence, this approach seems very hopeful. Unfortunately, Labb$\acute{\mathrm{e}}$ in \cite[p. 3204]{L19a} considered $P^{(a)}$ as the Green's function that solves $(-\Delta+ a) P^{(a)} = \delta$. In the 2D case there is no explicit form better than $P^{(a)}(x) = \frac{1}{2\pi} \text{Bess} (\sqrt{a} \lvert x \rvert)$ where Bess is the modified Bessel function of the second kind of index , but the following asymptotic behavior are known
\begin{equation}\label{Asymptotic of Bessel} 
\begin{cases}
D^{k} \text{Bess}(x) \sim (-1)^{k} \sqrt{\frac{\pi}{2x}} e^{-x} &\text{ as } \lvert x \rvert \to \infty,  \\
\text{Bess} (x) \sim \ln(x), \hspace{1mm} D^{k} \text{Bess}(x) \sim \frac{ (-1)^{k} (k-1)!}{x^{k}} &\text{ as } \lvert x \rvert \to 0,
\end{cases}
\end{equation}
and they played crucial role in the proof of \cite[Lemma 3.1]{L19a}. If we were to follow the analogous path, we would need to consider $P^{(a)}$ that solves for $x \in \mathbb{T}^{3}$, 
\begin{align*}
(\Lambda^{\frac{5}{2}} + a)Q^{(a)} = \delta; 
\end{align*}
however, at the time of writing this manuscript, the author could not determine if an asymptotic behavior of $Q^{(a)}$ that is analogous to \eqref{Asymptotic of Bessel} is known. 

Ultimately, we returned to the approach of \cite{AC15}. We had to decrease $\alpha$ by informally $\frac{1}{4}$ to obey the upper bound of one upon an application of \eqref{est 180}, which would decrease $\delta$ by $\frac{1}{4}$. This then creates the danger of violating the other condition of $\delta > 0$ in \eqref{est 180}; fortunately, we were able to find the appropriate range of parameters to make this work (see e.g. \eqref{Burgers' est 167}). We will see that the application of Lemma \ref{Burgers' Lemma A.2} is even more difficult. We will also face the same challenges of commutators \eqref{our commutator} (see \eqref{Burgers' Define B}). Again, we overcome this issue using the same approach in Remark \ref{Remark 2.2} (see \eqref{est 273}, \eqref{est 277}, and \eqref{est 276}).  
\end{remark}

\begin{remark}[Renormalization]\label{Remark 2.5}
As we described in \eqref{projection}-\eqref{alternative projection}, the 3D Leray projection has no simple form, and one of its major consequences is that the computations involving the renormalization constants increase significantly in comparison to the 2D case, often in an unexpected manner (see Remark \ref{Remark on renormalization constant}). E.g., in \eqref{est 0} we compute $\mathbb{E} [ \lvert \Delta_{m} (\nabla_{\text{sym}} \mathcal{L}_{\lambda} X \circlesign{\circ} P^{\lambda})_{(1,1)} (t) \rvert^{2} ]$ that informally consists of nine terms of the form $k_{i}k_{j}, i, j \in \{1,2,3\}$; we group $k_{2}^{2}$ and $k_{3}^{2}$ together with $k_{1}^{2}$ so that nine reduces to seven (see e.g. ``$4k_{1}k_{1}' + k_{2}k_{2}' + k_{3}k_{3}'$'' in \eqref{est 61}). Upon computing the second moment, it then produces 49 terms in \eqref{est 0}. The projection $\mathbb{P}_{L}$ in \eqref{projection} informally consists of $\Id - \frac{k\otimes k}{\lvert k \rvert^{2}}$ so that the product with itself produces four terms. Postponing the full descriptions of notations, further considering the product of $\nabla_{\text{sym}} \mathcal{L}_{\lambda} X$ with $P^{\lambda}$, each of these 49 terms produce $4\times 4 = 16$ terms instead of four (see \eqref{est 61}-\eqref{est 217}). At last, upon computing Wick products, each of the $49 \times 16$ terms creates three terms due to \eqref{Wick}. All of these considerations amount to exhaustive computations; we refer to Section \ref{Subsection 4.3} for more details. 
 \end{remark}

In what follows, we give preliminaries in Section \ref{Section 3}, prove Theorem \ref{Theorem 2.2} in Section \ref{Section 4}, Theorem \ref{Theorem 2.3} in Section \ref{Section 5}, and Proposition \ref{Proposition on Anderson Hamiltonian} in Section \ref{Proof of Proposition on Anderson Hamiltonian}. We leave additional details in the Appendix. 

\section{Preliminaries}\label{Section 3} 
\subsection{Further notations and assumptions}\label{Subsection 3.1}
\hfill\\ We write $\mathbb{M}^{n}$ to denote the set of all $n\times n$ matrices, and $C_{t}$ to indicate $\sup_{s\in [0,t]}$; e.g. $\lVert f \rVert_{C_{t}C_{x}} \triangleq \sup_{s\in[0,t]} \sup_{x\in \mathbb{T}^{3}} \lvert f(t,x) \rvert$. The heat kernel in our case will be denoted by $P_{t} \triangleq e^{ -\nu \Lambda^{\frac{5}{2}} t}$. Following \cite{HR24}, we define the symmetric tensor product as 
\begin{equation}\label{Define symmetric tensor}
u\otimes_{s} v \triangleq \frac{1}{2} ( u \otimes v + v \otimes u), 
\end{equation} 
so that $u\otimes_{s} v = v \otimes_{s} u$, and $(\nabla \phi)_{i,j} \triangleq \partial_{i} \phi_{j}$ for any $\phi \in C^{1}(\mathbb{T}^{3}; \mathbb{R}^{3})$, $i, j \in \{1,2, 3\}$, and 
\begin{equation}\label{symmetric gradient} 
(\nabla_{\text{sym}} \phi)_{i,j} = \frac{1}{2} (\partial_{i} \phi_{j} + \partial_{j} \phi_{i}). 
\end{equation} 

\subsection{Besov spaces and Bony's paraproducts}\label{Subsection 3.2}  
\hfill\\ We let $\chi$ and $\rho$ be smooth functions with compact support on $\mathbb{R}^{3}$ that are non-negative, and radial such that the support of $\chi$ is contained in a ball while that of $\rho$ in an annulus and 
\begin{align*}
& \chi(\xi) + \sum_{j\geq 0} \rho(2^{-j} \xi) = 1 \hspace{3mm} \forall \hspace{1mm} \xi \in \mathbb{R}^{3}, \\
&\supp (\chi) \cap  \supp (\rho(2^{-j} \cdot )) = \emptyset \hspace{1mm} \forall \hspace{1mm} j \in\mathbb{N}, \hspace{2mm} \supp (\rho(2^{-i}\cdot)) \cap \supp (\rho (2^{-j} \cdot)) = \emptyset \hspace{1mm} \text{ if } \lvert i-j \rvert > 1.
\end{align*}
Denoting by $\rho_{j}(\cdot) \triangleq \rho(2^{-j} \cdot)$, we define the Littlewood-Paley operators $\Delta_{j}$ for $j \in \mathbb{N}_{0} \cup \{-1\}$ by 
\begin{equation}
\Delta_{j} f \triangleq 
\begin{cases}
\mathcal{F}^{-1} (\chi) \ast f & \text{ if } j = -1, \\
\mathcal{F}^{-1} (\rho_{j}) \ast f & \text{ if } j \in \mathbb{N}_{0},
\end{cases}
\end{equation}
and inhomogeneous Besov spaces $B_{p,q}^{s} \triangleq \{f \in \mathcal{S}': \hspace{1mm} \lVert f \rVert_{B_{p,q}^{s}} < \infty \}$ where 
\begin{equation}
\lVert f \rVert_{B_{p,q}^{s}} \triangleq  \lVert 2^{sm} \lVert \Delta_{m} f \rVert_{L_{x}^{p}} \rVert_{l_{m}^{q}} \hspace{3mm} \forall \hspace{1mm} p, q \in [1,\infty], s \in \mathbb{R}. 
\end{equation} 
We define the low-frequency cut-off operator $S_{i} f \triangleq \sum_{-1  \leq j \leq i-1} \Delta_{j} f$ and Bony's paraproducts and remainder respectively as 
\begin{equation}  
f \prec g  \triangleq \sum_{i\geq -1} S_{i-1} f \Delta_{i} g   \text{ and }  f  \circ g  \triangleq \sum_{i\geq -1} \sum_{j: \lvert j\rvert \leq 1} \Delta_{i} f  \Delta_{i+j} g  
\end{equation}
so that $f g  = f  \prec g  + f  \succ g  + f  \circ g$, where $f  \succ g  = g  \prec f $ (see \cite[Sections 2.6.1 and 2.8.1]{BCD11}). We first extend such definitions to tensor products by 
\begin{subequations} 
\begin{align}
&f  \circlesign{\prec} g \triangleq \sum_{i \geq -1} S_{i-1} f  \otimes \Delta_{i} g, \\
&f  \circlesign{\succ} g \triangleq \sum_{i \geq -1} \Delta_{i} f  \otimes S_{i-1} g,\\
& f  \circlesign{\circ} g \triangleq \sum_{i \geq -1} \sum_{j: \lvert j \rvert \leq 1}\Delta_{i} f  \otimes \Delta_{i+j} g, 
\end{align}
\end{subequations} 
so that $f  \otimes g = f  \circlesign{\prec} g + f  \circlesign{\succ} g + f  \circlesign{\circ} g$. Following \cite{HR24}, we extend furthermore via 
\begin{equation}\label{est 92} 
f  \circlesign{\prec}_{s} g \triangleq \sum_{i \geq -1} S_{i-1} f  \otimes_{s} \Delta_{i} g \text{ and } f  \circlesign{\circ}_{s} g \triangleq \sum_{i \geq -1} \sum_{j: \lvert j \rvert \leq 1} \Delta_{i} f  \otimes_{s} \Delta_{i+j} g 
\end{equation} 
so that 
\begin{equation}\label{Bony's decomposition}
f  \otimes_{s} g = f  \circlesign{\prec}_{s} g + f  \circlesign{\succ}_{s} g + f  \circlesign{\circ}_{s} g 
\end{equation} 
where $f  \circlesign{\succ}_{s} g = g \circlesign{\prec}_{s} f $. Analogous definitions follow in cases with one or both of $f$ and $g$ are $\mathbb{M}^{2}$-valued (see \cite[p. 37]{HR24}). We recall from \cite{BCD11} that there exist constants $N_{1}, N_{2} \in \mathbb{N}$ such that 
\begin{equation}\label{est 152} 
\Delta_{m} (f \prec g) = \sum_{j: \lvert j-m \rvert \leq N_{1}} (S_{j-1} f) \Delta_{j} g \hspace{2mm} \text{ and } \hspace{2mm} \Delta_{m} ( f \circ g) = \Delta_{m}\sum_{i\geq m-N_{2}}  \sum_{j: \lvert j \rvert \leq 1} \Delta_{i} f \Delta_{i+j} g. 
\end{equation} 

For convenience we record some special cases of Bony's estimates. 
\begin{lemma}\label{Lemma 3.1} 
\rm{(\hspace{1sp}\cite[Proposition 3.1]{AC15} and \cite[Lemma A.1]{HR24} and \cite[p. 12]{GIP15})} Let $\alpha, \beta \in \mathbb{R}$. 
\begin{enumerate} 
\item  Then 
\begin{subequations}\label{Sobolev products}
\begin{align}
& \lVert f \prec g \rVert_{H^{\beta -\alpha}} \lesssim_{\alpha, \beta} \lVert f \rVert_{L^{2}} \lVert g \rVert_{\mathscr{C}^{\beta}}  \hspace{8mm}  \forall \hspace{1mm} f \in L^{2}, g \in \mathscr{C}^{\beta} \text{ if } \alpha > 0, \label{Sobolev products a} \\
& \lVert f \succ g \rVert_{H^{\alpha}} \lesssim_{\alpha} \lVert f \rVert_{H^{\alpha}} \lVert g \rVert_{L^{\infty}}  \hspace{11mm}  \forall \hspace{1mm} f \in H^{\alpha}, g \in L^{\infty},  \label{Sobolev products b} \\
& \lVert f \prec g \rVert_{H^{\alpha + \beta}} \lesssim_{\alpha, \beta} \lVert f \rVert_{H^{\alpha}} \lVert g \rVert_{\mathscr{C}^{\beta}} \hspace{7mm}  \forall \hspace{1mm} f \in H^{\alpha}, g \in \mathscr{C}^{\beta}  \text{ if } \alpha < 0,  \label{Sobolev products c}  \\
& \lVert f \succ g \rVert_{H^{\alpha + \beta}} \lesssim_{\alpha,\beta} \lVert f \rVert_{H^{\alpha}}  \lVert g \rVert_{\mathscr{C}^{\beta}}  \hspace{7mm}  \forall \hspace{1mm} f \in H^{\alpha}, g \in \mathscr{C}^{\beta} \text{ if } \beta < 0, \label{Sobolev products d} \\
& \lVert f \circ g \rVert_{H^{\alpha+ \beta}} \lesssim_{\alpha, \beta} \lVert f \rVert_{H^{\alpha}} \lVert g \rVert_{\mathscr{C}^{\beta}}  \hspace{8mm} \forall \hspace{1mm} f \in H^{\alpha}, g \in \mathscr{C}^{\beta}  \text{ if } \alpha + \beta > 0. \label{Sobolev products e} 
\end{align} 
\end{subequations} 
\item Let $p, q \in [1,\infty]$ such that $\frac{1}{r} = \frac{1}{p} + \frac{1}{q} \leq 1$. Then 
\begin{subequations}\label{Besov products} 
\begin{align}
& \lVert f \prec g \rVert_{B_{r, \infty}^{\alpha}} \lesssim_{\alpha,p,q,r} \lVert f \rVert_{L^{p}} \lVert g \rVert_{B_{q, \infty}^{\alpha}} \hspace{12mm} \forall \hspace{1mm} f \in L^{p}, g \in B_{q,\infty}^{\alpha}, \label{Besov products a}\\
& \lVert f \prec g \rVert_{B_{r,\infty}^{\alpha + \beta}} \lesssim_{\alpha, \beta, p, q, r} \lVert f \rVert_{B_{p,\infty}^{\beta}} \lVert g \rVert_{B_{q,\infty}^{\alpha}} \hspace{7mm} \forall \hspace{1mm} f \in B_{p,\infty}^{\beta}, g \in B_{q,\infty}^{\alpha} \text{ if } \beta < 0, \label{Besov products b}\\
& \lVert f \circ g \rVert_{B_{r,\infty}^{\alpha+ \beta}} \lesssim_{\alpha,\beta,p,q,r} \lVert f \rVert_{B_{p,\infty}^{\beta}} \lVert g \rVert_{B_{q,\infty}^{\alpha}} \hspace{8mm} \forall \hspace{1mm} f \in B_{p,\infty}^{\beta}, g \in B_{q,\infty}^{\alpha} \text{ if } \alpha + \beta > 0.\label{Besov products c}
\end{align}
\end{subequations}
One of the consequences is that $\lVert fg \rVert_{\mathscr{C}^{\min\{\alpha,\beta\}}} \lesssim \lVert f \rVert_{\mathscr{C}^{\alpha}} \lVert g \rVert_{\mathscr{C}^{\beta}}$ when $\alpha + \beta > 0$.
\end{enumerate} 
\end{lemma} 
The following well-known inequality is convenient for our estimates. 
\begin{lemma}\label{Lemma 3.2}
\rm{(\hspace{1sp}\cite[Lemma 2.5]{Y14b}, \cite[Lemma 3.2]{Y25d})} Suppose that $\sigma_{1}, \sigma_{2} < \frac{3}{2}$ satisfy $\sigma_{1} + \sigma_{2} > 0$. Then 
\begin{equation}
\lVert f g \rVert_{\dot{H}^{\sigma_{1} + \sigma_{2} - \frac{3}{2}}} \lesssim_{\sigma_{1}, \sigma_{2}} \lVert f \rVert_{\dot{H}^{\sigma_{1}}} \lVert g \rVert_{\dot{H}^{\sigma_{2}}} \hspace{3mm} \forall \hspace{1mm} f \in \dot{H}^{\sigma_{1}}(\mathbb{T}^{3}), g \in \dot{H}^{\sigma_{2}}(\mathbb{T}^{3}). 
\end{equation}     
 \end{lemma} 
 
The following commutator estimate is classical.   
\begin{lemma}\label{Lemma 3.3}
\rm{(\hspace{1sp}\cite{KP88} and \cite[Lemma 2.2]{CMZ07})} Let $p, p_{2}, p_{3} \in (1,\infty), p_{1}, p_{4} \in [1,\infty]$ satisfy $\frac{1}{p} = \frac{1}{p_{1}} + \frac{1}{p_{2}} = \frac{1}{p_{3}} + \frac{1}{p_{4}}$ and $s > 0$. Suppose that $f$ and $g$ are smooth and satisfy $\nabla f \in L^{p_{1}}, \Lambda^{s-1} g \in L^{p_{2}}, \Lambda^{s} f \in L^{p_{3}}$ and $g \in L^{p_{4}}$. Then 
\begin{equation}\label{Kato-Ponce}
\lVert \Lambda^{s}(fg) - f\Lambda^{s}g\rVert_{L^{p}} \lesssim (\lVert \nabla f\rVert_{L^{p_{1}}}\lVert \Lambda^{s-1}g\rVert_{L^{p_{2}}} + \lVert \Lambda^{s}f\rVert_{L^{p_{3}}}\lVert g\rVert_{L^{p_{4}}}).
\end{equation}
\end{lemma}

\begin{lemma}\label{Lemma 3.4}
\rm{(\hspace{1sp}\cite[Lemma 3.6]{ZZ15}} Let $f \in \mathscr{C}^{\alpha}(\mathbb{T}^{3})$ for any $\alpha \in \mathbb{R}$. Then, $\lVert \mathbb{P}_{L} f \rVert_{\mathscr{C}^{\alpha}} \lesssim \lVert f \rVert_{\mathscr{C}^{\alpha}}$. 
\end{lemma} 
 
 \begin{define}\label{Definition 3.1} 
\rm{(\hspace{1sp}\cite[Definition 4.1]{HR24})} Let $\mathfrak{h}: \hspace{1mm} [0,\infty) \mapsto [0,\infty)$ be a smooth function such that 
\begin{equation}
\mathfrak{h}(r) \triangleq  
\begin{cases}
1 & \text{ if } r \geq 1, \\
0 & \text{ if } r \leq \frac{1}{2}, 
\end{cases} 
\hspace{5mm} \mathfrak{l} \triangleq 1- \mathfrak{h}.
\end{equation} 
Then, we consider for any $\lambda > 0$
\begin{equation}\label{Lower and higher frequencies}
\check{\mathfrak{h}}_{\lambda}(x) \triangleq \mathcal{F}^{-1} \left( \mathfrak{h} \left( \frac{ \lvert \cdot \rvert}{\lambda} \right) \right) (x), \hspace{5mm} \check{\mathfrak{l}}_{\lambda}(x) \triangleq \mathcal{F}^{-1} \left( \mathfrak{l} \left(\frac{ \lvert \cdot \rvert}{\lambda} \right) \right)(x) 
\end{equation} 
and then the projections onto higher and lower frequencies respectively by 
\begin{equation}
\mathcal{H}_{\lambda}: \hspace{1mm} \mathcal{S}' \mapsto \mathcal{S}' \text{ by } \mathcal{H}_{\lambda} f \triangleq \check{\mathfrak{h}}_{\lambda} \ast f \text{ and } \mathcal{L}_{\lambda}: \hspace{1mm} \mathcal{S}' \mapsto \mathcal{S} \text{ by } \mathcal{L}_{\lambda} f \triangleq f - \mathcal{H}_{\lambda} f = \check{\mathfrak{l}}_{\lambda} \ast f.  
\end{equation} 
\end{define} 
The following is a straight-forward generalization of \cite[Lemmas 4.2-4.3]{HR24}.  
\begin{lemma}\label{Lemma 3.5} 
For any $p, q \in [1,\infty]$, $\alpha, \beta \in \mathbb{R}$ such that $\beta \geq \alpha$, 
\begin{equation}
\lVert \mathcal{L}_{\lambda} f \rVert_{B_{p,q}^{\beta}} \lesssim \lambda^{\beta - \alpha} \lVert f \rVert_{B_{p,q}^{\alpha}} \hspace{2mm} \forall \hspace{1mm} f \in B_{p,q}^{\alpha} \text{ and } \lVert \mathcal{H}_{\lambda} f \rVert_{B_{p,q}^{\alpha}} \lesssim \lambda^{\alpha - \beta} \lVert f \rVert_{B_{p,q}^{\beta}} \hspace{2mm} \forall \hspace{1mm} f \in B_{p,q}^{\beta}.
\end{equation} 
\end{lemma} 
 
Finally, we will need the following results in the proof of Proposition \ref{Proposition on Anderson Hamiltonian}. First, we define 
\begin{equation}\label{Burgers' Define sigma D} 
\sigma(D) f \triangleq \mathcal{F}^{-1} (\sigma \mathcal{F} f). 
\end{equation} 
\begin{lemma}\label{Burgers' Lemma A.1} 
\rm{(\hspace{1sp}\cite[Proposition 3.3]{AC15})} Let $\alpha, n \in \mathbb{R}$ and $\sigma: \hspace{1mm} \mathbb{R}^{d} \setminus \{0\} \mapsto \mathbb{R}$ be an infinitely differentiable function such that $\lvert D^{k} \sigma(x) \rvert \lesssim (1+ \lvert x \rvert)^{-n - k }$ for all $x \in \mathbb{R}^{d}$. For $f \in H^{\alpha}(\mathbb{T}^{d})$, we define $\sigma(D) f$ by \eqref{Burgers' Define sigma D}. Then $\sigma(D) f \in H^{\alpha + n} (\mathbb{T}^{d})$ and   
\begin{equation}\label{Burgers' Schauder}
\lVert \sigma(D) f \rVert_{H^{\alpha + n}} \lesssim_{\alpha, n} \lVert f \rVert_{H^{\alpha}}. 
\end{equation} 
\end{lemma} 

\begin{lemma}\label{Burgers' Lemma A.2} 
\rm{(\hspace{1sp}\cite[Proposition A.2]{AC15} and \cite[Lemma A.8]{GUZ20})} Let $\alpha \in (0,1), \beta \in \mathbb{R}$, and $f \in H^{\alpha} (\mathbb{T}^{d}), g \in \mathscr{C}^{\beta} (\mathbb{T}^{d})$, and $\sigma: \hspace{1mm} \mathbb{R}^{d} \setminus \{0\} \mapsto \mathbb{R}$ be an infinitely differentiable function such that $\lvert D^{k} \sigma(x) \rvert \lesssim (1+ \lvert x \rvert)^{-n-k}$ for all $x \in \mathbb{R}^{d}$ and $k \in \mathbb{N}_{0}^{d}$. Define 
\begin{equation}\label{Burgers' Define mathcal C}
\mathscr{C} ( f,g) \triangleq \sigma(D) (f\prec g) - f \prec \sigma(D) g.  
\end{equation} 
Then 
\begin{align}\label{Burgers' est 195}
\lVert \mathscr{C} (f,g) \rVert_{H^{\alpha + \beta + n - \delta}} \lesssim \lVert f \rVert_{H^{\alpha}} \lVert g \rVert_{\mathscr{C}^{\beta}} \hspace{3mm} \forall \hspace{1mm} \delta > 0. 
\end{align}
\end{lemma} 

\begin{lemma}\label{Burgers' Lemma A.3} 
\rm{(\hspace{1sp}\cite[Proposition 4.3]{AC15} and \cite[Proposition A.2]{GUZ20})} Let $\alpha \in (0,1), \beta, \gamma \in \mathbb{R}$ such that $\beta + \gamma < 0$ and $\alpha + \beta + \gamma > 0$. Define 
\begin{equation}\label{Burgers' Define R}
\mathcal{R}(f,g,h) \triangleq (f \prec g) \circ h - f (g \circ h)
\end{equation} 
for smooth functions. Then this trilinear operator can be extended to the product space $H^{\alpha} \times \mathscr{C}^{\beta} \times \mathscr{C}^{\gamma}$ and 
\begin{equation}\label{Burgers' Estimate on R} 
\lVert \mathcal{R} (f,g,h) \rVert_{H^{\alpha + \beta + \gamma - \delta}} \lesssim \lVert f \rVert_{H^{\alpha}} \lVert g \rVert_{\mathscr{C}^{\beta}} \lVert h \rVert_{\mathscr{C}^{\gamma}} \hspace{3mm} \forall \hspace{1mm} f \in H^{\alpha}, g \in \mathscr{C}^{\beta}, h \in \mathscr{C}^{\gamma}, \text{ and } \delta > 0. 
\end{equation} 
\end{lemma} 
 
\section{Proof of Theorem \ref{Theorem 2.2}}\label{Section 4}
\subsection{$L^{2}(\mathbb{T}^{3})$-Estimates}
\hfill\\ We introduce $Y$ as a solution to the following linear equation:
\begin{equation}\label{Equation of Y}
\partial_{t}Y + \mathbb{P}_{L} \divergence (2 X \otimes_{s} Y + X^{\otimes 2}) + \nu \Lambda^{\frac{5}{2}} Y= 0, \hspace{5mm} Y(0,\cdot) = 0
\end{equation} 
(recall Remark \ref{Remark 2.1}).  We define 
\begin{equation}\label{Define D}
D \triangleq 2(X+Y) 
\end{equation} 
so that $w \triangleq v - Y$ satisfies 
\begin{equation}\label{Equation of w}
\partial_{t} w + \mathbb{P}_{L} \divergence  (w^{\otimes 2} + D \otimes_{s} w + Y^{\otimes 2}) + \nu \Lambda^{\frac{5}{2}} w = 0, \hspace{5mm} w(0,\cdot) = u^{\text{in}}(\cdot). 
\end{equation} 
Considering \eqref{Regularity of X}, we see that 
\begin{equation}\label{Regularity of Y}
Y \in \mathscr{C}^{1 - 3 \kappa} (\mathbb{T}^{3}) 
\end{equation} 
$\mathbb{P}$-a.s. and thus the ill-defined terms upon $L^{2}(\mathbb{T}^{3})$-inner products on \eqref{Equation of w} with $w$ are   
\begin{equation}\label{Ill-defined products}
\langle w, \nu \Lambda^{\frac{5}{2}} w + \mathbb{P}_{L} \divergence (2X \otimes_{s} w) \rangle. 
\end{equation}

\begin{define}
Recall $P_{t} = e^{-\nu \Lambda^{\frac{5}{2}} t}$ from Section \ref{Subsection 3.1}. For any $\gamma > 0, T > 0$, and $\beta \in \mathbb{R}$, we define 
\begin{subequations}
\begin{align}
&\mathcal{M}_{T}^{\gamma} \mathscr{C}^{\beta} \triangleq \{ f: \hspace{1mm}t \mapsto t^{\gamma} \lVert f(t) \rVert_{\mathscr{C}^{\beta}} \text{ is continuous over } [0, T], \lVert f \rVert_{\mathcal{M}_{T}^{\gamma} \mathscr{C}^{\beta}} < \infty\},  \\
& \text{ where }\hspace{1mm}  \lVert f \rVert_{\mathcal{M}_{T}^{\gamma} \mathscr{C}^{\beta}} \triangleq \left\lVert t^{\gamma} \lVert f(t) \rVert_{\mathscr{C}^{\beta}} \right\rVert_{C_{T}}.  \label{MT gamma C beta norm}
\end{align} 
\end{subequations}  
Then $w \in \mathcal{M}_{T}^{\gamma} \mathscr{C}^{\beta}$ is a mild solution to \eqref{Equation of w} over $[0,T]$ if for all $t \in [0,T]$, 
\begin{equation}
w(t) = P_{t} u^{\text{in}} - \int_{0}^{t} P_{t-s} \mathbb{P}_{L} \divergence (w^{\otimes 2} + D \otimes_{s} w + Y^{\otimes 2}) (s) ds. 
\end{equation} 
\end{define} 

For any $\lambda \geq 1$ and $t \in [0, \infty)$, we define our enhanced noise by 
\begin{equation}
t \mapsto \left( \nabla_{\text{sym}} \mathcal{L}_{\lambda} X (t), \nabla_{\text{sym}} \mathcal{L}_{\lambda}X(t) \cdot P^{\lambda}(t) - r_{\lambda}(t) \Id \right)
\end{equation} 
where 
\begin{subequations}\label{Define P lambda and r lambda}
\begin{align}
& P^{\lambda}(t,x) \triangleq \left(1 + \frac{\nu  \Lambda^{\frac{5}{2}}}{2} \right)^{-1} \nabla_{\text{sym}} \mathcal{L}_{\lambda} X(t,x), \hspace{23mm} r_{\lambda} \triangleq r_{\lambda}^{1} + r_{\lambda}^{2}, \label{Define P lambda} \\
& r_{\lambda}^{1}(t) \triangleq \sum_{k\in\mathbb{Z}^{3} \setminus \{0\}} \frac{1}{4} \mathfrak{l} \left( \frac{\lvert k \rvert}{\lambda} \right)^{2} \left( \frac{1- e^{-2 \nu \lvert k \rvert^{\frac{5}{2}} t}}{2 \nu \lvert k \rvert^{\frac{1}{2}}} \right) \left(1+ \frac{ \nu \lvert k \rvert^{\frac{5}{2}}}{2} \right)^{-1},  \hspace{3mm} r_{\lambda}^{2}(t) \triangleq  \sum_{m=1}^{3} r_{\lambda}^{2,m}(t) \delta_{m,m} \label{Define r lambda1} \\
& \text{where } \hspace{1mm} r_{\lambda}^{2,m}(t) \triangleq  \sum_{k\in\mathbb{Z}^{3} \setminus \{0 \}} \frac{1}{4} \mathfrak{l} \left( \frac{\lvert k \rvert}{\lambda} \right)^{2} \left( \frac{1- e^{-2 \nu \lvert k \rvert^{\frac{5}{2}} t}}{2\nu \lvert k \rvert^{\frac{1}{2}}} \right) \left(1+ \frac{ \nu \lvert k \rvert^{\frac{5}{2}}}{2} \right)^{-1}\frac{ k_{m}^{2}}{\lvert k \rvert^{2}}, \label{Define r lambda2} \\
& \text{and } (\delta_{m,m} \Id)_{i,j} \triangleq 
\begin{cases}
1 & \text{ if } (i,j) = (m,j),\\
0 & \text{ otherwise}.
\end{cases} \label{Define delta m,m}
\end{align}
\end{subequations} 
For any $t \in [0,\infty)$, any $\kappa > 0$, and $\{\lambda^{i}\}_{i\in\mathbb{N}_{0}}$ to be defined in Definition \ref{Definition 4.2}, we define 
\begin{equation}\label{Define Lt kappa and Nt kappa}
L_{t}^{\kappa} \triangleq 1 + \lVert X \rVert_{C_{t} \mathscr{C}^{-\frac{1}{4} - \kappa}} + \lVert Y \rVert_{C_{t} \mathscr{C}^{1-\kappa}}, \hspace{3mm} N_{t}^{\kappa} \triangleq L_{t}^{\kappa} + \sup_{i\in\mathbb{N}_{0}} \lVert ( \nabla \mathcal{L}_{\lambda^{i}} X) \circlesign{\circ} P^{\lambda^{i}} - r_{\lambda^{i}}(t) \Id \rVert_{C_{t}\mathscr{C}^{-\kappa}}.
\end{equation} 
We will comment on the absence of $r_{\lambda}^{2}$ in the 2D case in Remark \ref{Remark on renormalization constant} (see \eqref{est 74}). 

\begin{proposition}\label{Proposition 4.1}
Recall that we fixed $(\Omega, \mathcal{F},\mathbb{P})$ upon introducing \eqref{STWN}-\eqref{Define STWN}. Then there exists a null set $\mathcal{N}'' \subset \Omega$ such that $N_{t}^{\kappa} (\omega) < \infty$ for all $\omega \in \Omega \setminus \mathcal{N}''$, all $t \geq 0$, and all $\kappa > 0$. 
\end{proposition}

The following result can be proven similarly to \cite[Proposition 3.2]{HR24}, \cite[Proposition 4.1]{Y23c}, and \cite[Proposition 4.2]{Y25d}. 
\begin{proposition}\label{Proposition 4.2}
Fix $\kappa \in (0, \frac{1}{3})$ that satisfies $\gamma \triangleq \frac{5}{4} - \frac{\kappa}{2} > 0$. Suppose that $X$ and $Y$ satisfy $X \in C([0,\infty); \mathscr{C}^{-\frac{1}{4} + \kappa} (\mathbb{T}^{3}))$ and $Y \in C([0,\infty); \mathscr{C}^{1-\kappa} (\mathbb{T}^{3}))$. Then, for all $u^{\text{in}} \in \mathscr{C}^{-1+ 2 \kappa}(\mathbb{T}^{3})$ that is mean-zero and divergence-free, \eqref{Equation of w} has a unique mild solution
\begin{equation}
w \in \mathcal{M}_{T^{\max}}^{\frac{2\gamma}{5}} \mathscr{C}^{\frac{1}{4} + \frac{3\kappa}{2}} \hspace{1mm} \text{ where } \hspace{1mm} T^{\max} ( L_{t}^{\kappa}, u^{\text{in}}) \in (0,\infty]. 
\end{equation} 
\end{proposition}

\begin{assume}\label{Assumption 4.3}
Because Proposition \ref{Proposition 4.2} guaranteed $w \in \mathcal{M}_{T^{\max}}^{\frac{2\gamma}{5}} \mathscr{C}^{\frac{1}{4} + \frac{3\kappa}{2}}$ for $\kappa \in (0, \frac{1}{3})$ and $\gamma = \frac{5}{4} - \frac{\kappa}{2}$, we assume hereafter that $u^{\text{in}} \in  L_{\sigma}^{2}  \cap \mathscr{C}^{-1+2\kappa}(\mathbb{T}^{3})$. 
\end{assume} 

We define $Q$ to solve 
\begin{equation}\label{Equation of Q}
( \partial_{t} + \nu \Lambda^{\frac{5}{2}})Q = 2 X, \hspace{3mm} Q(0,\cdot) = 0 
\end{equation} 
so that 
\begin{equation}\label{Regularity of Q}
\lVert Q(t) \rVert_{\mathscr{C}^{\gamma}} \lesssim \lVert X \rVert_{C_{t} \mathscr{C}^{-\frac{1}{4} - \kappa}} t^{\frac{9}{10} - \frac{2( \gamma + \kappa)}{5}} \hspace{3mm} \forall \hspace{1mm} \gamma < \frac{9}{4} - \kappa. 
\end{equation} 
We define $w^{\sharp}$ to satisfy 
\begin{equation}\label{Define w sharp}
w \triangleq - \mathbb{P}_{L} \divergence ( w \circlesign{\prec}_{s} Q) + w^{\sharp}
\end{equation} 
and  
\begin{equation}\label{Define commutator}
\mathcal{C}^{\circlesign{\prec}_{s}} (w, Q) \triangleq \partial_{t} (w \circlesign{\prec}_{s} Q) + \nu \Lambda^{\frac{5}{2}} (w \circlesign{\prec}_{s}Q) - w \circlesign{\prec}_{s} \partial_{t} Q - \nu w \circlesign{\prec}_{s} \Lambda^{\frac{5}{2}} Q 
\end{equation} 
so that 
\begin{equation}\label{Equation of w sharp}
\partial_{t} w^{\sharp} = - \nu \Lambda^{\frac{5}{2}} w^{\sharp} - \mathbb{P}_{L} \divergence \left( w^{\otimes 2} + D \otimes_{s} w - 2 X \circlesign{\succ}_{s} w - \mathcal{C}^{\circlesign{\prec}_{s}} (w, Q) + Y^{\otimes 2} \right). 
\end{equation} 
Let us also define 
\begin{equation}\label{Define QH, wH, and wL}
Q^{\mathcal{H}}(t) \triangleq \mathcal{H}_{\lambda_{t}} Q(t), \hspace{3mm} w^{\mathcal{H}}(t) \triangleq - \mathbb{P}_{L} \divergence ( w \circlesign{\prec}_{s} Q^{\mathcal{H}}), \hspace{3mm} w^{\mathcal{L}} \triangleq w - w^{\mathcal{H}}. 
\end{equation} 
Considering that $u^{\text{in}}$ is mean-zero from hypothesis of Theorem \ref{Theorem 2.2} and \eqref{Equation of w} implies that $w(t)$ is mean-zero for all $t > 0$, we will frequently make use of the fact that $w^{\mathcal{L}}(t) = w(t) + \mathbb{P}_{L} \divergence ( w \circlesign{\prec}_{s} Q^{\mathcal{H}})(t)$ is also mean-zero for all $t > 0$.  

\begin{define}\label{Definition 4.2}  
Fix any $\tau > 0$ and initial data $u^{\text{in}} \in L_{\sigma}^{2} \cap \mathscr{C}^{-1+ \kappa}(\mathbb{T}^{3})$ for some $\kappa > 0$. For $T^{\max} (\omega, u^{\text{in}})$ from Proposition \ref{Proposition 2.1}, we define a family of stopping times $\{T_{i}\}_{i\in\mathbb{N}_{0}}$ by 
\begin{equation}\label{Define T0 and Ti}
T_{0} \triangleq 0,  \hspace{5mm} T_{i+1} (\omega, u^{\text{in}}) \triangleq \inf\{t \geq T_{i}: \hspace{1mm} \lVert w(t) \rVert_{L^{2}} \geq i+1 \} \wedge T^{\max} (\omega, u^{\text{in}}). 
\end{equation} 
We set 
\begin{equation}\label{Define i0}
i_{0} (u^{\text{in}}) \triangleq \max\{i\in\mathbb{N}: \hspace{1mm} i \leq \lVert u^{\text{in}} \rVert_{L^{2}} \} 
\end{equation} 
so that $T_{i} = 0$ if and only if $i \leq i_{0} (u^{\text{in}})$.  We set $\lambda^{i} \triangleq (i+1)^{\tau}$, 
\begin{equation}\label{Define lambda t}
\lambda_{t} \triangleq 
\begin{cases}
( 1 + \lceil \lVert u^{\text{in}} \rVert_{L^{2}} \rceil )^{\tau} & \text{ if } t = 0, \\
(1+ \lVert w(T_{i}) \rVert_{L^{2}})^{\tau} & \text{ if } t > 0 \text{ and } t \in [T_{i}, T_{i+1}).
\end{cases}
\end{equation} 
As $u^{\text{in}} \in L_{\sigma}^{2}$, we have $i_{0} (u^{\text{in}}) < \infty$. Finally, $\lambda_{t} = \lambda^{i}$ for all $t \in [T_{i}, T_{i+1})$ such that $i \geq i_{0} (u^{\text{in}})$. 
\end{define}

\begin{proposition}\label{Proposition 4.4}
We fix any $\tau > 0$ from Definition \ref{Definition 4.2}, $\mathcal{N}$ from Proposition \ref{Proposition 2.1}, $\mathcal{N}''$ from Proposition \ref{Proposition 4.1}, and define $N_{t}^{\kappa}$ by \eqref{Define Lt kappa and Nt kappa}. Then, for any $\delta \geq 0$ and $\omega \in \Omega \setminus ( \mathcal{N}\cup \mathcal{N}'')$, there exists $C(\delta) > 0$ such that $w^{\mathcal{H}}$ defined by \eqref{Define QH, wH, and wL} satisfies 
\begin{equation}\label{Estimate on wH}
\lVert w^{\mathcal{H}}(t) \rVert_{H^{\frac{5}{4} - 2 \kappa - \delta}} \leq C(\delta) \left(1+ \lVert w(t) \rVert_{L^{2}} \right)^{1- \tau \delta} N_{t}^{\kappa} t^{\frac{\kappa}{5}} \hspace{3mm} \forall \hspace{1mm} t \in [0, T^{\max} (\omega, u^{\text{in}})).
\end{equation} 
\end{proposition}

\begin{remark}
The analogous estimate in \cite{HR24, Y23c, Y25d} were at the level $\lVert \cdot \rVert_{H^{1- 2 \kappa - \delta}}$. As we will see, this improvement to be able to bound $\lVert \cdot \rVert_{H^{\frac{5}{4} - 2 \kappa - \delta}}$ is crucial. 
\end{remark}

\begin{proof}[Proof of Proposition \ref{Proposition 4.4}]
We first estimate
\begin{equation}\label{est 63}
 \lVert w^{\mathcal{H}}(t) \rVert_{H^{\frac{5}{4} - 2 \kappa - \delta}}  \overset{\eqref{Define QH, wH, and wL}}{\lesssim} \lVert w \circlesign{\prec}_{s} Q^{\mathcal{H}} (t) \rVert_{H^{\frac{9}{4} - 2 \kappa - \delta}} \overset{\eqref{Sobolev products a}\eqref{Define QH, wH, and wL}}{\lesssim}  \lVert w(t) \rVert_{L^{2}} \lambda_{t}^{-\delta} \lVert Q(t) \rVert_{\mathscr{C}^{\frac{9}{4} - \frac{3\kappa}{2}}}.
\end{equation} 
Using the fact that $\lambda_{t}^{-\delta} \lesssim (1+ \lVert w(t) \rVert_{L^{2}})^{-\tau \delta}$ for all $t \in [0, T^{\max}(\omega, u^{\text{in}}))$, along with \eqref{Regularity of Q} and \eqref{Define Lt kappa and Nt kappa}, we can continue by  
\begin{align*}
\lVert w^{\mathcal{H}}(t) \rVert_{H^{\frac{5}{4} - 2 \kappa - \delta}} \lesssim  \lVert w(t) \rVert_{L^{2}} (1+ \lVert w(t) \rVert_{L^{2}})^{-\tau \delta} \lVert X \rVert_{C_{t} \mathscr{C}^{-\frac{1}{4} - \kappa}} t^{\frac{9}{10} - \frac{2 ( \frac{9}{4} - \frac{3\kappa}{2} + \kappa)}{5}} \lesssim (1+ \lVert w(t) \rVert_{L^{2}})^{1- \tau \delta} N_{t}^{\kappa} t^{\frac{\kappa}{5}} .
\end{align*}
This completes the proof of Proposition \ref{Proposition 4.4}. 
\end{proof}

Now we fix $i \in \mathbb{N}$ such that $i > i_{0} (u^{\text{in}})$, and $t \in [T_{i}, T_{i+1})$ for some $T_{i} < T_{i+1}$ and compute using \eqref{Define QH, wH, and wL}, \eqref{Equation of Q}, \eqref{Equation of w}, and \eqref{Define commutator}, 
\begin{equation}\label{Equation of wL}
\partial_{t} w^{\mathcal{L}} = - \nu \Lambda^{\frac{5}{2}} w^{\mathcal{L}}   - \mathbb{P}_{L} \divergence \left( w^{\otimes 2} + D \otimes_{s} w - 2 ( \mathcal{H}_{\lambda_{t}} X) \circlesign{\succ}_{s} w  - \mathcal{C}^{\circlesign{\prec}_{s}} (w, Q^{\mathcal{H}}) + Y^{\otimes 2} \right). 
\end{equation} 
Thus, by relying on \eqref{Define D} and \eqref{Define QH, wH, and wL}, we see that the $L^{2}(\mathbb{T}^{3})$-estimate of $w^{\mathcal{L}}$ can take the form of,
\begin{equation}\label{est 96} 
 \partial_{t} \lVert w^{\mathcal{L}} (t) \rVert_{L^{2}}^{2} = \sum_{k=1}^{4} \RomanI_{k}, 
\end{equation} 
where 
\begin{subequations}\label{Define Ik}
\begin{align}
& \RomanI_{1} \triangleq 2 \left\langle w^{\mathcal{L}}, - \nu \Lambda^{\frac{5}{2}} w^{\mathcal{L}} - \divergence  \left( 2 ( \mathcal{L}_{\lambda_{t}} X) \otimes_{s} w^{\mathcal{L}} \right) \right\rangle (t),  \label{Define I1} \\
& \RomanI_{2} \triangleq -2 \left\langle w^{\mathcal{L}}, \divergence \left( 2 ( \mathcal{H}_{\lambda_{t}} X) \otimes_{s} w^{\mathcal{L}} - 2 ( \mathcal{H}_{\lambda_{t}} X) \circlesign{\succ}_{s} w^{\mathcal{L}} \right) \right\rangle(t), \label{Define I2} \\
& \RomanI_{3} \triangleq -2 \left\langle w^{\mathcal{L}}, \divergence \left( 2X \otimes_{s} w^{\mathcal{H}} - 2 (\mathcal{H}_{\lambda_{t}} X) \circlesign{\succ}_{s} w^{\mathcal{H}} \right) \right\rangle(t), \label{Define I3}\\
& \RomanI_{4} \triangleq -2 \left\langle w^{\mathcal{L}}, \divergence \left( w^{\otimes 2} + 2 Y \otimes_{s} w - \mathcal{C}^{\circlesign{\prec}_{s}} (w, Q^{\mathcal{H}}) + Y^{\otimes 2} \right) \right\rangle (t).  \label{Define I4}
\end{align}
\end{subequations} 
By taking advantage of $\nabla\cdot \mathcal{L}_{\lambda_{t}} X = 0$, we can rewrite from \eqref{Define I1}
\begin{equation}\label{est 179}
\RomanI_{1}= - \nu \lVert w^{\mathcal{L}}(t) \rVert_{\dot{H}^{\frac{5}{4}}}^{2} - 2 \left\langle w^{\mathcal{L}}, [ \frac{\nu \Lambda^{\frac{5}{2}}}{2}  \Id + \nabla_{\text{sym}} \mathcal{L}_{\lambda_{t}} X ] w^{\mathcal{L}} \right\rangle (t). 
\end{equation} 
Thus, if we define 
\begin{equation}
\mathcal{A}_{t} \triangleq - \frac{\nu \Lambda^{\frac{5}{2}}}{2}  \Id - \nabla_{\text{sym}} X(t) - \infty \hspace{3mm} \forall \hspace{1mm} t \geq 0 
\end{equation} 
as the limit $\lambda \nearrow + \infty$ of 
\begin{equation}\label{Define At lambda}
\mathcal{A}_{t}^{\lambda} \triangleq - \frac{\nu \Lambda^{\frac{5}{2}}}{2}  \Id - \nabla_{\text{sym}} \mathcal{L}_{\lambda} X(t)  - \Id r_{\lambda}(t), 
\end{equation} 
we can finally obtain
\begin{equation}\label{est 97}
\RomanI_{1} = - \nu \lVert w^{\mathcal{L}}(t) \rVert_{\dot{H}^{\frac{5}{4}}}^{2} + 2 \langle w^{\mathcal{L}}, \mathcal{A}_{t}^{\lambda_{t}} w^{\mathcal{L}} \rangle (t) + r_{\lambda}(t) \lVert w^{\mathcal{L}}(t) \rVert_{L^{2}}^{2}. 
\end{equation} 
We proceed with estimates of $\RomanI_{2}, \RomanI_{3},$ and $\RomanI_{4}$ one by one. In what follows, the selection of parameters $\tau$ and $\eta$ are delicate requiring care, although the choice of $\kappa_{0}$ is straight-forward only under the requirement of sufficient smallness; we will elaborate on this issue in Remark \ref{Remark 4.2}. 

\begin{proposition}\label{Proposition 4.5}
Let $t \in [T_{i}, T_{i+1})$ and fix $\lambda_{t}$ with $\tau \in [1,\infty)$. Then, for any $\kappa_{0} \in (0,1)$, $\eta \in [\frac{5}{8} + \frac{\kappa_{0}}{2},1)$, and all $\kappa \in (0, \kappa_{0}]$, $\RomanI_{2}$ from \eqref{Define I2} satisfies 
\begin{equation}\label{est 98} 
\lvert \RomanI_{2} \rvert \lesssim \lVert w^{\mathcal{L}}(t) \rVert_{\dot{H}^{\eta}} N_{t}^{\kappa}. 
\end{equation}  
\end{proposition} 

\begin{proof}[Proof of Proposition \ref{Proposition 4.5}]
We rely on \eqref{Bony's decomposition}, \eqref{Sobolev products d}, \eqref{Sobolev products e}, and \eqref{Define Lt kappa and Nt kappa} to estimate 
\begin{align}
\lvert \RomanI_{2} \rvert \lesssim& \lVert w^{\mathcal{L}} \rVert_{\dot{H}^{\eta}} \lVert (\mathcal{H}_{\lambda_{t}} X) \circlesign{\prec}_{s} w^{\mathcal{L}} + (\mathcal{H}_{\lambda_{t}} X) \circlesign{\circ}_{s} w^{\mathcal{L}} \rVert_{\dot{H}^{1-\eta}} \nonumber \\
\lesssim&  \lVert w^{\mathcal{L}} \rVert_{\dot{H}^{\eta}} \lVert \mathcal{H}_{\lambda_{t}} X \rVert_{\mathscr{C}^{-\frac{1}{4} - \kappa}} \lVert w^{\mathcal{L}} \rVert_{H^{\frac{5}{4} - \eta + \kappa}} \lesssim \lVert w^{\mathcal{L}} \rVert_{\dot{H}^{\eta}}^{2} N_{t}^{\kappa}. 
\end{align}
This completes the proof of Proposition \ref{Proposition 4.5}. 
\end{proof}

\begin{proposition}\label{Proposition 4.6}
Let $t \in [T_{i}, T_{i+1})$ and fix $\lambda_{t}$ from \eqref{Define lambda t}. 
\begin{enumerate}[label=(\alph*)]
\item Let $\tau \in [2,\infty)$, $\kappa_{0} \in (0, \frac{1}{6})$, and $\eta \in [\frac{1}{\tau} + 3 \kappa_{0}, 1)$. Then, for all $\kappa \in (0, \kappa_{0}]$, $\RomanI_{3}$ from \eqref{Define I3} satisfies 
\begin{equation}\label{est 83}
\lvert \RomanI_{3} \rvert \lesssim (N_{t}^{\kappa})^{2} \lVert w^{\mathcal{L}}(t) \rVert_{\dot{H}^{\eta}} \lambda_{t}^{\frac{5}{4} + 2 \kappa - \eta}. 
\end{equation} 
\item Let $\tau \in [\frac{25}{12}, \infty)$,  $\kappa_{0} \in (0, \frac{1}{200})$, and $\eta \in \left( \max\{ \frac{1}{2\tau} + \frac{5}{8} + \kappa_{0}, \frac{2}{\tau} + 4 \kappa_{0} \}, 1 \right)$. Then, for all $\kappa \in (0, \kappa_{0}]$, $- \langle w^{\mathcal{L}}, \divergence (w^{\otimes 2}) \rangle$ within $\RomanI_{4}$ of \eqref{Define I4} satisfies 
\begin{equation}\label{est 84}
2 \lvert \langle w^{\mathcal{L}}, \divergence (w^{\otimes 2}) \rangle(t) \rvert \lesssim \lVert w^{\mathcal{L}}(t) \rVert_{\dot{H}^{\eta}} (\lVert w^{\mathcal{L}} (t) \rVert_{\dot{H}^{\eta}} + N_{t}^{\kappa}) N_{t}^{\kappa}. 
\end{equation} 
\end{enumerate} 
\end{proposition} 

\begin{proof}[Proof of Proposition \ref{Proposition 4.6}]  
\hfill\\ (a) First, we rewrite from \eqref{Define I3}
\begin{equation}\label{Split I3}
\lvert \RomanI_{3} \rvert \leq \RomanI_{31} +\RomanI_{32} 
\end{equation} 
where 
\begin{subequations}\label{Define I31 and I32}
\begin{align}
\RomanI_{31} \triangleq& 4 \left\lvert \langle w^{\mathcal{L}}, \divergence ( (\mathcal{L}_{\lambda_{t}} X) \otimes_{s} w^{\mathcal{H}} ) \rangle \right\rvert, \\
\RomanI_{32} \triangleq& 4 \left\lvert \langle w^{\mathcal{L}}, \divergence ( (\mathcal{H}_{\lambda_{t}}X)  \circlesign{\prec}_{s} w^{\mathcal{H}} + (\mathcal{H}_{\lambda_{t}} X) \circlesign{\circ}_{s} w^{\mathcal{H}} ) \rangle \right\rvert. 
\end{align}
\end{subequations}
For $\RomanI_{31}$, we estimate using \eqref{Bony's decomposition}
\begin{equation}\label{est 82}
\RomanI_{31} \lesssim \lVert w^{\mathcal{L}} \rVert_{\dot{H}^{\eta}} [ \lVert (\mathcal{L}_{\lambda_{t}} X) \circlesign{\succ}_{s} w^{\mathcal{H}} \rVert_{\dot{H}^{1-\eta}} +  \lVert (\mathcal{L}_{\lambda_{t}}X) \circlesign{\prec}_{s} w^{\mathcal{H}} \rVert_{\dot{H}^{1-\eta}} + \lVert (\mathcal{L}_{\lambda_{t}}X) \circlesign{\circ}_{s} w^{\mathcal{H}} \rVert_{\dot{H}^{1-\eta}} ]
\end{equation} 
where we can further estimate by \eqref{Sobolev products c},  \eqref{Sobolev products e}, \eqref{Define Lt kappa and Nt kappa}, and  \eqref{Estimate on wH}
\begin{subequations}\label{est 81}
\begin{align}
& \lVert (\mathcal{L}_{\lambda_{t}}X) \circlesign{\succ}_{s} w^{\mathcal{H}} \rVert_{\dot{H}^{1-\eta}} + \lVert (\mathcal{L}_{\lambda_{t}}X) \circlesign{\circ}_{s} w^{\mathcal{H}} \rVert_{\dot{H}^{1-\eta}} \nonumber \\
& \hspace{27mm}   \lesssim \lVert \mathcal{L}_{\lambda_{t}}X \rVert_{\mathscr{C}^{1+ \kappa - \eta}} \lVert w^{\mathcal{H}} \rVert_{H^{-\kappa}} \lesssim \lambda_{t}^{\frac{5}{4} + 2 \kappa - \eta} (N_{t}^{\kappa})^{2}, \\
& \lVert (\mathcal{L}_{\lambda_{t}}X) \circlesign{\prec}_{s} w^{\mathcal{H}} \rVert_{\dot{H}^{1-\eta}}  \lesssim \lVert \mathcal{L}_{\lambda_{t}}X \rVert_{\mathscr{C}^{-\frac{1}{4} - \kappa}} \lVert w^{\mathcal{H}} \rVert_{H^{\frac{5}{4} - \eta + \kappa}} \lesssim  (N_{t}^{\kappa})^{2}.
\end{align}
\end{subequations} 
Applying \eqref{est 81} to \eqref{est 82} gives us 
\begin{equation}\label{Estimate on I31}
\RomanI_{31}  \lesssim \lVert w^{\mathcal{L}} \rVert_{\dot{H}^{\eta}}  \lambda_{t}^{\frac{5}{4} + 2 \kappa - \eta} (N_{t}^{\kappa})^{2}. 
\end{equation} 
On the other hand, we estimate by \eqref{Sobolev products d}, \eqref{Sobolev products e}, and \eqref{Estimate on wH}, 
\begin{equation}\label{Estimate on I32}
\RomanI_{32}  \lesssim \lVert w^{\mathcal{L}} \rVert_{\dot{H}^{\eta}} \lVert \mathcal{H}_{\lambda_{t}}X \rVert_{\mathscr{C}^{-\frac{1}{4} - \kappa}} \lVert w^{\mathcal{H}} \rVert_{H^{\frac{5}{4} - \eta + \kappa}} \lesssim \lVert w^{\mathcal{L}} \rVert_{\dot{H}^{\eta}} (N_{t}^{\kappa})^{2}.
\end{equation} 
Applying \eqref{Estimate on I31} and \eqref{Estimate on I32} to \eqref{Split I3} we conclude 
\begin{equation}
\lvert \RomanI_{3} \rvert  \lesssim (N_{t}^{\kappa})^{2} \lVert w^{\mathcal{L}} \rVert_{\dot{H}^{\eta}} \lambda_{t}^{\frac{5}{4} + 2 \kappa - \eta}, 
\end{equation} 
which verifies \eqref{est 83}.  

(b) We first simplify using the divergence-free property and write 
\begin{equation}\label{est 86}
\langle w^{\mathcal{L}}, \divergence ( w^{\otimes 2}) \rangle = \langle w^{\mathcal{L}}, (w^{\mathcal{L}} \cdot\nabla) w^{\mathcal{H}} \rangle + \langle w^{\mathcal{L}}, \divergence (w^{\mathcal{H}})^{\otimes 2} \rangle. 
\end{equation} 
We estimate by Lemma \ref{Lemma 3.2} and \eqref{Estimate on wH}, 
\begin{subequations}\label{est 85} 
\begin{align}
& \lvert \langle w^{\mathcal{L}}, (w^{\mathcal{L}} \cdot \nabla) w^{\mathcal{H}} \rangle \rvert \lesssim \lVert w^{\mathcal{L}} \rVert_{\dot{H}^{\eta}}^{2} \lVert w^{\mathcal{H}} \rVert_{\dot{H}^{\frac{5}{2} - 2 \eta}} \lesssim  \lVert w^{\mathcal{L}} \rVert_{\dot{H}^{\eta}}^{2} N_{t}^{\kappa}, \\
& \lvert \langle w^{\mathcal{L}}, \divergence (w^{\mathcal{H}})^{\otimes 2} \rangle \rvert \lesssim \lVert w^{\mathcal{L}} \rVert_{\dot{H}^{\eta}} \lVert w^{\mathcal{H}} \rVert_{\dot{H}^{\frac{5}{4} - \frac{\eta}{2}}}^{2} \lesssim  \lVert w^{\mathcal{L}} \rVert_{\dot{H}^{\eta}} (N_{t}^{\kappa})^{2}. 
\end{align}
\end{subequations}
Applying \eqref{est 85} to \eqref{est 86}, we conclude that 
\begin{align*}
\langle w^{\mathcal{L}}, \divergence ( w^{\otimes 2}) \rangle \lesssim  \lVert w^{\mathcal{L}} \rVert_{\dot{H}^{\eta}} \left( \lVert w^{\mathcal{L}} \rVert_{\dot{H}^{\eta}} + N_{t}^{\kappa}  \right) N_{t}^{\kappa}, 
\end{align*}
which verifies \eqref{est 84}. This completes the proof of Proposition \ref{Proposition 4.6}. 
\end{proof}

The following estimate will involve the commutator we described in Remark \ref{Remark 2.2}. 
\begin{remark}\label{Remark 4.2}
At the step of reaching \eqref{est 91} from \eqref{est 89} due to \eqref{est 90}, we will need $\tau (2\kappa - \eta + \frac{1}{2}) + \frac{7}{5} + \frac{2\kappa}{5} \leq 0$ or equivalently 
\begin{equation}\label{lower bound of tau}
\left( \frac{1}{\eta - 2 \kappa - \frac{1}{2}} \right) \left( \frac{7}{5} + \frac{2\kappa}{5} \right) \leq \tau 
\end{equation} 
assuming that $\eta > \frac{1}{2} + 2 \kappa$, which is reasonable considering the hypothesis $\eta > \frac{1}{2\tau}+  \frac{5}{8} + \kappa_{0}$ in Proposition \ref{Proposition 4.6} (b). On the other hand, we can already see from \eqref{est 83} that we will need $\tau \left( \frac{5}{4} + 2 \kappa - \eta \right) \leq 1$ considering \eqref{Define lambda t} to be able to close the estimate (see \eqref{est 103}), or equivalently
\begin{equation}\label{upper bound of tau}
\tau \leq \frac{1}{ \frac{5}{4} + 2 \kappa - \eta} 
\end{equation}  
assuming that $\frac{5}{4} + 2 \kappa - \eta > 0$, which is reasonable considering the hypothesis $\eta < 1$ in Proposition \ref{Proposition 4.6} (b). To be able to find a suitable $\tau$ that satisfies both \eqref{lower bound of tau} and \eqref{upper bound of tau}, one can first temporarily assume $\kappa =0$ and realize that the requirement for $\eta$ is then $\frac{15}{16} \leq \eta$.  With this in mind, we choose $\eta \geq \frac{31}{32}$. With this lower bound fixed, it follows that $\kappa < \frac{1}{200}$ can be seen to admit the existence of $\tau$ that satisfies both \eqref{lower bound of tau}-\eqref{upper bound of tau}. We incorporate these considerations into the hypothesis of the following Proposition \ref{Proposition 4.7}. 
\end{remark}

\begin{proposition}\label{Proposition 4.7}
Let $t \in [T_{i}, T_{i+1})$ and fix $\lambda_{t}$ from \eqref{Define lambda t}. Let $\tau \in [\frac{40}{13}, \infty)$, $\eta \in [ \frac{31}{32}, 1)$, and $\kappa_{0} \in (0, \frac{1}{200})$. Then, for all $\kappa \in (0, \kappa_{0} ]$, $\RomanI_{4}$ from \eqref{Define I4} satisfies 
\begin{align}
\RomanI_{4} + 2 \langle w^{\mathcal{L}}, \divergence (w^{\otimes 2}) \rangle  \lesssim& N_{t}^{\kappa} \lVert w^{\mathcal{L}} \rVert_{\dot{H}^{1- \frac{3\kappa}{2}}} ( \lVert w^{\mathcal{L}} \rVert_{\dot{H}^{2\kappa}} + N_{t}^{\kappa}) + (N_{t}^{\kappa})^{2} \lVert w^{\mathcal{L}} \rVert_{\dot{H}^{1- \frac{3\kappa}{2}}}  \nonumber \\
&+ (N_{t}^{\kappa})^{3} \lVert w^{\mathcal{L}} \rVert_{\dot{H}^{\eta}} \lVert w^{\mathcal{L}} \rVert_{\dot{H}^{\frac{5}{4}}}^{\frac{3}{5} - \frac{2\kappa}{5}} + (N_{t}^{\kappa})^{2} \lVert w^{\mathcal{L}} \rVert_{\dot{H}^{1}} ( \lVert w^{\mathcal{L}} \rVert_{\dot{H}^{\frac{1}{4} + 3\kappa}} + N_{t}^{\kappa} ). \label{est 95} 
\end{align}
\end{proposition} 

\begin{proof}[Proof of Proposition \ref{Proposition 4.7}]
Due to \eqref{Define I4} we have 
\begin{equation}\label{est 88}  
\RomanI_{4}+  2 \langle w^{\mathcal{L}}, \divergence (w^{\otimes 2}) \rangle = -2 \langle w^{\mathcal{L}}, \divergence (2 Y \otimes_{s} w- \mathcal{C}^{\circlesign{\prec}_{s}}(w, Q^{\mathcal{H}}) + Y^{\otimes 2} ) \rangle.
\end{equation} 
We first estimate the terms other than $\mathscr{C}^{\circlesign{\prec}}_{s} (w, Q^{\mathcal{H}})$ that involves a commutator:
\begin{subequations}\label{est 93}  
\begin{align}
&2 \lvert \langle w^{\mathcal{L}}, \divergence (2 Y \otimes_{s} w ) \rangle \rvert \lesssim \lVert w^{\mathcal{L}} \rVert_{\dot{H}^{1- \frac{3\kappa}{2}}} \lVert Y \otimes_{s} w \rVert_{\dot{H}^{\frac{3\kappa}{2}}} \nonumber \\
&\overset{\eqref{Bony's decomposition}\eqref{Sobolev products}}{\lesssim} \lVert w^{\mathcal{L}} \rVert_{\dot{H}^{1- \frac{3\kappa}{2}}} \lVert Y \rVert_{\mathscr{C}^{1-\kappa}} \lVert w \rVert_{H^{2\kappa}} \overset{\eqref{Define Lt kappa and Nt kappa} \eqref{Estimate on wH}}{\lesssim} N_{t}^{\kappa} \lVert w^{\mathcal{L}} \rVert_{\dot{H}^{1- \frac{3\kappa}{2}}} ( \lVert w^{\mathcal{L}} \rVert_{\dot{H}^{2\kappa}} + N_{t}^{\kappa}), \\
& -2 \langle w^{\mathcal{L}}, \divergence ( Y^{\otimes 2}) \rangle  \lesssim \lVert w^{\mathcal{L}} \rVert_{\dot{H}^{1- \frac{3\kappa}{2}}} \left(\lVert Y \circlesign{\prec} Y \rVert_{\dot{H}^{\frac{3\kappa}{2}}} + \lVert Y \circlesign{\circ}_{s} Y \rVert_{\dot{H}^{\frac{3\kappa}{2}}} \right) \lesssim (N_{t}^{\kappa})^{2} \lVert w^{\mathcal{L}} \rVert_{\dot{H}^{1- \frac{3\kappa}{2}}}.
\end{align}
\end{subequations} 
Next, we first rewrite $\mathcal{C}^{\circlesign{\prec}_{s}} (w, Q^{\mathcal{H}})$ from \eqref{Define commutator} by utilizing \eqref{Equation of w} as
\begin{subequations}  
\begin{align}
 \mathcal{C}^{\circlesign{\prec}_{s}} (w, Q^{\mathcal{H}}) =& \left( \partial_{t} w + \nu \Lambda^{\frac{5}{2}} w \right) \circlesign{\prec}_{s} Q^{\mathcal{H}}  \nonumber \\
 &+ \nu \left( \Lambda^{\frac{5}{2}} (w \circlesign{\prec}_{s} Q^{\mathcal{H}} ) - (\Lambda^{\frac{5}{2}} w) \circlesign{\prec}_{s} Q^{\mathcal{H}} - w \circlesign{\prec}_{s} \Lambda^{\frac{5}{2}} Q^{\mathcal{H}} \right) \label{est 167}\\
=& -  \left[ \mathbb{P}_{L} \divergence (w^{\otimes 2} + D \otimes_{s} w + Y^{\otimes 2} ) \right] \circlesign{\prec}_{s} Q^{\mathcal{H}}  \nonumber \\
& \hspace{10mm} + \nu \left( \Lambda^{\frac{5}{2}} (w \circlesign{\prec}_{s} Q^{\mathcal{H}}) -  ( \Lambda^{\frac{5}{2}} w) \circlesign{\prec}_{s} Q^{\mathcal{H}} - w \circlesign{\prec}_{s} \Lambda^{\frac{5}{2}} Q^{\mathcal{H}} \right), \label{est 87}
\end{align}
\end{subequations}   
where as we emphasized in \eqref{HR24 commutator}-\eqref{our commutator} and Remark \ref{Remark 2.2}, Leibniz rule in time derivative can be used but not in space. We will use the identity \eqref{est 87} in this proof of Proposition \ref{Proposition 4.7} while \eqref{est 167} will be useful in the proof of Proposition \ref{Proposition 5.2}. 

Considering \eqref{est 87} in \eqref{est 88}, we see that our remaining task now is to estimate 
\begin{equation}\label{Split to Ck}
2 \langle w^{\mathcal{L}}, \divergence C^{\circlesign{\prec}_{s}}(w, Q^{\mathcal{H}}) \rangle = \sum_{k=1}^{4} C_{k},
\end{equation}
where 
\begin{subequations}\label{Define Ck}
\begin{align}
& C_{1} \triangleq -2 \langle w^{\mathcal{L}}, \divergence \left( [ \mathbb{P}_{L} \divergence (w^{\otimes 2} ) ] \circlesign{\prec}_{s} Q^{\mathcal{H}}  \right) \rangle, \label{Define C1}\\
& C_{2} \triangleq -2 \langle w^{\mathcal{L}}, \divergence \left( [ \mathbb{P}_{L} \divergence (D \otimes_{s} w) ] \circlesign{\prec}_{s} Q^{\mathcal{H}} \right) \rangle,  \label{Define C2}\\
& C_{3} \triangleq -2 \langle w^{\mathcal{L}}, \divergence \left( [ \mathbb{P}_{L} \divergence (Y^{\otimes 2}) ] \circlesign{\prec}_{s} Q^{\mathcal{H}}\right) \rangle,  \label{Define C3}\\
& C_{4} \triangleq 2 \langle w^{\mathcal{L}}, \divergence \left( \Lambda^{\frac{5}{2}} (w \circlesign{\prec}_{s} Q^{\mathcal{H}}) - (\Lambda^{\frac{5}{2}} w) \circlesign{\prec}_{s} Q^{\mathcal{H}} - w \circlesign{\prec}_{s} \Lambda^{\frac{5}{2}} Q^{\mathcal{H}} \right) \rangle. \label{Define C4}
\end{align}
\end{subequations} 

Let us first work on $C_{1}$ as follows: using $2 \kappa - \eta + \frac{1}{2} \leq 0$ due to $\eta \geq \frac{5}{8}$ from hypothesis, Lemma \ref{Lemma 3.2}, $\lambda_{t}^{-\delta} \lesssim (1+ \lVert w(t) \rVert_{L^{2}})^{-\tau \delta}$ for all $\delta \geq 0$ from the proof of Proposition \ref{Proposition 4.4}, and the Gagliardo-Nirenberg inequality of $\lVert f \rVert_{\dot{H}^{\frac{3}{8} - \frac{\kappa}{4}}}  \lesssim \lVert f \rVert_{L^{2}}^{\frac{7}{10} + \frac{\kappa}{5}} \lVert f \rVert_{\dot{H}^{\frac{5}{4}}}^{\frac{3}{10} - \frac{\kappa}{5}}$, we obtain 
\begin{align} 
&C_{1} \lesssim \lVert w^{\mathcal{L}} \rVert_{\dot{H}^{\eta}} \lVert \mathbb{P}_{L} \divergence (w^{\otimes 2}) \rVert_{H^{-\frac{7}{4} - \frac{\kappa}{2}}} \lVert Q^{\mathcal{H}} \rVert_{\mathscr{C}^{\frac{11}{4} - \eta + \frac{\kappa}{2}}}  \nonumber\\
&\hspace{3mm}\lesssim \lVert w^{\mathcal{L}} \rVert_{\dot{H}^{\eta}} \lVert w^{\otimes 2} \rVert_{H^{-\frac{3}{4} - \frac{\kappa}{2}}} \lambda_{t}^{2\kappa - \eta + \frac{1}{2}} \lVert Q \rVert_{\mathscr{C}^{\frac{9}{4} - \frac{3\kappa}{2}}} \label{est 89}\\
&\overset{\eqref{Regularity of Q} \eqref{Define Lt kappa and Nt kappa} \eqref{Estimate on wH}}{\lesssim} N_{t}^{\kappa} \lVert w^{\mathcal{L}} \rVert_{\dot{H}^{\eta}}  \left( \lVert w^{\mathcal{L}} \rVert_{\dot{H}^{\frac{3}{8} - \frac{\kappa}{4}}} + (1+ \lVert w \rVert_{L^{2}})^{1- \tau [\frac{7}{8} - \frac{7 \kappa}{4}]}N_{t}^{\kappa}\right)^{2} (1+ \lVert w  \rVert_{L^{2}})^{\tau (2\kappa - \eta + \frac{1}{2})}   \nonumber \\
& \lesssim N_{t}^{\kappa} \lVert w^{\mathcal{L}} \rVert_{\dot{H}^{\eta}} \lVert w^{\mathcal{L}} \rVert_{\dot{H}^{\frac{5}{4}}}^{\frac{3}{5} - \frac{2\kappa}{5}} \lVert w^{\mathcal{L}} \rVert_{L^{2}}^{\frac{7}{5} + \frac{2\kappa}{5}} (1+ \lVert w \rVert_{L^{2}})^{\tau (2 \kappa - \eta + \frac{1}{2})} + (N_{t}^{\kappa})^{3} \lVert w^{\mathcal{L}} \rVert_{\dot{H}^{\eta}} (1+ \lVert w \rVert_{L^{2}})^{\tau (2\kappa - \eta + \frac{1}{2})}. \nonumber 
\end{align}
Then, utilizing the fact that $N_{t}^{\kappa} \geq 1$ from \eqref{Define Lt kappa and Nt kappa} and 
\begin{equation}\label{est 90}
 \tau(2\kappa - \eta + \frac{1}{2}) + \frac{7}{5} + \frac{2\kappa}{5}  \leq \frac{40}{13} (2 \kappa - \frac{31}{32} + \frac{1}{2}) + \frac{7}{5} + \frac{2\kappa}{5} \leq \left(\frac{426}{65}\right) \frac{1}{200} - \frac{11}{260}  \leq 0 
\end{equation} 
due to the hypothesis of $\eta \geq \frac{40}{13}, \eta \geq \frac{31}{32}$, and $\kappa < \frac{1}{200}$, we conclude 
\begin{equation}\label{est 91}
C_{1}\lesssim (N_{t}^{\kappa})^{3} \lVert w^{\mathcal{L}} \rVert_{\dot{H}^{\eta}} \lVert w^{\mathcal{L}} \rVert_{\dot{H}^{\frac{5}{4}}}^{\frac{3}{5} - \frac{2\kappa}{5}}. 
\end{equation}
 
Second, we rewrite using \eqref{Define D} from \eqref{Define C2}
\begin{equation}\label{Split C2}
C_{2} = \sum_{k=1}^{2} C_{2k} 
\end{equation} 
where 
\begin{subequations}\label{Define C21 and C22}
\begin{align}
& C_{21} \triangleq - 4 \langle w^{\mathcal{L}}, \divergence \left( [ \mathbb{P}_{L} \divergence (Y \otimes_{s} w) ] \circlesign{\prec}_{s} Q^{\mathcal{H}} \right) \rangle, \\
& C_{22}  \triangleq - 4 \langle w^{\mathcal{L}}, \divergence \left( [ \mathbb{P}_{L} \divergence (X \otimes_{s} w) ] \circlesign{\prec}_{s} Q^{\mathcal{H}} \right) \rangle.
\end{align}
\end{subequations}
For a subsequent purpose, we point out that we can estimate for any $\eta \in (0,1)$ and $\kappa_{0} \in (0, \frac{1}{200})$, 
\begin{align}
&C_{21} \lesssim  \lVert w^{\mathcal{L}} \rVert_{\dot{H}^{\eta}} \lVert  [ \mathbb{P}_{L} \divergence ( Y \otimes_{s} w) ] \circlesign{\prec}_{s} Q^{\mathcal{H}} \rVert_{\dot{H}^{1-\eta}} \nonumber  \\
&\overset{\eqref{Sobolev products c} \eqref{Regularity of Q}\eqref{Define Lt kappa and Nt kappa} \eqref{Bony's decomposition}}{\lesssim} N_{t}^{\kappa} \lVert w^{\mathcal{L}} \rVert_{\dot{H}^{\eta}} [ \lVert Y \circlesign{\succ}_{s} w \rVert_{H^{-\frac{1}{4} - \eta + \frac{3\kappa}{2}}} + \lVert Y \circlesign{\prec}_{s} w \rVert_{H^{-\frac{1}{4} - \eta + \frac{3\kappa}{2}}} + \lVert Y \circlesign{\circ}_{s} w \rVert_{H^{\kappa}}]  \nonumber \\
&\overset{\eqref{Sobolev products}  \eqref{Define Lt kappa and Nt kappa} \eqref{Estimate on wH}}{\lesssim} (N_{t}^{\kappa})^{2} \lVert w^{\mathcal{L}} \rVert_{\dot{H}^{\eta}} [ \lVert w^{\mathcal{L}} \rVert_{\dot{H}^{\kappa}} + (1+ \lVert w \rVert_{L^{2}})^{1- \tau (\frac{5}{4} - 3 \kappa)} N_{t}^{\kappa}], \label{est 120}   
\end{align} 
and therefore, our current hypothesis on $\tau$ and $\kappa$ that implies $1- \tau (\frac{5}{4} - 3 \kappa) < 0$ gives us 
\begin{equation}\label{Estimate on C21}
C_{21}  \lesssim (N_{t}^{\kappa})^{2} \lVert w^{\mathcal{L}} \rVert_{\dot{H}^{\eta}} ( \lVert w^{\mathcal{L}} \rVert_{H^{\kappa}} + N_{t}^{\kappa}).
\end{equation} 
Next, for a subsequent purpose again we point out that we can estimate for any $\eta \in (0, 1)$ and $\kappa_{0} \in (0, \frac{1}{200})$, 
\begin{align}
C_{22}  \lesssim& \lVert w^{\mathcal{L}} \rVert_{\dot{H}^{\eta}} \lVert [ \mathbb{P}_{L} \divergence ( X \otimes_{s} w) ] \circlesign{\prec}_{s} Q^{\mathcal{H}} \rVert_{\dot{H}^{1-\eta}} \nonumber \\
\overset{\eqref{Regularity of Q} \eqref{Define Lt kappa and Nt kappa} \eqref{Sobolev products}}{\lesssim}&  N_{t}^{\kappa} \lVert w^{\mathcal{L}} \rVert_{\dot{H}^{\eta}} [ \lVert X \rVert_{\mathscr{C}^{-\frac{1}{4} - \kappa}} \lVert w \rVert_{H^{-\eta + \frac{5\kappa}{2}}} + \lVert X \rVert_{\mathscr{C}^{-\frac{1}{4} - \kappa}} \lVert w \rVert_{H^{-\eta + \frac{5\kappa}{2}}} + \lVert X \rVert_{\mathscr{C}^{-\frac{1}{4} - \kappa}} \lVert w \rVert_{H^{\frac{1}{4} + \frac{3\kappa}{2}}} ]  \nonumber \\
\overset{\eqref{Estimate on wH}}{\lesssim}& (N_{t}^{\kappa})^{2} \lVert w^{\mathcal{L}} \rVert_{\dot{H}^{\eta}} [ \lVert w^{\mathcal{L}} \rVert_{\dot{H}^{\frac{1}{4} + \frac{3\kappa}{2}}} + (1+ \lVert w \rVert_{L^{2}})^{1- \tau (1- \frac{7\kappa}{2})} N_{t}^{\kappa} t^{\frac{\kappa}{5}} ]; \label{est 122}
\end{align}
utilizing that $1 - \tau (1- \frac{7\kappa}{2}) \leq 0$ due to our current hypothesis, we conclude that 
\begin{equation}
C_{22}  \lesssim (N_{t}^{\kappa})^{2} \lVert w^{\mathcal{L}} \rVert_{\dot{H}^{\eta}} [ \lVert w^{\mathcal{L}} \rVert_{\dot{H}^{\frac{1}{4} + \frac{3\kappa}{2}}} + N_{t}^{\kappa}].\label{Estimate on C22}
\end{equation} 
Applying \eqref{Estimate on C21} and \eqref{Estimate on C22} to \eqref{Split C2} allows us to conclude
\begin{equation}\label{Estimate on C2}
C_{2}  \lesssim (N_{t}^{\kappa})^{2} \lVert w^{\mathcal{L}} \rVert_{\dot{H}^{\eta}} [ \lVert w^{\mathcal{L}} \rVert_{\dot{H}^{\frac{1}{4} + \frac{3\kappa}{2}}} + N_{t}^{\kappa} ].
\end{equation} 

Next, we estimate $C_{3}$; for a subsequence purpose again, we point out that this estimate is valid for any $\eta \in (0,1)$ and $\kappa_{0} \in (0, \frac{1}{200})$: 
\begin{align}\label{Estimate on C3}
C_{3}  \lesssim&  \lVert w^{\mathcal{L}} \rVert_{\dot{H}^{\eta}} \lVert [ \mathbb{P}_{L} \divergence (Y^{\otimes 2} ) ] \circlesign{\prec}_{s} Q^{\mathcal{H}} \rVert_{\dot{H}^{1-\eta}} \overset{\eqref{Sobolev products d}}{\lesssim} \lVert w^{\mathcal{L}} \rVert_{\dot{H}^{\eta}} \lVert \mathbb{P}_{L} \divergence (Y^{\otimes 2}) \rVert_{H^{-\frac{5}{4} - \eta + \frac{3\kappa}{2}}} \lVert Q^{\mathcal{H}}  \rVert_{\mathscr{C}^{\frac{9}{4} - \frac{3\kappa}{2}}} \nonumber \\
& \hspace{20mm} \overset{\eqref{Regularity of Q}}{\lesssim} \lVert w^{\mathcal{L}} \rVert_{\dot{H}^{\eta}} \lVert Y^{\otimes 2} \rVert_{H^{-\frac{1}{4} - \eta + \frac{3\kappa}{2}}}   \lVert X \rVert_{C_{t} \mathscr{C}^{-\frac{1}{4} - \kappa}} t^{\frac{\kappa}{5}}  \overset{\eqref{Define Lt kappa and Nt kappa}}{\lesssim} (N_{t}^{\kappa})^{3} \lVert w^{\mathcal{L}} \rVert_{\dot{H}^{\eta}}.
\end{align}

We finally come to $C_{4}$ of \eqref{Define C4}. After all the failed attempts we described in Remark \ref{Remark 2.2}, what we were able to do is to make a careful observation of ``$\circlesign{\prec}_{s}$'' throughout in our commutator \eqref{Define C4}, use Besov space techniques combined with Lemma \ref{Lemma 3.3} to deduce a sufficient estimate. First, we split to 
\begin{equation}\label{Split C4}
C_{4} \lesssim \lVert w^{\mathcal{L}} \rVert_{\dot{H}^{1}} \lVert \Lambda^{\frac{5}{2}} (w \circlesign{\prec}_{s} Q^{\mathcal{H}}) - (\Lambda^{\frac{5}{2}} w) \circlesign{\prec}_{s} Q^{\mathcal{H}} - w \circlesign{\prec}_{s} \Lambda^{\frac{5}{2}} Q^{\mathcal{H}}  \rVert_{L^{2}}  \lesssim \sum_{k=1}^{2}  \lVert w^{\mathcal{L}} \rVert_{\dot{H}^{1}}   C_{4k}
\end{equation} 
where 
\begin{equation}\label{Define C41 and C42}
C_{41} \triangleq \lVert ( \Lambda^{\frac{5}{2}} w) \circlesign{\prec}_{s} Q^{\mathcal{H}} \rVert_{L^{2}} \hspace{2mm} \text{ and } \hspace{2mm} 
C_{42} \triangleq \lVert \Lambda^{\frac{5}{2}} (w \circlesign{\prec}_{s} Q^{\mathcal{H}}) - w \circlesign{\prec}_{s} \Lambda^{\frac{5}{2}} Q^{\mathcal{H}} \rVert_{L^{2}}.  
\end{equation} 

We estimate 
\begin{equation}\label{Estimate on C41}
C_{41} \overset{\eqref{Sobolev products c}\eqref{Regularity of Q}}{\lesssim} \lVert w \rVert_{\dot{H}^{\frac{1}{4} + \frac{3\kappa}{2}}}  \lVert X \rVert_{C_{t} \mathscr{C}^{-\frac{1}{4} - \kappa}}  t^{\frac{\kappa}{5}}  \overset{\eqref{Define Lt kappa and Nt kappa}}{\lesssim} \lVert w \rVert_{\dot{H}^{\frac{1}{4} + \frac{3\kappa}{2}}}  N_{t}^{\kappa}.
\end{equation} 
On the other hand, we first write by \eqref{est 92}, 
\begin{equation*}
 \Lambda^{\frac{5}{2}} (w \circlesign{\prec}_{s} Q^{\mathcal{H}}) - w \circlesign{\prec}_{s} \Lambda^{\frac{5}{2}} Q^{\mathcal{H}} =  \sum_{m\geq -1} [\Lambda^{\frac{5}{2}} \left((S_{m-1} w) \otimes_{s} \Delta_{m} Q^{\mathcal{H}} \right) -  (S_{m-1} w) \otimes_{s} \Delta_{m} \Lambda^{\frac{5}{2}} Q^{\mathcal{H}} ],
\end{equation*} 
rely on \eqref{Kato-Ponce} and Sobolev embedding of $\dot{H}^{\kappa} (\mathbb{T}^{3}) \hookrightarrow L^{\frac{6}{3-2\kappa}} (\mathbb{T}^{3})$, and Young's inequality for convolution to estimate 
\begin{align}
C_{42}\lesssim& \sum_{m\geq -1} \lVert S_{m-1} w \rVert_{\dot{H}^{1+ \kappa}} \lVert \Lambda^{\frac{3}{2}} \Delta_{m} Q^{\mathcal{H}} \rVert_{L^{\frac{3}{\kappa}}} + \lVert S_{m-1} w \rVert_{\dot{H}^{\frac{5}{2}}} \lVert \Delta_{m} Q^{\mathcal{H}} \rVert_{L^{\infty}} \nonumber  \\
\lesssim& \left( \sum_{m\geq -1} 2^{- \lvert m \rvert (\frac{3}{4} - \frac{3\kappa}{2})} \ast_{m} 2^{m( \frac{1}{4} + \frac{5\kappa}{2})} \lVert \Delta_{m} w \rVert_{L^{2}} \right) \lVert Q^{\mathcal{H}}  \rVert_{\mathscr{C}^{\frac{9}{4} - \frac{3\kappa}{2}}}   \nonumber \\
& + \left( \sum_{m\geq -1} 2^{- \lvert m \rvert ( \frac{9}{4} - \frac{3\kappa}{2} )} \ast_{m} 2^{m ( \frac{1}{4} + \frac{3\kappa}{2})} \lVert \Delta_{m} w \rVert_{L^{2}} \right) \lVert Q^{\mathcal{H}} \rVert_{\mathscr{C}^{\frac{9}{4} - \frac{3\kappa}{2}}}   \nonumber \\
\overset{\eqref{Regularity of Q}}{\lesssim}&  \lVert X \rVert_{C_{t} \mathscr{C}^{-\frac{1}{4} - \kappa}} t^{\frac{\kappa}{5}}  \lVert 2^{- \lvert m \rvert (\frac{3}{4} - \frac{3\kappa}{2})} \rVert_{l_{m \geq -1}^{1}} \lVert w \rVert_{B_{2,1}^{\frac{1}{4} + \frac{5\kappa}{2}}} \overset{\eqref{Define Lt kappa and Nt kappa}}{\lesssim} N_{t}^{\kappa} \lVert w \rVert_{\dot{H}^{\frac{1}{4} + 3 \kappa}},  \label{Estimate on C42} 
\end{align}
where the last inequality used the fact that $B_{p,q}^{s+\epsilon} \subset B_{p,\infty}^{s}$ for any $q \in [1,\infty]$ as long as $\epsilon > 0$.  Applying \eqref{Estimate on C41} and \eqref{Estimate on C42} to \eqref{Split C4} allows us to estimate 
\begin{align}
C_{4}  \lesssim& \lVert w^{\mathcal{L}} \rVert_{\dot{H}^{1}}   \left( \lVert w \rVert_{\dot{H}^{\frac{1}{4} + \frac{3\kappa}{2}}}  N_{t}^{\kappa} + N_{t}^{\kappa} \lVert w \rVert_{\dot{H}^{\frac{1}{4} + 3 \kappa}}  \right) \label{Estimate on C4} \\
\overset{\eqref{Estimate on wH}}{\lesssim}& N_{t}^{\kappa} \lVert w^{\mathcal{L}} \rVert_{\dot{H}^{1}} \left( \lVert w^{\mathcal{L}} \rVert_{\dot{H}^{\frac{1}{4} + 3\kappa}} + (1+ \lVert w \rVert_{L^{2}})^{1- \tau (1- \frac{7\kappa}{2})}  N_{t}^{\kappa} t^{\frac{\kappa}{5}} \right) \lesssim N_{t}^{\kappa} \lVert w^{\mathcal{L}} \rVert_{\dot{H}^{1}} [ \lVert w^{\mathcal{L}} \rVert_{\dot{H}^{\frac{1}{4} + 3 \kappa}} + N_{t}^{\kappa}] \nonumber 
\end{align} 
where the last inequality relied on the fact that $1-\tau (1-5 \kappa) \leq 0$ due to our hypothesis of $\tau \geq \frac{40}{13}$ and $\kappa < \frac{1}{200}$. 
We can now conclude by applying \eqref{est 91}, \eqref{Estimate on C2}, \eqref{Estimate on C3}, and \eqref{Estimate on C4} to \eqref{Split to Ck} that 
\begin{equation}\label{est 94}
2 \langle w^{\mathcal{L}}, \divergence C^{\circlesign{\prec}_{s}}(w, Q^{\mathcal{H}}) \rangle \lesssim ( N_{t}^{\kappa})^{3} \lVert w^{\mathcal{L}} \rVert_{\dot{H}^{\eta}} \lVert w^{\mathcal{L}} \rVert_{\dot{H}^{\frac{5}{4}}}^{\frac{3}{5} - \frac{2\kappa}{5}} + (N_{t}^{\kappa})^{2} \lVert w^{\mathcal{L}} \rVert_{\dot{H}^{1}} [ \lVert w^{\mathcal{L}} \rVert_{\dot{H}^{\frac{1}{4} + 3 \kappa}} + N_{t}^{\kappa} ]. 
\end{equation} 
At last, applying \eqref{est 93} and \eqref{est 94} to \eqref{est 88} gives us 
\begin{align}
\RomanI_{4}+  2 \langle w^{\mathcal{L}}, \divergence (w^{\otimes 2}) \rangle  \lesssim& N_{t}^{\kappa} \lVert w^{\mathcal{L}} \rVert_{\dot{H}^{1- \frac{3\kappa}{2}}} ( \lVert w^{\mathcal{L}} \rVert_{\dot{H}^{2\kappa}} + N_{t}^{\kappa}) + (N_{t}^{\kappa})^{2} \lVert w^{\mathcal{L}} \rVert_{\dot{H}^{1- \frac{3\kappa}{2}}}  \nonumber \\
&+  ( N_{t}^{\kappa})^{3} \lVert w^{\mathcal{L}} \rVert_{\dot{H}^{\eta}} \lVert w^{\mathcal{L}} \rVert_{\dot{H}^{\frac{5}{4}}}^{\frac{3}{5} - \frac{2\kappa}{5}} + (N_{t}^{\kappa})^{2} \lVert w^{\mathcal{L}} \rVert_{\dot{H}^{1}} [ \lVert w^{\mathcal{L}} \rVert_{\dot{H}^{\frac{1}{4} + 3 \kappa}} + N_{t}^{\kappa} ],
\end{align}
which verifies \eqref{est 95}. This completes the proof of Proposition \ref{Proposition 4.7}.  
\end{proof}

Applying \eqref{est 97}, \eqref{est 98}, \eqref{est 83}, \eqref{est 84}, and \eqref{est 95} to \eqref{est 96}, we have the following corollary. 

\begin{corollary}
Fix $\lambda_{t}$ from \eqref{Define lambda t} with $\tau = \frac{40}{13}$ and $\kappa_{0} \in (0, \frac{1}{200})$. Then there exists a constant $C >0$ such that for all $\kappa \in (0, \kappa_{0}]$, all $i \in \mathbb{N}_{0}$, and all $t \in [T_{i}, T_{i+1})$, 
\begin{align}
\partial_{t} \lVert w^{\mathcal{L}}(t) \rVert_{L^{2}}^{2} \leq& - \nu \lVert w^{\mathcal{L}} \rVert_{\dot{H}^{\frac{5}{4}}}^{2} + 2 \langle w^{\mathcal{L}}, \mathcal{A}_{t}^{\lambda_{t}} w^{\mathcal{L}} \rangle (t) + r_{\lambda} (t) \lVert w^{\mathcal{L}}(t) \rVert_{L^{2}}^{2} \nonumber \\
&+ C  \Bigg( ( N_{t}^{\kappa})^{2}  \lambda_{t}^{\frac{9}{32}} \lVert w^{\mathcal{L}}(t) \rVert_{\dot{H}^{\frac{31}{32} + 2 \kappa}} + (N_{t}^{\kappa})^{2} ( \lVert w^{\mathcal{L}} \rVert_{\dot{H}^{1- \frac{3\kappa}{2}}} + \lVert w^{\mathcal{L}} \rVert_{\dot{H}^{1-\frac{3\kappa}{2}}}^{2})  \nonumber \\
& \hspace{5mm} + N_{t}^{\kappa} \lVert w^{\mathcal{L}} \rVert_{\dot{H}^{\frac{31}{32} + 2 \kappa}} \lVert w^{\mathcal{L}} \rVert_{\dot{H}^{\frac{5}{4}}}^{\frac{3}{5} - \frac{2\kappa}{5}} + (N_{t}^{\kappa})^{2} \lVert w^{\mathcal{L}} \rVert_{\dot{H}^{1}} [ \lVert w^{\mathcal{L}} (t) \rVert_{\dot{H}^{\frac{1}{4} + 3\kappa}} + N_{t}^{\kappa} ] \Bigg). \label{est 99}
\end{align} 
\end{corollary} 

\begin{proposition}\label{Proposition 4.9}
Fix $\lambda_{t}$ from \eqref{Define lambda t} with $\tau = \frac{40}{13}$ and $\kappa_{0} \in (0, \frac{1}{200})$. Then, there exists a constant $C_{1} > 0$ and increasing continuous functions $C_{2}$ and $C_{3}$ from $\mathbb{R}_{\geq 0}$ to $\mathbb{R}_{\geq 0}$, specifically 
\begin{align}\label{Define C2 and C3} 
C_{2}(N_{t}^{\kappa}) \approx (N_{t}^{\kappa})^{5} + \mathbf{m}(N_{t}^{\kappa}) \hspace{1mm} \text{ and } C_{3}(N_{t}^{\kappa}) \approx (N_{t}^{\kappa})^{\frac{10}{1+ 6 \kappa}}  
\end{align}
where $\mathbf{m}$ is the map from Proposition \ref{Proposition on Anderson Hamiltonian}, such that for all $\kappa \in (0, \kappa_{0}]$, all $i \in \mathbb{N}_{0}$ such that $i \geq i_{0} (u^{\text{in}})$ for $i_{0}(u^{\text{in}})$ from \eqref{Define i0}, and all $t \in [T_{i}, T_{i+1})$, 
\begin{align}
\partial_{t} \lVert w^{\mathcal{L}}(t) \rVert_{L^{2}}^{2} \leq& - \frac{\nu}{2} \lVert w^{\mathcal{L}}(t) \rVert_{\dot{H}^{\frac{5}{4}}}^{2}  \nonumber  \\
&+ \left( C_{1} \ln(\lambda_{t}) + C_{2} (N_{t}^{\kappa}) \right) [ \lVert w^{\mathcal{L}} (t) \rVert_{L^{2}}^{2} + \lVert w^{\mathcal{L}} (T_{i}) \rVert_{L^{2}}^{2} ] + C_{3} (N_{t}^{\kappa}), \label{est 105}
\end{align}  
so that 
\begin{equation}\label{est 106}
\sup_{t \in [T_{i}, T_{i+1})} \lVert w^{\mathcal{L}}(t) \rVert_{L^{2}}^{2} + \frac{\nu}{2} \int_{T_{i}}^{T_{i+1}} \lVert w^{\mathcal{L}} (s) \rVert_{\dot{H}^{\frac{5}{4}}}^{2} ds \leq e^{\mu (T_{i+1} - T_{i})}[  \lVert w^{\mathcal{L}}(T_{i}) \rVert_{L^{2}}^{2} + C_{3} (N_{T_{i+1}}^{\kappa})], 
\end{equation} 
where 
\begin{equation}\label{Define mu}
\mu \triangleq C_{1} \ln(\lambda_{T_{i}}) + C_{2} (N_{T_{i+1}}^{\kappa}).
\end{equation}
\end{proposition}

\begin{proof}[Proof of Proposition \ref{Proposition 4.9}]
We fix an arbitrary $t \in [T_{i}, T_{i+1})$. In contrast to \cite{HR24, Y23c, Y25d}, we have an extra $r_{\lambda}^{2}(t)$ of \eqref{Define r lambda1} as part of the $r_{\lambda}(t)$ inside our estimate \eqref{est 99}; in any event, every entry of $r_{\lambda}(t) \Id$ is non-negative. Therefore, applying Proposition \ref{Proposition on Anderson Hamiltonian}, similarly to \cite[p. 23]{HR24} (also \cite[Equation (120a)]{Y23c}), for $\mathbf{m}$ from Proposition \ref{Proposition on Anderson Hamiltonian} (2), 
\begin{equation}\label{est 102}
 \langle w^{\mathcal{L}}, \mathcal{A}_{t}^{\lambda_{t}}w^{\mathcal{L}} \rangle(t) \leq \mathbf{m}(N_{t}^{\kappa}) \lVert w^{\mathcal{L}}\rVert_{L^{2}}^{2}. 
\end{equation} 
Additionally, our hypothesis of $\tau = \frac{40}{13}$ crucially indicates that $\tau (\frac{9}{32}) = \frac{45}{52} < 1$ so that  
\begin{equation}\label{est 103}
\lambda_{t}^{\frac{9}{32}} \lesssim \lVert w^{\mathcal{L}} (t) \rVert_{L^{2}} + 1 
\end{equation} 
due to \eqref{Estimate on wH}. Applying \eqref{est 102}, \eqref{est 103}, and \eqref{logarithmic growth} to \eqref{est 99} gives us for the constant $c> 0$ from \eqref{logarithmic growth},  
\begin{align}
\partial_{t} \lVert w^{\mathcal{L}}(t) \rVert_{L^{2}}^{2} \leq& - \nu \lVert w^{\mathcal{L}} \rVert_{\dot{H}^{\frac{5}{4}}}^{2} + 2\mathbf{m}(N_{t}^{\kappa}) \lVert w^{\mathcal{L}} \rVert_{L^{2}}^{2} + c \ln(\lambda_{t})\lVert w^{\mathcal{L}}(t) \rVert_{L^{2}}^{2} \nonumber \\
&+ C  \Bigg( ( N_{t}^{\kappa})^{2} [\lVert w^{\mathcal{L}}(t) \rVert_{L^{2}} +1] \lVert w^{\mathcal{L}}(t) \rVert_{\dot{H}^{\frac{31}{32} + 2 \kappa}} + (N_{t}^{\kappa})^{2} ( \lVert w^{\mathcal{L}} \rVert_{\dot{H}^{1- \frac{3\kappa}{2}}} + \lVert w^{\mathcal{L}} \rVert_{\dot{H}^{1-\frac{3\kappa}{2}}}^{2})  \nonumber \\
& \hspace{5mm} + N_{t}^{\kappa} \lVert w^{\mathcal{L}} \rVert_{\dot{H}^{\frac{31}{32} + 2 \kappa}} \lVert w^{\mathcal{L}} \rVert_{\dot{H}^{\frac{5}{4}}}^{\frac{3}{5} - \frac{2\kappa}{5}} + (N_{t}^{\kappa})^{2} \lVert w^{\mathcal{L}} \rVert_{\dot{H}^{1}} [ \lVert w^{\mathcal{L}} (t) \rVert_{\dot{H}^{\frac{1}{4} + 3\kappa}} + N_{t}^{\kappa} ] \Bigg). \label{est 104}
\end{align}
Now we apply a series of Gagliardo-Nirenberg inequalities, namely 
\begin{align*}
& \lVert f(t) \rVert_{\dot{H}^{\frac{31}{32} + 2 \kappa}} \lesssim \lVert f(t) \rVert_{L^{2}}^{\frac{9 - 64 \kappa}{40}} \lVert f (t) \rVert_{\dot{H}^{\frac{5}{4}}}^{\frac{31 + 64 \kappa}{40}}, \hspace{3mm} \lVert f(t) \rVert_{\dot{H}^{1-\frac{3\kappa}{2}}} \lesssim  \lVert f(t) \rVert_{L^{2}}^{\frac{1+ 6 \kappa}{5}} \lVert f (t) \rVert_{\dot{H}^{\frac{5}{4}}}^{\frac{4-6\kappa}{5}}, \\
&\lVert f(t) \rVert_{\dot{H}^{1}}  \lesssim \lVert f(t) \rVert_{L^{2}}^{\frac{1}{5}} \lVert f(t) \rVert_{\dot{H}^{\frac{5}{4}}}^{\frac{4}{5}}, \hspace{13mm} \lVert f(t) \rVert_{\dot{H}^{\frac{1}{4} + 3 \kappa}}  \lesssim  \lVert f(t) \rVert_{L^{2}}^{\frac{4}{5} - \frac{12 \kappa}{5}} \lVert f (t) \rVert_{\dot{H}^{\frac{5}{4}}}^{\frac{1 + 12 \kappa}{5}},
\end{align*}
and Young's inequalities to to \eqref{est 104} to deduce \eqref{est 105} as follows:
\begin{align*}
\partial_{t} \lVert w^{\mathcal{L}}(t) \rVert_{L^{2}}^{2} \leq&  - \frac{\nu}{2} \lVert w^{\mathcal{L}} \rVert_{\dot{H}^{\frac{5}{4}}}^{2} + 2\mathbf{m}(N_{t}^{\kappa}) \lVert w^{\mathcal{L}} \rVert_{L^{2}}^{2} + c \ln(\lambda_{t})\lVert w^{\mathcal{L}}(t) \rVert_{L^{2}}^{2} \nonumber \\
&+ C  \Bigg(   \left( N_{t}^{\kappa} \right)^{\frac{160}{49 - 64 \kappa}}\left[ \lVert w^{\mathcal{L}}(t) \rVert_{L^{2}}^{2} + 1 \right]  +   \lVert w^{\mathcal{L}}(t) \rVert_{L^{2}}^{2} + (N_{t}^{\kappa})^{2( \frac{5}{1+ 6 \kappa})}  \nonumber \\
& \hspace{5mm} + (N_{t}^{\kappa})^{\frac{80}{25 - 48 \kappa}} ( \lVert w^{\mathcal{L}}(t) \rVert_{L^{2}}^{2} + 1) + ( N_{t}^{\kappa})^{5} \left( \lVert w^{\mathcal{L}}(t) \rVert_{L^{2}}^{2} + 1 \right) \Bigg)\\
\leq& -\frac{\nu}{2} \lVert w^{\mathcal{L}}(t) \rVert_{\dot{H}^{\frac{5}{4}}}^{2} + \left( C_{1} \ln(\lambda_{t}) + C_{2}(N_{t}^{\kappa}) \right) [ \lVert w^{\mathcal{L}} (t) \rVert_{L^{2}}^{2} + \lVert w^{\mathcal{L}} (T_{i}) \rVert_{L^{2}}^{2} ] + C_{3}(N_{t}^{\kappa}). 
\end{align*}
Next, using the fact that $\mu \geq 1$, for all $t \in [T_{i}, T_{i+1})$,
\begin{equation}\label{est 107}
\lVert w^{\mathcal{L}}(t) \rVert_{L^{2}}^{2}  + \frac{\nu}{2} \int_{T_{i}}^{t} \lVert w^{\mathcal{L}}(s) \rVert_{\dot{H}^{\frac{5}{4}}}^{2} ds \leq e^{\mu (T_{i+1} - T_{i})} [  \lVert w^{\mathcal{L}}(T_{i}) \rVert_{L^{2}}^{2} + C_{3} (N_{T_{i+1}}^{\kappa})], 
\end{equation} 
on which taking supremum over all $t \in [T_{i}, T_{i+1})$ gives us \eqref{est 106}. This completes the proof of Proposition \ref{Proposition 4.9}. 
\end{proof} 

\begin{proposition}\label{Proposition 4.10}
Fix $\lambda_{t}$ from \eqref{Define lambda t} with $\tau = \frac{40}{13}$ and $\kappa_{0} \in (0, \frac{1}{200})$. If $T_{i+1} < T^{\max} \wedge t$ for any $i \in \mathbb{N}$ such that $i \geq i_{0} (u^{\text{in}})$ with $i_{0}(u^{\text{in}})$ from \eqref{Define i0} and any $t > 0$, then, for all $\kappa \in (0, \kappa_{0}]$, there exist 
\begin{equation}
C(N_{t}^{\kappa}) = 2 C N_{t}^{\kappa} \text{ and } \tilde{C}(N_{t}^{\kappa}) \triangleq \max \left\{ C_{1} (\frac{40}{13}), 2C_{2}(N_{t}^{\kappa}), C N_{t}^{\kappa} + C^{2} (N_{t}^{\kappa})^{2} + C_{3} (N_{t}^{\kappa}) \right\}
\end{equation} 
where $C$ is the universal constant from \eqref{Estimate on wH} and $C_{1}, C_{2}, C_{3}$ from Proposition \ref{Proposition 4.9}, such that 
\begin{equation}\label{est 110}
T_{i+1} - T_{i} \geq \frac{1}{ \tilde{C} (N_{t}^{\kappa}) (\ln(1+i) + 1)} \ln \left( \frac{i^{2} + 2i - C(N_{t}^{\kappa})}{i^{2} + \tilde{C} (N_{t}^{\kappa} )} \right). 
\end{equation} 
\end{proposition}

\begin{proof}[Proof of Proposition \ref{Proposition 4.10}]
We compute from \eqref{est 107}, using \eqref{Define mu}, 
\begin{equation}\label{est 108}
T_{i+1} - T_{i} \geq \frac{1}{C_{1} \ln(\lambda_{T_{i}}) + C_{2} (N_{t}^{\kappa})} \ln \left( \frac{ \lVert w^{\mathcal{L}} (T_{i+1} - ) \rVert_{L^{2}}^{2}}{  \lVert w^{\mathcal{L}} (T_{i}) \rVert_{L^{2}}^{2} + C_{3} (N_{t}^{\kappa} ) } \right). 
\end{equation} 
Moreover, applications of \eqref{Estimate on wH} and \eqref{Define T0 and Ti} give us 
\begin{equation}\label{est 109} 
\lVert w^{\mathcal{L}}(T_{i+1} - ) \rVert_{L^{2}} \geq i +1 - C(i+1)^{-1} N_{t}^{\kappa}  \hspace{3mm} \text{ and } \hspace{3mm} \lVert w^{\mathcal{L}}\left(T_{i}) \rVert_{L^{2}} \leq i + C ( \frac{N_{t}^{\kappa}}{i} \right). 
\end{equation} 
Applying \eqref{est 109} to \eqref{est 108} gives us the desired \eqref{est 110}. This completes the proof of Proposition \ref{Proposition 4.10}. 
\end{proof}

\subsection{Higher-order estimates}
We come to the $\dot{H}^{\epsilon}(\mathbb{T}^{3})$-estimate of $w^{\mathcal{L}}$ next. 
\begin{remark}\label{Remark 4.3}
In Proposition \ref{Proposition 4.7}, we have already achieved an estimate on the difficult commutator term. The analogous proof of $\dot{H}^{\epsilon}(\mathbb{T}^{2})$-estimate in \cite[Lemma 5.3]{HR24} for the 2D Navier-Stokes equations forced by STWN follows its $L^{2}(\mathbb{T}^{2})$-estimate similarly. Therefore, we are tempted to feel that $\dot{H}^{\epsilon}(\mathbb{T}^{3})$-estimate should impose no significant difficulties. In fact, as we elaborated in Remark \ref{Remark 2.1}, extending the local solution to global-in-time required only $\epsilon > 0$ in the 2D case while it will require $\epsilon > \frac{1}{2}$ in the 3D case. A close inspection of the proof of \cite[Lemma 5.3]{HR24} shows that their restriction of ``$\epsilon \in (0, \kappa)$'' stems from the low regularity of their ``$Y \in \mathscr{C}^{2\kappa}$'' on \cite[p. 8]{HR24} which in turn is rooted from their additional force $\zeta$. As can be seen in \eqref{Regularity of Y}, we have $Y \in \mathscr{C}^{1 - 3 \kappa} (\mathbb{T}^{3})$ for any $\kappa > 0$. This suggests that, if all other terms do not create major difficulties, we can expect to obtain $\dot{H}^{\epsilon}(\mathbb{T}^{3})$-estimate for $\epsilon < 1$.  

Unfortunately, once again, the commutator stands in our way. Recall that we only had to consider ``$\lVert w^{\mathcal{L}} \rVert_{\dot{H}^{1}} \lVert \Lambda^{\frac{5}{2}} (w \circlesign{\prec}_{s} Q^{\mathcal{H}}) - (\Lambda^{\frac{5}{2}} w) \circlesign{\prec}_{s} Q^{\mathcal{H}} - w \circlesign{\prec}_{s} \Lambda^{\frac{5}{2}} Q^{\mathcal{H}}  \rVert_{L^{2}}$'' in \eqref{Split C4}. In the $\dot{H}^{\epsilon}(\mathbb{T}^{3})$-estimate, informally we have diffusion of $\lVert w^{\mathcal{L}} \rVert_{\dot{H}^{\frac{5}{4} + \epsilon}}^{2}$ and a nonlinear term of 
\begin{align*}
 2 \langle \Lambda^{2\epsilon} w^{\mathcal{L}}, \divergence \left( \Lambda^{\frac{5}{2}} (w \circlesign{\prec}_{s} Q^{\mathcal{H}}) - (\Lambda^{\frac{5}{2}} w) \circlesign{\prec}_{s} Q^{\mathcal{H}} - w \circlesign{\prec}_{s} \Lambda^{\frac{5}{2}} Q^{\mathcal{H}} \right) \rangle 
\end{align*}
(cf. \eqref{Define C4}). If we bound this by 
\begin{align*}
 \lVert w^{\mathcal{L}} \rVert_{\dot{H}^{1+2\epsilon}} \lVert \Lambda^{\frac{5}{2}} (w \circlesign{\prec}_{s} Q^{\mathcal{H}}) - (\Lambda^{\frac{5}{2}} w) \circlesign{\prec}_{s} Q^{\mathcal{H}} - w \circlesign{\prec}_{s} \Lambda^{\frac{5}{2}} Q^{\mathcal{H}}  \rVert_{L^{2}},
\end{align*}
then the diffusion $\lVert w^{\mathcal{L}} \rVert_{\dot{H}^{\frac{5}{4} + \epsilon}}^{2}$ will demand $1+ 2 \epsilon \leq \frac{5}{4} + \epsilon$ or equivalently $\epsilon \leq \frac{1}{4}$ which is not good enough as we need $\epsilon >\frac{1}{2}$. This shows that we have no choice but to apply some derivative, at least $\Lambda^{\alpha}$ for $\alpha > \frac{1}{4}$, on $\divergence \left( \Lambda^{\frac{5}{2}} (w \circlesign{\prec}_{s} Q^{\mathcal{H}}) - (\Lambda^{\frac{5}{2}} w) \circlesign{\prec}_{s} Q^{\mathcal{H}} - w \circlesign{\prec}_{s} \Lambda^{\frac{5}{2}} Q^{\mathcal{H}} \right)$, although we are not aware of \eqref{Li and D'Ancona} in $\dot{H}^{s}$-norm in the literature, as we elaborated in Remark \ref{Remark 2.2} (a)-(b).  We overcome this obstacle again via careful interpolation of Besov space taking advantage of ``$\circlesign{\prec}_{s}$'' that is present inside the commutator throughout, and obtain the necessary estimate for any $\epsilon < \frac{9}{10}$. With our heuristic in mind, we believe that  $\frac{9}{10}$ is not sharp and we may be able to improve such a restriction of $\epsilon < \frac{9}{10}$ to $\epsilon < 1$; however, we choose to not pursue this direction because any $\epsilon > \frac{1}{2}$ suffices for our subsequent purpose. 
\end{remark} 

\begin{proposition}\label{Proposition 4.11}
Fix $\lambda_{t}$ from \eqref{Define lambda t} with $\tau =\frac{40}{13}$ and $\kappa_{0} \in (0, \frac{1}{200})$. Then the following holds for any $\kappa \in (0, \kappa_{0})$ and $\epsilon \in (0, \frac{9}{10})$. Suppose that there exist $M > 1$ and $T > 0$ such that 
\begin{equation}\label{Define M}
\lVert w^{\mathcal{L}}(0) \rVert_{\dot{H}^{\epsilon}}^{2} +  \sup_{t \in [0, T \wedge T^{\max}]} \lVert w^{\mathcal{L}} (t) \rVert_{L^{2}}^{2} +\frac{\nu}{2} \int_{0}^{T \wedge T^{\max}} \lVert w^{\mathcal{L}} (s) \rVert_{\dot{H}^{\frac{5}{4}}}^{2} ds \leq M. 
\end{equation} 
Then, there exists $C(T, M, N_{T}^{\kappa}) \in (0,\infty)$ such that 
\begin{equation}\label{est 139}
\sup_{t\in [0, T \wedge T^{\max} ]} \lVert w^{\mathcal{L}} (t) \rVert_{\dot{H}^{\epsilon}}^{2} \leq C( T, M, N_{T}^{\kappa}). 
\end{equation} 
\end{proposition}

\begin{proof}[Proof of Proposition \ref{Proposition 4.11}]
Identically to the derivation of \eqref{est 96}-\eqref{Define Ik}, we obtain  
\begin{equation}\label{est 111} 
 \partial_{t} \lVert w^{\mathcal{L}} (t) \rVert_{\dot{H}^{\epsilon}}^{2} = \sum_{k=1}^{4} \RomanII_{k}, 
\end{equation} 
where 
\begin{subequations}\label{Define IIk}
\begin{align}
& \RomanII_{1} \triangleq 2 \left\langle (-\Delta)^{\epsilon} w^{\mathcal{L}}, - \nu \Lambda^{\frac{5}{2}} w^{\mathcal{L}} - \divergence  \left( 2 ( \mathcal{L}_{\lambda_{t}} X) \otimes_{s} w^{\mathcal{L}} \right) \right\rangle (t),  \label{Define II1} \\
& \RomanII_{2} \triangleq -2 \left\langle (-\Delta)^{\epsilon}w^{\mathcal{L}}, \divergence \left( 2 ( \mathcal{H}_{\lambda_{t}} X) \otimes_{s} w^{\mathcal{L}} - 2 ( \mathcal{H}_{\lambda_{t}} X) \circlesign{\succ}_{s} w^{\mathcal{L}} \right) \right\rangle(t), \label{Define II2} \\
& \RomanII_{3} \triangleq -2 \left\langle (-\Delta)^{\epsilon}w^{\mathcal{L}}, \divergence \left( 2X \otimes_{s} w^{\mathcal{H}} - 2 (\mathcal{H}_{\lambda_{t}} X) \circlesign{\succ}_{s} w^{\mathcal{H}} \right) \right\rangle(t), \label{Define II3}\\
& \RomanII_{4} \triangleq -2 \left\langle (-\Delta)^{\epsilon}w^{\mathcal{L}}, \divergence \left( w^{\otimes 2} + 2 Y \otimes_{s} w - \mathcal{C}^{\circlesign{\prec}_{s}} (w, Q^{\mathcal{H}}) + Y^{\otimes 2} \right) \right\rangle (t).  \label{Define II4}
\end{align}
\end{subequations}
We start with our estimate on $\RomanII_{1}$ from \eqref{Define II1} using \eqref{Bony's decomposition}:
\begin{align}
\RomanII_{1} \leq& - 2 \nu \lVert w^{\mathcal{L}} \rVert_{\dot{H}^{\frac{5}{4} + \epsilon}}^{2}  \nonumber \\
&+ C \lVert w^{\mathcal{L}} \rVert_{\dot{H}^{1}}  \left( \lVert \mathcal{L}_{\lambda_{t}}X  \circlesign{\succ}_{s} w^{\mathcal{L}} \rVert_{\dot{H}^{2\epsilon}} + \lVert \mathcal{L}_{\lambda_{t}}X \circlesign{\prec}_{s} w^{\mathcal{L}} \rVert_{\dot{H}^{2\epsilon}}  + \lVert \mathcal{L}_{\lambda_{t}}X \circlesign{\circ}_{s} w^{\mathcal{L}} \rVert_{\dot{H}^{2\epsilon}} \right). \label{est 113}
\end{align}
We can compute using the definition of $\lambda_{t}$ from \eqref{Define lambda t}, via Gagliardo-Nirenberg inequality of $\lVert f \rVert_{\dot{H}^{2\epsilon}} \lesssim \lVert f \rVert_{\dot{H}^{\epsilon}}^{1- \frac{4\epsilon}{5}} \lVert f \rVert_{\dot{H}^{\frac{5}{4} + \epsilon}}^{\frac{4\epsilon}{5}}$, 
\begin{subequations}\label{est 112} 
\begin{align}
& \lVert \mathcal{L}_{\lambda_{t}}X \circlesign{\succ}_{s} w^{\mathcal{L}} \rVert_{\dot{H}^{2\epsilon}} \overset{\eqref{Sobolev products c}}{\lesssim}  \lambda_{t}^{\frac{13}{40} + 2 \epsilon} \lVert X(t) \rVert_{\mathscr{C}^{-\frac{1}{4} - \kappa}} \lVert w^{\mathcal{L}} (t) \rVert_{H^{-\frac{3}{40} + \kappa}} 
\overset{ \eqref{Define Lt kappa and Nt kappa} \eqref{Define M} \eqref{Estimate on wH}}{\leq}  C(M, N_{t}^{\kappa}), \\
& \lVert \mathcal{L}_{\lambda_{t}}X \circlesign{\prec}_{s} w^{\mathcal{L}} \rVert_{\dot{H}^{2\epsilon}} + \lVert \mathcal{L}_{\lambda_{t}}X \circlesign{\circ}_{s} w^{\mathcal{L}} \rVert_{\dot{H}^{2\epsilon}}  \overset{\eqref{Sobolev products b}\eqref{Sobolev products e}}{\lesssim} \lVert \mathcal{L}_{\lambda_{t}}X \rVert_{L^{\infty}} \lVert w^{\mathcal{L}} \rVert_{\dot{H}^{2\epsilon}}  \\
& \hspace{10mm} \overset{\eqref{Define lambda t}\eqref{Define Lt kappa and Nt kappa}}{\lesssim} (1+ \lVert w(T_{i})  \rVert_{L^{2}})N_{t}^{\kappa} \lVert w^{\mathcal{L}} \rVert_{\dot{H}^{\epsilon}}^{1- \frac{4\epsilon}{5}} \lVert w^{\mathcal{L}} \rVert_{\dot{H}^{\frac{5}{4} + \epsilon}}^{\frac{4\epsilon}{5}} \overset{\eqref{Define M}}{\leq} C(M, N_{t}^{\kappa}) \lVert w^{\mathcal{L}} \rVert_{\dot{H}^{\epsilon}}^{1- \frac{4\epsilon}{5}} \lVert w^{\mathcal{L}} \rVert_{\dot{H}^{\frac{5}{4} + \epsilon}}^{\frac{4\epsilon}{5}}. \nonumber 
\end{align}
\end{subequations} 
By applying \eqref{est 112} to \eqref{est 113} we conclude
\begin{equation}\label{est 135}
\RomanII_{1} \leq - \frac{31 \nu}{16} \lVert w^{\mathcal{L}} \rVert_{\dot{H}^{\frac{5}{4} + \epsilon}}^{2} + C(M, N_{t}^{\kappa}) (1+ \lVert w^{\mathcal{L}} \rVert_{\dot{H}^{\epsilon}}^{2}).
\end{equation} 

Next, we rewrite $\RomanII_{2}$ in \eqref{Define II2} as 
\begin{align*}
\RomanII_{2} = -2 \left\langle (-\Delta)^{\epsilon}w^{\mathcal{L}}, \divergence \left( 2 ( \mathcal{H}_{\lambda_{t}} X) \circlesign{\prec}_{s} w^{\mathcal{L}} + 2 ( \mathcal{H}_{\lambda_{t}} X) \circlesign{\circ}_{s} w^{\mathcal{L}} \right) \right\rangle(t)
\end{align*}
and estimate using \eqref{Sobolev products d}, \eqref{Sobolev products e}, \eqref{Define Lt kappa and Nt kappa}, and \eqref{Define M}, along with interpolation inequalities of $\lVert f \rVert_{\dot{H}^{1+ \epsilon}} \lesssim \lVert f \rVert_{L^{2}}^{\frac{1+ \epsilon}{(\frac{5}{4} + \epsilon)}} \lVert f \rVert_{\dot{H}^{\frac{5}{4} + \epsilon}}^{\frac{1+ \epsilon}{(\frac{5}{4} + \epsilon)}}$ and $\lVert f \rVert_{\dot{H}^{\frac{1}{4} + \kappa + \epsilon}} \lesssim \lVert f \rVert_{L^{2}}^{\frac{1- \kappa}{(\frac{5}{4} + \epsilon)}} \lVert f \rVert_{\dot{H}^{\frac{5}{4} + \epsilon}}^{\frac{\frac{1}{4} + \kappa + \epsilon}{(\frac{5}{4} + \epsilon)}}$, 
\begin{align}
\RomanII_{2}  \lesssim& \lVert w^{\mathcal{L}} \rVert_{\dot{H}^{1+ \epsilon}} \lVert \mathcal{H}_{\lambda_{t}}X \rVert_{\mathscr{C}^{-\frac{1}{4} - \kappa}} \lVert w^{\mathcal{L}} \rVert_{H^{\frac{1}{4} + \kappa + \epsilon}} \lesssim N_{t}^{\kappa} \lVert w^{\mathcal{L}} \rVert_{L^{2}}^{\frac{\frac{1}{4}}{(\frac{5}{4} + \epsilon)}} \lVert w^{\mathcal{L}} \rVert_{\dot{H}^{\frac{5}{4} + \epsilon}}^{\frac{1+ \epsilon}{(\frac{5}{4} + \epsilon)}} \lVert w^{\mathcal{L}} \rVert_{L^{2}}^{\frac{1-\kappa}{(\frac{5}{4} + \epsilon)}} \lVert w^{\mathcal{L}} \rVert_{\dot{H}^{\frac{5}{4} + \epsilon}}^{\frac{\frac{1}{4} + \kappa + \epsilon}{(\frac{5}{4} + \epsilon)}}   \nonumber \\
& \hspace{10mm} \leq C(M, N_{t}^{\kappa}) \lVert w^{\mathcal{L}} \rVert_{\dot{H}^{\frac{5}{4} + \epsilon}}^{\frac{\frac{5}{4} + \kappa + 2 \epsilon}{(\frac{5}{4} + \epsilon)}}\leq \frac{\nu}{16} \lVert w^{\mathcal{L}} \rVert_{\dot{H}^{\frac{5}{4} + \epsilon}}^{2} + C(M, N_{t}^{\kappa}). \label{est 136}
\end{align} 

Next, concerning $\RomanII_{3}$ in \eqref{Define II3}, we first rewrite as 
\begin{equation}\label{Split II3}
\RomanII_{3}  = \sum_{k=1}^{2} \RomanII_{3,k}, 
\end{equation} 
where
\begin{subequations}
\begin{align}
& \RomanII_{3,1} \triangleq - 4 \langle (-\Delta)^{\epsilon} w^{\mathcal{L}}, \divergence \left( ( \mathcal{L}_{\lambda_{t}}X) \otimes_{s} w^{\mathcal{H}} \right) \rangle (t), \label{Define II3,1}\\
& \RomanII_{3,2} \triangleq - 4 \langle (-\Delta)^{\epsilon} w^{\mathcal{L}}, \divergence  \left( ( \mathcal{H}_{\lambda_{t}}X) \circlesign{\prec}_{s} w^{\mathcal{H}} + (\mathcal{H}_{\lambda_{t}} X) \circlesign{\circ}_{s} w^{\mathcal{H}} \right) \rangle (t). \label{Define II3,2}
\end{align}
\end{subequations}
To estimate $\RomanII_{3}$ from \eqref{Define II3,1}, we recall Proposition \ref{Proposition 4.6} (a) that led to $\lVert ( \mathcal{L}_{\lambda_{t}} X) \otimes_{s} w^{\mathcal{H}} \rVert_{\dot{H}^{1-\eta}} \lesssim \lambda_{t}^{\frac{5}{4} + 2 \kappa - \eta} (N_{t}^{\kappa})^{2}$ due to \eqref{est 81} under its hypothesis of ``$\eta \in [\frac{1}{\tau} + 3 \kappa_{0}, 1)$.'' Considering that we have $\eta = \frac{40}{13}$ in our current hypothesis, we can choose ``$\eta$'' $= \frac{7}{20}$ so that 
\begin{align*}
\lVert (\mathcal{L}_{\lambda_{t}}X \otimes_{s} w^{\mathcal{H}} \rVert_{\dot{H}^{\frac{13}{20}}} \lesssim \lambda_{t}^{\frac{9}{10} + 2 \kappa} (N_{t}^{\kappa})^{2},  
\end{align*}
and estimate using Gagliardo-Nirenberg inequality of $\lVert f \rVert_{\dot{H}^{\frac{7}{20} + 2\epsilon}} \lesssim \lVert f \rVert_{\dot{H}^{\epsilon}}^{\frac{2(9 - 10 \epsilon)}{25}} \lVert f \rVert_{\dot{H}^{\frac{5}{4} + \epsilon}}^{\frac{7+ 20 \epsilon}{25}}$ since $\epsilon < \frac{9}{10}$ by hypothesis, 
\begin{align}
&\RomanI_{3,1} \lesssim \lVert w^{\mathcal{L}} \rVert_{\dot{H}^{\frac{7}{20} + 2 \epsilon}} \lVert (\mathcal{L}_{\lambda_{t}}X) \otimes_{s} w^{\mathcal{H}} \rVert_{\dot{H}^{\frac{13}{20}}} \lesssim  \lVert w^{\mathcal{L}} \rVert_{\dot{H}^{\epsilon}}^{\frac{2(9- 10 \epsilon)}{25}} \lVert w^{\mathcal{L}} \rVert_{\dot{H}^{\frac{5}{4} + \epsilon}}^{\frac{7+ 20 \epsilon}{25}} \lambda_{t}^{\frac{9}{10} + 2 \kappa} (N_{t}^{\kappa})^{2} \label{est 114} \\ 
& \hspace{6mm} \overset{\eqref{Define lambda t} \eqref{Define M} \eqref{Estimate on wH}}{\lesssim}   \lVert w^{\mathcal{L}} \rVert_{\dot{H}^{\epsilon}}^{\frac{2(9- 10 \epsilon)}{25}} \lVert w^{\mathcal{L}} \rVert_{\dot{H}^{\frac{5}{4} + \epsilon}}^{\frac{7+ 20 \epsilon}{25}} C(M, N_{t}^{\kappa})  \leq \frac{\nu}{32} \lVert w^{\mathcal{L}} \rVert_{\dot{H}^{\frac{5}{4} + \epsilon}}^{2} + C(M, N_{t}^{\kappa}) ( \lVert w^{\mathcal{L}} \rVert_{\dot{H}^{\epsilon}}^{2} + 1). \nonumber 
\end{align}
Next, using the Gagliardo-Nirenberg inequality of $\lVert f \rVert_{\dot{H}^{2\epsilon + 3 \kappa}} \lesssim \lVert f \rVert_{\dot{H}^{\epsilon}}^{1- \frac{4}{5} (\epsilon + 3 \kappa)} \lVert f \rVert_{\dot{H}^{\frac{5}{4} + \epsilon}}^{\frac{4}{5} (\epsilon + 3 \kappa)}$, we estimate from \eqref{Define II3,2} 
\begin{align}
\RomanI_{3,2} \overset{\eqref{Sobolev products d} \eqref{Sobolev products e}}{\lesssim}& \lVert w^{\mathcal{L}} \rVert_{\dot{H}^{2\epsilon + 3 \kappa}} \lVert \mathcal{H}_{\lambda_{t}} X \rVert_{\mathscr{C}^{-\frac{1}{4} - \kappa}} \lVert w^{\mathcal{H}} \rVert_{H^{\frac{5}{4} - 2 \kappa}}  \label{est 115}\\
\overset{\eqref{Estimate on wH} \eqref{Define Lt kappa and Nt kappa} \eqref{Define M} }{\leq}& C(M, N_{t}^{\kappa}) \lVert w^{\mathcal{L}} \rVert_{\dot{H}^{\epsilon}}^{1 - \frac{4}{5} (\epsilon + 3 \kappa)} \lVert w^{\mathcal{L}} \rVert_{\dot{H}^{\frac{5}{4} + \epsilon}}^{\frac{4}{5} (\epsilon + 3 \kappa)} \leq \frac{\nu}{32} \lVert w^{\mathcal{L}} \rVert_{\dot{H}^{\frac{5}{4} + \epsilon}}^{2} + C(M, N_{t}^{\kappa}) (\lVert w^{\mathcal{L}} \rVert_{\dot{H}^{\epsilon}}^{2} + 1). \nonumber 
\end{align}
Applying \eqref{est 114} and \eqref{est 115} to \eqref{Split II3} gives us 
\begin{equation}\label{est 137}
\RomanII_{3} \leq \frac{\nu}{16} \lVert w^{\mathcal{L}} \rVert_{\dot{H}^{\frac{5}{4} + \epsilon}}^{2} + C(M, N_{t}^{\kappa}) ( \lVert w^{\mathcal{L}} \rVert_{\dot{H}^{\epsilon}}^{2} + 1).
\end{equation} 

Next, within $\RomanII_{4}$, we first work on $- 2 \langle (-\Delta)^{\epsilon} w^{\mathcal{L}}, \divergence (w^{\otimes 2}) \rangle (t)$, which we rewrite as 
\begin{align}
&- 2 \langle (-\Delta)^{\epsilon} w^{\mathcal{L}}, \divergence (w^{\otimes 2}) \rangle (t)   \nonumber \\
\overset{\eqref{Define symmetric tensor}}{=}&  -2 \langle (-\Delta)^{\epsilon} w^{\mathcal{L}}, \divergence \left( (w^{\mathcal{L}})^{\otimes 2} + 2w^{\mathcal{L}} \otimes_{s} w^{\mathcal{H}} + (w^{\mathcal{H}})^{\otimes 2} \right) \rangle (t).  \label{est 117}
\end{align}
Concerning the product of lower-frequency terms, we can use the Kato-Ponce commutator estimate \eqref{Kato-Ponce} and Sobolev embeddings of $\dot{H}^{\frac{5}{4}}(\mathbb{T}^{3}) \hookrightarrow L^{12} (\mathbb{T}^{3})$ and $\dot{H}^{\frac{1}{4}}(\mathbb{T}^{3}) \hookrightarrow L^{\frac{12}{5}} (\mathbb{T}^{3})$ to compute 
\begin{align}
&  -2 \langle (-\Delta)^{\epsilon} w^{\mathcal{L}}, \divergence \left( (w^{\mathcal{L}} )^{\otimes 2} \right) \rangle (t) = -2 \int_{\mathbb{T}^{3}} (-\Delta)^{\frac{\epsilon}{2}} w^{\mathcal{L}} \cdot [ (-\Delta)^{\frac{\epsilon}{2}}, (w^{\mathcal{L}} \cdot\nabla) ] w^{\mathcal{L}} dx  \nonumber\\
\overset{\eqref{Kato-Ponce}}{\lesssim}&  \lVert w^{\mathcal{L}} \rVert_{\dot{H}^{\frac{5}{4} + \epsilon}} \lVert w^{\mathcal{L}} \rVert_{\dot{H}^{\epsilon}} \lVert w^{\mathcal{L}} \rVert_{\dot{H}^{\frac{5}{4}}} \leq \frac{\nu}{32} \lVert w^{\mathcal{L}} \rVert_{\dot{H}^{\frac{5}{4} + \epsilon}}^{2} + C \lVert w^{\mathcal{L}} \rVert_{\dot{H}^{\epsilon}}^{2} \lVert w^{\mathcal{L}} \rVert_{\dot{H}^{\frac{5}{4}}}^{2};  \label{est 116}
\end{align}
we note that this estimate demonstrates the energy-criticality (recall \eqref{Energy-critical}) of our case and new ideas will be necessary for further improvements beyond our case. Next, we estimate the high-low frequency term by Lemma \ref{Lemma 3.2} and the Gagliardo-Nirenberg inequality of $\lVert f \rVert_{\dot{H}^{\frac{37}{40} + \epsilon + \kappa}} \lesssim \lVert f \rVert_{\dot{H}^{\epsilon}}^{\frac{13 - 40 \kappa}{50}} \lVert f \rVert_{\dot{H}^{\frac{5}{4} + \epsilon}}^{\frac{37 + 40 \kappa}{50}}$, 
\begin{align}
& -4 \langle (-\Delta)^{\epsilon} w^{\mathcal{L}}, \divergence \left( w^{\mathcal{L}} \otimes_{s} w^{\mathcal{H}}  \right) \rangle (t) \lesssim  \lVert w^{\mathcal{L}} \rVert_{\dot{H}^{\frac{37}{40} + \epsilon + \kappa}} \lVert w^{\mathcal{L}} \rVert_{\dot{H}^{\frac{1}{2} + \epsilon + \kappa}} \lVert w^{\mathcal{H}} \rVert_{\dot{H}^{\frac{43}{40} - 2 \kappa}}  \nonumber \\
& \hspace{15mm} \overset{\eqref{Estimate on wH}}{\lesssim} C(N_{t}^{\kappa}, M) \lVert w^{\mathcal{L}} \rVert_{\dot{H}^{\epsilon}}^{\frac{43 - 80 \kappa}{50}} \lVert w^{\mathcal{L}} \rVert_{\dot{H}^{\frac{5}{4} + \epsilon}}^{\frac{57 + 80 \kappa}{50}} \leq \frac{\nu}{32} \lVert w^{\mathcal{L}} \rVert_{\dot{H}^{\frac{5}{4} + \epsilon}}^{2} + C(N_{t}^{\kappa}, M) \lVert w^{\mathcal{L}} \rVert_{\dot{H}^{\epsilon}}^{2}. \label{est 118}
\end{align}
Next, we rely on Lemma \ref{Lemma 3.2} and the Gagliardo-Nirenberg inequality of $\lVert f \rVert_{\dot{H}^{\frac{1}{4} + 2 \epsilon + 4 \kappa}} \lesssim \lVert f \rVert_{\dot{H}^{\epsilon}}^{\frac{4 - 4 (\epsilon + 4 \kappa)}{5}} \lVert f \rVert_{\dot{H}^{\frac{5}{4} + \epsilon}}^{\frac{1+ 4( \epsilon + 4 \kappa)}{5}}$ to estimate the  higher frequency term by 
\begin{align}
& -2 \langle (-\Delta)^{\epsilon} w^{\mathcal{L}}, \divergence (w^{\mathcal{H}})^{\otimes 2} \rangle (t) \lesssim \lVert w^{\mathcal{L}} \rVert_{\dot{H}^{\frac{1}{4} + 2 \epsilon + 4 \kappa}} \lVert (w^{\mathcal{H}})^{\otimes 2} \rVert_{\dot{H}^{\frac{3}{4} - 4 \kappa}}  \nonumber \\
&  \hspace{15mm} \lesssim  \lVert w^{\mathcal{L}} \rVert_{\dot{H}^{\epsilon}}^{\frac{4- 4 (\epsilon + 4 \kappa)}{5}} \lVert w^{\mathcal{L}} \rVert_{\dot{H}^{\frac{5}{4} + \epsilon}}^{\frac{1+ 4 (\epsilon + 4 \kappa)}{5}} \lVert w^{\mathcal{H}} \rVert_{\dot{H}^{\frac{9}{8} - 2 \kappa}}^{2}  \overset{\eqref{Estimate on wH}  \eqref{Define M} }{\lesssim} C(N_{t}^{\kappa})  \lVert w^{\mathcal{L}} \rVert_{\dot{H}^{\epsilon}}^{\frac{4- 4 (\epsilon + 4 \kappa)}{5}} \lVert w^{\mathcal{L}} \rVert_{\dot{H}^{\frac{5}{4} + \epsilon}}^{\frac{1+ 4 (\epsilon + 4 \kappa)}{5}}  \nonumber\\
& \hspace{15mm} \leq \frac{\nu}{32} \lVert w^{\mathcal{L}} \rVert_{\dot{H}^{\frac{5}{4} + \epsilon}}^{2} + C(M, N_{t}^{\kappa}) (\lVert w^{\mathcal{L}} \rVert_{\dot{H}^{\epsilon}}^{2} + 1).\label{est 119}
\end{align}
In sum, we are ready to conclude by applying \eqref{est 116}, \eqref{est 118}, and \eqref{est 119} to \eqref{est 117} to conclude
\begin{equation}\label{est 133} 
- 2 \langle (-\Delta)^{\epsilon} w^{\mathcal{L}}, \divergence (w^{\otimes 2}) \rangle (t)  \leq  \frac{3\nu}{32} \lVert w^{\mathcal{L}} \rVert_{\dot{H}^{\frac{5}{4} + \epsilon}}^{2} + C(M, N_{t}^{\kappa}) ( 1+ \lVert w^{\mathcal{L}} \rVert_{\dot{H}^{\frac{5}{4}}}^{2}) \lVert w^{\mathcal{L}} \rVert_{\dot{H}^{\epsilon}}^{2}. 
\end{equation} 
Next, we estimate the last non-commutator term in $\RomanII_{4}$ by relying on \eqref{Sobolev products}, \eqref{Estimate on wH}, and the Gagliardo-Nirenberg inequalities of $\lVert f  \rVert_{\dot{H}^{\frac{3}{40} + 2 \epsilon + 3 \kappa}} \lesssim \lVert f \rVert_{\dot{H}^{\epsilon}}^{\frac{47}{50} - \frac{4\epsilon}{5} - \frac{12\kappa}{5}} \lVert f \rVert_{\dot{H}^{\frac{5}{4} + \epsilon}}^{\frac{3}{50} + \frac{4\epsilon}{5} + \frac{12 \kappa}{5}}$ and $\lVert f \rVert_{\dot{H}^{\frac{37}{40} - 2 \kappa}} \lesssim \lVert f \rVert_{L^{2}}^{\frac{13 + 80 \kappa}{50}} \lVert f \rVert_{\dot{H}^{\frac{5}{4}}}^{\frac{37 - 80 \kappa}{50}}$, 
\begin{align}
& -2 \langle (-\Delta)^{\epsilon} w^{\mathcal{L}}, \divergence \left( 2 Y \otimes_{s} w + Y^{\otimes 2} \right) \rangle(t) \lesssim \lVert w^{\mathcal{L}} \rVert_{\dot{H}^{\frac{3}{40} + 2 \epsilon + 3 \kappa}} \left( \lVert Y \otimes_{s} w \rVert_{\dot{H}^{\frac{37}{40} - 3 \kappa}} + \lVert Y^{\otimes 2} \rVert_{\dot{H}^{\frac{37}{40}-3\kappa}} \right) \nonumber \\
& \hspace{20mm} \lesssim  \lVert w^{\mathcal{L}} \rVert_{\dot{H}^{\epsilon}}^{\frac{47}{50} - \frac{4\epsilon}{5} - \frac{12\kappa}{5}} \lVert w^{\mathcal{L}} \rVert_{\dot{H}^{\frac{5}{4} + \epsilon}}^{\frac{3}{50} + \frac{4\epsilon}{5} + \frac{12 \kappa}{5}} \left[ \lVert Y \rVert_{\mathscr{C}^{1-\kappa}} ( \lVert w^{\mathcal{L}} \rVert_{H^{\frac{37}{40} - 2 \kappa}} + \lVert w^{\mathcal{H}} \rVert_{H^{\frac{37}{40} - 2 \kappa}} ) + (N_{t}^{\kappa})^{2} \right]  \nonumber  \\ 
& \hspace{20mm}  \leq \frac{\nu}{32} \lVert w^{\mathcal{L}} \rVert_{\dot{H}^{\frac{5}{4} + \epsilon}}^{2} + C(M, N_{t}^{\kappa})  \left( 1+ \lVert w^{\mathcal{L}} \rVert_{\dot{H}^{\epsilon}}^{2} \right) \left( 1+ \lVert w^{\mathcal{L}} \rVert_{\dot{H}^{\frac{5}{4}}}^{2} \right). \label{est 134} 
\end{align}
At last, we come to the commutator involving $\mathcal{C}^{\circlesign{\prec}_{s}} (w, Q^{\mathcal{H}})$ in \eqref{Define II4}, which, using \eqref{est 87} again, we rewrite it as 
\begin{align}\label{Split tilde Ck}
2 \left\langle (-\Delta)^{\epsilon}w^{\mathcal{L}}, \divergence \left( \mathcal{C}^{\circlesign{\prec}_{s}} (w, Q^{\mathcal{H}}) \right) \right\rangle (t)  = \sum_{k=1}^{4} \tilde{C}_{k} 
\end{align}
where 
\begin{subequations}\label{Define tilde Ck}
\begin{align}
& \tilde{C}_{1} \triangleq -2 \langle (-\Delta)^{\epsilon} w^{\mathcal{L}}, \divergence \left( [ \mathbb{P}_{L} \divergence (w^{\otimes 2} ) ] \circlesign{\prec}_{s} Q^{\mathcal{H}}  \right) \rangle, \label{Define tilde C1}\\
& \tilde{C}_{2} \triangleq -2 \langle  (-\Delta)^{\epsilon}w^{\mathcal{L}}, \divergence \left( [ \mathbb{P}_{L} \divergence (D \otimes_{s} w) ] \circlesign{\prec}_{s} Q^{\mathcal{H}} \right) \rangle,  \label{Define tilde C2}\\
& \tilde{C}_{3} \triangleq -2 \langle  (-\Delta)^{\epsilon}w^{\mathcal{L}}, \divergence \left( [ \mathbb{P}_{L} \divergence (Y^{\otimes 2}) ] \circlesign{\prec}_{s} Q^{\mathcal{H}}\right) \rangle,  \label{Define tilde C3}\\
& \tilde{C}_{4} \triangleq 2 \langle  (-\Delta)^{\epsilon}w^{\mathcal{L}}, \divergence \left( \Lambda^{\frac{5}{2}} (w \circlesign{\prec}_{s} Q^{\mathcal{H}}) - (\Lambda^{\frac{5}{2}} w) \circlesign{\prec}_{s} Q^{\mathcal{H}} - w \circlesign{\prec}_{s} \Lambda^{\frac{5}{2}} Q^{\mathcal{H}} \right) \rangle. \label{Define tilde C4}
\end{align}
\end{subequations} 
First, we estimate $\tilde{C}_{1}$ from \eqref{Define tilde C1} using the Gagliardo-Nirenberg inequalities of $\lVert f \rVert_{\dot{H}^{2\epsilon}} \lesssim \lVert f \rVert_{\dot{H}^{\epsilon}}^{\frac{5- 4 \epsilon}{5}} \lVert f \rVert_{\dot{H}^{\frac{5}{4} + \epsilon}}^{\frac{4\epsilon}{5}}$ and $\lVert f \rVert_{\dot{H}^{\frac{5+ 6 \kappa}{8}}} \lesssim\lVert f \rVert_{L^{2}}^{\frac{5+ 6 \kappa}{10}} \lVert f \rVert_{\dot{H}^{\frac{5}{4}}}^{\frac{5-6\kappa}{10}}$ as follows:
\begin{align}
&\tilde{C}_{1} \overset{\eqref{Sobolev products c}}{\lesssim} \lVert w^{\mathcal{L}} \rVert_{\dot{H}^{\epsilon}}^{\frac{5- 4 \epsilon}{5}} \lVert w^{\mathcal{L}} \rVert_{\dot{H}^{\frac{5}{4} + \epsilon}}^{\frac{4\epsilon}{5}} \lVert \mathbb{P}_{L} \divergence (w^{\otimes 2}) \rVert_{H^{-\frac{5}{4} + \frac{3\kappa}{2}}} \lVert Q^{\mathcal{H}} \rVert_{\mathscr{C}^{\frac{9}{4} - \frac{3\kappa}{2}}} \nonumber \\
&\overset{\eqref{Regularity of Q} \eqref{Define Lt kappa and Nt kappa}}{\lesssim} C(N_{t}^{\kappa}) \lVert w^{\mathcal{L}} \rVert_{\dot{H}^{\epsilon}}^{\frac{5- 4 \epsilon}{5}} \lVert w^{\mathcal{L}} \rVert_{\dot{H}^{\frac{5}{4} + \epsilon}}^{\frac{4\epsilon}{5}} \lVert w \rVert_{\dot{H}^{\frac{5+ 6 \kappa}{8}}}^{2}  \nonumber \\
&\overset{ \eqref{Estimate on wH} \eqref{Define M}}{\lesssim} C(M, N_{t}^{\kappa})  \lVert w^{\mathcal{L}} \rVert_{\dot{H}^{\epsilon}}^{\frac{5-4\epsilon}{5}} \lVert w^{\mathcal{L}} \rVert_{\dot{H}^{\frac{5}{4} + \epsilon}}^{\frac{4\epsilon}{5}} \left( \lVert w^{\mathcal{L}} \rVert_{\dot{H}^{\frac{5}{4}}}^{\frac{5-6\kappa}{10}} + N_{t}^{\kappa} \right)^{2} \nonumber \\
& \hspace{15mm} \leq \frac{\nu}{32} \lVert w^{\mathcal{L}} \rVert_{\dot{H}^{\frac{5}{4} + \epsilon}}^{2} + C(M, N_{t}^{\kappa}) (1+\lVert w^{\mathcal{L}} \rVert_{\dot{H}^{\epsilon}}^{2}) (1+ \lVert w^{\mathcal{L}} \rVert_{\dot{H}^{\frac{5}{4}}}^{2} ). \label{est 128}
\end{align}
For $\tilde{C}_{2}$ from \eqref{Define tilde C2}, we first rewrite it by \eqref{Define D} as 
\begin{equation}\label{Split tilde C2}
\tilde{C}_{2} = \sum_{k=1}^{2} \tilde{C}_{2k}
\end{equation} 
where 
\begin{subequations}\label{est 121}
\begin{align} 
& \tilde{C}_{21} = -4 \langle  (-\Delta)^{\epsilon}w^{\mathcal{L}}, \divergence \left( [ \mathbb{P}_{L} \divergence (Y \otimes_{s} w) ] \circlesign{\prec}_{s} Q^{\mathcal{H}} \right) \rangle, \label{est 121a}\\
& \tilde{C}_{22} = -4 \langle  (-\Delta)^{\epsilon}w^{\mathcal{L}}, \divergence \left( [ \mathbb{P}_{L} \divergence (X \otimes_{s} w) ] \circlesign{\prec}_{s} Q^{\mathcal{H}} \right). \rangle \label{est 121b}
\end{align}
\end{subequations} 
Now, we recall the estimate of ``$\lVert  [ \mathbb{P}_{L} \divergence ( Y \otimes_{s} w) ] \circlesign{\prec}_{s} Q^{\mathcal{H}} \rVert_{\dot{H}^{1-\eta}}  \lesssim (N_{t}^{\kappa})^{2} [ \lVert w^{\mathcal{L}} \rVert_{\dot{H}^{\kappa}} + (1+ \lVert w \rVert_{L^{2}})^{1- \tau (\frac{5}{4} - 3 \kappa)} N_{t}^{\kappa}]$'' within \eqref{est 120} that was valid for any $\eta \in (0,1)$ and $\kappa_{0} \in (0, \frac{1}{200})$. We use it with ``$\eta$'' = $\kappa$ so that 
\begin{align*}
\lVert  [ \mathbb{P}_{L} \divergence ( Y \otimes_{s} w) ] \circlesign{\prec}_{s} Q^{\mathcal{H}} \rVert_{\dot{H}^{1-\kappa}}  \lesssim (N_{t}^{\kappa})^{2} [ \lVert w^{\mathcal{L}} \rVert_{\dot{H}^{\kappa}} + (1+ \lVert w \rVert_{L^{2}})^{1- \frac{40}{13}(\frac{5}{4} - 3 \kappa)} N_{t}^{\kappa}]
\end{align*}
and estimate $\tilde{C}_{21}$ from \eqref{est 121a}, along with the Gagliardo-Nirenberg inequality of $\lVert f \rVert_{\dot{H}^{2\epsilon + \kappa}} \lesssim \lVert f \rVert_{\dot{H}^{\epsilon}}^{1- \frac{4(\epsilon + \kappa)}{5}} \lVert f \rVert_{\dot{H}^{\frac{5}{4} + \epsilon}}^{\frac{4( \epsilon + \kappa)}{5}}$ and $\lVert f \rVert_{\dot{H}^{\kappa}} \lesssim \lVert f \rVert_{L^{2}}^{1- \frac{4\kappa}{5}} \lVert f \rVert_{\dot{H}^{\frac{5}{4}}}^{\frac{4\kappa}{5}}$,  by  
\begin{align}
&  \tilde{C}_{21} \lesssim \lVert w^{\mathcal{L}} \rVert_{\dot{H}^{2\epsilon + \kappa}} \lVert [ \mathbb{P}_{L} \divergence ( Y \otimes_{s} w) ] \circlesign{\prec}_{s} Q^{\mathcal{H}} \rVert_{\dot{H}^{1-\kappa}} \nonumber \\
& \hspace{25mm} \overset{\eqref{est 120}}{\lesssim} (N_{t}^{\kappa})^{2} \lVert w^{\mathcal{L}} \rVert_{\dot{H}^{\epsilon}}^{1- \frac{4(\epsilon + \kappa)}{5}} \lVert w^{\mathcal{L}} \rVert_{\dot{H}^{\frac{5}{4} + \epsilon}}^{\frac{4(\epsilon + \kappa)}{5}} \left( \lVert w^{\mathcal{L}} \rVert_{L^{2}}^{1- \frac{4\kappa}{5}} \lVert w^{\mathcal{L}} \rVert_{\dot{H}^{\frac{5}{4}}}^{\frac{4\kappa}{5}} + N_{t}^{\kappa} \right)  \nonumber\\
&\hspace{25mm} \overset{\eqref{Define M}}{\leq}  \frac{\nu}{64} \lVert w^{\mathcal{L}} \rVert_{\dot{H}^{\frac{5}{4} + \epsilon}}^{2} + C(M, N_{t}^{\kappa}) (1+\lVert w^{\mathcal{L}} \rVert_{\dot{H}^{\epsilon}}^{2}) (1+ \lVert w^{\mathcal{L}} \rVert_{\dot{H}^{\frac{5}{4}}}^{2} ). \label{est 124} 
\end{align}
Similarly, we recall the estimate of ``$\lVert [ \mathbb{P}_{L} \divergence ( X \otimes_{s} w) ] \circlesign{\prec}_{s} Q^{\mathcal{H}} \rVert_{\dot{H}^{1-\eta}}  \lesssim  (N_{t}^{\kappa})^{2}  [ \lVert w^{\mathcal{L}} \rVert_{\dot{H}^{\frac{1}{4} + \frac{3\kappa}{2}}} + (1+ \lVert w \rVert_{L^{2}})^{1- \tau (1- \frac{7\kappa}{2})} N_{t}^{\kappa} t^{\frac{\kappa}{5}} ]$'' from \eqref{est 122} that is valid for any $\eta \in (0,1)$ and $\kappa_{0} \in (0, \frac{1}{200})$. Thus, we choose ``$\eta$'' = $\kappa$ so that 
\begin{align*}
\lVert [ \mathbb{P}_{L} \divergence ( X \otimes_{s} w) ] \circlesign{\prec}_{s} Q^{\mathcal{H}} \rVert_{\dot{H}^{1-\kappa}}  \lesssim  (N_{t}^{\kappa})^{2}  [ \lVert w^{\mathcal{L}} \rVert_{\dot{H}^{\frac{1}{4} + \frac{3\kappa}{2}}} + (1+ \lVert w \rVert_{L^{2}})^{1- \frac{40}{13} (1- \frac{7\kappa}{2})} N_{t}^{\kappa} t^{\frac{\kappa}{5}} ]
\end{align*}
and use it to estimate $\tilde{C}_{22}$ of \eqref{est 121b}, along with the Gagliardo-Nirenberg inequalities of $\lVert f \rVert_{\dot{H}^{2\epsilon + \kappa}} \lesssim \lVert f \rVert_{\dot{H}^{\epsilon}}^{1- \frac{4(\epsilon + \kappa)}{5}} \lVert f \rVert_{\dot{H}^{\frac{5}{4} + \epsilon}}^{\frac{4( \epsilon + \kappa)}{5}}$ and $\lVert f \rVert_{\dot{H}^{\frac{1}{4} + \frac{3\kappa}{2}}} \lesssim \lVert f \rVert_{L^{2}}^{\frac{4-6 \kappa}{5}} \lVert f \rVert_{\dot{H}^{\frac{5}{4}}}^{\frac{1+ 6 \kappa}{5}}$ by 
\begin{align}
&  \tilde{C}_{22} \lesssim \lVert w^{\mathcal{L}} \rVert_{\dot{H}^{2\epsilon + \kappa}} \lVert [ \mathbb{P}_{L} \divergence (X \otimes_{s} w) ] \circlesign{\prec}_{s} Q^{\mathcal{H}} \rVert_{\dot{H}^{1-\kappa}}  \nonumber \\
& \hspace{20mm} \overset{\eqref{est 122}}{\lesssim} \lVert w^{\mathcal{L}} \rVert_{\dot{H}^{2\epsilon + \kappa}} (N_{t}^{\kappa})^{2} [ \lVert w^{\mathcal{L}} \rVert_{\dot{H}^{\frac{1}{4} + \frac{3\kappa}{2}}} + (1+ \lVert w \rVert_{L^{2}})^{1- \tau (1- \frac{7\kappa}{2})} N_{t}^{\kappa} t^{\frac{\kappa}{5}} ]   \nonumber\\
& \hspace{20mm} \overset{\eqref{Define M}}{\lesssim} C(N_{t}^{\kappa}, M) \lVert w^{\mathcal{L}} \rVert_{\dot{H}^{\epsilon}}^{1- \frac{4}{5} (\epsilon + \kappa)} \lVert w^{\mathcal{L}} \rVert_{\dot{H}^{\frac{5}{4} + \epsilon}}^{\frac{4}{5} (\epsilon + \kappa)} [  \lVert w^{\mathcal{L}} \rVert_{\dot{H}^{\frac{5}{4}}}^{\frac{1+  6 \kappa}{5}} + 1]  \nonumber \\ 
& \hspace{20mm} \leq  \frac{\nu}{64} \lVert w^{\mathcal{L}} \rVert_{\dot{H}^{\frac{5}{4} + \epsilon}}^{2} + C(M, N_{t}^{\kappa})   (\lVert w^{\mathcal{L}} \rVert_{\dot{H}^{\epsilon}}^{2} + 1) (\lVert w^{\mathcal{L}} \rVert_{\dot{H}^{\frac{5}{4}}}^{2} + 1).  \label{est 123}
\end{align}
Applying \eqref{est 124} and \eqref{est 123} to \eqref{Split tilde C2}, we conclude that 
\begin{equation}\label{est 129}
\tilde{C}_{2} \leq  \frac{\nu}{32} \lVert w^{\mathcal{L}} \rVert_{\dot{H}^{\frac{5}{4} + \epsilon}}^{2} + C(M, N_{t}^{\kappa}) ( \lVert w^{\mathcal{L}} \rVert_{\dot{H}^{\epsilon}}^{2} + 1) ( \lVert w^{\mathcal{L}} \rVert_{\dot{H}^{\frac{5}{4}}}^{2} + 1).
\end{equation} 
Next, we recall the estimate of ``$\lVert [ \mathbb{P}_{L} \divergence (Y^{\otimes 2} ) ] \circlesign{\prec}_{s} Q^{\mathcal{H}} \rVert_{\dot{H}^{1-\eta}} \lesssim  (N_{t}^{\kappa})^{3}$'' for any $\eta \in (0,1)$ and $\kappa_{0} \in (0, \frac{1}{200})$ from \eqref{Estimate on C3} and utilize it with ``$\eta$'' = $\kappa$ so that 
\begin{align*}
\lVert [ \mathbb{P}_{L} \divergence (Y^{\otimes 2} ) ] \circlesign{\prec}_{s} Q^{\mathcal{H}} \rVert_{\dot{H}^{1-\kappa}} \lesssim  (N_{t}^{\kappa})^{3}
\end{align*}
to estimate from \eqref{Define tilde C3}, along with the Gagliardo-Nirenberg inequality of $\lVert f \rVert_{\dot{H}^{2\epsilon + \kappa}} \lesssim \lVert f \rVert_{\dot{H}^{\epsilon}}^{1- \frac{4(\epsilon + \kappa)}{5}} \lVert f \rVert_{\dot{H}^{\frac{5}{4} + \epsilon}}^{\frac{4( \epsilon + \kappa)}{5}}$
\begin{align}
\tilde{C}_{3}\lesssim& \lVert w^{\mathcal{L}} \rVert_{\dot{H}^{2\epsilon + \kappa}} \lVert [ \mathbb{P}_{L} \divergence (Y^{\otimes 2}) ] \circlesign{\prec}_{s} Q^{\mathcal{H}} \rVert_{\dot{H}^{1-\kappa}} \overset{\eqref{Estimate on C3}}{\lesssim} \lVert w^{\mathcal{L}} \rVert_{\dot{H}^{\epsilon}}^{1- \frac{4(\epsilon + \kappa)}{5}} \lVert w^{\mathcal{L}} \rVert_{\dot{H}^{\frac{5}{4} + \epsilon}}^{\frac{4(\epsilon + \kappa)}{5}} (N_{t}^{\kappa})^{3}  \nonumber \\
& \hspace{35mm}  \leq \frac{\nu}{32} \lVert w^{\mathcal{L}} \rVert_{\dot{H}^{\frac{5}{4} + \epsilon}}^{2} + C(N_{t}^{\kappa}) ( \lVert w^{\mathcal{L}} \rVert_{\dot{H}^{\epsilon}}^{2} + 1).  \label{est 130}
\end{align}
We now come to $\tilde{C}_{4}$ in \eqref{Define tilde C4}. Similarly to \eqref{Define C41 and C42}, we treat the second term within the commutator first using the Gagliardo-Nirenberg inequality of $\lVert f \rVert_{\dot{H}^{\epsilon+ \frac{3\kappa}{2}}} \lesssim \lVert f \rVert_{\dot{H}^{\epsilon}}^{1- \frac{6\kappa}{5}} \lVert f \rVert_{\dot{H}^{\frac{5}{4} + \epsilon}}^{\frac{6\kappa}{5}}$:
\begin{align}
& -2 \langle (-\Delta)^{\epsilon} w^{\mathcal{L}}, \divergence \left( ( \Lambda^{\frac{5}{2}} w) \circlesign{\prec}_{s} Q^{\mathcal{H}} \right) \rangle \overset{\eqref{Sobolev products c}}{\lesssim} \lVert w^{\mathcal{L}} \rVert_{\dot{H}^{\frac{5}{4} + \epsilon}} \lVert \Lambda^{\frac{5}{2}} w \rVert_{H^{-\frac{5}{2} + \epsilon + \frac{3\kappa}{2}}} \lVert Q^{\mathcal{H}} \rVert_{\mathscr{C}^{\frac{9}{4} - \frac{3\kappa}{2}}}  \nonumber \\ 
& \hspace{25mm} \overset{\eqref{Regularity of Q} \eqref{Estimate on wH} }{\lesssim} \lVert w^{\mathcal{L}} \rVert_{\dot{H}^{\frac{5}{4} + \epsilon}} \left( \lVert w^{\mathcal{L}} \rVert_{\dot{H}^{\epsilon}}^{1- \frac{6\kappa}{5}} \lVert w^{\mathcal{L}} \rVert_{\dot{H}^{\frac{5}{4} + \epsilon}}^{\frac{6\kappa}{5}} + ( 1+ \lVert w(t) \rVert_{L^{2}})^{1- \frac{40}{13} (\frac{5}{4} - \epsilon - \frac{7\kappa}{2})} N_{t}^{\kappa} t^{\frac{\kappa}{5}} \right) \nonumber \\
& \hspace{25mm} \overset{\eqref{Define Lt kappa and Nt kappa} \eqref{Define M}}{\leq} \frac{\nu}{64} \lVert w^{\mathcal{L}} \rVert_{\dot{H}^{\frac{5}{4} + \epsilon}}^{2} + C(M, N_{t}^{\kappa}) ( \lVert w^{\mathcal{L}} \rVert_{\dot{H}^{\epsilon}}^{2} + 1 ). \label{est 126}
\end{align} 
Next, for $\alpha> 0$ to be determined subsequently, we estimate for $N_{1} \in \mathbb{N}$ from \eqref{est 152}, 
\begin{align}
&\lVert   \Lambda^{\frac{5}{2}} (w \circlesign{\prec}_{s} Q^{\mathcal{H}} ) - w \circlesign{\prec}_{s} \Lambda^{\frac{5}{2}} Q^{\mathcal{H}}  \rVert_{H^{\alpha}}^{2}    \nonumber \\
\overset{\eqref{est 152}}{=}& \sum_{j\geq -1} 2^{2j\alpha} \left\lVert \sum_{m: \lvert j-m \rvert \leq N_{1}}  \left(\Lambda^{\frac{5}{2}} \left( (S_{m-1} w) \Delta_{m} Q^{\mathcal{H}} \right) - (S_{m-1} w) \Lambda^{\frac{5}{2}} \Delta_{m} Q^{\mathcal{H}}  \right) \right\rVert_{L^{2}}^{2} \nonumber  \\
\overset{\eqref{Kato-Ponce}}{\lesssim}&  \sum_{j\geq -1} 2^{2j\alpha} \left[ \lVert \nabla S_{j-1} w \rVert_{L^{\frac{6}{3-2\kappa}}} \lVert \Lambda^{\frac{3}{2}} \Delta_{j} Q^{\mathcal{H}} \rVert_{L^{\frac{3}{\kappa}}} + \lVert \Lambda^{\frac{5}{2}} S_{j-1} w \rVert_{L^{2}} \lVert \Delta_{j} Q^{\mathcal{H}} \rVert_{L^{\infty}} \right]^{2}\nonumber  \\
\lesssim& \lVert Q^{\mathcal{H}} \rVert_{\mathscr{C}^{\frac{9}{4} - \frac{3\kappa}{2}}}^{2}  \Bigg[ \left\lVert w^{-\lvert j \rvert (\frac{3}{4} - \alpha - \frac{3\kappa}{2})} \ast_{j} 2^{j(\frac{1}{4} + \alpha + \frac{5\kappa}{2})} \lVert \Delta_{j} w \rVert_{L^{2}} \right\rVert_{l^{2}_{j\geq -1}}  \nonumber \\
& \hspace{10mm} + \left\lVert 2^{-\lvert j \rvert (\frac{9}{4} - \alpha - \frac{3\kappa}{2})} \ast_{j} 2^{j(\frac{1}{4} + \alpha  + \frac{3\kappa}{2})} \lVert \Delta_{j} w\rVert_{L^{2}} \right\rVert_{l^{2}_{j\geq -1}} \Bigg]^{2}   \overset{\eqref{Regularity of Q} \eqref{Define Lt kappa and Nt kappa}}{\lesssim} (N_{t}^{\kappa})^{2} \lVert w \rVert_{\dot{H}^{\frac{1}{4} + \alpha + \frac{5\kappa}{2}}}^{2}, \label{est 125} 
\end{align}
where the last inequality used Young's inequality for convolution and we see the requirement of $\alpha < \frac{3}{4} - \frac{3\kappa}{2}$. Thus, we select $\alpha = \frac{3}{4} - 2 \kappa$ and utilize \eqref{est 125} to estimate the rest of the terms in $\tilde{C}_{4}$ using the Gagliardo-Nirenberg inequalities of $\lVert f \rVert_{\dot{H}^{\frac{1}{4} + 2 \epsilon + 2\kappa}} \lesssim \lVert f \rVert_{\dot{H}^{\epsilon}}^{\frac{4 - 4(\epsilon + 2\kappa)}{5}} \lesssim \lVert f \rVert_{\dot{H}^{\frac{5}{4} + \epsilon}}^{\frac{1+ 4(\epsilon + 2 \kappa)}{5}}$ and $\lVert f \rVert_{\dot{H}^{1+ \frac{\kappa}{2}}} \lesssim \lVert f \rVert_{L^{2}}^{\frac{1-2\kappa}{5}} \lVert f \rVert_{\dot{H}^{\frac{5}{4}}}^{\frac{4+ 2 \kappa}{5}}$ by  
\begin{align}
& 2 \langle (-\Delta)^{\epsilon} w^{\mathcal{L}}, \divergence \left( \Lambda^{\frac{5}{2}} (w \circlesign{\prec}_{s} Q^{\mathcal{H}} ) - w \circlesign{\prec}_{s} \Lambda^{\frac{5}{2}} Q^{\mathcal{H}}  \right) \rangle  \nonumber \\
\lesssim& \lVert w^{\mathcal{L}} \rVert_{H^{\frac{1}{4} + 2 \epsilon + 2 \kappa}} \lVert  (\Lambda^{\frac{5}{2}} (w \circlesign{\prec}_{s} Q^{\mathcal{H}} ) - w \circlesign{\prec}_{s} \Lambda^{\frac{5}{2}} Q^{\mathcal{H}}  \rVert_{H^{\frac{3}{4} - 2\kappa}}   \nonumber \\
\overset{\eqref{est 125}}{\lesssim}& \lVert w^{\mathcal{L}} \rVert_{\dot{H}^{\frac{1}{4} + 2 \epsilon + 2 \kappa}} N_{t}^{\kappa} \lVert w \rVert_{\dot{H}^{1+ \frac{\kappa}{2}}}  \lesssim N_{t}^{\kappa} \lVert w^{\mathcal{L}} \rVert_{\dot{H}^{\epsilon}}^{\frac{4-4 (\epsilon + 2 \kappa)}{5}} \lVert w^{\mathcal{L}} \rVert_{\dot{H}^{\frac{5}{4} + \epsilon}}^{\frac{1+ 4(\epsilon + 2 \kappa)}{5}} [ \lVert w^{\mathcal{L}} \rVert_{\dot{H}^{1+ \frac{\kappa}{2}}} +  \lVert w^{\mathcal{H}} \rVert_{\dot{H}^{1+ \frac{\kappa}{2}}} ]  \nonumber \\
& \hspace{15mm} \overset{\eqref{Estimate on wH} \eqref{Define M}}{\leq}   \frac{\nu}{64} \lVert w^{\mathcal{L}} \rVert_{\dot{H}^{\frac{5}{4} + \epsilon}}^{2} + C(M, N_{t}^{\kappa}) ( \lVert w^{\mathcal{L}} \rVert_{\dot{H}^{\epsilon}}^{2} + 1) ( \lVert w^{\mathcal{L}} \rVert_{\dot{H}^{\frac{5}{4}}}^{2} + 1). \label{est 127}
\end{align}
Applying \eqref{est 126} and \eqref{est 127} to \eqref{Define tilde C4} gives us 
\begin{equation}\label{est 131}
\tilde{C}_{4} \leq  \frac{\nu}{32} \lVert w^{\mathcal{L}} \rVert_{\dot{H}^{\frac{5}{4} + \epsilon}}^{2} + C(M, N_{t}^{\kappa}) ( \lVert w^{\mathcal{L}} \rVert_{\dot{H}^{\epsilon}}^{2} + 1) ( \lVert w^{\mathcal{L}} \rVert_{\dot{H}^{\frac{5}{4}}}^{2} + 1).
\end{equation} 

At last, applying \eqref{est 128}, \eqref{est 129}, \eqref{est 130}, and \eqref{est 131} to \eqref{Split tilde Ck} gives us 
\begin{align}
&2 \left\langle (-\Delta)^{\epsilon}w^{\mathcal{L}}, \divergence \left( \mathcal{C}^{\circlesign{\prec}_{s}} (w, Q^{\mathcal{H}}) \right) \right\rangle (t)  \nonumber\\
\leq& \frac{\nu}{8} \lVert w^{\mathcal{L}} \rVert_{\dot{H}^{\frac{5}{4} + \epsilon}}^{2} + C(M, N_{t}^{\kappa}) ( \lVert w^{\mathcal{L}} \rVert_{\dot{H}^{\epsilon}}^{2} + 1) ( \lVert w^{\mathcal{L}} \rVert_{\dot{H}^{\frac{5}{4}}}^{2} + 1). \label{est 132}
\end{align} 
Applying \eqref{est 133}, \eqref{est 134}, and \eqref{est 132} to \eqref{Define II4}, 
\begin{equation}\label{est 138}
\RomanII_{4}\leq \frac{\nu}{4} \lVert w^{\mathcal{L}} \rVert_{\dot{H}^{\frac{5}{4} + \epsilon}}^{2} + C(M, N_{t}^{\kappa}) ( \lVert w^{\mathcal{L}} \rVert_{\dot{H}^{\epsilon}}^{2} + 1) ( \lVert w^{\mathcal{L}} \rVert_{\dot{H}^{\frac{5}{4}}}^{2} + 1).
\end{equation} 
Finally, applying \eqref{est 135}, \eqref{est 136}, \eqref{est 137}, and \eqref{est 138} to \eqref{est 111} gives us 
\begin{align*}
 \partial_{t} \lVert w^{\mathcal{L}} (t) \rVert_{\dot{H}^{\epsilon}}^{2}  \leq - \frac{25\nu}{16} \lVert w^{\mathcal{L}} \rVert_{\dot{H}^{\frac{5}{4} + \epsilon}}^{2}   +  C(M, N_{t}^{\kappa}) ( \lVert w^{\mathcal{L}} \rVert_{\dot{H}^{\epsilon}}^{2} + 1) ( \lVert w^{\mathcal{L}} \rVert_{\dot{H}^{\frac{5}{4}}}^{2} + 1)
\end{align*}
so that Gronwall's inequality implies the desired result \eqref{est 139} and completes the proof of Proposition \ref{Proposition 4.11}. 
\end{proof} 

\begin{proposition}\label{Proposition 4.12}
Suppose that $u^{\text{in}} \in L_{\sigma}^{2}   \cap \mathscr{C}^{-1+ 2 \kappa}(\mathbb{T}^{3})$ for some $\kappa > 0$ (recall Assumption \ref{Assumption 4.3}). If $T^{\max} < \infty$, then $\limsup_{t\nearrow T^{\max}} \lVert w(t) \rVert_{L^{2}} = \infty$.  
\end{proposition}

\begin{proof}[Proof of Proposition \ref{Proposition 4.12}]
By Proposition \ref{Proposition 4.2}, for any initial data $u^{\text{in}} \in \mathscr{C}^{-1+ 2 \kappa} \cap L_{\sigma}^{2}$, we know that there exists $T^{\max} = T^{\max} (\{L_{t}^{\kappa}\}_{t\geq 0}, u^{\text{in}}) \in (0,\infty]$ and a unique mild solution $w \in \mathcal{M}_{T^{\max}}^{\frac{2\gamma}{5}} \mathscr{C}^{\frac{1}{4} + \frac{3\kappa}{2}}$ with $\gamma = \frac{5}{4} - \frac{\kappa}{2}$ so that $\sup_{t\in [0,T^{\max} ]} t^{\frac{1}{2} - \frac{\kappa}{5}} \lVert w^{\mathcal{L}} (t) \rVert_{\mathscr{C}^{\frac{1}{4} + \frac{3\kappa}{2}}} < \infty$. It follows that 
\begin{equation}\label{est 140}
\lVert w^{\mathcal{L}}(t) \rVert_{H^{\zeta}} < \infty \hspace{3mm}\forall \hspace{1mm} t \in [0, T^{\max}), \zeta < \frac{5}{4} - 2\kappa. 
\end{equation} 
Suppose that there exists some $i_{\max} \in \mathbb{N}_{0}$ such that $T_{i} = T^{\max}$ for all $i \geq i_{\max}$. Then, due to the Sobolev embedding of  
\begin{equation}\label{needed embedding} 
H^{\frac{9}{10} - \kappa}(\mathbb{T}^{3}) \hookrightarrow \mathscr{C}^{-1+ 2\kappa} (\mathbb{T}^{3}),
\end{equation} 
and Proposition \ref{Proposition 4.11}, we can reach a contradiction to $T^{\max}$ and conclude that $T_{i} < T^{\max}$ for all $i\in\mathbb{N}$. This completes the proof of Proposition \ref{Proposition 4.12}.
\end{proof}

We are now ready to prove Theorem \ref{Theorem 2.2}. Its proof is independent of spatial dimension or diffusion, and hence similar to previous works, mainly \cite[Theorem 2.5]{HR24}. Therefore, we leave it in the Appendix \ref{Proof of Theorem 2.2}. 

\subsection{Anderson Hamiltonian}\label{Subsection 4.2}
\hfill\\ We define 
\begin{equation}\label{Burgers' est 21} 
\mathcal{E}^{-\frac{5}{4} - \kappa} \triangleq \mathscr{C}^{-\frac{5}{4} - \kappa} (\mathbb{T}^{3}; \mathbb{M}^{3}) \times \mathscr{C}^{-2\kappa} (\mathbb{T}^{3}; \mathbb{M}^{3})
\end{equation} 
and 
\begin{equation}\label{Define K}
\mathcal{K}^{-\frac{5}{4} - \kappa} \triangleq \overline{\{ ( \Theta_{1}, P \circlesign{\circ} \Theta_{1}  - c: \hspace{1mm} \Theta_{1} \in C^{\infty} (\mathbb{T}^{3}; \mathbb{M}^{3}), c \in \mathbb{R} \}}
\end{equation} 
where the closure is taken w.r.t. $\mathcal{E}^{-\frac{5}{4} - \kappa}$-topology. In order to define 
\begin{equation}\label{Define U}
\mathcal{U} (\Theta) \triangleq -\frac{\nu \Lambda^{\frac{5}{2}}}{2} \Id - \Theta_{1} \hspace{3mm} \text{ for any } \Theta = (\Theta_{1}, \Theta_{2}) \in \mathcal{K}^{-\frac{5}{4} - \kappa} 
\end{equation} 
for any $\kappa > 0$, we define the space of strongly paracontrolled distributions 
\begin{subequations}\label{Burgers' est 24}
\begin{align}
& \chi_{\kappa}(\Theta) \triangleq \{ \phi \in H^{\frac{5}{4}-\kappa}: \hspace{1mm} \phi^{\sharp} \triangleq \phi + \phi \prec P \in H^{\frac{9}{4}-2\kappa} \} \hspace{1mm} \text{ where } \hspace{1mm} P \triangleq \left( 1+  \frac{\nu \Lambda^{\frac{5}{2}}}{2} \right)^{-1} \Theta_{1}, \label{Burgers' est 24a} \\
& \lVert \phi \rVert_{\chi_{\kappa}} \triangleq \lVert \phi \rVert_{H^{\frac{5}{4}-\kappa}} + \lVert \phi^{\sharp} \rVert_{H^{\frac{9}{4}-2\kappa}}, \label{Burgers' est 24b}
\end{align}
\end{subequations}
and rely on the following result; we postpone its proof to  Section \ref{Proof of Proposition on Anderson Hamiltonian} for completeness. 

\begin{proposition}\label{Proposition on Anderson Hamiltonian}
\rm{(\hspace{1sp}\cite[Proposition 6.1]{HR24}, \cite[Proposition 2.4]{Y25d})} Define $\mathcal{K}^{-\frac{5}{4} - \kappa}$ by  \eqref{Define K} and then $\mathcal{K} \triangleq \cup_{0 < \kappa < \kappa_{0}}$ for any $\kappa_{0} > 0$, $C_{\text{op}}$ to be the space of all closed self-adjoint operators with the graph distance where the convergence in this distance is implied by the convergence in the resolvent sense. Then there exist $\kappa_{0} > 0$ and a unique map $\mathcal{U}: \hspace{1mm} \mathcal{K} \mapsto C_{\text{op}}$ such that the following holds.
\begin{enumerate}
\item For any $\Theta = (\Theta_{1}, \Theta_{2}) \in \left(C^{\infty} (\mathbb{T}^{3})\right)^{2} \cap \mathcal{K}$ and $\phi \in H^{\frac{5}{2}}(\mathbb{T}^{3})$, 
\begin{equation}\label{Burgers' est 25}
\mathcal{U} (\Theta) \phi = - \frac{\nu \Lambda^{\frac{5}{2}}}{2} \phi - \Theta_{1} \circlesign{\prec} \phi - \Theta_{1} \circlesign{\succ} \phi - \Theta_{1} \circlesign{\circ} \phi^{\sharp} - \phi \circlesign{\prec} \Theta_{2} - C^{0} (\phi, P, \Theta_{1})
\end{equation} 
where 
\begin{equation}\label{Burgers' est 26}
C^{0} ( \phi, P, \Theta_{1}) \triangleq \Theta_{1} \circlesign{\circ} ( \phi \circlesign{\prec} P) - \phi \circlesign{\prec} (P \circlesign{\circ} \Theta_{1}).
\end{equation} 
In particular, if $\Theta_{2}= P \circlesign{\circ} \Theta_{1}$, then $\mathcal{U} (\Theta) \phi = - \frac{\nu \Lambda^{\frac{5}{2}}}{2}\phi - \Theta_{1} \phi$ where $P \triangleq \left(1+ \frac{\nu \Lambda^{\frac{5}{2}}}{2}  \right)^{-1} \Theta_{1}$.
\item For any $\{\Theta^{n} \}_{n\in\mathbb{N}} \subset C^{\infty} (\mathbb{T}^{3}; \mathbb{M}^{3}) \times C^{\infty} (\mathbb{T}^{3}; \mathbb{M}^{3})$ such that $\Theta^{n} \to \Theta$ in $\mathcal{K}^{-\frac{5}{4} - \kappa}$ as $n\nearrow + \infty$ for some $\kappa \in (0, \kappa_{0})$ and $\Theta \in \mathcal{K}^{-\frac{5}{4} - \kappa}$, $\mathcal{U}(\Theta^{n})$ converges to $\mathcal{U}(\Theta)$ in resolvent sense. Moreover, for any $\kappa \in (0, \kappa_{0})$, there exist two continuous maps $\mathbf{m}, \mathbf{c}: \hspace{1mm} \mathcal{K}^{-\frac{5}{4} - \kappa} \mapsto \mathbb{R}_{+}$ such that 
\begin{equation}\label{est 100}
[ \mathbf{m} (\Theta), \infty) \subset \rho( \mathcal{U} (\Theta)) \hspace{3mm} \forall \hspace{1mm} \Theta \in \mathcal{K}^{-\frac{5}{4} - \kappa}, 
\end{equation} 
where $\rho( \mathcal{U} (\Theta))$ is the resolvent set of $\mathcal{U}(\Theta)$ that satisfies 
\begin{equation}\label{est 65}
\lVert \left( \mathcal{U} (\Theta) +  m \right)^{-1} \phi \rVert_{\chi_{\kappa}} \leq \mathbf{c}(\Theta) \lVert \phi \rVert_{L^{2}} \hspace{3mm} \forall \hspace{1mm} m \geq \mathbf{m}(\Theta). 
\end{equation} 
 \end{enumerate} 
\end{proposition} 

\subsection{Renormalization}\label{Subsection 4.3}
\hfill\\ We define $\{\beta_{j}(k)\}_{k \in \mathbb{Z}^{3} \setminus \{0\}}$ to be a family of $\mathbb{C}$-valued two-sided Brownian motions such that 
\begin{equation}\label{Brownian motion} 
\mathbb{E} [ \partial_{t} \beta_{i}(t,k) \partial_{t} \beta_{j}(s, k') ] = \delta(t-s) 1_{\{ k = - k'\}} 1_{\{ i = j \}}
\end{equation} 
and 
\begin{equation}\label{Define ek}
e_{k}(x) \triangleq e^{i2 \pi k \cdot x} \left( \Id - \frac{k\otimes k}{\lvert k \rvert^{2}} \right), 
\end{equation} 
so that we can rewrite, considering \eqref{projection}, 
\begin{equation}\label{est 66} 
\mathbb{P}_{L} \mathbb{P}_{\neq 0} \xi(t,x) = \sum_{k\in\mathbb{Z}^{3} \setminus \{0\}} e_{k}(x) \partial_{t} \beta(t,k).
\end{equation} 
We define 
\begin{equation}\label{Define F and F lambda}
F(t,k) \triangleq \int_{0}^{t} e^{- \nu \lvert k \rvert^{\frac{5}{2}} (t-s)} d \beta(s,k), \hspace{3mm} F^{\lambda}(t,k) \triangleq \int_{0}^{t} e^{-\nu \lvert k \rvert^{\frac{5}{2}}(t-s)} \mathfrak{l} \left( \frac{\lvert k \rvert}{\lambda} \right) d \beta(s,k)
\end{equation} 
so that 
\begin{equation}\label{est 67}
\mathcal{L}_{\lambda} X(t,x) = \sum_{k\in\mathbb{Z}^{3} \setminus \{0\}} e_{k}(x) F^{\lambda}(t,k).
\end{equation}

\begin{proposition}\label{Proposition 4.14}
For any $\kappa > 0$, we define $\mathcal{K}^{-\frac{5}{4} - \kappa}$ by \eqref{Define K}, $P^{\lambda}$ and $r_{\lambda}$ by \eqref{Define P lambda and r lambda}. Then, for any $t \geq 0$, there exists a distribution $\nabla_{\text{sym}} X(t) \diamondsuit P_{t} \in \mathscr{C}^{-\kappa} (\mathbb{T}^{3}; \mathbb{M}^{3})$ such that 
\begin{equation}\label{est 68} 
\left( \nabla_{\text{sym}} \mathcal{L}_{\lambda^{n}}X, \nabla_{\text{sym}} \mathcal{L}_{\lambda^{n}} X \circlesign{\circ} P^{\lambda^{n}} - r_{\lambda}(t) \Id \right) \to \left( \nabla_{\text{sym}} X, \nabla_{\text{sym}} X \diamondsuit P \right) 
\end{equation}  
as $n\nearrow + \infty$ both in $L^{p} (\Omega; C_{\text{loc}} (\mathbb{R}_{+}; \mathcal{K}^{-\frac{5}{4} - \kappa} ) )$ for any $p \in [1,\infty)$ and $\mathbb{P}$-a.s. Finally, there exists a constant $c > 0$ such that for all $\lambda \geq 1$, 
\begin{equation}\label{logarithmic growth}
r_{\lambda}(t) \leq c \ln (\lambda) 
\end{equation} 
uniformly over all $t \geq 0$. 
\end{proposition}

\begin{proof}[Proof of Proposition \ref{Proposition 4.14}]  
We focus on the convergence of $\nabla_{\text{sym}} \mathcal{L}_{\lambda^{n}} X \diamondsuit P^{\lambda^{n}} - r_{\lambda}(t) \Id \to \nabla_{\text{sym}} X \diamondsuit P$ as $n\nearrow \infty$ in $L^{p}(\Omega; C_{\text{loc}}(\mathbb{R}_{+}; \mathscr{C}^{-2\kappa} (\mathbb{T}^{3}; \mathbb{M}^{3})))$ for any $p\in[1,\infty)$ because the convergence of $\nabla_{\text{sym}} \mathcal{L}_{\lambda^{n}} X$ to $\nabla_{\text{sym}} X$ is easier to show. We denote $X_{\lambda} \triangleq \mathcal{L}_{\lambda} X$ for brevity and compute from \eqref{Define P lambda}
\begin{align}
4 \nabla_{\text{sym}} \mathcal{L}_{\lambda} X \circlesign{\circ} P^{\lambda} =& \begin{pmatrix}
\partial_{1} X_{\lambda}^{1} + \partial_{1} X_{\lambda}^{1} & \partial_{1} X_{\lambda}^{2} + \partial_{2} X_{\lambda}^{1} & \partial_{1} X_{\lambda}^{3} + \partial_{3} X_{\lambda}^{1} \\
\partial_{2} X_{\lambda}^{1} + \partial_{1} X_{\lambda}^{2} & \partial_{2} X_{\lambda}^{2} + \partial_{2} X_{\lambda}^{2} & \partial_{2} X_{\lambda}^{3} + \partial_{3} X_{\lambda}^{2} \\
\partial_{3} X_{\lambda}^{1} + \partial_{1} X_{\lambda}^{3} & \partial_{3} X_{\lambda}^{2} + \partial_{2} X_{\lambda}^{3} & \partial_{3} X_{\lambda}^{3} + \partial_{3} X_{\lambda}^{3} 
\end{pmatrix}   \nonumber \\
&  \circlesign{\circ} \left(1+ \frac{\nu \Lambda^{\frac{5}{2}}}{2}  \right)^{-1}   \begin{pmatrix}
\partial_{1} X_{\lambda}^{1} + \partial_{1} X_{\lambda}^{1} & \partial_{1} X_{\lambda}^{2} + \partial_{2} X_{\lambda}^{1} & \partial_{1} X_{\lambda}^{3} + \partial_{3} X_{\lambda}^{1} \\
\partial_{2} X_{\lambda}^{1} + \partial_{1} X_{\lambda}^{2} & \partial_{2} X_{\lambda}^{2} + \partial_{2} X_{\lambda}^{2} & \partial_{2} X_{\lambda}^{3} + \partial_{3} X_{\lambda}^{2} \\
\partial_{3} X_{\lambda}^{1} + \partial_{1} X_{\lambda}^{3} & \partial_{3} X_{\lambda}^{2} + \partial_{2} X_{\lambda}^{3} & \partial_{3} X_{\lambda}^{3} + \partial_{3} X_{\lambda}^{3} 
\end{pmatrix}  \label{est 72} 
\end{align}
so that each entry consists of four terms of the form 
\begin{align}
& \partial_{i} X_{\lambda}^{j} \circ \left( 1 + \frac{\nu \Lambda^{\frac{5}{2}}}{2}  \right)^{-1} \partial_{l} X_{\lambda}^{m}(x)  \nonumber \\
=& - \sum_{k,k' \in \mathbb{Z}^{3}: k' \neq 0, k \neq k'} \sum_{c,d \geq -1:  \lvert c-d \rvert \leq 1} e^{i 2 \pi k \cdot x} \rho_{c}(k-k') \rho_{d}(k') \mathfrak{l} \left( \frac{ \lvert k-k' \rvert}{\lambda} \right) \mathfrak{l} \left( \frac{\lvert k' \rvert}{\lambda} \right) \left(1+ \frac{\nu \lvert k' \rvert^{\frac{5}{2}}}{2} \right)^{-1} (k-k')_{i} k_{l}'  \nonumber \\
& \times  \left[ F_{j} (k-k') - \frac{(k-k')_{j}}{\lvert k-k' \rvert^{2}} \sum_{r=1}^{3} (k-k')_{r} F_{r} (k-k') \right] \left[ F_{m} (k') - \frac{k_{m}'}{\lvert k'\rvert^{2}} \sum_{r=1}^{3} k_{r}' F_{r}(k') \right]\label{est 71}
\end{align}
where we used \eqref{est 67} and \eqref{Define ek}. Now for all $\alpha, \gamma \in \{1,2,3\}$, we can compute using \eqref{Define F and F lambda} and \eqref{Brownian motion}
\begin{equation}\label{Expectation F}
 \mathbb{E} [ F_{\alpha}(t, k-k') F_{\gamma} (t,k') ] = \frac{1- e^{-2 \nu \lvert k' \rvert^{\frac{5}{2}} t}}{2 \nu \lvert k' \rvert^{\frac{5}{2}}} 1_{\{ k - k' = -k' \}} 1_{\{ \alpha = \gamma \}}
\end{equation} 
so that taking mathematical expectation on \eqref{est 71} gives us 
\begin{align}
& \mathbb{E} \left[  \partial_{i} X_{\lambda}^{j} \circ \left( 1 + \frac{\nu \Lambda^{\frac{5}{2}}}{2}  \right)^{-1} \partial_{l} X_{\lambda}^{m}  \right] (t,x)   \nonumber \\
=& \sum_{k \in \mathbb{Z}^{3} \setminus \{0\}} \mathfrak{l} \left( \frac{ \lvert k \rvert}{\lambda} \right)^{2} \left( \frac{1- e^{-2 \nu \lvert k \rvert^{\frac{5}{2}} t}}{2 \nu \lvert k \rvert^{\frac{5}{2}}} \right) \left(1+ \frac{ \nu \lvert k \rvert^{\frac{5}{2}}}{2} \right)^{-1} k_{i} k_{l} \left[ 1_{\{j=m \}} - \frac{k_{j} k_{m}}{\lvert k \rvert^{2}} \right].  \label{est 73} 
\end{align}
We are now ready to compute $(i,j)$-entry of \eqref{est 72}, $i, j \in \{1,2,3\}$, thanks to \eqref{est 73}. E.g., 
\begin{subequations}
\begin{align}
&\mathbb{E} [ \left( 4 \nabla_{\text{sym}} \mathcal{L}_{\lambda} X \circlesign{\circ} P^{\lambda} \right)_{1,1} ]  = \sum_{k \in \mathbb{Z}^{3} \setminus \{0\}} \mathfrak{l} \left( \frac{ \lvert k \rvert}{\lambda} \right)^{2} \left( \frac{1- e^{-2 \nu \lvert k \rvert^{\frac{5}{2}} t}}{2\nu \lvert k \rvert^{\frac{5}{2}}} \right) \left(1+ \frac{\nu \lvert k \rvert^{\frac{5}{2}}}{2} \right)^{-1} [k_{1}^{2} + \lvert k \rvert^{2}], \label{est 76a}\\
& \mathbb{E} [ \left( 4 \nabla_{\text{sym}} \mathcal{L}_{\lambda} X \circlesign{\circ} P^{\lambda} \right)_{1,2} ]  = \sum_{k \in \mathbb{Z}^{3} \setminus \{0\}} \mathfrak{l} \left( \frac{\lvert k \rvert}{\lambda} \right)^{2} \left( \frac{1- e^{ - 2 \nu \lvert k \rvert^{\frac{5}{2}} t}}{2 \nu \lvert k \rvert^{\frac{5}{2}}} \right) \left(1+ \frac{ \nu \lvert k \rvert^{\frac{5}{2}}}{2} \right)^{-1}  k_{1} k_{2},\label{est 76b}
\end{align}
\end{subequations} 
and other terms can be computed similarly, which lead to 
\begin{equation}\label{est 75}
\mathbb{E} [  4 \nabla_{\text{sym}} \mathcal{L}_{\lambda} X \circlesign{\circ} P^{\lambda}  ] =  \sum_{k \in \mathbb{Z}^{3} \setminus \{0\}} \mathfrak{l} \left( \frac{\lvert k \rvert}{\lambda} \right)^{2} \left( \frac{1- e^{ - 2 \nu \lvert k \rvert^{\frac{5}{2}} t}}{2 \nu \lvert k \rvert^{\frac{1}{2}}} \right) \left(1+ \frac{ \nu \lvert k \rvert^{\frac{5}{2}}}{2} \right)^{-1}  \left[\Id + \frac{k \otimes k}{\lvert k \rvert^{2}} \right].
\end{equation}  

\begin{remark}\label{Remark on renormalization constant}
In contrast to \eqref{est 75}, the analogous computation for the 2D Navier-Stokes equations yields 
\begin{align}\label{est 74}
\mathbb{E} [  4 \nabla_{\text{sym}} \mathcal{L}_{\lambda} X \circlesign{\circ} P^{\lambda}  ] =  \sum_{k \in \mathbb{Z}^{2} \setminus \{0\}} \mathfrak{l} \left( \frac{\lvert k \rvert}{\lambda} \right)^{2} \left( \frac{1- e^{ - 2 \nu \lvert k \rvert^{2} t}}{2 \nu} \right) \left(1+ \frac{ \nu \lvert k \rvert^{2}}{2} \right)^{-1} \Id
\end{align}  
(see \cite[pp. 31--32]{HR24}). Considering that the 2D case is the 3D case with zero third velocity component and independence of $x_{3}$, one may expect \eqref{est 74} with the $2\times 2$ identity matrix simply replaced by $3 \times 3$ case, and hence be surprised to see the additional $\frac{k\otimes k}{\lvert k \rvert^{2}}$ in \eqref{est 75}. In short, for $(1,2)$-entry for example, ``$k_{i} k_{l} \left[ 1_{\{j=m \}} - \frac{k_{j} k_{m}}{\lvert k \rvert^{2}} \right]$'' in \eqref{est 73} in the 3D case results in 
\begin{align*}
& \Bigg[ 2 k_{1} k_{1} [ 1_{\{ 1 = 2 \}} - \frac{k_{1}k_{2}}{\lvert k \rvert^{2}} ] + 2 k_{1} k_{2} [ 1_{\{ 1 =1 \}} - \frac{k_{1}k_{1}}{\lvert k \rvert^{2}} ] + 2 k_{1} k_{2} [ 1_{\{ 2 = 2 \}} - \frac{k_{2}k_{2}}{\lvert k \rvert^{2}} ] + 2 k_{2} k_{2} [ 1_{\{ 1 = 2 \}} - \frac{k_{1}k_{2}}{\lvert k \rvert^{2}} ] \\
& \hspace{5mm}  + k_{1} k_{3} [ 1_{\{ 3 =2 \}} - \frac{k_{3} k_{2}}{\lvert k \rvert^{2}} ] + k_{1}k_{2} [ 1_{\{ 3=3 \}} - \frac{k_{3}k_{3}}{\lvert k \rvert^{2}} ] + k_{3}k_{3} [ 1_{\{ 1 =2 \}} - \frac{k_{1}k_{2}}{\lvert k \rvert^{2}} ] + k_{3} k_{2} [ 1_{\{ 1= 3 \}} - \frac{k_{1} k_{3}}{\lvert k \rvert^{2}} ] \Bigg]
\end{align*}
which simplifies to $k_{1}k_{2}$; on the other hand, the analogous term in the 2D case becomes  
\begin{align*}
\Bigg[ 2 k_{1} k_{1} [ 1_{\{ 1 = 2 \}} - \frac{k_{1}k_{2}}{\lvert k \rvert^{2}} ] + 2 k_{1} k_{2} [ 1_{\{ 1 =1 \}} - \frac{k_{1}k_{1}}{\lvert k \rvert^{2}} ] + 2 k_{1} k_{2} [ 1_{\{ 2 = 2 \}} - \frac{k_{2}k_{2}}{\lvert k \rvert^{2}} ] + 2 k_{2} k_{2} [ 1_{\{ 1 = 2 \}} - \frac{k_{1}k_{2}}{\lvert k \rvert^{2}} ],
\end{align*}
which vanishes, verifying that both \eqref{est 74} in the 2D case and \eqref{est 75} in the 3D case are correct. 
\end{remark} 
We further see that the the non-diagonal entries in $\frac{ k \otimes k}{\lvert k \rvert^{2}}$ of \eqref{est 75} cancel out because $ \mathfrak{l} \left( \frac{\lvert k \rvert}{\lambda} \right)^{2} \left( \frac{1- e^{ - 2 \nu \lvert k \rvert^{\frac{5}{2}} t}}{2 \nu \lvert k \rvert^{\frac{1}{2}}} \right) \left(1+ \frac{ \nu \lvert k \rvert^{\frac{5}{2}}}{2} \right)^{-1}$ is radial, allowing us to conclude using \eqref{Define P lambda and r lambda} that 
\begin{equation}\label{est 219}
\mathbb{E} [  4 \nabla_{\text{sym}} \mathcal{L}_{\lambda} X \circlesign{\circ} P^{\lambda} ] = r_{\lambda}(t) \Id. 
\end{equation} 
We define 
\begin{equation}\label{Define psi0}
\psi_{0}(k,k') \triangleq \sum_{c, d \geq -1: \lvert c-d \rvert \leq 1} \rho_{c}(k) \rho_{d}(k').
\end{equation} 
In what follows, we consider the $(1,1)$-entry among the nine entries due to similarity; this implies according to \eqref{Define r lambda2} that we consider 
\begin{align}
& \mathbb{E} [ \lvert \Delta_{m} \left( ( \nabla_{\text{sym}} \mathcal{L}_{\lambda} X \circlesign{\circ} P^{\lambda} )_{1,1} - (r_{\lambda}^{1} + r_{\lambda}^{2,1} ) \right) (t) \rvert^{2} ]  \nonumber \\
=& \mathbb{E} [\lvert \Delta_{m} ( \nabla_{\text{sym}} \mathcal{L}_{\lambda} X \circlesign{\circ} P^{\lambda} )_{1,1} (t) \rvert^{2} ] - \Delta_{m} (r_{\lambda}^{1} + r_{\lambda}^{2,1})^{2}(t) 1_{\{ m = -1 \}}.  \label{est 78}
\end{align} 
Focusing on the first term, we can compute using \eqref{est 71}, \eqref{Define F and F lambda}, and \eqref{Define psi0}
\begin{subequations}\label{est 0}
\begin{align}
&  \mathbb{E} [ \lvert \Delta_{m} ( \nabla_{\text{sym}} \mathcal{L}_{\lambda} X \circlesign{\circ} P^{\lambda})_{1,1} (t) \rvert^{2} ]  = \sum_{i,j=1}^{7} \RomanI_{i,j}, \\
& \RomanI_{i,j} \triangleq \frac{1}{4} \mathbb{E} \Bigg[ \sum_{k, k', \tilde{k},\tilde{k}' \in \mathbb{Z}^{3} \setminus \{0\}} e^{i2 \pi (k+k') \cdot x} \overline{ e^{i2\pi(\tilde{k} + \tilde{k}') \cdot x}} \rho_{m} (k+k') \rho_{m}( \tilde{k} + \tilde{k}') \psi_{0} (k, k') \psi_{0} (\tilde{k}, \tilde{k}') \nonumber \\
&\hspace{10mm} \times \mathfrak{l} \left( \frac{\lvert k \rvert}{\lambda} \right)   \mathfrak{l} \left( \frac{\lvert k' \rvert}{\lambda} \right)   \mathfrak{l} \left( \frac{\lvert \tilde{k} \rvert}{\lambda} \right)   \mathfrak{l} \left( \frac{\lvert \tilde{k}' \rvert}{\lambda} \right) \left( \frac{1}{\nu \lvert k'\rvert^{\frac{5}{2}} + 2} \right) \left( \frac{1}{\nu \lvert \tilde{k}' \rvert^{\frac{5}{2}} + 2} \right)   \tilde{\RomanI}_{i,j} \Bigg] 
\end{align}
\end{subequations} 
for all $i, j \in \{1, \hdots, 7 \}$, and e.g. 
\begin{subequations} 
\begin{align}
\tilde{\RomanI}_{1,1} \triangleq& (4k_{1}k_{1}' + k_{2}k_{2}' + k_{3}k_{3}') \Bigg[ \int_{0}^{t}\int_{0}^{t} e^{- \nu \lvert k \rvert^{\frac{5}{2}} (t-s) - \nu \lvert k' \rvert^{\frac{5}{2}} (t-s')} d \beta_{1}(s,k) d \beta_{1}(s', k') \label{est 61} \\
& \hspace{20mm} - \frac{k_{1}'}{\lvert k ' \rvert^{2}} \sum_{r=1}^{3}k_{r}' \int_{0}^{t} \int_{0}^{t} e^{- \nu \lvert k \rvert^{\frac{5}{2}} (t-s) - \nu \lvert k' \rvert^{\frac{5}{2}} (t-s')} d \beta_{1}(s,k) d \beta_{r} (s', k')\nonumber \\
& \hspace{20mm} - \frac{k_{1}}{\lvert k \rvert^{2}} \sum_{r=1}^{3} k_{r} \int_{0}^{t} \int_{0}^{t} e^{- \nu \lvert k' \rvert^{\frac{5}{2}}(t-s') - \nu \lvert k \rvert^{\frac{5}{2}} (t-s)} d \beta_{1}(s', k') d \beta_{r}(s, k)\nonumber \\
& \hspace{20mm} + \frac{k_{1}k_{1}'}{\lvert k \rvert^{2} \lvert k' \rvert^{2}} \sum_{r=1}^{3} \sum_{w=1}^{3} k_{r} k_{w}' \int_{0}^{t} \int_{0}^{t} e^{- \nu \lvert k \rvert^{\frac{5}{2}}(t-s) - \nu \lvert k' \rvert^{\frac{5}{2}}(t-s')} d \beta_{r}(s,k) d \beta_{w}(s', k')  \Bigg]\nonumber \\
& \times (4\tilde{k}_{1}\tilde{k}_{1}' + \tilde{k}_{2}\tilde{k}_{2}' + \tilde{k}_{3}\tilde{k}_{3}') \Bigg[ \int_{0}^{t}\int_{0}^{t} e^{- \nu \lvert \tilde{k} \rvert^{\frac{5}{2}} (t-\tilde{s}) - \nu \lvert \tilde{k}' \rvert^{\frac{5}{2}} (t-\tilde{s}')} \overline{d \beta_{1}(\tilde{s},\tilde{k}) d \beta_{1}(\tilde{s}', \tilde{k}')} \nonumber\\
& \hspace{20mm} - \frac{\tilde{k}_{1}'}{\lvert \tilde{k} ' \rvert^{2}} \sum_{r=1}^{3}\tilde{k}_{r}' \int_{0}^{t} \int_{0}^{t} e^{- \nu \lvert \tilde{k} \rvert^{\frac{5}{2}} (t-\tilde{s}) - \nu \lvert \tilde{k}' \rvert^{\frac{5}{2}} (t-\tilde{s}')} \overline{d \beta_{1}(\tilde{s},\tilde{k}) d \beta_{r} (\tilde{s}', \tilde{k}')}\nonumber \\
& \hspace{20mm} - \frac{\tilde{k}_{1}}{\lvert \tilde{k} \rvert^{2}} \sum_{r=1}^{3} \tilde{k}_{r} \int_{0}^{t} \int_{0}^{t} e^{ - \nu \lvert \tilde{k} \rvert^{\frac{5}{2}} (t-\tilde{s}) - \nu \lvert \tilde{k}' \rvert^{\frac{5}{2}}(t-\tilde{s}') } \overline{d \beta_{1}(\tilde{s}', \tilde{k}') d \beta_{r}(\tilde{s}, \tilde{k})}  \nonumber \\
& \hspace{20mm} + \frac{\tilde{k}_{1}\tilde{k}_{1}'}{\lvert \tilde{k} \rvert^{2} \lvert \tilde{k}' \rvert^{2}} \sum_{r=1}^{3} \sum_{w=1}^{3} \tilde{k}_{r} \tilde{k}_{w}' \int_{0}^{t} \int_{0}^{t} e^{- \nu \lvert \tilde{k} \rvert^{\frac{5}{2}}(t-\tilde{s}) - \nu \lvert \tilde{k}' \rvert^{\frac{5}{2}}(t-\tilde{s}')} \overline{d \beta_{r}(\tilde{s},\tilde{k}) d \beta_{w}(\tilde{s}', \tilde{k}')}  \Bigg], \nonumber \\ 
\tilde{\RomanI}_{1,2} \triangleq& (4k_{1}k_{1}' + k_{2}k_{2}' + k_{3}k_{3}') \Bigg[ \int_{0}^{t}\int_{0}^{t} e^{- \nu \lvert k \rvert^{\frac{5}{2}} (t-s) - \nu \lvert k' \rvert^{\frac{5}{2}} (t-s')} d \beta_{1}(s,k) d \beta_{1}(s', k') \label{est 217}\\
& \hspace{20mm} - \frac{k_{1}'}{\lvert k ' \rvert^{2}} \sum_{r=1}^{3}k_{r}' \int_{0}^{t} \int_{0}^{t} e^{- \nu \lvert k \rvert^{\frac{5}{2}} (t-s) - \nu \lvert k' \rvert^{\frac{5}{2}} (t-s')} d \beta_{1}(s,k) d \beta_{r} (s', k') \nonumber \\
& \hspace{20mm} - \frac{k_{1}}{\lvert k \rvert^{2}} \sum_{r=1}^{3} k_{r} \int_{0}^{t} \int_{0}^{t} e^{- \nu \lvert k' \rvert^{\frac{5}{2}}(t-s') - \nu \lvert k \rvert^{\frac{5}{2}} (t-s)} d \beta_{1}(s', k') d \beta_{r}(s, k) \nonumber \\
& \hspace{20mm} + \frac{k_{1}k_{1}'}{\lvert k \rvert^{2} \lvert k' \rvert^{2}} \sum_{r=1}^{3} \sum_{w=1}^{3} k_{r} k_{w}' \int_{0}^{t} \int_{0}^{t} e^{- \nu \lvert k \rvert^{\frac{5}{2}}(t-s) - \nu \lvert k' \rvert^{\frac{5}{2}}(t-s')} d \beta_{r}(s,k) d \beta_{w}(s', k')  \Bigg] \nonumber \\
& \times \tilde{k}_{1} \tilde{k}_{2}' \Bigg[ \int_{0}^{t}\int_{0}^{t} e^{- \nu \lvert \tilde{k} \rvert^{\frac{5}{2}} (t-\tilde{s}) - \nu \lvert \tilde{k}' \rvert^{\frac{5}{2}} (t-\tilde{s}')} \overline{d \beta_{2}(\tilde{s},\tilde{k}) d \beta_{1}(\tilde{s}', \tilde{k}')} \nonumber \\
& \hspace{20mm} - \frac{\tilde{k}_{1}'}{\lvert \tilde{k} ' \rvert^{2}} \sum_{r=1}^{3}\tilde{k}_{r}' \int_{0}^{t} \int_{0}^{t} e^{- \nu \lvert \tilde{k} \rvert^{\frac{5}{2}} (t-\tilde{s}) - \nu \lvert \tilde{k}' \rvert^{\frac{5}{2}} (t-\tilde{s}')} \overline{d \beta_{2}(\tilde{s},\tilde{k}) d \beta_{r} (\tilde{s}', \tilde{k}')} \nonumber \\
& \hspace{20mm} - \frac{\tilde{k}_{2}}{\lvert \tilde{k} \rvert^{2}} \sum_{r=1}^{3} \tilde{k}_{r} \int_{0}^{t} \int_{0}^{t} e^{- \nu \lvert \tilde{k} \rvert^{\frac{5}{2}}(t-\tilde{s}) - \nu \lvert \tilde{k}' \rvert^{\frac{5}{2}} (t-\tilde{s}')} \overline{d \beta_{1}(\tilde{s}', \tilde{k}') d \beta_{r}(\tilde{s}, \tilde{k})} \nonumber \\
& \hspace{20mm} + \frac{\tilde{k}_{2}\tilde{k}_{1}'}{\lvert \tilde{k} \rvert^{2} \lvert \tilde{k}' \rvert^{2}} \sum_{r=1}^{3} \sum_{w=1}^{3} \tilde{k}_{r} \tilde{k}_{w}' \int_{0}^{t} \int_{0}^{t} e^{- \nu \lvert \tilde{k} \rvert^{\frac{5}{2}}(t-\tilde{s}) - \nu \lvert \tilde{k}' \rvert^{\frac{5}{2}}(t-\tilde{s}')} \overline{d \beta_{r}(\tilde{s},\tilde{k}) d \beta_{w}(\tilde{s}', \tilde{k}')}  \Bigg], \nonumber  \\
\tilde{\RomanI}_{2,1} \triangleq& k_{1}k_{2}' \Bigg[ \int_{0}^{t}\int_{0}^{t} e^{- \nu \lvert k \rvert^{\frac{5}{2}} (t-s) - \nu \lvert k' \rvert^{\frac{5}{2}} (t-s')} d \beta_{2}(s,k) d \beta_{1}(s', k')  \label{Define tilde I2,1} \\
& \hspace{20mm} - \frac{k_{1}'}{\lvert k ' \rvert^{2}} \sum_{r=1}^{3}k_{r}' \int_{0}^{t} \int_{0}^{t} e^{- \nu \lvert k \rvert^{\frac{5}{2}} (t-s) - \nu \lvert k' \rvert^{\frac{5}{2}} (t-s')} d \beta_{2}(s,k) d \beta_{r} (s', k')  \nonumber \\
& \hspace{20mm} - \frac{k_{2}}{\lvert k \rvert^{2}} \sum_{r=1}^{3} k_{r} \int_{0}^{t} \int_{0}^{t} e^{- \nu \lvert k \rvert^{\frac{5}{2}}(t-s) - \nu \lvert k' \rvert^{\frac{5}{2}} (t-s')} d \beta_{1}(s', k') d \beta_{r}(s, k)  \nonumber \\
& \hspace{20mm} + \frac{k_{2}k_{1}'}{\lvert k \rvert^{2} \lvert k' \rvert^{2}} \sum_{r=1}^{3} \sum_{w=1}^{3} k_{r} k_{w}' \int_{0}^{t} \int_{0}^{t} e^{- \nu \lvert k \rvert^{\frac{5}{2}}(t-s) - \nu \lvert k' \rvert^{\frac{5}{2}}(t-s')} d \beta_{r}(s,k) d \beta_{w}(s', k')  \Bigg]  \nonumber \\
& \times (4\tilde{k}_{1}\tilde{k}_{1}' + \tilde{k}_{2}\tilde{k}_{2}' + \tilde{k}_{3}\tilde{k}_{3}') \Bigg[ \int_{0}^{t}\int_{0}^{t} e^{- \nu \lvert \tilde{k} \rvert^{\frac{5}{2}} (t-\tilde{s}) - \nu \lvert \tilde{k}' \rvert^{\frac{5}{2}} (t-\tilde{s}')} \overline{d \beta_{1}(\tilde{s},\tilde{k}) d \beta_{1}(\tilde{s}', \tilde{k}')}  \nonumber \\
& \hspace{20mm} - \frac{\tilde{k}_{1}'}{\lvert \tilde{k} ' \rvert^{2}} \sum_{r=1}^{3}\tilde{k}_{r}' \int_{0}^{t} \int_{0}^{t} e^{- \nu \lvert \tilde{k} \rvert^{\frac{5}{2}} (t-\tilde{s}) - \nu \lvert \tilde{k}' \rvert^{\frac{5}{2}} (t-\tilde{s}')} \overline{d \beta_{1}(\tilde{s},\tilde{k}) d \beta_{r} (\tilde{s}', \tilde{k}')}  \nonumber \\
& \hspace{20mm} - \frac{\tilde{k}_{1}}{\lvert \tilde{k} \rvert^{2}} \sum_{r=1}^{3} \tilde{k}_{r} \int_{0}^{t} \int_{0}^{t} e^{- \nu \lvert \tilde{k}' \rvert^{\frac{5}{2}}(t-\tilde{s}') - \nu \lvert \tilde{k} \rvert^{\frac{5}{2}} (t-\tilde{s})} \overline{d \beta_{1}(\tilde{s}', \tilde{k}') d \beta_{r}(\tilde{s}, \tilde{k})}  \nonumber \\
& \hspace{20mm} + \frac{\tilde{k}_{1}\tilde{k}_{1}'}{\lvert \tilde{k} \rvert^{2} \lvert \tilde{k}' \rvert^{2}} \sum_{r=1}^{3} \sum_{w=1}^{3} \tilde{k}_{r} \tilde{k}_{w}' \int_{0}^{t} \int_{0}^{t} e^{- \nu \lvert \tilde{k} \rvert^{\frac{5}{2}}(t-\tilde{s}) - \nu \lvert \tilde{k}' \rvert^{\frac{5}{2}}(t-\tilde{s}')} \overline{d \beta_{r}(\tilde{s},\tilde{k}) d \beta_{w}(\tilde{s}', \tilde{k}')}  \Bigg]. \nonumber 
\end{align}
\end{subequations}
For brevity, we choose to not write out the rest of the terms $\{\tilde{\RomanI}_{i,j}\}_{i,j \in \{1, \hdots, 7\}: (i,j) \neq (1,1), (1,2), (2,1)}$. We note that as much as some symmetry is desired for reduction of computations, unfortunately, $\RomanI_{i,j} \neq \RomanI_{j,i}$ although they certainly are similar. Let us define for brevity
\begin{subequations}\label{Define ABCD}
\begin{align}
& A \triangleq \int_{0}^{t}\int_{0}^{t} e^{- \nu \lvert k \rvert^{\frac{5}{2}} (t-s) - \nu \lvert k' \rvert^{\frac{5}{2}} (t-s')}, \hspace{3mm} \tilde{A} \triangleq \int_{0}^{t}\int_{0}^{t} e^{- \nu \lvert \tilde{k} \rvert^{\frac{5}{2}} (t-\tilde{s}) - \nu \lvert \tilde{k}' \rvert^{\frac{5}{2}} (t-\tilde{s}')}, \\
&B_{j} \triangleq \frac{k_{j}}{\lvert k \rvert^{2}} A, \hspace{3mm} B_{j}' \triangleq \frac{k_{j}'}{\lvert k' \rvert^{2}} A, \hspace{3mm} \tilde{B}_{j} \triangleq \frac{\tilde{k}_{j}}{\lvert \tilde{k} \rvert^{2}} \tilde{A}, \hspace{3mm} \tilde{B}_{j}' \triangleq \frac{\tilde{k}_{j}'}{\lvert \tilde{k}' \rvert^{2}} \tilde{A}, \hspace{3mm} j \in \{1,2,3\}, \\ 
& D_{ij} \triangleq \frac{k_{i} k_{j}'}{\lvert k \rvert^{2} \lvert k' \rvert^{2}} A, \hspace{3mm} \tilde{D}_{ij} \triangleq \frac{ \tilde{k}_{i} \tilde{k}_{j}'}{\lvert \tilde{k} \rvert^{2} \lvert \tilde{k}' \rvert^{2}} \tilde{A} \hspace{3mm} i, j \in \{1,2,3\}, 
\end{align}
\end{subequations} 
so that we can compute using \eqref{est 61} and \eqref{Define ABCD},  
\begin{align}
\mathbb{E}[ \tilde{\RomanI}_{1,1} ] &= (4k_{1}k_{1}' + k_{2}k_{2}' + k_{3}k_{3}') (4\tilde{k}_{1}\tilde{k}_{1}' + \tilde{k}_{2}\tilde{k}_{2}' + \tilde{k}_{3}\tilde{k}_{3}') \nonumber\\
& \times  \mathbb{E} \Bigg[  \Bigg( A d \beta_{1}(s,k) d \beta_{1}(s', k') - B_{1}' d \beta_{1}(s,k) \sum_{r=1}^{3} k_{r}' d \beta_{r}(s', k')  \nonumber\\
& \hspace{20mm} - B_{1} d \beta_{1}(s', k') \sum_{r=1}^{3} k_{r} d \beta_{r}(s,k) + D_{11} \sum_{r=1}^{3} k_{r}  d \beta_{r}(s,k) \sum_{w=1}^{3} k_{w}' d \beta_{w}(s', k') \Bigg) \nonumber\\
& \hspace{5mm}  \times \Bigg( \tilde{A} \overline{ d \beta_{1} (\tilde{s}, \tilde{k}) d \beta_{1} (\tilde{s}', \tilde{k}' ) } - \tilde{B}_{1}' \overline{ d \beta_{1} (\tilde{s}, \tilde{k}) \sum_{r=1}^{3} \tilde{k}_{r}'  d \beta_{r}( \tilde{s}', \tilde{k}')}  \nonumber\\
& \hspace{20mm} - \tilde{B}_{1} \overline{ d \beta_{1} (\tilde{s}', \tilde{k}') \sum_{r=1}^{3} \tilde{k}_{r} d \beta_{r}(\tilde{s}, \tilde{k})} + \tilde{D}_{11} \overline{\sum_{r=1}^{3} \tilde{k}_{r}d \beta_{r}(\tilde{s}, \tilde{k})  \sum_{w=1}^{3} \tilde{k}_{w}' d \beta_{w} (\tilde{s}', \tilde{k}' ) } \Bigg]  \nonumber \\
=&   (4k_{1}k_{1}' + k_{2}k_{2}' + k_{3}k_{3}') (4\tilde{k}_{1}\tilde{k}_{1}' + \tilde{k}_{2}\tilde{k}_{2}' + \tilde{k}_{3}\tilde{k}_{3}') \sum_{i=1}^{16} \tilde{\RomanI}_{1,1,i}.  \label{est 1}
\end{align}
To continue, we turn to 
\begin{equation}\label{Wick} 
\mathbb{E} [ \xi_{1} \xi_{2} \xi_{3} \xi_{4}] = \mathbb{E} [\xi_{2} \xi_{3}] \mathbb{E} [\xi_{1} \xi_{4}] + \mathbb{E} [\xi_{2} \xi_{4}] \mathbb{E} [\xi_{1} \xi_{3}] + \mathbb{E} [\xi_{3} \xi_{4}] \mathbb{E} [\xi_{1} \xi_{2}] 
\end{equation}   
(cf. \cite{J97}, more precisely \cite{Y21c}) to estimate e.g. 
\begin{align*}
&\tilde{\RomanI}_{1,1,1} \triangleq A \tilde{A} \mathbb{E} [ d \beta_{1}(s,k) d \beta_{1}(s', k') \overline{ d \beta_{1} (\tilde{s}, \tilde{k}) d \beta_{1}(\tilde{s}', \tilde{k}') } ] \\
=& A \tilde{A} ds ds' d \tilde{s} d \tilde{s}' [  \delta(s'-\tilde{s}) 1_{\{k' = - \tilde{k} \}} \delta(s- \tilde{s}') 1_{\{ k = - \tilde{k}' \}}  \\
& \hspace{10mm}  + \delta(s'- \tilde{s}') 1_{\{ k ' = - \tilde{k}' \}} \delta(s- \tilde{s}) 1_{\{ k = - \tilde{k} \}} + \delta(\tilde{s} - \tilde{s}') 1_{\{ \tilde{k} = - \tilde{k} ' \}} \delta(s- s') 1_{ \{ k = -k' \}} ], 
\end{align*} 
and 
\begin{align*}
&\tilde{\RomanI}_{1,1,2} \triangleq -A \tilde{B}_{1}' \mathbb{E} [ d \beta_{1}(s,k) d \beta_{1}(s', k') \overline{ d \beta_{1} (\tilde{s}, \tilde{k}) \sum_{r=1}^{3} \tilde{k}_{r}' d \beta_{r}(\tilde{s}', \tilde{k}') } ] \\
=& -A \tilde{B}_{1}' ds ds' d \tilde{s} d \tilde{s}' [  \delta(s'-\tilde{s}) 1_{\{k' = - \tilde{k} \}} \delta(s- \tilde{s}') 1_{\{ k = - \tilde{k}' \}} \tilde{k}_{1}'  \\
&\hspace{10mm}  + \delta(s'- \tilde{s}') 1_{\{ k ' = - \tilde{k}' \}} \tilde{k}_{1}'  \delta(s- \tilde{s}) 1_{\{ k = - \tilde{k} \}} + \delta(\tilde{s} - \tilde{s}') 1_{\{ \tilde{k} = - \tilde{k} ' \}}  \tilde{k}_{1}' \delta(s- s') 1_{ \{ k = -k' \}} ]. 
\end{align*} 
Computing $\tilde{\RomanI}_{1,1,i}$ for $i \in \{3, \hdots, 16\}$ similarly leads to    
\begin{align}
\RomanI_{1,1} =& \frac{1}{4} \sum_{k, k' \in \mathbb{Z}^{3} \setminus \{0\}} e^{i4 \pi (k+k') \cdot x} \rho_{m}^{2}(k+k') \lvert\psi_{0}(k,k') \rvert^{2} \mathfrak{l} \left( \frac{\lvert k \rvert}{\lambda} \right)^{2} \mathfrak{l} \left( \frac{\lvert k'\rvert}{\lambda} \right)^{2}  \nonumber \\
& \hspace{10mm} \times \left( \frac{1}{\nu \lvert k'\rvert^{\frac{5}{2}} + 2} \right)  \left( \frac{1}{\nu \lvert k \rvert^{\frac{5}{2}} + 2} \right) \left( \frac{1- e^{-2 \nu \lvert k \rvert^{\frac{5}{2}} t}}{2 \nu \lvert k \rvert^{\frac{5}{2}}} \right) \left( \frac{1- e^{-2 \nu \lvert k' \rvert^{\frac{5}{2}} t}}{2 \nu \lvert k' \rvert^{\frac{5}{2}}} \right)   \nonumber  \\
&\hspace{10mm} \times  ( 4k_{1}k_{1}' + k_{2}k_{2}' + k_{3}k_{3}')^{2} \Bigg[ 1 - \frac{k_{1}^{2}}{\lvert k\rvert^{2}} - \frac{ (k_{1}')^{2}}{\lvert k'\rvert^{2}} + \frac{ k_{1}^{2} (k_{1}')^{2}}{\lvert k \rvert^{2} \lvert k'\rvert^{2}} \Bigg]  \nonumber \\
& +  \frac{1}{4} \sum_{k, k' \in \mathbb{Z}^{3} \setminus \{0\}} e^{i4 \pi (k+k') \cdot x} \rho_{m}^{2}(k+k') \lvert\psi_{0}(k,k') \rvert^{2} \mathfrak{l} \left( \frac{\lvert k \rvert}{\lambda} \right)^{2} \mathfrak{l} \left( \frac{\lvert k'\rvert}{\lambda} \right)^{2} \nonumber \\
& \hspace{10mm} \times \left( \frac{1}{\nu \lvert k'\rvert^{\frac{5}{2}} + 2} \right)^{2}   \left( \frac{1- e^{-2 \nu \lvert k \rvert^{\frac{5}{2}} t}}{2 \nu \lvert k \rvert^{\frac{5}{2}}} \right) \left( \frac{1- e^{-2 \nu \lvert k' \rvert^{\frac{5}{2}} t}}{2 \nu \lvert k' \rvert^{\frac{5}{2}}} \right)   \nonumber \\
&\hspace{10mm} \times  ( 4k_{1}k_{1}' + k_{2}k_{2}' + k_{3}k_{3}')^{2} \Bigg[ 1 - \frac{k_{1}^{2}}{\lvert k\rvert^{2}} - \frac{ (k_{1}')^{2}}{\lvert k'\rvert^{2}} + \frac{ k_{1}^{2} (k_{1}')^{2}}{\lvert k \rvert^{2} \lvert k'\rvert^{2}} \Bigg]  \nonumber \\
&+ \frac{1}{4} \sum_{k, \tilde{k} \in\mathbb{Z}^{3} \setminus \{0\}} \rho_{m}^{2}(0) \mathfrak{l} \left( \frac{\lvert k \rvert}{\lambda} \right)^{2} \mathfrak{l} \left( \frac{\lvert \tilde{k} \rvert}{\lambda} \right)^{2} \left( \frac{1}{\nu \lvert k \rvert^{\frac{5}{2}} + 2} \right) \left( \frac{1}{\nu \lvert \tilde{k} \rvert^{\frac{5}{2}} + 2} \right)  \left( \frac{1- e^{-2 \nu \lvert k \rvert^{\frac{5}{2}} t}}{2 \nu \lvert k \rvert^{\frac{5}{2}}} \right) \left( \frac{1- e^{-2 \nu \lvert \tilde{k} \rvert^{\frac{5}{2}} t}}{2 \nu \lvert \tilde{k} \rvert^{\frac{5}{2}}} \right)  \nonumber \\
& \hspace{10mm}  \times (4k_{1}^{2} + k_{2}^{2} + k_{3}^{2}) (4 \tilde{k}_{1}^{2} + \tilde{k}_{2}^{2} + \tilde{k}_{3}^{2})  \Bigg[1 - \frac{k_{1}^{2}}{\lvert k\rvert^{2}} - \frac{ \tilde{k}_{1}^{2}}{\lvert \tilde{k}\rvert^{2}} + \frac{ k_{1}^{2} \tilde{k}_{1}^{2}}{\lvert k \rvert^{2} \lvert \tilde{k}\rvert^{2}}   \Bigg] \label{est 26}, 
\end{align}
where the first, second, and third sums correspond respectively to groups that consists of 
\begin{align*}
&1_{\{ k' = -\tilde{k}\}} 1_{\{ k = - \tilde{k}'\}} \delta(s'- \tilde{s}) \delta(s- \tilde{s}'),\\
&1_{\{ k' = - \tilde{k}\}} 1_{\{ k = - \tilde{k}\}} \delta(s' - \tilde{s}') \delta(s- \tilde{s}),\\
&1_{\{ \tilde{k} = - \tilde{k}'\}} 1_{\{ k = -k' \}} \delta(s-s') \delta(\tilde{s} - \tilde{s}').
\end{align*}   
Working from \eqref{est 217}, we can also compute 
\begin{align}
\RomanI_{1,2} =& \frac{1}{4}  \sum_{k, k' \in \mathbb{Z}^{3} \setminus \{0\}} e^{i4 \pi (k+k') \cdot x}  \rho_{m}^{2}(k+k')  \lvert \psi_{0} (k, k') \rvert^{2}   \mathfrak{l} \left( \frac{\lvert k \rvert}{\lambda} \right)^{2} \mathfrak{l} \left( \frac{\lvert k'\rvert}{\lambda} \right)^{2}    \left( \frac{1}{\nu \lvert k'\rvert^{\frac{5}{2}} + 2} \right) \left( \frac{1}{\nu \lvert k \rvert^{\frac{5}{2}} + 2} \right)    \nonumber    \\
& \hspace{8mm}  \times  \left( \frac{ 1- e^{-2\nu \lvert k \rvert^{\frac{5}{2}} t}}{2\nu \lvert k \rvert^{\frac{5}{2}}} \right) \left( \frac{ 1- e^{-2\nu \lvert k' \rvert^{\frac{5}{2}} t}}{2\nu \lvert k' \rvert^{\frac{5}{2}}} \right)  k_{1}' k_{2}  (4k_{1}k_{1}' + k_{2}k_{2}' + k_{3}k_{3}') \Bigg[  - \frac{k_{1}' k_{2}'}{\lvert k' \rvert^{2}} + \frac{k_{1}^{2} k_{1}' k_{2}'}{\lvert k \rvert^{2} \lvert k'\rvert^{2}} \Bigg]  \nonumber \\
+& \frac{1}{4} \sum_{k, k' \in \mathbb{Z}^{3} \setminus \{0\}} e^{i4 \pi (k+k') \cdot x}  \rho_{m}^{2}(k+k') \lvert \psi_{0} (k, k') \rvert^{2}    \mathfrak{l} \left( \frac{\lvert k \rvert}{\lambda} \right)^{2} \mathfrak{l} \left( \frac{\lvert k'\rvert}{\lambda} \right)^{2}   \left( \frac{1}{\nu \lvert k'\rvert^{\frac{5}{2}} + 2} \right)^{2}   \nonumber   \\
& \hspace{8mm}  \times \left( \frac{ 1- e^{-2\nu \lvert k \rvert^{\frac{5}{2}} t}}{2\nu \lvert k \rvert^{\frac{5}{2}}} \right) \left( \frac{ 1- e^{-2\nu \lvert k' \rvert^{\frac{5}{2}} t}}{2\nu \lvert k' \rvert^{\frac{5}{2}}} \right)  (4k_{1}k_{1}' + k_{2}k_{2}' + k_{3}k_{3}') k_{1} k_{2}'   \Bigg[  -\frac{k_{1} k_{2}}{\lvert k \rvert^{2}} + \frac{k_{1}k_{2} (k_{1}')^{2}}{\lvert k \rvert^{2} \lvert k'\rvert^{2}} \Bigg]  \nonumber \\
+& \frac{1}{4} \sum_{k, \tilde{k} \in \mathbb{Z}^{3} \setminus \{0\}}  \rho_{m}^{2}(0)    \mathfrak{l} \left( \frac{\lvert k \rvert}{\lambda} \right)^{2}  \mathfrak{l}  \left( \frac{\lvert \tilde{k} \rvert}{\lambda} \right)^{2} \left( \frac{1}{\nu \lvert k\rvert^{\frac{5}{2}} + 2} \right) \left( \frac{1}{\nu \lvert \tilde{k} \rvert^{\frac{5}{2}} + 2} \right)  \left( \frac{ 1- e^{-2\nu \lvert k \rvert^{\frac{5}{2}} t}}{2\nu \lvert k \rvert^{\frac{5}{2}}} \right) \left( \frac{ 1- e^{-2\nu \lvert \tilde{k} \rvert^{\frac{5}{2}} t}}{2\nu \lvert \tilde{k} \rvert^{\frac{5}{2}}} \right)   \nonumber \\
& \hspace{8mm} \times  (4k_{1}k_{1} + k_{2}k_{2} + k_{3}k_{3}) \tilde{k}_{1} \tilde{k}_{2}    \Bigg[ - \frac{ \tilde{k}_{1}\tilde{k}_{2}}{\lvert \tilde{k} \rvert^{2}} + \frac{k_{1}^{2} \tilde{k}_{1} \tilde{k}_{2}}{\lvert k \rvert^{2} \lvert \tilde{k} \rvert^{2}}  \Bigg]. \label{est 27}
\end{align} 
We remark on the following deduction that can be made due to some symmetry and reduce computations. 
\begin{remark}\label{Remark on deduction}  
Looking at the definitions of $\tilde{\RomanI}_{1,2}$ in \eqref{est 217} and  $\tilde{\RomanI}_{2,1}$ in \eqref{Define tilde I2,1}, we observe that we only have to take the result of $\tilde{\RomanI}_{1,2}$ and carefully swap $k \leftrightarrow \tilde{k}$ and $s \leftrightarrow \tilde{s}$. Then, in the first sum that consists of $1_{ \{ k' = - \tilde{k} \}} 1_{\{ k = - \tilde{k}' \}} \delta(s'- \tilde{s}) \delta(s- \tilde{s}')$, we have $k \to \tilde{k} \to - k'$, $k' \to \tilde{k}' \to -k$, $s \to \tilde{s} \to s', s' \to \tilde{s}' \to s$, so that we only need to swap $k \leftrightarrow k'$ and $s \leftrightarrow s'$ in \eqref{est 27} to obtain the appropriate version for $\tilde{\RomanI}_{2,1}$. Similarly, in the second sum that consists of $1_{\{ k' = - \tilde{k} \}} 1_{\{ k = - \tilde{k} \}} \delta(s' - \tilde{s}') \delta(s- \tilde{s})$, we have $k \to \tilde{k} \to - k$, $k' \to \tilde{k}' \to - k'$, $s\to \tilde{s} \to s$, and $s' \to \tilde{s}' \to s'$, so that the second sum remains the same as that in the $\tilde{\RomanI}_{1,2}$. Finally, in the third sum corresponding to $1_{ \{  \tilde{k} = - \tilde{k}'\}} 1_{ \{ k = - k' \}} \delta(s- s') \delta(\tilde{s} - \tilde{s}')$, we have $k \to \tilde{k} \to - \tilde{k}'$, $\tilde{k} \to k \to -k'$, $s \to \tilde{s} \to \tilde{s}'$, and $s' \to \tilde{s}' \to \tilde{s}$, but swapping $k \leftrightarrow k', s \leftrightarrow s'$ once more implies that we only need to swap $k \leftrightarrow \tilde{k}$ and $s \leftrightarrow \tilde{s}$ in \eqref{est 27} to obtain the appropriate version of $\tilde{\RomanI}_{2,1}$. Consequently, 
\begin{align}
\RomanI_{2,1} = & \frac{1}{4}  \sum_{k, k' \in \mathbb{Z}^{3} \setminus \{0\}} e^{i4 \pi (k+k') \cdot x}  \rho_{m}^{2}(k+k')  \lvert \psi_{0} (k, k') \rvert^{2}   \mathfrak{l} \left( \frac{\lvert k \rvert}{\lambda} \right)^{2} \mathfrak{l} \left( \frac{\lvert k'\rvert}{\lambda} \right)^{2}    \left( \frac{1}{\nu \lvert k'\rvert^{\frac{5}{2}} + 2} \right) \left( \frac{1}{\nu \lvert k \rvert^{\frac{5}{2}} + 2} \right)       \nonumber \\
&   \times  \left( \frac{ 1- e^{-2\nu \lvert k \rvert^{\frac{5}{2}} t}}{2\nu \lvert k \rvert^{\frac{5}{2}}} \right) \left( \frac{ 1- e^{-2\nu \lvert k' \rvert^{\frac{5}{2}} t}}{2\nu \lvert k' \rvert^{\frac{5}{2}}} \right) k_{1} k_{2}'  (4k_{1}'k_{1} + k_{2}'k_{2} + k_{3}'k_{3}) \Bigg[  - \frac{k_{1} k_{2}}{\lvert k \rvert^{2}} + \frac{(k_{1}')^{2} k_{1} k_{2}}{\lvert k' \rvert^{2} \lvert k\rvert^{2}} \Bigg]  \nonumber \\
+& \frac{1}{4} \sum_{k, k' \in \mathbb{Z}^{3} \setminus \{0\}} e^{i4 \pi (k+k') \cdot x}  \rho_{m}^{2}(k+k') \lvert \psi_{0} (k, k') \rvert^{2}    \mathfrak{l} \left( \frac{\lvert k \rvert}{\lambda} \right)^{2} \mathfrak{l} \left( \frac{\lvert k'\rvert}{\lambda} \right)^{2}   \left( \frac{1}{\nu \lvert k'\rvert^{\frac{5}{2}} + 2} \right)^{2}   \nonumber  \\
&   \times  \left( \frac{ 1- e^{-2\nu \lvert k \rvert^{\frac{5}{2}} t}}{2\nu \lvert k \rvert^{\frac{5}{2}}} \right) \left( \frac{ 1- e^{-2\nu \lvert k' \rvert^{\frac{5}{2}} t}}{2\nu \lvert k' \rvert^{\frac{5}{2}}} \right)  (4k_{1}k_{1}' + k_{2}k_{2}' + k_{3}k_{3}') k_{1} k_{2}'   \Bigg[  -\frac{k_{1} k_{2}}{\lvert k \rvert^{2}} + \frac{k_{1}k_{2} (k_{1}')^{2}}{\lvert k \rvert^{2} \lvert k'\rvert^{2}} \Bigg]  \nonumber \\
+& \frac{1}{4} \sum_{k, \tilde{k} \in \mathbb{Z}^{3} \setminus \{0\}}  \rho_{m}^{2}(0)    \mathfrak{l} \left( \frac{\lvert k \rvert}{\lambda} \right)^{2}  \mathfrak{l}  \left( \frac{\lvert \tilde{k} \rvert}{\lambda} \right)^{2} \left( \frac{1}{\nu \lvert k\rvert^{\frac{5}{2}} + 2} \right) \left( \frac{1}{\nu \lvert \tilde{k} \rvert^{\frac{5}{2}} + 2} \right)     \nonumber \\
&  \times \left( \frac{ 1- e^{-2\nu \lvert k \rvert^{\frac{5}{2}} t}}{2\nu \lvert k \rvert^{\frac{5}{2}}} \right) \left( \frac{ 1- e^{-2\nu \lvert \tilde{k} \rvert^{\frac{5}{2}} t}}{2\nu \lvert \tilde{k} \rvert^{\frac{5}{2}}} \right) (4\tilde{k}_{1}\tilde{k}_{1} + \tilde{k}_{2}\tilde{k}_{2} + \tilde{k}_{3}\tilde{k}_{3}) k_{1} k_{2}    \Bigg[ - \frac{ k_{1} k_{2}}{\lvert  k \rvert^{2}} + \frac{\tilde{k}_{1}^{2} k_{1} k_{2}}{\lvert \tilde{k} \rvert^{2} \lvert k \rvert^{2}}  \Bigg].  \label{first reduction}
\end{align}
We emphasize that $\RomanI_{2,1} \neq \RomanI_{1,2}$.
\end{remark}
We can compute $\{\RomanI_{i,j}\}_{i,j \in \{1, \hdots, 7\}: (i,j) \neq (1,1), (1,2), (2,1)}$ similarly. We can then define 
\begin{align}
\mathcal{Y} (k, k') \triangleq& k_{1}^{2} ( k_{1}')^{2} + \lvert k \rvert^{2} (k_{1}')^{2} + \lvert k' \rvert^{2} k_{1}^{2} + 6 k_{1} k_{1}' (k \cdot k') + (k\cdot k')^{2} \nonumber \\
& - 4 \left( \frac{k_{1}^{2}}{\lvert k \rvert^{2}} + \frac{ ( k_{1}')^{2}}{\lvert k' \rvert^{2}} \right) [ 2 k_{1} k_{1}' (k\cdot k') + (k\cdot k')^{2} ] + \frac{16 k_{1}^{2} (k_{1}')^{2}}{\lvert k \rvert^{2} \lvert k' \rvert^{2}} (k\cdot k')^{2}, \label{est 57} 
\end{align}
and deduce 
\begin{align}
&  \mathbb{E} [ \lvert \Delta_{m} ( \nabla_{\text{sym}} \mathcal{L}_{\lambda} X \circlesign{\circ} P^{\lambda})_{1,1} (t) \rvert^{2} ] \nonumber \\
=&  \frac{1}{4} \sum_{k, k' \in \mathbb{Z}^{3} \setminus \{0\}} e^{i4 \pi (k+k') \cdot x} \rho_{m}^{2}(k+k') \lvert\psi_{0}(k,k') \rvert^{2} \mathfrak{l}  \left( \frac{\lvert k \rvert}{\lambda} \right)^{2} \mathfrak{l} \left( \frac{\lvert k'\rvert}{\lambda} \right)^{2} \left( \frac{1}{\nu \lvert k'\rvert^{\frac{5}{2}} + 2} \right)  \left( \frac{1}{\nu \lvert k \rvert^{\frac{5}{2}} + 2} \right)\nonumber  \\
& \hspace{5mm} \times  \left( \frac{1- e^{-2 \nu \lvert k \rvert^{\frac{5}{2}} t}}{2 \nu \lvert k \rvert^{\frac{5}{2}}} \right) \left( \frac{1- e^{-2 \nu \lvert k' \rvert^{\frac{5}{2}} t}}{2 \nu \lvert k' \rvert^{\frac{5}{2}}} \right)  \mathcal{Y} (k, k') \nonumber \\
& +  \frac{1}{4} \sum_{k, k' \in \mathbb{Z}^{3} \setminus \{0\}} e^{i4 \pi (k+k') \cdot x} \rho_{m}^{2}(k+k') \lvert\psi_{0}(k,k') \rvert^{2} \mathfrak{l} \left( \frac{\lvert k \rvert}{\lambda} \right)^{2} \mathfrak{l} \left( \frac{\lvert k'\rvert}{\lambda} \right)^{2} \left( \frac{1}{\nu \lvert k'\rvert^{\frac{5}{2}} + 2} \right)^{2} \nonumber  \\
& \hspace{5mm} \times   \left( \frac{1- e^{-2 \nu \lvert k \rvert^{\frac{5}{2}} t}}{2 \nu \lvert k \rvert^{\frac{5}{2}}} \right) \left( \frac{1- e^{-2 \nu \lvert k' \rvert^{\frac{5}{2}} t}}{2 \nu \lvert k' \rvert^{\frac{5}{2}}} \right)   \mathcal{Y}(k,k') \nonumber \\  
&+ \frac{1}{4} \sum_{k, \tilde{k} \in\mathbb{Z}^{3} \setminus \{0\}} \rho_{m}^{2}(0) \mathfrak{l} \left( \frac{\lvert k \rvert}{\lambda} \right)^{2} \mathfrak{l} \left( \frac{\lvert \tilde{k} \rvert}{\lambda} \right)^{2} \left( \frac{1}{\nu \lvert k \rvert^{\frac{5}{2}} + 2} \right) \left( \frac{1}{\nu \lvert \tilde{k} \rvert^{\frac{5}{2}} + 2} \right)   \nonumber \\
& \hspace{5mm} \times  \left( \frac{1- e^{-2 \nu \lvert k \rvert^{\frac{5}{2}} t}}{2 \nu \lvert k \rvert^{\frac{5}{2}}} \right) \left( \frac{1- e^{-2 \nu \lvert \tilde{k} \rvert^{\frac{5}{2}} t}}{2 \nu \lvert \tilde{k} \rvert^{\frac{5}{2}}} \right)  [ k_{1}^{2} \tilde{k}_{1}^{2} + k_{1}^{2} \lvert \tilde{k} \rvert^{2} + \tilde{k}_{1}^{2} \lvert k \rvert^{2} + \lvert k \rvert^{2} \lvert \tilde{k} \rvert^{2} ].  \label{est 56}
\end{align} 
Applying \eqref{est 56} to \eqref{est 78}, and using the definitions of $r_{\lambda}$, we obtain 
\begin{align}
& \mathbb{E} [ \lvert \Delta_{m} \left( ( \nabla_{\text{sym}} \mathcal{L}_{\lambda} X \circlesign{\circ} P^{\lambda} )_{1,1} - (r_{\lambda}^{1} + r_{\lambda}^{2,1} ) \right)(t) \rvert^{2} ] \label{est 79}\\
=&  \frac{1}{4} \sum_{k, k' \in \mathbb{Z}^{3} \setminus \{0\}} e^{i4 \pi (k+k') \cdot x} \rho_{m}^{2}(k+k') \lvert\psi_{0}(k,k') \rvert^{2} \mathfrak{l}  \left( \frac{\lvert k \rvert}{\lambda} \right)^{2} \mathfrak{l} \left( \frac{\lvert k'\rvert}{\lambda} \right)^{2} \left( \frac{1}{\nu \lvert k'\rvert^{\frac{5}{2}} + 2} \right) \nonumber  \\
& \hspace{5mm} \times \left[  \left( \frac{1}{\nu \lvert k' \rvert^{\frac{5}{2}} + 2} \right) +   \left( \frac{1}{\nu \lvert k \rvert^{\frac{5}{2}} + 2} \right) \right] \left( \frac{1- e^{-2 \nu \lvert k \rvert^{\frac{5}{2}} t}}{2 \nu \lvert k \rvert^{\frac{5}{2}}} \right) \left( \frac{1- e^{-2 \nu \lvert k' \rvert^{\frac{5}{2}} t}}{2 \nu \lvert k' \rvert^{\frac{5}{2}}} \right)  \mathcal{Y} (k, k') \nonumber \\
&+ \frac{1}{4} \sum_{k, \tilde{k} \in\mathbb{Z}^{3} \setminus \{0\}} \rho_{m}^{2}(0) \mathfrak{l} \left( \frac{\lvert k \rvert}{\lambda} \right)^{2} \mathfrak{l} \left( \frac{\lvert \tilde{k} \rvert}{\lambda} \right)^{2} \left( \frac{1}{\nu \lvert k \rvert^{\frac{5}{2}} + 2} \right)    \nonumber \\
& \hspace{5mm} \times \left( \frac{1}{\nu \lvert \tilde{k} \rvert^{\frac{5}{2}} + 2} \right) \left( \frac{1- e^{-2 \nu \lvert k \rvert^{\frac{5}{2}} t}}{2 \nu \lvert k \rvert^{\frac{5}{2}}} \right) \left( \frac{1- e^{-2 \nu \lvert \tilde{k} \rvert^{\frac{5}{2}} t}}{2 \nu \lvert \tilde{k} \rvert^{\frac{5}{2}}} \right)  [ k_{1}^{2} \tilde{k}_{1}^{2} + k_{1}^{2} \lvert \tilde{k} \rvert^{2} + \tilde{k}_{1}^{2} \lvert k \rvert^{2} + \lvert k \rvert^{2} \lvert \tilde{k} \rvert^{2} ]  \nonumber \\
& - \frac{1}{16}\Delta_{-1} \Bigg[\sum_{k \in \mathbb{Z}^{3} \setminus \{0\}}  \mathfrak{l} \left( \frac{\lvert k \rvert}{\lambda} \right)^{2} \left( \frac{1- e^{-2 \nu \lvert k \rvert^{\frac{5}{2}} t}}{2 \nu \lvert k \rvert^{\frac{1}{2}}} \right) \left( \frac{2}{2 + \nu \lvert k \rvert^{\frac{5}{2}}} \right)  \nonumber \\
& \hspace{22mm} +  \mathfrak{l} \left( \frac{\lvert k \rvert}{\lambda} \right)^{2} \left( \frac{1- e^{-2 \nu \lvert k \rvert^{\frac{5}{2}} t}}{2 \nu \lvert k \rvert^{\frac{1}{2}}} \right) \left( \frac{2}{2 + \nu \lvert k \rvert^{\frac{5}{2}}} \right) \frac{k_{1}^{2}}{\lvert k \rvert^{2}} \Bigg]^{2} (t), \nonumber 
\end{align}
where the second sum and the last term cancel out. After the cancellations, making use of the facts that $\mathcal{Y}(k,k')$ from \eqref{est 57} satisfies  $\lvert \mathcal{Y} (k, k') \rvert \lesssim \lvert k \rvert^{2} \lvert k' \rvert^{2}$, that $\rho_{m}(k), \rho_{c}(k-k'), \rho_{d}(k')$ imply $\lvert k \rvert \approx 2^{m}, \lvert k' \rvert \gtrsim 2^{m}$, and $m \lesssim c$, we can estimate by symmetry 
\begin{align}
& \mathbb{E} [ \lvert \Delta_{m} \left( ( \nabla_{\text{sym}} \mathcal{L}_{\lambda} X \circlesign{\circ} P^{\lambda} )_{1,1} - (r_{\lambda}^{1} + r_{\lambda}^{2,1} ) \right)(t) \rvert^{2} ]  \nonumber \\ 
\lesssim& \sum_{k, k' \in \mathbb{Z}^{3} \setminus \{0\}} \rho_{m}^{2} (k+ k') \lvert \psi_{0} (k, k') \rvert^{2} \mathfrak{l} \left( \frac{ \lvert k \rvert}{\lambda} \right)^{2} \mathfrak{l} \left( \frac{ \lvert k' \rvert}{\lambda} \right) \left( \frac{1}{\nu \lvert k' \rvert^{\frac{5}{2}}+ 2} \right)^{2} \frac{1}{\lvert k \rvert^{\frac{1}{2}} \lvert k' \rvert^{\frac{1}{2}}}  \nonumber \\
\lesssim& \sum_{k,k' \in \mathbb{Z}^{3} \setminus \{0\}: \lvert k \rvert \approx 2^{m}, \lvert k' \rvert \gtrsim 2^{m}} \lvert k' \rvert \left( \sum_{c: m \lesssim c} \frac{1}{2^{\frac{c}{2}}} \right)^{2} \left( \frac{1}{\lvert k' \rvert^{\frac{5}{2}} + 2} \right)^{2} \frac{1}{2^{\frac{m}{2}} \lvert k' \rvert^{\frac{1}{2}}}   \lesssim 1. \label{est 80} 
\end{align}  
Finally, an application of Gaussian hypercontractivity theorem (e.g. \cite[Theorem 3.50]{J97}) concludes that for any $p\in [2,\infty)$, thanks to \eqref{est 80}
\begin{align*}
& \sup_{\lambda \geq 1} \mathbb{E} [ \lVert \left( ( \nabla_{\text{sym}} \mathcal{L}_{\lambda} X \circlesign{\circ} P^{\lambda})_{1,1} - (r_{\lambda}^{1} + r_{\lambda}^{2,1}) \right) (t) \rVert_{B_{p,p}^{-\kappa}}^{p} ] \\
\lesssim& \sup_{\lambda \geq 1} \sum_{m\geq -1} 2^{-\kappa m p} \int_{\mathbb{T}^{3}}  \lVert \Delta_{m} \left( ( \nabla_{\text{sym}} \mathcal{L}_{\lambda} X \circlesign{\circ} P^{\lambda} )_{1,1} - (r_{\lambda}^{1} + r_{\lambda}^{2,1}) \right)(t)  \rVert_{L_{\omega}^{2}}^{p}  dx \lesssim 1. 
\end{align*}
We can show the analogous convergence of the other eight $(i,j)$-entries for $i,j\in \{1,2,3\}$, $(i,j) \neq (1,1)$, and hence conclude the convergence of  \eqref{est 68} in $L^{p} (\Omega; C_{\text{loc}} (\mathbb{R}_{+}; \mathcal{K}^{-\frac{5}{4} - \kappa} ) )$ for any $p \in [1,\infty)$. The convergence $\mathbb{P}$-a.s. follows similarly and we refer to the proofs in \cite[p. 33]{HR24}, \cite[p. 44]{Y23c}, and \cite[Section B.4]{Y25d}. This completes the proof of Proposition \ref{Proposition 4.14}. 
\end{proof}

\section{Proof of Theorem \ref{Theorem 2.3}}\label{Section 5}  
We start with the long-awaited definition of the HL weak solution, 

\begin{define}\label{Definition 5.1} 
Given any $u^{\text{in}} \in L_{\sigma}^{2}$ and any $\kappa \in (0, \frac{1}{2}), v \in C([0,\infty); \mathcal{S}(\mathbb{T}^{3}; \mathbb{R}^{3}))$ is called a global-in-time high-low (HL) weak solution to \eqref{Equation of v} starting from $u^{\text{in}}$ if $w = v- Y$ from \eqref{Equation of w}, where $Y$ is the unique solution to the linear equation \eqref{Equation of Y}, satisfies the following conditions. 
\begin{enumerate}[label=(\alph*)]
\item For any $T > 0,$  there exists a $\lambda_{T} > 0$ such that for any $\lambda \geq \lambda_{T}$, there is a pair $(w^{\mathcal{L},\lambda}, w^{\mathcal{H},\lambda})$ such that 
\begin{subequations}\label{est 153}
\begin{align}
& w^{\mathcal{L},\lambda} \in L^{\infty} ( 0,T; L_{\sigma}^{2}) \cap L^{2} (0,T; \dot{H}^{\frac{5}{4}}), \label{est 153a} \\
& w^{\mathcal{H},\lambda} \in L^{\infty} ( 0, T; L_{\sigma}^{2})\cap L^{2} (0, T; B_{p,2}^{\frac{5}{4} - 2 \kappa}) \hspace{3mm} \forall \hspace{1mm} p \in \left[1, \frac{12}{1+ 4 \kappa} \right], \label{est 153b}
\end{align}
\end{subequations}
that satisfies 
\begin{equation}\label{est 149}
w^{\mathcal{H},\lambda}(t) = - \mathbb{P}_{L} \divergence ( w \circlesign{\prec}_{s} \mathcal{H}_{\lambda} Q), \hspace{3mm} w(t) = w^{\mathcal{L},\lambda}(t) + w^{\mathcal{H},\lambda} (t) 
\end{equation} 
(cf. \eqref{Define QH, wH, and wL}) for all $t \in [0,T]$ with $Q$ defined by \eqref{Equation of Q}. 
\item $w$ solves \eqref{Equation of w} distributionally; i.e., for any $T > 0$ and any $\phi \in C^{\infty} ([0,T] \times \mathbb{T}^{3})$ that is divergence-free, 
\begin{align}
& \langle w(T), \phi(T) \rangle - \langle w(0), \phi(0) \rangle = \int_{0}^{T} \langle w, \partial_{t} \phi - \nu \Lambda^{\frac{5}{2}} \phi \rangle  \nonumber \\
& \hspace{10mm} + \langle w, ( w\cdot\nabla) \phi \rangle + \frac{1}{2} \langle D, (w\cdot\nabla) \phi \rangle + \frac{1}{2} \langle w, (D \cdot\nabla) \phi \rangle + \langle Y, (Y \cdot\nabla) \phi \rangle dt. 
\end{align}
 \end{enumerate} 
\end{define} 

\begin{remark}
In comparison, \cite[Definition 7.1]{HR24} for the 2D Navier-Stokes equations with full Laplacian as diffusion had $w^{\mathcal{H},\lambda} \in L^{2} (0, T; B_{4,\infty}^{1-\kappa})$ for any $\kappa > 0$. In \cite[Definition 5.1]{Y23c} for the 2D MHD system with full Laplacian as diffusion improved both integrability and summability to $w^{\mathcal{H},\lambda} \in L^{2} (0, T; B_{p,2}^{1- 2\kappa})$ for all $p \in [1, \frac{2}{\kappa}]$ for any $\kappa > 0$. For the 1D Burgers' equation with full Laplacian as diffusion but forced by $\Lambda^{\frac{1}{2}}$ applied on STWN, \cite[Definition 5.1]{Y25d} was able to improve furthermore to $w^{\mathcal{H},\lambda} \in L^{2}(0, T; B_{\infty, 2}^{1-2\kappa})$ for any $\kappa > 0$. In each case, the strength of diffusion, singularity of the force, and spatial dimensions play roles determining such regularity. In this manuscript we were able to deduce $w^{\mathcal{H},\lambda} \in L^{2} (0, T; B_{p,2}^{\frac{5}{4} - 2 \kappa})$ for all $p \in [1, \frac{12}{1+ 4 \kappa}]$ and this will play a crucial role in the proof of Proposition \ref{Proposition 5.2}. (See Remark \ref{Remark 5.2} and \eqref{Reason 1}-\eqref{Reason 4}).
\end{remark} 

\subsection{Existence}
\hfill\\ First, we present our main result on the existence of HL weak solution to \eqref{Equation of v}.
\begin{proposition}\label{Proposition 5.1}
Let $\mathcal{N}'' \subset \Omega$ be the null set from Proposition \ref{Proposition 4.1}. Then, for any $\omega \in \Omega \setminus \mathcal{N}''$ and $u^{\text{in}} \in L_{\sigma}^{2}$, there exists a HL weak solution to \eqref{Equation of v} starting from $u^{\text{in}}$. 
\end{proposition}

\begin{proof}[Proof of Proposition \ref{Proposition 5.1}]
For any $n \in \mathbb{N}_{0}$ we define $X^{n} \triangleq \mathcal{L}_{n} X$, where $X$ solves \eqref{Equation of X}. We define $Y_{n}$ to be the corresponding solution to \eqref{Equation of Y} with $X$ replaced by $X^{n}$. Then, similarly to \eqref{Define D}, we define 
\begin{equation}\label{Define Dn}
D^{n} \triangleq 2(X^{n} + Y^{n})
\end{equation} 
and $w^{n}$ to be the solution to 
\begin{equation}\label{Equation of wn}
\partial_{t} w^{n} + \mathbb{P}_{L} \divergence  ((w^{n})^{\otimes 2} + D^{n} \otimes_{s} w^{n} + (Y^{n})^{\otimes 2}) + \nu \Lambda^{\frac{5}{2}} w^{n} = 0, \hspace{5mm} w^{n}(0,x) = \mathcal{L}_{n} u^{\text{in}}(x) 
\end{equation} 
(cf. \eqref{Equation of w}). Additionally, we define
\begin{subequations}\label{Define Lt n kappa and Nt n kappa}  
\begin{align}
&L_{t}^{n, \kappa} \triangleq 1 + \lVert X^{n} \rVert_{C_{t} \mathscr{C}^{-\frac{1}{4} - \kappa}} + \lVert Y^{n} \rVert_{C_{t} \mathscr{C}^{1-\kappa}}, \label{Define Lt n kappa and Nt n kappa 1}   \\
& N_{t}^{n, \kappa} \triangleq L_{t}^{n, \kappa} + \sup_{i\in\mathbb{N}} \left\lVert ( \nabla \mathcal{L}_{\lambda^{i}} X^{n}) \circlesign{\circ} P^{\lambda^{i}, n} - r_{\lambda^{i}}^{n}(t) \Id \right\rVert_{C_{t}\mathscr{C}^{-\kappa}}, \hspace{5mm} \bar{N}_{t}^{\kappa} (\omega) \triangleq \sup_{n\in\mathbb{N}} N_{t}^{n,\kappa} (\omega), \label{Define Lt n kappa and Nt n kappa 2}  
\end{align}
\end{subequations} 
(cf. \eqref{Define Lt kappa and Nt kappa}) with $\{\lambda^{i}\}_{i \in \mathbb{N}}$ from Definition \ref{Definition 4.2}, 
\begin{subequations}\label{Define P lambda n and r lambda n}
\begin{align}
& P^{\lambda, n}(t,x) \triangleq \left(1 + \frac{\nu \Lambda^{\frac{5}{2}}}{2}  \right)^{-1} \nabla_{\text{sym}} \mathcal{L}_{\lambda} X^{n}(t,x), \hspace{25mm} r_{\lambda}^{n} \triangleq r_{\lambda}^{1,n} + r_{\lambda}^{2,n}, \label{Define P lambda n} \\
& r_{\lambda}^{1,n}(t) \triangleq \sum_{k\in\mathbb{Z}^{3} \setminus \{0\}} \frac{1}{4} \mathfrak{l} \left( \frac{\lvert k \rvert}{\lambda} \right) \mathfrak{l} \left( \frac{\lvert k \rvert}{n} \right)  \left( \frac{1- e^{-2 \nu \lvert k \rvert^{\frac{5}{2}} t}}{2 \nu \lvert k \rvert^{\frac{1}{2}}} \right) \left(1+ \frac{ \nu \lvert k \rvert^{\frac{5}{2}}}{2} \right)^{-1},  r_{\lambda}^{2,n}(t) \triangleq  \sum_{m=1}^{3} r_{\lambda}^{2,n,m} \delta_{m,m}  \label{Define r lambda1 n} \\
&\text{where } \hspace{1mm} r_{\lambda}^{2,n,m} \triangleq  \sum_{k\in\mathbb{Z}^{3} \setminus \{0 \}} \frac{1}{4} \mathfrak{l} \left( \frac{\lvert k \rvert}{\lambda} \right)  \left( \frac{\lvert k \rvert}{n} \right)  \left( \frac{1- e^{-2 \nu \lvert k \rvert^{\frac{5}{2}} t}}{2\nu \lvert k \rvert^{\frac{1}{2}}} \right) \left(1+ \frac{ \nu \lvert k \rvert^{\frac{5}{2}}}{2} \right)^{-1}\frac{ k_{m}^{2}}{\lvert k \rvert^{2}}, \label{Define r lambda2 n}
\end{align}
\end{subequations} 
and $(\delta_{m,m} \Id)$ is defined in \eqref{Define delta m,m} (cf. \eqref{Define P lambda and r lambda}). Identically to \eqref{logarithmic growth} we can show that $r_{\lambda}^{1,n} (t) \lesssim \ln ( \lambda \wedge n)$ and $r_{\lambda}^{2,n} (t) \lesssim \ln(\lambda \wedge n)$. Moreover, we see that 
\begin{equation}
\bar{N}_{t}^{\kappa} (\omega) < \infty \hspace{3mm} \forall \hspace{1mm} \omega \in \Omega \setminus \mathcal{N}'' \hspace{1mm} \text{ and } \hspace{1mm} \lim_{n\to\infty} N_{t}^{n,\kappa}(\omega) = N_{t}^{\kappa}(\omega), 
\end{equation} 
for $\mathcal{N}''$ from Proposition \ref{Proposition 4.1} and $N_{t}^{\kappa}$ from \eqref{Define Lt kappa and Nt kappa}. For all $i \in \mathbb{N}_{0}$ we define 
\begin{equation}\label{Define Ti 0 and T i+1 n}
T_{0}^{n} \triangleq 0 \hspace{1mm} \text{ and } \hspace{1mm} T_{i+1}^{n} (\omega, u^{\text{in}}) \triangleq \inf\{t \geq T_{i}^{n}: \hspace{1mm} \lVert w^{n}(t) \rVert_{L^{2}} \geq i +1 \},  
\end{equation} 
(cf. \eqref{Define T0 and Ti}) and 
\begin{equation}\label{Define lambda t n}
\lambda_{0}^{n} \triangleq \lambda_{0} \hspace{1mm} \text{ and } \hspace{1mm} \lambda_{t}^{n} \triangleq \left(1+ \lVert w^{n} (T_{i}^{n}) \rVert_{L^{2}} \right)^{\frac{40}{13}} \hspace{1mm} \text{ if } t > 0 \text{ and } t \in [T_{i}^{n}, T_{i+1}^{n})
\end{equation} 
where $\lambda_{0}$ was defined in \eqref{Define lambda t}. Similarly to \eqref{Equation of Q} and \eqref{Define QH, wH, and wL}, we define $Q^{n}$ as a solution to
\begin{equation}\label{Equation of Qn}
( \partial_{t} + \nu \Lambda^{\frac{5}{2}})Q^{n} = 2 X^{n}, \hspace{3mm} Q^{n}(0) = 0 
\end{equation} 
and then 
\begin{equation}\label{Define Qn H, wn H, and wn L}
Q^{n, \mathcal{H}}(t) \triangleq \mathcal{H}_{\lambda_{t}^{n}} Q^{n}(t), \hspace{3mm} w^{n,\mathcal{H}}(t) \triangleq - \mathbb{P}_{L} \divergence ( w^{n} \circlesign{\prec}_{s} Q^{n,\mathcal{H}}), \hspace{3mm} w^{n,\mathcal{L}} \triangleq w^{n} - w^{n,\mathcal{H}}. 
\end{equation}
We can repeat the proof up to Proposition \ref{Proposition 4.9} to obtain, for $\kappa_{0} > 0$ sufficiently small, a constant $C_{1} > 0$ and appropriate increasing continuous functions $C_{2}$ and $C_{3}$ from $\mathbb{R}_{\geq 0}$ to $\mathbb{R}_{\geq 0}$ such that 
\begin{align}
&\sup_{t\in [T_{i}^{n}, T_{i+1}^{n})} \lVert w^{n,\mathcal{L}}(t) \rVert_{L^{2}}^{2} + \frac{\nu}{2} \int_{T_{i}^{n}}^{T_{i+1}^{n}} \lVert w^{n,\mathcal{L}} (s) \rVert_{\dot{H}^{\frac{5}{4}}}^{2} ds \nonumber \\
\leq&  e^{(T_{i+1}^{n} - T_{i}^{n}) [ C_{1} \ln ( \lambda_{T_{i}^{n}}) + C_{2} (N_{T_{i+1}^{n}}^{\kappa})]} [ \lVert w^{n,\mathcal{L}} (T_{i}^{n} ) \rVert_{L^{2}}^{2} + C_{3} (N_{T_{i+1}^{n}}^{n,\kappa})]
\end{align} 
for all $\kappa \in (0, \kappa_{0}]$ and $i \in \mathbb{N}$ such that $i \geq i_{0} (u^{\text{in}})$.  Similarly to Proposition \ref{Proposition 4.10} and the proof of Theorem \ref{Theorem 2.2}, this leads to, uniformly over all $n \in \mathbb{N}$ and $i \geq i_{0} (u^{\text{in}})$, 
\begin{equation}\label{est 288}
T_{i+1}^{n} - T_{i}^{n} \geq \frac{1}{ \tilde{C} (\bar{N}_{T_{i+1}^{n}}^{\kappa}) ( \ln(1+i) + 1)} \ln \left( \frac{i^{2} + 2 i - C ( \bar{N}_{T_{i+1}^{n}}^{\kappa})}{i^{2} + \tilde{C} (\bar{N}_{T_{i+1}^{n}}^{\kappa})} \right) 
\end{equation} 
for some constants $C( \bar{N}_{T_{i+1}^{n}}^{\kappa})$ and $\tilde{C}(\bar{N}_{T_{i+1}^{n}}^{\kappa})$ (see \eqref{est 110}) and thus for every $T> 0$, $i \in \mathbb{N}$ such that $i > i_{0} (u^{\text{in}})$, there exists $\mathfrak{t} (i, \bar{N}_{T}^{\kappa}) \in (0, T]$ such that 
\begin{equation}
\inf_{n\in\mathbb{N}_{0}} T_{i}^{n} \geq \mathfrak{t} (i, \bar{N}_{T}^{\kappa}), \hspace{3mm} \mathfrak{t} (i, \bar{N}_{T}^{\kappa}) = T \hspace{2mm} \forall \hspace{1mm} i \text{ sufficiently large}. 
\end{equation}  
Therefore, for all $T > 0$ and $\kappa > 0$ sufficiently small, there exists a constant $C(T, \bar{N}_{T}^{\kappa}) > 0$ such that 
\begin{equation}\label{est 142}
\sup_{n\in\mathbb{N}} \left[ \lVert w^{n,\mathcal{L}}  \rVert_{C_{T} L^{2}}^{2} + \nu \int_{0}^{T} \lVert w^{n, \mathcal{L}} (t) \rVert_{\dot{H}^{\frac{5}{4}}}^{2} dt \right] \leq C(T, \bar{N}_{T}^{\kappa}). 
\end{equation} 
Additionally, we can find $\bar{\lambda}_{T} > 0$ such that 
\begin{equation}\label{est 143}
\lambda_{t}^{n} \leq \bar{\lambda}_{T} \hspace{3mm} \forall \hspace{1mm} t \in [0, T], n \in \mathbb{N}. 
\end{equation} 
Hence, extending the previous definition \eqref{Define Qn H, wn H, and wn L} to 
\begin{equation}\label{est 148}
w^{n, \mathcal{H},\lambda} \triangleq - \mathbb{P}_{L} \divergence (w^{n} \circlesign{\prec}_{s} \mathcal{H}_{\lambda} Q^{n} ), \hspace{3mm} w^{n, \mathcal{L},\lambda} \triangleq w^{n} - w^{n, \mathcal{H},\lambda} \hspace{3mm} \forall \hspace{1mm} \lambda \geq \bar{\lambda}_{T}, 
\end{equation} 
we see that for all $\lambda \geq \bar{\lambda}_{T}$ so that $\lambda \geq \lambda_{t}^{n}$ for all $t \in [0,T]$ and all $n \in \mathbb{N}$ due to \eqref{est 143}, we have 
\begin{equation}\label{est 145}
\sup_{n\in\mathbb{N}} \left[ \lVert w^{n, \mathcal{L},\lambda} \rVert_{C_{T}L^{2}}^{2} + \nu \int_{0}^{T} \lVert w^{n,\mathcal{L},\lambda} (t) \rVert_{\dot{H}^{\frac{5}{4}}}^{2} dt \right] \leq C( \lambda, T, \bar{N}_{T}^{\kappa}). 
\end{equation} 
Next, for all $\kappa > 0$ sufficiently small, due to \eqref{Estimate on wH}, \eqref{est 142}, and \eqref{Define Lt n kappa and Nt n kappa 2}, 
\begin{equation}\label{est 144}
\sup_{n\in\mathbb{N}} \left[ \lVert w^{n}  \rVert_{C_{T} L^{2}}^{2} + \nu  \int_{0}^{T} \lVert w^{n} (t) \rVert_{\dot{H}^{\alpha}}^{2} dt \right] \leq C( T, \bar{N}_{T}^{\kappa}) \hspace{3mm} \forall \hspace{1mm} \alpha \leq \frac{37}{40} - 2 \kappa.
\end{equation} 
 As a consequence of \eqref{Define Qn H, wn H, and wn L}, \eqref{Sobolev products c}, \eqref{est 144}, and \eqref{est 145}, 
\begin{equation}\label{est 146}
\sup_{n\in\mathbb{N}} \left[ \lVert w^{n} \rVert_{L^{2} (0,T; \dot{H}^{\frac{5}{4} - \frac{3\kappa}{2}})}^{2}  \right] = \sup_{n\in\mathbb{N}} \left[ \lVert w^{n, \mathcal{H}} +  w^{n, \mathcal{L}} \rVert_{L^{2} (0,T; \dot{H}^{\frac{5}{4} - \frac{3\kappa}{2}})}^{2}  \right] \leq C(T, \bar{N}_{T}^{\kappa}). 
\end{equation} 
It follows that for all $p\in [2, \frac{12}{1+ 4 \kappa}]$, for $N_{1}$ from \eqref{est 152}, due to H$\ddot{\mathrm{o}}$lder's inequality and Young's inequality for convolution, and Bernstein's inequality, 
\begin{align}
&  \lVert w^{n, \mathcal{H},\lambda} \rVert_{L^{2} (0, T; B_{p,2}^{\frac{5}{4} - 2 \kappa})}^{2} \lesssim  \int_{0}^{T} \sum_{m \geq -1} \left\lvert 2^{m( \frac{9}{4} - 2 \kappa)} \lVert \sum_{j: \lvert j-m \rvert \leq N_{1}} (S_{j-1} w^{n})(t) \Delta_{j} \mathcal{H}_{\lambda} Q^{n}(t) \rVert_{L^{p}} \right\rvert^{2} dt \nonumber \\
\lesssim& \int_{0}^{T} \sum_{m\geq -1} \left\lvert 2^{m(\frac{9}{4} - \frac{3\kappa}{2})} \lVert \Delta_{m} \mathcal{H}_{\lambda} Q^{n} (t)\rVert_{L^{\infty}} \sum_{l: l \leq m  - 2} 2^{(l-m) \frac{\kappa}{2}} 2^{-l (\frac{\kappa}{2})} \lVert \Delta_{l} w^{n}(t) \rVert_{L^{p}} \right\rvert^{2} dt  \nonumber  \\
\lesssim& \int_{0}^{T} \lVert Q^{n} \rVert_{\mathscr{C}^{\frac{9}{4} - \frac{3\kappa}{2}}}^{2}  \left\lVert 2^{-m (\frac{\kappa}{2})}  2^{m3(\frac{1}{2} - \frac{1}{p})}\lVert  \Delta_{m} w^{n}(t) \rVert_{L^{2}} \right\rVert_{l^{2}}^{2} dt  \overset{\eqref{Equation of Qn} \eqref{Define Lt n kappa and Nt n kappa 2} \eqref{est 146}}{\leq} C(T, \bar{N}_{T}^{\kappa}),  \label{est 151}
\end{align}
where the last inequality used the hypothesis that $p \leq \frac{12}{1+4 \kappa}$. 

Next, we estimate $\partial_{t}w^{n}$ in negative Sobolev space. To do so, we estimate separately for all $t \in [0, T]$, $\kappa \in (0, \frac{1}{5})$, using the Gagliardo-Nirenberg inequality of $\lVert f \rVert_{\dot{H}^{\frac{1}{4} + 3 \kappa}} \lesssim \lVert f \rVert_{L^{2}}^{\frac{4-20\kappa}{5 - 8 \kappa}} \lVert f \rVert_{\dot{H}^{\frac{5}{4} - 2 \kappa}}^{\frac{1+ 12 \kappa}{5-8\kappa}}$, 
\begin{subequations}\label{est 147}
\begin{align}
& \lVert (w^{n})^{\otimes 2} (t) \rVert_{H^{-\frac{1}{4} - 2 \kappa}} \lesssim \lVert w^{n} (t) \rVert_{L^{\frac{24}{7+ 8 \kappa}}}^{2} \lesssim \lVert w^{n} (t) \rVert_{L^{2}} \lVert w^{n} (t) \rVert_{\dot{H}^{\frac{5}{4} - 2 \kappa}}, \\
& \lVert D^{n} \otimes w^{n} (t) \rVert_{H^{-\frac{1}{4} - 2 \kappa}} \overset{\eqref{Bony's decomposition}  \eqref{Sobolev products} \eqref{Define Lt n kappa and Nt n kappa 2}}{\lesssim} \bar{N}_{T}^{\kappa} \lVert w^{n} (t)  \rVert_{H^{\frac{1}{4} + 3 \kappa}}   \lesssim \bar{N}_{T}^{\kappa} \lVert w^{n} (t) \rVert_{L^{2}}^{\frac{4 - 20 \kappa}{5 - 8 \kappa}} \lVert w^{n} (t) \rVert_{\dot{H}^{\frac{5}{4} - 2 \kappa}}^{\frac{1 + 12 \kappa}{5 - 8 \kappa}}, \\
& \lVert (Y_{n})^{\otimes 2}(t)  \rVert_{H^{-\frac{1}{4} - 2\kappa}} \lesssim \lVert (Y^{n})^{\otimes 2} (t) \rVert_{L^{2}}  \overset{\eqref{Define Lt n kappa and Nt n kappa 2}}{\lesssim} (\bar{N}_{T}^{\kappa})^{2}, 
\end{align}
\end{subequations} 
which lead to, along with \eqref{est 144} and \eqref{est 146},  
\begin{equation}
\sup_{n\in\mathbb{N}} \lVert \partial_{t} w^{n} \rVert_{L^{2} (0, T; H^{-\frac{5}{4} - 2 \kappa})}^{2} \lesssim  ( \bar{N}_{T}^{\kappa})^{4} \sup_{n\in\mathbb{N}} (1+ \lVert w^{n} \rVert_{C_{T} L^{2}}^{2}) (1+ \lVert w^{n} \rVert_{L^{2}(0, T; \dot{H}^{\frac{5}{4} - 2 \kappa})}^{2}) \leq C(T, \bar{N}_{T}^{\kappa}). 
\end{equation}
This, along with \eqref{est 144} and \eqref{est 146}, allows us to rely on Lions-Aubins compactness lemma (e.g. \cite[Lemma 4 (i)]{S90}) and Banach-Alaoglu theorem to deduce the existence of a subsequence $\{w^{n_{k}}\}_{k} \subset \{w^{n}\}_{n}$, which we relabel as $\{w^{n}\}_{n}$, and its limit $w \in L^{\infty} (0, T; L^{2} (\mathbb{T}^{3}))$ $\cap L^{2} (0, T; H^{\frac{5}{4} - \frac{3\kappa}{2}} (\mathbb{T}^{3}))$ such that 
\begin{subequations}\label{convergence} 
\begin{align}
& w^{n} \overset{\ast}{\rightharpoonup} w \hspace{3mm} \text{weak}^{\ast} \text{ in } L^{\infty} (0, T; L^{2} (\mathbb{T}^{3})), \label{convergence 1} \\
& w^{n} \rightharpoonup w \hspace{3mm} \text{weakly in } L^{2} (0, T; H^{\frac{5}{4}- \frac{3\kappa}{2}} (\mathbb{T}^{3})),  \label{convergence 2} \\
& w^{n} \to w \hspace{3mm} \text{strongly in } L^{2} (0, T; H^{\beta} (\mathbb{T}^{3})) \hspace{3mm} \forall \hspace{1mm} \beta \in \left(-\frac{5}{4} -2\kappa, \frac{5}{4}- \frac{3\kappa}{2} \right).  \label{convergence 3} 
\end{align}
\end{subequations} 
With these convergence results in hand, it readily follows that for all $i, j \in \{1,2,3\}$,  
\begin{align*}
& \lVert w_{i}^{n}w_{j}^{n}  - w_{i} w_{j} \rVert_{L^{2} (0, T; L^{1}(\mathbb{T}^{3}))},  \\
& \lVert D^{n} \otimes_{s} w^{n} - D \otimes_{s} w \rVert_{L^{2} (0, T; H^{-\frac{1}{4} - 3 \kappa} ( \mathbb{T}^{3}))}, \\
&\lVert Y_{i}^{n} Y_{j}^{n} - Y_{i}Y_{j} \rVert_{C([0, T]; \mathscr{C}^{1- 5 \kappa} (\mathbb{T}^{3}))}
\end{align*}
all vanish as $n\to\infty$ and consequently, $w$ solves \eqref{Equation of w} distributionally. We can also compute starting from \eqref{est 148} and \eqref{est 149} that 
\begin{equation*}
\int_{0}^{T} \lVert w^{n, \mathcal{H}, \lambda} - w^{\mathcal{H},\lambda} \rVert_{\dot{H}^{\frac{5}{4} - 4 \kappa}}^{2} dt \lesssim \int_{0}^{T} \lVert w^{n} - w \rVert_{H^{-\kappa}}^{2} \lVert \mathcal{H}_{\lambda} Q^{n} \rVert_{\mathscr{C}^{\frac{9}{4} - 3 \kappa}}^{2} + \lVert w \rVert_{H^{-\kappa}}^{2} \lVert \mathcal{H}_{\lambda} (Q^{n} - Q) \rVert_{\mathscr{C}^{\frac{9}{4} - 3 \kappa}}^{2} dt  \to 0 
\end{equation*}
as $n\to\infty$ due to \eqref{convergence 3} and \eqref{Define Lt n kappa and Nt n kappa 2}, which verifies \eqref{est 149}. 

Finally, concerning the regularity of $w^{\mathcal{L},\lambda}$ and $w^{\mathcal{H},\lambda}$, we see from \eqref{est 151} that $w^{\mathcal{H},\lambda} \in L^{2} (0, T; B_{p,2}^{\frac{5}{4} - 2 \kappa})$ for all $p \in \left[1, \frac{12}{1+ 4 \kappa} \right]$ as claimed in \eqref{est 153b}. The other claim that $w^{\mathcal{H},\lambda} \in L^{\infty} (0, T; L_{\sigma}^{2})$ in \eqref{est 153b} follows from \eqref{est 145} and \eqref{est 144} since $w^{\mathcal{H},\lambda} = w - w^{\mathcal{L},\lambda}$ from \eqref{est 149}. At last, the regularity of \eqref{est 153a} follows from \eqref{est 145}. This completes the proof of Proposition \ref{Proposition 5.1}.
\end{proof}

\subsection{Uniqueness}
\hfill\\ We come to the proof of uniqueness of HL weak solution. Once again, the commutator term will present us a challenge in \eqref{Define III4}. The techniques we applied in \eqref{Define C4} and \eqref{Define tilde C4} can allow us to overcome this obstacle. The new difficulty we described in Section \ref{Subsection 1.1} and Remark \ref{Remark 2.3} will come about  from \eqref{Define III5}. 

\begin{proposition}\label{Proposition 5.2}
Let $\mathcal{N}''$ be the null set from Proposition \ref{Proposition 4.1}. Then, for any $\omega \in \Omega \setminus \mathcal{N}''$, any $u^{\text{in}} \in L_{\sigma}^{2}$, and $\kappa \in (0, \frac{3}{20})$ sufficiently small, there exists at most one HL weak solution starting from $u^{\text{in}}$. 
\end{proposition}  
The reason for $\kappa < \frac{3}{20}$ will be made clear in \eqref{Reason 3}.

\begin{proof}[Proof of Proposition \ref{Proposition 5.2}]
Suppose that $v \triangleq w+Y$ and $\bar{v} \triangleq \bar{w} + Y$ are two HL weak solutions, both starting from $u^{\text{in}}$. We define 
\begin{equation}\label{Define z, z l lambda, z h lambda}
z \triangleq w - \bar{w}, \hspace{3mm} z^{\mathcal{L},\lambda} \triangleq w^{\mathcal{L},\lambda} - \bar{w}^{\mathcal{L},\lambda}, \hspace{3mm} z^{\mathcal{H},\lambda} \triangleq z - z^{\mathcal{L},\lambda}. 
\end{equation} 
It follows that 
\begin{align}
& \partial_{t} z^{\mathcal{L},\lambda} + \nu \Lambda^{\frac{5}{2}} z^{\mathcal{L},\lambda}  =  - \mathbb{P}_{L} \divergence \Bigg( 2 \mathcal{L}_{\lambda} X \otimes_{s} z^{\mathcal{L},\lambda} + 2 \mathcal{H}_{\lambda} X \otimes_{s} z^{\mathcal{L},\lambda} - 2 \mathcal{H}_{\lambda} X \circlesign{\succ}_{s} z^{\mathcal{L},\lambda} + 2 X \otimes_{s} z^{\mathcal{H},\lambda}  \nonumber \\
& \hspace{33mm} - 2 \mathcal{H}_{\lambda} X \circlesign{\succ}_{s} z^{\mathcal{H},\lambda} - \mathcal{C}^{\circlesign{\prec}_{s}} (z, \mathcal{H}_{\lambda} Q) + w^{\otimes 2} - \bar{w}^{\otimes 2} + 2 Y \otimes_{s} z \Bigg)\label{est 154}
\end{align}  
so that taking $L^{2}(\mathbb{T}^{3})$-inner products with $z^{\mathcal{L},\lambda}$ gives us 
\begin{equation}\label{est 178}
\frac{1}{2} \partial_{t} \lVert z^{\mathcal{L},\lambda} \rVert_{L^{2}}^{2} + \nu \lVert z^{\mathcal{L},\lambda} \rVert_{\dot{H}^{\frac{5}{4}}}^{2} = \sum_{k=1}^{5} \RomanIII_{k} 
\end{equation}  
where 
\begin{subequations}
\begin{align}
&\RomanIII_{1} \triangleq -2 \langle z^{\mathcal{L},\lambda}, \mathbb{P}_{L} \divergence ( \mathcal{L}_{\lambda} X \otimes_{s} z^{\mathcal{L},\lambda}) \rangle (t),  \label{Define III1}\\
& \RomanIII_{2} \triangleq -2 \langle z^{\mathcal{L},\lambda}, \mathbb{P}_{L} \divergence ( \mathcal{H}_{\lambda} X \otimes_{s} z^{\mathcal{L},\lambda} - \mathcal{H}_{\lambda} X \circlesign{\succ}_{s} z^{\mathcal{L},\lambda} ) \rangle (t), \label{Define III2}\\
&\RomanIII_{3} \triangleq -2 \langle z^{\mathcal{L},\lambda}, \mathbb{P}_{L} \divergence ( X \otimes_{s} z^{\mathcal{H},\lambda} - \mathcal{H}_{\lambda} X \circlesign{\succ}_{s} z^{\mathcal{H},\lambda} ) \rangle (t), \label{Define III3}\\
&\RomanIII_{4} \triangleq - \langle z^{\mathcal{L},\lambda}, \mathbb{P}_{L} \divergence \left( \mathcal{C}^{\circlesign{\prec}_{s}} (z, \mathcal{H}_{\lambda} Q) \right) \rangle(t), \label{Define III4}\\
&\RomanIII_{5} \triangleq - \langle z^{\mathcal{L},\lambda}, \mathbb{P}_{L} \divergence  \left( w^{\otimes 2} - \bar{w}^{\otimes 2} + 2Y \otimes_{s} z  \right) \rangle(t).  \label{Define III5}
\end{align}
\end{subequations} 
Considering \eqref{est 153} we define 
\begin{equation}\label{Define lambda norm and MT}
\lVert \lvert f \rvert \rVert_{\lambda} \triangleq \lVert f^{\mathcal{L}, \lambda} \rVert_{\dot{H}^{\frac{5}{4}}} + \lVert f^{\mathcal{H},\lambda} \rVert_{B_{\frac{5}{2}, 2}^{\frac{19}{20}}} \hspace{1mm} \text{ and } \hspace{1mm} M_{T} \triangleq \lVert (w^{\mathcal{L},\lambda}, w^{\mathcal{H},\lambda}, \bar{w}^{\mathcal{L},\lambda}, \bar{w}^{\mathcal{H},\lambda}) \rVert_{L_{T}^{\infty} L_{x}^{2}}.
\end{equation} 
\begin{remark}\label{Remark 5.2}
Considering \eqref{est 153}, we see that $\lVert \lvert w \rvert \rVert_{\lambda}, \lVert \lvert \bar{w} \rvert \rVert_{\lambda} \in L_{T}^{2}$ and $M_{T} < \infty$. In comparison, in the case of the 2D Navier-Stokes equations in \cite{HR24}, \cite[Definition 7.1]{HR24} has ``$w^{\mathcal{L},\lambda} \in L^{2}([0,T]; H^{1})$'' and ``$w^{\mathcal{H},\lambda} \in L^{2}([0, T]; B_{4,\infty}^{1-\delta})$'' for any $\delta \in (0,1)$ and \cite[Equation (7.18)]{HR24} has ``$\lVert \lvert \phi \rvert \rVert_{\lambda} = \lVert \phi^{\mathcal{L},\lambda} \rVert_{H^{1}} + \lVert \phi^{\mathcal{H},\lambda} \rVert_{B_{4,\infty}^{1-3\kappa}}$.'' Considering \eqref{est 153}, we can replace $\lVert f^{\mathcal{H},\lambda} \rVert_{B_{\frac{5}{2}, 2}^{\frac{19}{20}}}$ in \eqref{Define lambda norm and MT} with e.g. $\lVert f^{\mathcal{H},\lambda} \rVert_{B_{11,2}^{\frac{5}{4} - 2 \kappa}}$, and yet, we strategically chose otherwise; the reason will become clear in \eqref{Reason 1}-\eqref{Reason 4}. 
\end{remark}  

We start with the following estimate that will be useful throughout the proof of Proposition \ref{Proposition 5.2}: for all $s \in [0, \frac{5}{4} - \frac{3\kappa}{2})$, there exists a universal constant $C \geq 0$ such that 
\begin{align}
\lVert z \rVert_{H^{s}}& \overset{\eqref{Define z, z l lambda, z h lambda} \eqref{est 149}}{\leq}  \lVert z^{\mathcal{L},\lambda} \rVert_{H^{s}} + C \lVert z \circlesign{\prec}_{s} \mathcal{H}_{\lambda} Q \rVert_{H^{s+1}} \nonumber \\
&\overset{\eqref{Sobolev products c}}{\leq}  \lVert z^{\mathcal{L},\lambda} \rVert_{H^{s}} + C \lVert z \rVert_{H^{s}} \lambda^{-\frac{\kappa}{4}} \lVert Q \rVert_{\mathscr{C}^{\frac{9}{4} - \frac{5\kappa}{4}}} \overset{\eqref{Regularity of Q} \eqref{Define Lt kappa and Nt kappa}}{\leq} \lVert z^{\mathcal{L},\lambda} \rVert_{H^{s}} + C \lVert z \rVert_{H^{s}} \lambda^{-\frac{\kappa}{4}} N_{T}^{\kappa} t^{\frac{\kappa}{10}}  \label{est 165}
\end{align}
and consequently due to \eqref{Define z, z l lambda, z h lambda}, for $\lambda \gg 1$ sufficiently large, 
\begin{equation}\label{est 156} 
\lVert z \rVert_{H^{s}} \leq 2 \lVert z^{\mathcal{L}, \lambda} \rVert_{H^{s}}  \hspace{1mm} \text{ so that } \hspace{1mm} \lVert z^{\mathcal{H},\lambda} \rVert_{H^{s}}  \leq 3 \lVert z^{\mathcal{L},\lambda} \rVert_{H^{s}}.
\end{equation} 

First, we can use the divergence-free property to estimate from \eqref{Define III1}, 
\begin{align}
\RomanIII_{1} =& \int_{\mathbb{T}^{3}} ( z^{\mathcal{L},\lambda} \cdot\nabla) z^{\mathcal{L},\lambda} \cdot \mathcal{L}_{\lambda} X dx \lesssim \lVert z^{\mathcal{L},\lambda} \rVert_{L^{2}} \lVert \nabla z^{\mathcal{L},\lambda} \rVert_{L^{2}} \lVert \mathcal{L}_{\lambda} X \rVert_{L^{\infty}}  \nonumber \\
\leq& C(\lambda, N_{T}^{\kappa})  \lVert z^{\mathcal{L},\lambda} \rVert_{L^{2}} \lVert z^{\mathcal{L},\lambda} \rVert_{\dot{H}^{\frac{5}{4}}} \leq \frac{\nu}{32} \lVert z^{\mathcal{L},\lambda} \rVert_{\dot{H}^{\frac{5}{4}}}^{2} + C(\lambda, N_{T}^{\kappa}) \lVert z^{\mathcal{L},\lambda} \rVert_{L^{2}}^{2}.\label{est 155}
\end{align}

Second, we rewrite \eqref{Define III2} via \eqref{Bony's decomposition}, and estimate using the Gagliardo-Nirenberg inequalities of $\lVert f \rVert_{\dot{H}^{1-\kappa}} \lesssim \lVert f \rVert_{L^{2}}^{\frac{1+ 4 \kappa}{5} } \lVert f \rVert_{\dot{H}^{\frac{5}{4}}}^{\frac{4- 4 \kappa}{5}}$ and $\lVert f \rVert_{\dot{H}^{\frac{1}{4} + 2 \kappa}} \lesssim \lVert f \rVert_{L^{2}}^{\frac{4-8\kappa}{5}} \lVert f \rVert_{\dot{H}^{\frac{5}{4}}}^{\frac{1+ 8 \kappa}{5}}$,  
\begin{align}
\RomanIII_{2} =& - 2 \langle z^{\mathcal{L},\lambda}, \mathbb{P}_{L} \divergence \left( \mathcal{H}_{\lambda}X \circlesign{\prec}_{s} z^{\mathcal{L},\lambda} + \mathcal{H}_{\lambda} X \circlesign{\circ}_{s} z^{\mathcal{L},\lambda} \right) \rangle (t) \nonumber\\
\overset{\eqref{Sobolev products c} \eqref{Define Lt kappa and Nt kappa}}{\lesssim}& C(N_{T}^{\kappa}) \lVert z^{\mathcal{L},\lambda} \rVert_{L^{2}}^{\frac{5- 4 \kappa}{5}} \lVert z^{\mathcal{L},\lambda} \rVert_{\dot{H}^{\frac{5}{4}}}^{\frac{5+ 4 \kappa}{5}} \leq \frac{\nu}{32} \lVert z^{\mathcal{L},\lambda} \rVert_{\dot{H}^{\frac{5}{4}}}^{2} + C(N_{T}^{\kappa}) \lVert z^{\mathcal{L},\lambda} \rVert_{L^{2}}^{2}. \label{Estimate on III2}
\end{align}

 Third, we estimate of $\RomanIII_{3}$ from \eqref{Define III3} using Gagliardo-Nirenberg inequalities of $\lVert f \rVert_{\dot{H}^{1}} \lesssim \lVert f \rVert_{L^{2}}^{\frac{1}{5} } \lVert f \rVert_{\dot{H}^{\frac{5}{4}}}^{\frac{4}{5}}$ and $\lVert f \rVert_{\dot{H}^{\frac{1}{4} +  2\kappa}} \lesssim \lVert f \rVert_{L^{2}}^{\frac{4-8\kappa}{5}} \lVert f \rVert_{\dot{H}^{\frac{5}{4}}}^{\frac{1+ 8 \kappa}{5}}$, 
\begin{align}
\RomanIII_{3} \overset{\eqref{Bony's decomposition}}{=}&  -2 \langle z^{\mathcal{L},\lambda}, \mathbb{P}_{L} \divergence ( \mathcal{L}_{\lambda} X \otimes_{s} z^{\mathcal{H},\lambda} + \mathcal{H}_{\lambda} X \circlesign{\prec}_{s} z^{\mathcal{H},\lambda} + \mathcal{H}_{\lambda} X \circlesign{\circ}_{s} z^{\mathcal{H},\lambda} ) \rangle (t)  \nonumber \\
\overset{\eqref{Sobolev products d} \eqref{Sobolev products e}}{\lesssim}& \lVert z^{\mathcal{L},\lambda} \rVert_{L^{2}}^{\frac{1}{5}} \lVert z^{\mathcal{L},\lambda} \rVert_{\dot{H}^{\frac{5}{4}}}^{\frac{4}{5}} \left( \lVert \mathcal{L}_{\lambda} X \rVert_{L^{\infty}} \lVert z^{\mathcal{H},\lambda} \rVert_{L^{2}} + \lVert \mathcal{H}_{\lambda} X \rVert_{\mathscr{C}^{-\frac{1}{4} - \kappa}} \lVert z^{\mathcal{H},\lambda} \rVert_{H^{\frac{1}{4} + 2\kappa}} \right) \nonumber\\ 
\overset{\eqref{Define Lt kappa and Nt kappa} \eqref{est 156}}{\lesssim}& C(\lambda, N_{T}^{\kappa})\lVert z^{\mathcal{L}, \lambda} \rVert_{L^{2}}^{\frac{1}{5}} \lVert z^{\mathcal{L},\lambda} \rVert_{\dot{H}^{\frac{5}{4}}}^{\frac{4}{5}}  \lVert z^{\mathcal{L},\lambda} \rVert_{H^{\frac{1}{4} + 2\kappa}} \leq \frac{\nu}{32} \lVert z^{\mathcal{L},\lambda} \rVert_{\dot{H}^{\frac{5}{4}}}^{2} + C(\lambda, N_{T}^{\kappa}) \lVert z^{\mathcal{L},\lambda} \rVert_{L^{2}}^{2}. \label{Estimate on III3}
\end{align}

Next, we rewrite $\RomanIII_{4}$ from \eqref{Define III4} using \eqref{est 167} as
\begin{equation}\label{Split III4}
\RomanIII_{4} = \sum_{k=1}^{3} \RomanIII_{4,k}
\end{equation} 
where 
\begin{subequations}
\begin{align}
& \RomanIII_{4,1} \triangleq - \langle z^{\mathcal{L},\lambda}, \divergence \left( [ \partial_{t} z + \nu \Lambda^{\frac{5}{2}} z ] \circlesign{\prec}_{s} \mathcal{H}_{\lambda} Q\right) \rangle(t), \label{Define III4,1}\\
& \RomanIII_{4,2} \triangleq  \nu \langle z^{\mathcal{L}, \lambda}, \divergence \left( (\Lambda^{\frac{5}{2}} z) \circlesign{\prec}_{s} \mathcal{H}_{\lambda} Q  \right) \rangle (t),  \label{Define III4,2}\\
& \RomanIII_{4,3} \triangleq -\nu \langle z^{\mathcal{L}, \lambda}, \divergence \left( \Lambda^{\frac{5}{2}} (z \circlesign{\prec}_{s} \mathcal{H}_{\lambda} Q) - z \circlesign{\prec}_{s} \Lambda^{\frac{5}{2}} \mathcal{H}_{\lambda} Q \right) \rangle (t).  \label{Define III4,3}
\end{align}
\end{subequations}
We can rewrite using \eqref{Define z, z l lambda, z h lambda}, \eqref{Equation of w}, and \eqref{Define D} 
\begin{equation}\label{est 157}
\partial_{t} z + \nu \Lambda^{\frac{5}{2}} z =  - \mathbb{P}_{L} \divergence \left( z \otimes w + \bar{w} \otimes z + 2 [X+Y] \otimes_{s} z \right).
\end{equation} 
We estimate from \eqref{Define III4,1}
\begin{align}
\RomanIII_{4,1} \overset{\eqref{Sobolev products c}}{\lesssim}& \lVert z^{\mathcal{L},\lambda} \rVert_{\dot{H}^{\frac{5}{4}}} \lVert \partial_{t} z + \nu \Lambda^{\frac{5}{2}} z \rVert_{H^{-\frac{5}{2} + \frac{3\kappa}{2}}} \lVert \mathcal{H}_{\lambda} Q \rVert_{\mathscr{C}^{\frac{9}{4} - \frac{3\kappa}{2}}} \nonumber \\
\overset{\eqref{est 157} \eqref{Regularity of Q} \eqref{Define Lt kappa and Nt kappa}}{\lesssim}& N_{T}^{\kappa}  \lVert z^{\mathcal{L},\lambda} \rVert_{\dot{H}^{\frac{5}{4}}} \lVert  z \otimes w + \bar{w} \otimes z + 2 [X+Y] \otimes_{s} z  \rVert_{H^{-\frac{3}{2} + \frac{3\kappa}{2}}}. \label{est 158} 
\end{align} 
We can estimate using the embedding of $L^{\frac{6}{6- 3 \kappa}}(\mathbb{T}^{3})\subset H^{-\frac{3}{2} + \frac{3\kappa}{2}}(\mathbb{T}^{3})$ and $H^{\frac{3\kappa}{2}}(\mathbb{T}^{3}) \subset L^{\frac{6}{3-3\kappa}}(\mathbb{T}^{3})$, \eqref{Bony's decomposition}, and \eqref{Sobolev products}, 
\begin{align}
&\lVert  z \otimes w + \bar{w} \otimes z + 2 [X+Y] \otimes_{s} z  \rVert_{H^{-\frac{3}{2} + \frac{3\kappa}{2}}}  \lesssim \lVert z \otimes w \rVert_{L^{\frac{6}{6- 3 \kappa}}} + \lVert \bar{w} \otimes z \rVert_{L^{\frac{6}{6- 3\kappa}}}  \label{est 159} \\
& \hspace{20mm} + \lVert X \otimes_{s} z \rVert_{H^{-\frac{3}{2} + \frac{3\kappa}{2}}} + \lVert Y \otimes_{s} z \rVert_{H^{-\frac{3}{2} + \frac{3\kappa}{2}}} \lesssim \lVert z \rVert_{H^{\frac{3\kappa}{2}}} ( \lVert w \rVert_{L^{2}} + \lVert \bar{w} \rVert_{L^{2}}) + N_{T}^{\kappa} \lVert z \rVert_{H^{\frac{1}{4} + 2 \kappa}} \nonumber 
\end{align}
and conclude from \eqref{est 158} 
\begin{equation}\label{est 160}
\RomanIII_{4,1} \leq \frac{\nu}{32} \lVert z^{\mathcal{L},\lambda} \rVert_{\dot{H}^{\frac{5}{4}}}^{2} + C(N_{T}^{\kappa}) \left( \lVert z \rVert_{H^{\frac{3\kappa}{2}}}^{2} [ \lVert w \rVert_{L^{2}}^{2} + \lVert \bar{w} \rVert_{L^{2}}^{2} ] + \lVert z \rVert_{\dot{H}^{\frac{1}{4} + 2 \kappa}}^{2} \right). 
\end{equation} 
Next, we can estimate directly from \eqref{Define III4,2}, 
\begin{equation}\label{Estimate on III4,2}
\RomanIII_{4,2} \overset{\eqref{Sobolev products c}}{\lesssim} \lVert z^{\mathcal{L},\lambda} \rVert_{\dot{H}^{\frac{5}{4}}} \lVert \Lambda^{\frac{5}{2}} z \rVert_{H^{-\frac{5}{2} + \frac{3\kappa}{2}}} \lVert \mathcal{H}_{\lambda} Q \rVert_{\mathscr{C}^{\frac{9}{4} - \frac{3\kappa}{2}}}  \overset{\eqref{Regularity of Q} \eqref{Define Lt kappa and Nt kappa}}{\leq}  \frac{\nu}{32} \lVert z^{\mathcal{L},\lambda} \rVert_{\dot{H}^{\frac{5}{4}}}^{2} + C(N_{T}^{\kappa}) \lVert z \rVert_{H^{\frac{3\kappa}{2}}}^{2}. 
\end{equation} 
Next, concerning $\RomanIII_{4,3}$ of \eqref{Define III4,3}, we estimate for $N_{1}$ from \eqref{est 152}
\begin{align}
& \lVert \Lambda^{\frac{5}{2}} ( z \circlesign{\prec}_{s} \mathcal{H}_{\lambda} Q) - z \circlesign{\prec}_{s} \Lambda^{\frac{5}{2}} \mathcal{H}_{\lambda} Q \rVert_{H^{-\frac{1}{4}}}^{2}   \nonumber \\
\overset{\eqref{est 152}}{=}& \sum_{j\geq -1} 2^{-\frac{j}{2}} \left\lVert \sum_{m: \lvert j-m \rvert \leq N_{1}} \Lambda^{\frac{5}{2}} \left( ( S_{m-1} z ) \Delta_{m} \mathcal{H}_{\lambda} Q \right) - S_{m-1} z \Lambda^{\frac{5}{2}} \Delta_{m} \mathcal{H}_{\lambda} Q \right\rVert_{L^{2}}^{2}   \nonumber \\
\overset{\eqref{Kato-Ponce}}{\lesssim}& \sum_{j\geq -1} 2^{-\frac{j}{2}} \left( \lVert S_{j-1} z \rVert_{\dot{H}^{1+ \kappa}} \lVert \Lambda^{\frac{3}{2}} \Delta_{j} \mathcal{H}_{\lambda} Q \rVert_{L^{\infty}} +\lVert S_{j-1} z \rVert_{\dot{H}^{\frac{5}{2}} } \lVert \Delta_{j} \mathcal{H}_{\lambda} Q \rVert_{L^{\infty}} \right)^{2} \nonumber  \\
\lesssim&   \sum_{j\geq -1} \Bigg( \sum_{l \leq j-2} 2^{(l-j) (1- \frac{3\kappa}{2})} 2^{l(\frac{5\kappa}{2})} \lVert \Delta_{l} z \rVert_{L^{2}} 2^{j(\frac{9}{4} - \frac{3\kappa}{2})} \lVert \Delta_{j} \mathcal{H}_{\lambda} Q \rVert_{L^{\infty}}  \nonumber \\
&\hspace{12mm}  + 2^{(l-j)(\frac{5}{2} - \frac{3\kappa}{2})} 2^{l(\frac{3\kappa}{2})} \lVert \Delta_{l} z \rVert_{L^{2}} 2^{j( \frac{9}{4} - \frac{3\kappa}{2})} \lVert \Delta_{j} \mathcal{H}_{\lambda} Q \rVert_{L^{\infty}} \Bigg)^{2}  \overset{\eqref{Regularity of Q} \eqref{Define Lt kappa and Nt kappa}}{\lesssim} (N_{T}^{\kappa})^{2} \lVert z \rVert_{\dot{H}^{\frac{5\kappa}{2}}}^{2} \label{est 161}
\end{align}
which allows us to conclude from \eqref{Define III4,3}
\begin{align}
\RomanIII_{4,3} \lesssim&\lVert z^{\mathcal{L},\lambda} \rVert_{H^{\frac{5}{4}}} \lVert \Lambda^{\frac{5}{2}} (z \circlesign{\prec}_{s} \mathcal{H}_{\lambda} Q) - z \circlesign{\prec}_{s} \Lambda^{\frac{5}{2}} \mathcal{H}_{\lambda} Q \rVert_{H^{-\frac{1}{4}}}  \nonumber \\
& \hspace{20mm} \overset{\eqref{est 161}}{\lesssim} \lVert z^{\mathcal{L},\lambda} \rVert_{\dot{H}^{\frac{5}{4}}} N_{T}^{\kappa} \lVert z \rVert_{\dot{H}^{\frac{5\kappa}{2}}} \leq \frac{\nu}{32} \lVert z^{\mathcal{L},\lambda} \rVert_{\dot{H}^{\frac{5}{4}}}^{2} + C(N_{T}^{\kappa}) \lVert z \rVert_{\dot{H}^{\frac{5\kappa}{2}}}^{2}. \label{est 162}
\end{align} 
Applying \eqref{est 160}, \eqref{Estimate on III4,2}, and \eqref{est 162} to \eqref{Split III4}, we can conclude with the Gagliardo-Nirenberg inequality of $\lVert f \rVert_{\dot{H}^{\frac{1}{4} + 2\kappa}} \lesssim \lVert f \rVert_{L^{2}}^{\frac{3}{4} - 2 \kappa} \lVert f \rVert_{\dot{H}^{1}}^{\frac{1}{4} + 2 \kappa}$ that 
\begin{align}
&\RomanIII_{4} \leq \frac{3\nu}{32} \lVert z^{\mathcal{L},\lambda} \rVert_{\dot{H}^{\frac{5}{4}}}^{2} + C(N_{T}^{\kappa}) \lVert z \rVert_{\dot{H}^{\frac{1}{4} + 2 \kappa}}^{2}[ \lVert w \rVert_{L^{2}}^{2} + \lVert \bar{w} \rVert_{L^{2}}^{2} + 1 ] \nonumber \\
& \hspace{15mm} \overset{\eqref{est 149} \eqref{est 156} \eqref{Define lambda norm and MT}}{\leq} \frac{3\nu}{32} \lVert z^{\mathcal{L},\lambda} \rVert_{\dot{H}^{\frac{5}{4}}}^{2} + C(N_{T}^{\kappa}, M_{T})  \lVert z^{\mathcal{L},\lambda} \rVert_{L^{2}}^{\frac{3}{2} - 4 \kappa} \lVert z^{\mathcal{L},\lambda} \rVert_{\dot{H}^{\frac{5}{4}}}^{\frac{1}{2}+ 4 \kappa} \nonumber\\
& \hspace{15mm} \leq \frac{\nu}{8}  \lVert z^{\mathcal{L},\lambda} \rVert_{\dot{H}^{\frac{5}{4}}}^{2} + C (N_{T}^{\kappa}, M_{T}) \lVert z^{\mathcal{L},\lambda} \rVert_{L^{2}}^{2}. \label{Estimate on III4} 
\end{align} 

We come to the estimate of $\RomanIII_{5}$ of \eqref{Define III5}. First, we rewrite using \eqref{Define z, z l lambda, z h lambda}
\begin{equation}\label{Split III5}
\RomanIII_{5} =\sum_{k=1}^{4} \RomanIII_{5,k}
\end{equation}
where 
\begin{subequations}\label{Define III5,k}
\begin{align}
& \RomanIII_{5,1} \triangleq - \langle z^{\mathcal{L},\lambda}, \divergence \left( z^{\mathcal{L},\lambda} \otimes w^{\mathcal{L},\lambda} + \bar{w}^{\mathcal{L},\lambda} \otimes z^{\mathcal{L},\lambda}  \right) \rangle (t), \label{Define III5,1} \\
&\RomanIII_{5,2} \triangleq - \langle z^{\mathcal{L},\lambda}, \divergence \left( z^{\mathcal{L},\lambda} \otimes w^{\mathcal{H},\lambda} + \bar{w}^{\mathcal{H},\lambda} \otimes z^{\mathcal{L},\lambda}  \right) \rangle (t),  \label{Define III5,2}\\
& \RomanIII_{5,3} \triangleq  - \langle z^{\mathcal{L},\lambda}, \divergence \left( z^{\mathcal{H},\lambda} \otimes w^{\mathcal{L},\lambda} + z^{\mathcal{H},\lambda} \otimes w^{\mathcal{H},\lambda} + \bar{w}^{\mathcal{L},\lambda} \otimes z^{\mathcal{H},\lambda} + \bar{w}^{\mathcal{H},\lambda} \otimes z^{\mathcal{H},\lambda}  \right) \rangle (t), \label{Define III5,3}\\
& \RomanIII_{5,4} \triangleq  -2 \langle z^{\mathcal{L},\lambda},  \divergence ( Y \otimes_{s} z ) \rangle (t); \label{Define III5,4}
\end{align}
\end{subequations} 
the difficult terms are those involving $z^{\mathcal{H},\lambda}$ which we separated in \eqref{Define III5,3}. Let us estimate the most straight-forward piece $\RomanIII_{5,4}$ from \eqref{Define III5,4} first:
\begin{equation}\label{Estimate on III5,4}
\RomanIII_{5,4}  \lesssim \lVert z^{\mathcal{L},\lambda} \rVert_{\dot{H}^{\frac{5}{4}}} \lVert Y \rVert_{L^{\infty}} \lVert z \rVert_{L^{2}}  \overset{\eqref{est 156} \eqref{Define Lt kappa and Nt kappa}}{\lesssim} N_{T}^{\kappa} \lVert z^{\mathcal{L},\lambda} \rVert_{\dot{H}^{\frac{5}{4}}}   \lVert z^{\mathcal{L},\lambda} \rVert_{L^{2}}. 
\end{equation} 
Next, we estimate from \eqref{Define III5,1} using the embedding of $\dot{H}^{\frac{5}{4}} (\mathbb{T}^{3}) \hookrightarrow L^{12} (\mathbb{T}^{3})$ 
\begin{align}\label{Estimate on III5,1}
\RomanIII_{5,1} \lesssim&  \lVert z^{\mathcal{L},\lambda} \rVert_{\dot{H}^{\frac{5}{4}}} \left( \lVert z^{\mathcal{L},\lambda}  \otimes w^{\mathcal{L},\lambda} \rVert_{L^{\frac{12}{7}}} + \lVert \bar{w}^{\mathcal{L},\lambda} \otimes z^{\mathcal{L},\lambda} \rVert_{L^{\frac{12}{7}}} \right) \nonumber\\
\lesssim& \lVert z^{\mathcal{L},\lambda} \rVert_{\dot{H}^{\frac{5}{4}}} \lVert z^{\mathcal{L},\lambda} \rVert_{L^{2}} \left( \lVert w^{\mathcal{L},\lambda} \rVert_{\dot{H}^{\frac{5}{4}}} + \lVert \bar{w}^{\mathcal{L},\lambda} \rVert_{\dot{H}^{\frac{5}{4}}} \right) \overset{\eqref{Define lambda norm and MT}}{\lesssim}  \lVert z^{\mathcal{L},\lambda} \rVert_{\dot{H}^{\frac{5}{4}}} \lVert z^{\mathcal{L},\lambda} \rVert_{L^{2}} \lVert  \lvert (w, \bar{w})  \rvert \rVert_{\lambda}.
\end{align}
Next, we estimate from \eqref{Define III5,2} 
\begin{equation}\label{est 163}
\RomanIII_{5,2} \lesssim  \lVert z^{\mathcal{L},\lambda} \rVert_{\dot{H}^{\frac{5}{4}}} \lVert z^{\mathcal{L},\lambda} \rVert_{L^{\frac{24}{7}}} \left( \lVert w^{\mathcal{H},\lambda} \rVert_{L^{\frac{24}{7}}} + \lVert \bar{w}^{\mathcal{H},\lambda} \rVert_{L^{\frac{24}{7}}} \right). 
\end{equation} 
To continue, we recall that $B_{\frac{24}{7},2}^{0} \subset L^{\frac{24}{7}}$ (e.g. \cite[Theorem 6.4.4]{BL76}) and further estimate by H$\ddot{\mathrm{o}}$lder's and Bernstein's inequalities (e.g. \cite[Lemma 2.1]{BCD11})  
\begin{align}
\lVert f \rVert_{L^{\frac{24}{7}}} \lesssim& \lVert f \rVert_{B_{\frac{24}{7}, 2}^{0}} \lesssim \left( \sum_{m\geq -1} 2^{m3(\frac{1}{2} - \frac{7}{24})} \lVert \Delta_{m} f \rVert_{L^{2}} 2^{m3(\frac{2}{5} - \frac{7}{24})} \lVert \Delta_{m}f \rVert_{L^{\frac{5}{2}}} \right)^{\frac{1}{2}} \nonumber\\
& \hspace{10mm} \lesssim \left( \sum_{m\geq -1} \lVert \Delta_{m} f \rVert_{L^{2}}^{2} \right)^{\frac{1}{4}} \left( \sum_{m\geq -1} \left\lvert 2^{m(\frac{19}{20})} \lVert \Delta_{m} f \rVert_{L^{\frac{5}{2}}} \right\rvert^{2} \right)^{\frac{1}{4}} \lesssim \lVert f \rVert_{L^{2}}^{\frac{1}{2}} \lVert f \rVert_{B_{\frac{5}{2}, 2}^{\frac{19}{20}}}^{\frac{1}{2}}. \label{est 164}
\end{align}
Then, along with the Gagliardo-Nirenberg inequality of $ \lVert f \rVert_{L^{\frac{24}{7}}} \lesssim \lVert f \rVert_{L^{2}}^{\frac{1}{2}} \lVert f \rVert_{\dot{H}^{\frac{5}{4}}}^{\frac{1}{2}}$, we can continue to estimate from \eqref{est 163} 
\begin{align}
\RomanIII_{5,2}\lesssim& \lVert z^{\mathcal{L},\lambda} \rVert_{L^{2}}^{\frac{1}{2}} \lVert z^{\mathcal{L},\lambda} \rVert_{\dot{H}^{\frac{5}{4}}}^{\frac{3}{2}} \left( \lVert w^{\mathcal{H},\lambda} \rVert_{L^{2}}^{\frac{1}{2}} \lVert w^{\mathcal{H},\lambda} \rVert_{B_{\frac{5}{2}, 2}^{\frac{19}{20}}}^{\frac{1}{2}} +  \lVert \bar{w}^{\mathcal{H},\lambda} \rVert_{L^{2}}^{\frac{1}{2}} \lVert \bar{w}^{\mathcal{H},\lambda} \rVert_{B_{\frac{5}{2}, 2}^{\frac{19}{20}}}^{\frac{1}{2}} \right) \nonumber \\
& \hspace{10mm}  \overset{\eqref{Define lambda norm and MT}}{\lesssim}  \lVert z^{\mathcal{L},\lambda} \rVert_{L^{2}}^{\frac{1}{2}} \lVert z^{\mathcal{L},\lambda} \rVert_{\dot{H}^{\frac{5}{4}}}^{\frac{3}{2}}  M_{T}^{\frac{1}{2}} \lVert\lvert (w, \bar{w})\rvert \rVert_{\lambda}^{\frac{1}{2}}. \label{Estimate on III5,2}
\end{align} 

\begin{remark}\label{Remark 5.3} 
We are now ready to discuss the difficulty of $\RomanIII_{5,3}$ in \eqref{Define III5,3}. We explain the difficulty using only $\langle z^{\mathcal{L},\lambda}, \divergence \left( z^{\mathcal{H},\lambda} \otimes w^{\mathcal{L},\lambda}  \right) \rangle (t)$ as an example because the identical difficulties apply to other terms within $\RomanIII_{5,3}$. As we described in Remark \ref{Remark 2.3}, we can proceed similarly to \eqref{est 60},  
\begin{equation}\label{est 169}
 \langle z^{\mathcal{L},\lambda}, \divergence \left( z^{\mathcal{H},\lambda} \otimes w^{\mathcal{L},\lambda}  \right) \rangle 
 \lesssim \lVert z^{\mathcal{L},\lambda} \rVert_{\dot{H}^{\frac{5}{4}}} \lVert z^{\mathcal{H},\lambda} \otimes w^{\mathcal{L},\lambda} \rVert_{H^{-\frac{1}{4}}} \lesssim  \lVert z^{\mathcal{L},\lambda} \rVert_{\dot{H}^{\frac{5}{4}}}  \lVert z^{\mathcal{H},\lambda} \rVert_{L^{12}} \lVert w^{\mathcal{L},\lambda} \rVert_{L^{2}}. 
\end{equation} 
To handle the $\lVert z^{\mathcal{H},\lambda} \rVert_{L^{12}}$, we can estimate similarly to \eqref{est 165}-\eqref{est 156}, using the embeddings $\dot{H}^{\frac{5}{4}}(\mathbb{T}^{3}) \hookrightarrow L^{12} (\mathbb{T}^{3})$ and $B_{12,2}^{1} \subset W^{1,12}$ and $L^{12} \subset B_{12,12}^{0}$ from \cite[Theorem 6.4.4]{BL76}, for $N_{1}$ from \eqref{est 152}, we can find $C_{1}\geq 0$ such that 
\begin{equation}\label{est 280}
\lVert f \rVert_{L^{12}} \leq C_{1} \lVert f \rVert_{\dot{H}^{\frac{5}{4}}} 
\end{equation}
and $C_{2} \geq 0$ that satisfy 
\begin{align}
\lVert z \rVert_{L^{12}} \overset{ \eqref{Define z, z l lambda, z h lambda} \eqref{est 149}}{\leq}&  C_{1} \lVert z^{\mathcal{L},\lambda} \rVert_{\dot{H}^{\frac{5}{4}}} + C \lVert z \circlesign{\prec}_{s} \mathcal{H}_{\lambda} Q \rVert_{W^{1,2}}  \leq C_{1} \lVert z^{\mathcal{L},\lambda} \rVert_{\dot{H}^{\frac{5}{4}}} + C  \lVert z \circlesign{\prec}_{s} \mathcal{H}_{\lambda} Q \rVert_{B_{12,2}^{1}} \nonumber \\
\overset{\eqref{est 152} }{=}&  C_{1} \lVert z^{\mathcal{L},\lambda} \rVert_{\dot{H}^{\frac{5}{4}}} + C \left( \sum_{m\geq -1} \left\lvert 2^{m} \lVert \sum_{j: \lvert j-m \rvert \leq N_{1}} (S_{j-1} z ) \otimes_{s} \Delta_{j} \mathcal{H}_{\lambda} Q \rVert_{L^{12}} \right\rvert^{2} \right)^{\frac{1}{2}}  \nonumber\\
 \overset{\eqref{Regularity of Q} \eqref{Define Lt kappa and Nt kappa}}{\leq}& C_{1} \lVert z^{\mathcal{L},\lambda} \rVert_{\dot{H}^{\frac{5}{4}}} + C N_{T}^{\kappa} \lambda^{-\frac{\kappa}{2}}   \lVert z \rVert_{B_{12,12}^{0}} \leq  C_{1} \lVert z^{\mathcal{L},\lambda} \rVert_{\dot{H}^{\frac{5}{4}}} + C_{2}(N_{T}^{\kappa}) \lambda^{-\frac{\kappa}{2}}   \lVert z \rVert_{L^{12}} \label{est 166} 
\end{align}
and consequently for $\lambda \gg 1$ sufficiently large, 
\begin{equation}\label{est 168}
\lVert z \rVert_{L^{12}} \leq 2 C_{1} \lVert z^{\mathcal{L},\lambda} \rVert_{\dot{H}^{\frac{5}{4}}}  \hspace{1mm} \text{ and } \hspace{1mm} \lVert z^{\mathcal{H},\lambda} \rVert_{L^{12}} \overset{\eqref{est 280}}{\leq} 3 C_{1} \lVert z^{\mathcal{L},\lambda} \rVert_{\dot{H}^{\frac{5}{4}}}. 
\end{equation}  
 Yet, applying \eqref{est 168} to \eqref{est 169} gives us 
\begin{align}
 \langle z^{\mathcal{L},\lambda}, \divergence \left( z^{\mathcal{H},\lambda} \otimes w^{\mathcal{L},\lambda}  \right) \rangle 
 \overset{\eqref{est 169}}{\lesssim}   \lVert z^{\mathcal{L},\lambda} \rVert_{\dot{H}^{\frac{5}{4}}}  \lVert z^{\mathcal{H},\lambda} \rVert_{L^{12}} \lVert w^{\mathcal{L},\lambda} \rVert_{L^{2}} \overset{\eqref{est 168}}{\lesssim}   \lVert z^{\mathcal{L},\lambda} \rVert_{\dot{H}^{\frac{5}{4}}}^{2}  \lVert w^{\mathcal{L},\lambda} \rVert_{L^{2}} \label{est 170}
\end{align} 
which is not of a favorable form to prove uniqueness, typically a constant multiple of either $M_{T}\lVert z^{\mathcal{L},\lambda} \rVert_{L^{2}}^{2} (1+ \lVert \lvert w \rvert \rVert_{\lambda}^{2} + \lVert \lvert \bar{w} \rvert \rVert_{\lambda}^{2})$. In comparison, the analogous bound in the 2D case was (see e.g. \cite[Equation (192)]{Y23c})
\begin{align*}
 \lVert z^{\mathcal{H},\lambda} \rVert_{L^{2}}^{\frac{1}{2}} \lVert z^{\mathcal{L},\lambda} \rVert_{H^{1}}^{\frac{3}{2}} \lVert (w^{\mathcal{L},\lambda}, w^{\mathcal{H},\lambda}) \rVert_{L^{2}}^{\frac{1}{2}}  \lVert w\rvert \rVert_{\lambda}^{\frac{1}{2}},
\end{align*}
to which, applying an analogue of \eqref{est 156} with $s = 0$ and Young's inequality gives us a bound of 
\begin{align*}
C\lVert z^{\mathcal{L},\lambda} \rVert_{L^{2}}^{\frac{1}{2}} \lVert z^{\mathcal{L},\lambda} \rVert_{H^{1}}^{\frac{3}{2}} \lVert (w^{\mathcal{L},\lambda}, w^{\mathcal{H},\lambda}) \rVert_{L^{2}}^{\frac{1}{2}}  \lVert \lvert w \rvert \rVert_{\lambda}^{\frac{1}{2}} \leq \frac{\nu}{32} \lVert z^{\mathcal{L},\lambda} \rVert_{H^{1}}^{2} + C \lVert z^{\mathcal{L},\lambda} \rVert_{L^{2}}^{2}  \lVert (w^{\mathcal{L},\lambda}, w^{\mathcal{H},\lambda}) \rVert_{L^{2}}^{2}  \lVert \lvert w \rvert \rVert_{\lambda}^{2},
\end{align*} 
which is of the favorable form to complete this proof of uniqueness. Thus, we must approach $\RomanIII_{5,3}$ differently from $\RomanIII_{5,1}$ or $\RomanIII_{5,2}$ or the 2D case. 
\end{remark}
We apply H$\ddot{\mathrm{o}}$lder's inequalities to start our estimate with 
\begin{equation}\label{est 172} 
\RomanIII_{5,3} \lesssim \lVert z^{\mathcal{L},\lambda} \rVert_{W^{1, \frac{12}{5}}} \lVert z^{\mathcal{H},\lambda} \rVert_{L^{4}} \left( \lVert w^{\mathcal{L},\lambda} \rVert_{L^{3}} + \lVert w^{\mathcal{H},\lambda} \rVert_{L^{3}} + \lVert \bar{w}^{\mathcal{L},\lambda} \rVert_{L^{3}} + \lVert \bar{w}^{\mathcal{H},\lambda} \rVert_{L^{3}} \right).
\end{equation} 
Inspired by \cite[Equation (187)]{Y23c}, which showed $\lVert f \rVert_{L^{4}(\mathbb{T}^{2})} \lesssim \lVert f \rVert_{L^{2}(\mathbb{T}^{2})}^{\frac{1}{2}} \lVert f \rVert_{B_{\infty, 2}^{0}(\mathbb{T}^{2})}^{\frac{1}{2}}$, of which its proof shows independence of spatial dimension, we estimate using $B_{4,2}^{0} \subset L^{4}$ from \cite[Theorem 6.4.4]{BL76}
\begin{equation}\label{est 171} 
\lVert f \rVert_{L^{4}} \lesssim \lVert f \rVert_{B_{4,2}^{0}} \lesssim \left( \sum_{m\geq -1} \lVert \Delta_{m} f \rVert_{L^{2}} \lVert \Delta_{m} f \rVert_{L^{\infty}} \right)^{\frac{1}{2}} \lesssim \lVert f \rVert_{H^{\frac{1}{4}}}^{\frac{1}{2}} \lVert f \rVert_{B_{\infty,2}^{-\frac{1}{4}}}^{\frac{1}{2}},
\end{equation} 
which is also independent of spatial dimension because the estimates depend only on interpolation and H$\ddot{\mathrm{o}}$lder's inequalities. The heuristic behind our ``choice'' of $-\frac{1}{4}$ in \eqref{est 171} is $-\frac{1}{4} + \frac{3}{2} = \frac{5}{4}$ where $\frac{3}{2}$ represents the informal amount of derivatives one has to give up to bound $L^{\infty}(\mathbb{T}^{3})$-norm by $L^{2}(\mathbb{T}^{3})$-norm by Bernstein's inequality, as it will be clear in \eqref{est 173}, while $\frac{5}{4}$ represents the strength of the diffusion. By applying \eqref{est 171} and the embedding $\dot{H}^{\frac{1}{4}} (\mathbb{T}^{3}) \hookrightarrow L^{\frac{12}{5}} (\mathbb{T}^{3})$ to \eqref{est 172} we obtain
\begin{equation}\label{est 176}
\RomanIII_{5,3}  \lesssim \lVert z^{\mathcal{L},\lambda} \rVert_{\dot{H}^{\frac{5}{4}}} \lVert z^{\mathcal{H},\lambda} \rVert_{\dot{H}^{\frac{1}{4}}}^{\frac{1}{2}} \lVert z^{\mathcal{H},\lambda} \rVert_{B_{\infty,2}^{-\frac{1}{4}}}^{\frac{1}{2}} \left( \lVert w^{\mathcal{L},\lambda} \rVert_{L^{3}} + \lVert w^{\mathcal{H},\lambda} \rVert_{L^{3}} + \lVert \bar{w}^{\mathcal{L},\lambda} \rVert_{L^{3}} + \lVert \bar{w}^{\mathcal{H},\lambda} \rVert_{L^{3}} \right).
\end{equation} 
Now, we handle $\lVert z^{\mathcal{H},\lambda} \rVert_{B_{\infty, 2}^{-\frac{1}{4}}}$ similarly to \eqref{est 165}-\eqref{est 156}, \eqref{est 166}-\eqref{est 168} by finding the universal constant $C_{1}\geq 0$ such that the following consequence of the Bernstein's inequality of 
\begin{equation}\label{est 174} 
\lVert f \rVert_{B_{\infty,2}^{-\frac{1}{4}}} \leq C_{1} \lVert f \rVert_{H^{\frac{5}{4}}}
\end{equation} 
holds, and $C_{2} \geq 0$ such that for $N_{1}$ from \eqref{est 152} and all $\kappa \in (0, \frac{3}{4})$, 
\begin{align}
&\lVert z \rVert_{B_{\infty, 2}^{-\frac{1}{4}}}  \overset{\eqref{Define z, z l lambda, z h lambda} \eqref{est 149}}{\leq} \lVert z^{\mathcal{L},\lambda} \rVert_{B_{\infty,2}^{-\frac{1}{4}}} + \lVert \mathbb{P}_{L} \divergence (z \circlesign{\prec}_{s} \mathcal{H}_{\lambda} Q) \rVert_{B_{\infty,2}^{-\frac{1}{4}}}  \nonumber \\
&\overset{\eqref{est 174}}{\leq} C_{1} \left( \sum_{m\geq -1} \left\lvert 2^{m(-\frac{1}{4} + 3(\frac{1}{2} ) ) } \lVert \Delta_{m} z^{\mathcal{L},\lambda} \rVert_{L^{2}} \right\rvert^{2} \right)^{\frac{1}{2}} + C \left( \sum_{m\geq -1} \left\lvert 2^{\frac{3m}{4}} \lVert \Delta_{m} (z \circlesign{\prec}_{s} \mathcal{H}_{\lambda} Q ) \rVert_{L^{\infty}} \right\rvert^{2} \right)^{\frac{1}{2}}  \nonumber \\
&\overset{\eqref{est 152} }{\leq} C_{1} \lVert z^{\mathcal{L},\lambda} \rVert_{\dot{H}^{\frac{5}{4}}} + C \left( \sum_{m\geq -1} \left\lvert 2^{\frac{3m}{4}} \lVert \sum_{j: \lvert j-m \rvert \leq N_{1}} (S_{j-1} z) \otimes_{s} \Delta_{j} \mathcal{H}_{\lambda} Q \rVert_{L^{\infty}} \right\rvert^{2} \right)^{\frac{1}{2}} \nonumber \\
&\leq C_{1} \lVert z^{\mathcal{L},\lambda} \rVert_{\dot{H}^{\frac{5}{4}}} + C \left( \sum_{m\geq -1} 2^{2m ( \frac{9}{4} - 2 \kappa)} \lVert \Delta_{m} \mathcal{H}_{\lambda} Q \rVert_{L^{\infty}}^{2} \left\lvert \sum_{l: l \leq m - 2} 2^{(l-m) ( \frac{3}{2} - 2 \kappa)} 2^{-l(\frac{3}{2} - 2 \kappa)} \lVert \Delta_{l} z \rVert_{L^{\infty}} \right\rvert^{2} \right)^{\frac{1}{2}} \nonumber \ \\
&\overset{\eqref{Regularity of Q}\eqref{Define Lt kappa and Nt kappa}}{\leq} C_{1} \lVert z^{\mathcal{L},\lambda} \rVert_{\dot{H}^{\frac{5}{4}}}  + C_{2} \lambda^{-\frac{\kappa}{2}} N_{T}^{\kappa} \lVert z \rVert_{B_{\infty,2}^{-\frac{1}{4}}}.\label{est 173} 
\end{align} 
Consequently, for $\lambda \gg 1$ sufficiently large, 
\begin{equation}\label{est 175}
\lVert z \rVert_{B_{\infty,2}^{-\frac{1}{4}}} \leq 2 C_{1} \lVert z^{\mathcal{L},\lambda} \rVert_{\dot{H}^{\frac{5}{4}}}  \hspace{1mm} \text{ so that } \hspace{1mm}  \lVert z^{\mathcal{H},\lambda} \rVert_{B_{\infty,2}^{-\frac{1}{4}}} \overset{\eqref{Define z, z l lambda, z h lambda} \eqref{est 174}}{\leq} 3 C_{1} \lVert z^{\mathcal{L},\lambda} \rVert_{\dot{H}^{\frac{5}{4}}}.
\end{equation} 
We apply \eqref{est 175} to \eqref{est 176} to deduce, along with the Gagliardo-Nirenberg inequality of $\lVert f \rVert_{H^{\frac{1}{4}}} \lesssim \lVert f \rVert_{L^{2}}^{\frac{4}{5}} \lVert f \rVert_{H^{\frac{5}{4}}}^{\frac{1}{5}}$, 
\begin{align}
\RomanIII_{5,3}  \lesssim&  \lVert z^{\mathcal{L},\lambda} \rVert_{\dot{H}^{\frac{5}{4}}}^{\frac{3}{2}}  \lVert z^{\mathcal{H},\lambda} \rVert_{\dot{H}^{\frac{1}{4}}}^{\frac{1}{2}}  \left( \lVert w^{\mathcal{L},\lambda} \rVert_{L^{3}} + \lVert w^{\mathcal{H},\lambda} \rVert_{L^{3}} + \lVert \bar{w}^{\mathcal{L},\lambda} \rVert_{L^{3}} + \lVert \bar{w}^{\mathcal{H},\lambda} \rVert_{L^{3}} \right) \nonumber\\   
\overset{\eqref{est 156} }{\lesssim}& \lVert z^{\mathcal{L},\lambda} \rVert_{\dot{H}^{\frac{5}{4}}}^{\frac{8}{5}} \lVert z^{\mathcal{L},\lambda} \rVert_{L^{2}}^{\frac{2}{5}} \left( \lVert w^{\mathcal{L},\lambda} \rVert_{L^{3}} + \lVert w^{\mathcal{H},\lambda} \rVert_{L^{3}} + \lVert \bar{w}^{\mathcal{L},\lambda} \rVert_{L^{3}} + \lVert \bar{w}^{\mathcal{H},\lambda} \rVert_{L^{3}} \right).  \label{est 177}
\end{align}
To deal with $( \lVert w^{\mathcal{L},\lambda} \rVert_{L^{3}} + \lVert w^{\mathcal{H},\lambda} \rVert_{L^{3}} + \lVert \bar{w}^{\mathcal{L},\lambda} \rVert_{L^{3}} + \lVert \bar{w}^{\mathcal{H},\lambda} \rVert_{L^{3}})$, we can bound $\lVert w^{\mathcal{L},\lambda} \rVert_{L^{3}}$ and $\lVert \bar{w}^{\mathcal{L},\lambda}\rVert_{L^{3}}$  by the Gagliardo-Nirenberg inequality of $\lVert f \rVert_{L^{3}} \lesssim \lVert f \rVert_{L^{2}}^{\frac{3}{5}} \lVert f \rVert_{\dot{H}^{\frac{5}{4}}}^{\frac{2}{5}}$, 
\begin{equation}\label{First L3 estimate}
\lVert w^{\mathcal{L},\lambda} \rVert_{L^{3}} + \lVert \bar{w}^{\mathcal{L},\lambda} \rVert_{L^{3}} \lesssim ( \lVert w^{\mathcal{L},\lambda} \rVert_{L^{2}} + \lVert \bar{w}^{\mathcal{L},\lambda} \rVert_{L^{2}})^{\frac{3}{5}} (\lVert w^{\mathcal{L},\lambda} \rVert_{\dot{H}^{\frac{5}{4}}} + \lVert \bar{w}^{\mathcal{L},\lambda} \rVert_{\dot{H}^{\frac{5}{4}}} )^{\frac{2}{5}}.
\end{equation} 
We cannot estimate $\lVert w^{\mathcal{H},\lambda} \rVert_{L^{3}} + \lVert \bar{w}^{\mathcal{H},\lambda} \rVert_{L^{3}}$ similarly to \eqref{First L3 estimate} because from \eqref{est 153b} we know that $w^{\mathcal{H},\lambda}, \bar{w}^{\mathcal{H},\lambda}$ has the regularity of $L^{2}(0, T; B_{p,2}^{\frac{5}{4} - 2 \kappa})$ for $p \in [1, \frac{12}{1+ 4 \kappa} ]$, not necessarily $L^{2} (0, T; \dot{H}^{\frac{5}{4}})$. As we described in Remark \ref{Remark 5.2}, this where our choice of $\lVert \lvert \cdot \rvert \rVert_{\lambda}$ is decided. First, we estimate using $B_{3,2}^{0} \subset L^{3}$ from \cite[Theorem 6.4.4]{BL76} and Bernstein's inequality, for $p \in [1,3]$ to be subsequently determined, 
\begin{align}
\lVert f \rVert_{L^{3}} \lesssim& \lVert f \rVert_{B_{3,2}^{0}} \lesssim \left( \sum_{m\geq -1} \left\lvert \left(2^{m3(\frac{1}{2} - \frac{1}{3})} \lVert \Delta_{m} f \rVert_{L^{2}} \right)^{\frac{3}{5}}  \left(2^{m3 (\frac{1}{p}- \frac{1}{3})} \lVert \Delta_{m} f \rVert_{L^{p}} \right)^{\frac{2}{5}} \right\rvert^{2} \right)^{\frac{1}{2}} \nonumber  \\
\approx& \left( \sum_{m\geq -1} \left\lvert \lVert \Delta_{m} f \rVert_{L^{2}}^{\frac{3}{5}} \left( 2^{m (\frac{3}{p} - \frac{1}{4} )} \lVert \Delta_{m} f \rVert_{L^{p}} \right)^{\frac{2}{5}} \right\rvert^{2} \right)^{\frac{1}{2}}, \label{Reason 1}
\end{align} 
where the Bernstein's inequality crucially relies on $p \in [1,3]$. In addition to $p \in [1,3]$, considering $L^{2} (0, T; B_{p,2}^{\frac{5}{4} - 2 \kappa})$ in \eqref{est 153b} for $p \in [1, \frac{12}{1+ 4 \kappa}]$, we require $\frac{3}{p} - \frac{1}{4} \leq \frac{5}{4} - 2 \kappa$ which boils down to 
\begin{equation}\label{Reason 2} 
\frac{6}{3-4\kappa} \leq p.
\end{equation}  
Such a lower bound of $\frac{6}{3-4\kappa} \leq p$ and $p \leq 3$ implies that we can choose $p = \frac{5}{2}$ for example, which has a consequence of $\frac{3}{p} - \frac{1}{4} = \frac{19}{20}$ if 
\begin{equation}\label{Reason 3}
\kappa < \frac{3}{20}
\end{equation} 
which is valid by hypothesis of Proposition \ref{Proposition 5.2}. With these choices, we may continue from \eqref{Reason 1} by 
\begin{equation}\label{Reason 4} 
\lVert f \rVert_{L^{3}} \overset{\eqref{Reason 1}}{\lesssim} \left( \sum_{m\geq -1} \left\lvert \lVert \Delta_{m} f \rVert_{L^{2}}^{\frac{3}{5}} \left( 2^{m (\frac{19}{20} )} \lVert \Delta_{m} f \rVert_{L^{p}} \right)^{\frac{2}{5}} \right\rvert^{2} \right)^{\frac{1}{2}} \lesssim \lVert f \rVert_{L^{2}}^{\frac{3}{5}} \lVert f \rVert_{B_{\frac{5}{2}, 2}^{\frac{19}{20}}}^{\frac{2}{5}}.
\end{equation} 
At last, we are able to obtain by \eqref{Reason 4} the analogous estimate to \eqref{First L3 estimate} that is fit for $w^{\mathcal{H},\lambda}, \bar{w}^{\mathcal{H},\lambda}$:
\begin{equation}\label{Second L3 estimate}
\lVert w^{\mathcal{H},\lambda} \rVert_{L^{3}} + \lVert \bar{w}^{\mathcal{H},\lambda} \rVert_{L^{3}} \lesssim \left( \lVert w^{\mathcal{H},\lambda} \rVert_{L^{2}} + \lVert \bar{w}^{\mathcal{H},\lambda} \rVert_{L^{2}} \right)^{\frac{3}{5}} \left( \lVert w^{\mathcal{H},\lambda} \rVert_{B_{\frac{5}{2},2}^{\frac{19}{20}}} + \lVert \bar{w}^{\mathcal{H},\lambda} \rVert_{B_{\frac{5}{2},2}^{\frac{19}{20}}} \right)^{\frac{2}{5}}. 
\end{equation} 
Applying \eqref{First L3 estimate} and \eqref{Second L3 estimate} to \eqref{est 177} gives us, using the definition from \eqref{Define lambda norm and MT}, 
\begin{align}
\RomanIII_{5,3} &\lesssim  \lVert z^{\mathcal{L},\lambda} \rVert_{\dot{H}^{\frac{5}{4}}}^{\frac{8}{5}} \lVert z^{\mathcal{L},\lambda} \rVert_{L^{2}}^{\frac{2}{5}}  \lVert ( w^{\mathcal{L},\lambda}, \bar{w}^{\mathcal{L},\lambda}, w^{\mathcal{H},\lambda}, \bar{w}^{\mathcal{H},\lambda} ) \rVert_{L^{2}}^{\frac{3}{5}} \nonumber\\
& \times \left( \lVert (w^{\mathcal{L},\lambda},\bar{w}^{\mathcal{L},\lambda}) \rVert_{\dot{H}^{\frac{5}{4}}} + \lVert (w^{\mathcal{H},\lambda}, \bar{w}^{\mathcal{H},\lambda}) \rVert_{B_{\frac{5}{2}, 2}^{\frac{19}{20}}}\right)^{\frac{2}{5}}  \lesssim M_{T}^{\frac{3}{5}} \lVert z^{\mathcal{L},\lambda} \rVert_{\dot{H}^{\frac{5}{4}}}^{\frac{8}{5}} \lVert z^{\mathcal{L},\lambda} \rVert_{L^{2}}^{\frac{2}{5}} \lVert \lvert (w, \bar{w}) \rvert \rVert_{\lambda}^{\frac{2}{5}}. \label{Estimate on III5,3}
\end{align} 
We can conclude by applying \eqref{Estimate on III5,1}, \eqref{Estimate on III5,2}, \eqref{Estimate on III5,3}, and \eqref{Estimate on III5,4} to \eqref{Split III5} and Young's inequality  
\begin{align}
&\RomanIII_{5} \lesssim \lVert z^{\mathcal{L},\lambda} \rVert_{\dot{H}^{\frac{5}{4}}} \lVert z^{\mathcal{L},\lambda} \rVert_{L^{2}} \lVert  \lvert (w, \bar{w})  \rvert \rVert_{\lambda} + \lVert z^{\mathcal{L},\lambda} \rVert_{L^{2}}^{\frac{1}{2}} \lVert z^{\mathcal{L},\lambda} \rVert_{\dot{H}^{\frac{5}{4}}}^{\frac{3}{2}}  M_{T}^{\frac{1}{2}} \lVert\lvert (w, \bar{w})\rvert \rVert_{\lambda}^{\frac{1}{2}}  \nonumber\\
& + M_{T}^{\frac{3}{5}} \lVert z^{\mathcal{L},\lambda} \rVert_{\dot{H}^{\frac{5}{4}}}^{\frac{8}{5}} \lVert z^{\mathcal{L},\lambda} \rVert_{L^{2}}^{\frac{2}{5}} \lVert \lvert (w, \bar{w}) \rvert \rVert_{\lambda}^{\frac{2}{5}} + N_{T}^{\kappa} \lVert z^{\mathcal{L},\lambda} \rVert_{\dot{H}^{\frac{5}{4}}}   \lVert z^{\mathcal{L},\lambda} \rVert_{L^{2}} \nonumber\\
& \hspace{25mm}  \leq \frac{\nu}{32} \lVert z^{\mathcal{L},\lambda} \rVert_{\dot{H}^{\frac{5}{4}}}^{2} + C(M_{T}) \lVert z^{\mathcal{L},\lambda} \rVert_{L^{2}}^{2} (1+ \lVert \lvert (w, \bar{w}) \rvert \rVert_{\lambda}^{2}).\label{Estimate on III5}
\end{align}

Finally, applying \eqref{est 155}, \eqref{Estimate on III2}, \eqref{Estimate on III3}, \eqref{Estimate on III4}, and \eqref{Estimate on III5} to \eqref{est 178} gives us 
\begin{equation}
\frac{1}{2} \partial_{t} \lVert z^{\mathcal{L},\lambda} \rVert_{L^{2}}^{2} + \nu \lVert z^{\mathcal{L},\lambda} \rVert_{\dot{H}^{\frac{5}{4}}}^{2}
\leq \frac{\nu}{2} \lVert z^{\mathcal{L},\lambda} \rVert_{\dot{H}^{\frac{5}{4}}}^{2} + C( \lambda, M_{T}, N_{T}^{\kappa}) \lVert z^{\mathcal{L},\lambda} \rVert_{L^{2}}^{2} \left(1+ \lVert \lvert (w, \bar{w}) \rvert \rVert_{\lambda}^{2} \right).
\end{equation} 
Subtracting $\frac{\nu}{2} \lVert z^{\mathcal{L},\lambda} \rVert_{\dot{H}^{\frac{5}{4}}}^{2}$ from both sides and applying Gronwall's inequality, while 
recalling that $\lVert \lvert w \rvert \rVert_{\lambda}, \lVert \lvert \bar{w} \rvert \rVert_{\lambda} \in L_{T}^{2}$ from Remark \ref{Remark 5.2}, completes the proof of Proposition \ref{Proposition 5.2}. 
\end{proof} 

\section{Proof of Proposition \ref{Proposition on Anderson Hamiltonian}}\label{Proof of Proposition on Anderson Hamiltonian}
The purpose of this section is to sketch the proof of a general result that implies Proposition \ref{Proposition on Anderson Hamiltonian}. As we described in Remark \ref{Remark 2.4} in details, we will rely on the approach of \cite{AC15} making adjustments upon applications of Lemmas \ref{Burgers' Lemma A.2}-\ref{Burgers' Lemma A.3}, and estimating commutators carefully. Following \cite[p. 29]{HR24}, we consider a scalar-valued case in this section because its extension to the vector-valued case is straight-forward. As we described in \cite{Y25d}, it suffices to extend \cite{AC15} that treated the 2D case with $-\Delta$ as diffusion and an additive white-in-space noise $\eta \in \mathscr{C}^{-1 - \kappa} (\mathbb{T}^{2})$ as a force to the 3D case with $\Lambda^{\frac{5}{2}}$ as diffusion and $\eta  \in \mathscr{C}^{-\frac{5}{4} - \kappa}(\mathbb{T}^{3})$.  

\begin{define}
We define 
\begin{equation}\label{Burgers' Define E alpha}
\mathcal{E}^{\alpha} \triangleq \mathscr{C}^{\alpha} \times \mathscr{C}^{2 \alpha + \frac{5}{2}}, \hspace{2mm} \alpha \in \mathbb{R}, 
\end{equation} 
and the space of enhanced noise,  
\begin{equation}\label{Burgers' Define K alpha}
\mathcal{K}^{\alpha} \triangleq \overline{\{ (\eta, -\eta \circ \sigma(D) \eta -c ): \hspace{1mm} \eta \in C^{\infty}, c \in \mathbb{R} \}}, \hspace{1mm} \text{ where } \hspace{1mm} \sigma(D) \triangleq - ( 1 + \Lambda^{\frac{5}{2}})^{-1}, 
\end{equation} 
where the closure is taken w.r.t. $\mathcal{E}^{\alpha}$-topology (cf. \eqref{Burgers' est 21}). A general element of $\mathcal{K}^{\alpha}$ will be denoted by $\Theta \triangleq (\Theta_{1}, \Theta_{2})$. If $\eta \in \mathscr{C}^{\alpha}$ is $\Theta_{1}$, then $\Theta$ is said to be an enhancement (or lift) of $\eta$. 
\end{define} 
The case $\alpha$ = $-\frac{5}{4} - \kappa$, and $\nu =2$ applies to Proposition \ref{Proposition on Anderson Hamiltonian}; recall $\nabla X \in C_{t} \mathscr{C}_{x}^{-\frac{5}{4}-\kappa}$ $\mathbb{P}$-a.s. due to \eqref{Define Lt kappa and Nt kappa} and Proposition \ref{Proposition 4.1}.   
\begin{define}\label{Burgers' Definition 6.2}
Let $\alpha < -\frac{5}{4}$ and $\eta \in \mathscr{C}^{\alpha}$. For $\gamma \leq \alpha + \frac{5}{2}$, we define the space of distributions which are paracontrolled by $\sigma(D) \eta$ as 
\begin{equation}\label{Burgers' Define D eta gamma}
\mathcal{D}_{\eta}^{\gamma} \triangleq \{ f \in H^{\gamma}: \hspace{1mm} f^{\sharp} \triangleq f - f \prec \sigma(D) \eta \in H^{2\gamma - \frac{1}{4}}  \}. 
\end{equation} 
The space $\mathcal{D}_{\eta}^{\gamma}$, equipped with the following scalar product, is a Hilbert space: 
\begin{equation}\label{Burgers' D eta gamma norm}
\langle f,g \rangle_{\mathcal{D}_{\eta}^{\gamma}} \triangleq \langle f,g \rangle_{H^{\gamma}} + \langle f^{\sharp}, g^{\sharp} \rangle_{H^{2\gamma - \frac{1}{4}}} \hspace{3mm} \forall \hspace{1mm} f, g \in \mathcal{D}_{\eta}^{\gamma}. 
\end{equation} 
\end{define} 

The following definition, especially the product $\eta \circ f$ within $\eta f$, can be justified by Proposition \ref{Burgers' Proposition 6.1} (1), which is analogous to  \cite[Proposition 4.8]{AC15}.
\begin{define}\label{Burgers' Definition 6.3}
Let $\alpha \in (-\frac{19}{12}, -\frac{5}{4})$, $\gamma \in (\frac{1}{8} -\frac{\alpha}{2}, \alpha + \frac{5}{2}]$, and $\Theta = (\eta, \Theta_{2}) \in \mathcal{K}^{\alpha}$. We define the linear operator 
\begin{equation}\label{Burgers' Define H}
\mathcal{H}: \hspace{1mm} \mathcal{D}_{\eta}^{\gamma} \mapsto H^{\gamma - \frac{5}{2}} \text{ by } \mathcal{H}f \triangleq  -\Lambda^{\frac{5}{2}} f - \eta f \text{ where } \eta f = \eta \prec f + \eta \succ f + \eta \circ f  
\end{equation} 
(cf. \eqref{Define U}). 
\end{define} 
\begin{proposition}\label{Burgers' Proposition 6.1}
Let $\alpha \in (-\frac{19}{12}, -\frac{5}{4})$ and $\gamma \in (\frac{1}{8}-\frac{\alpha}{2}, \alpha + \frac{5}{2})$. 
\begin{enumerate}
\item Denote $\Theta = (\eta, \Theta_{2}) \in \mathcal{K}^{\alpha}$ as an enhancement of $\eta \in \mathscr{C}^{\alpha}$, and $f \in \mathcal{D}_{\eta}^{\gamma}$. Then we can define 
\begin{equation}\label{Burgers' est 165}
f \circ \eta = f \Theta_{2} + \mathcal{R} (f, \sigma(D) \eta, \eta) + f^{\sharp} \circ \eta 
\end{equation} 
and we have the following bound: for all $\kappa > 0$, 
\begin{equation}\label{Burgers' est 163} 
\lVert f \circ \eta \rVert_{H^{2 \alpha + \frac{5}{2} - \kappa}} \lesssim \lVert f \rVert_{\mathcal{D}_{\eta}^{\gamma}} \lVert \Theta \rVert_{\mathcal{E}^{\alpha}} ( 1+ \lVert \Theta \rVert_{\mathcal{E}^{\alpha}}). 
\end{equation} 

\item Let $\Theta = (\eta, \Theta_{2}) \in \mathcal{K}^{\alpha}$, and $\{ \Theta^{n} \}_{n \in\mathbb{N}}$ where 
\begin{equation}
\Theta^{n} \triangleq ( \eta_{n}, -\eta_{n} \circ ( \sigma(D) \eta_{n}) - c_{n}), \hspace{3mm} c_{n} \in \mathbb{R}, 
\end{equation} 
a family of smooth functions such that 
\begin{equation}\label{Burgers' est 169}
\Theta^{n} \to \Theta \hspace{1mm} \text{ in }\mathcal{E}^{\alpha} \hspace{1mm} \text{ as } n \nearrow \infty. 
\end{equation} 
Let $f^{n}$ be a smooth approximation of $f \in \mathcal{D}_{\eta}^{\gamma}$ such that 
\begin{equation}\label{Burgers' est 170}
\lVert f - f_{n} \rVert_{H^{\gamma}} + \lVert f_{n}^{\sharp} - f^{\sharp} \rVert_{H^{2\gamma - \frac{1}{4}}} \to 0 \hspace{1mm} \text{ as } n \nearrow \infty, \hspace{1mm} \text{ where } f_{n}^{\sharp} \triangleq f_{n} - f_{n} \prec \sigma(D) \eta_{n}. 
\end{equation} 
Then, for all $\kappa > 0$, 
\begin{equation}\label{Burgers' est 171} 
\lVert f_{n} \circ \eta_{n} - f \circ \eta \rVert_{H^{2 \alpha + \frac{5}{2} - \kappa}} \to 0 \hspace{3mm} \text{ as } n \nearrow \infty. 
\end{equation} 
\end{enumerate} 
\end{proposition}
\begin{proof}[Proof of Proposition \ref{Burgers' Proposition 6.1}]
\hfill\\ (1) The hypothesis of $\gamma > \frac{1}{8} - \frac{\alpha}{2}$ and $\alpha > - \frac{19}{12}$ assures that $2\alpha + \frac{5}{2} + \gamma > 0$ so that \eqref{Sobolev products c}-\eqref{Sobolev products e} allow us to estimate the first term in \eqref{Burgers' est 165} by  
\begin{equation}\label{Burgers' est 166} 
\lVert f \Theta_{2} \rVert_{H^{2 \alpha + \frac{5}{2} - \kappa}} \lesssim \lVert f \prec \Theta_{2} \rVert_{H^{2 \alpha + \frac{5}{2} - \kappa}} + \lVert f \succ \Theta_{2} \rVert_{H^{2 \alpha + \frac{5}{2}}} + \lVert f \circ \Theta_{2} \rVert_{H^{2 \alpha + \frac{5}{2} + \gamma}} \lesssim \lVert f \rVert_{\mathcal{D}_{\eta}^{\gamma}} \lVert \Theta \rVert_{\mathcal{E}^{\alpha}}.
\end{equation} 
We come to $\mathcal{R} (f, \sigma(D) \eta, \eta)$ in \eqref{Burgers' est 165}. Thanks to our hypothesis that $-\frac{19}{12} < \alpha < - \frac{5}{4}$, we have $\frac{1}{8} - \frac{\alpha}{2} \in (\frac{3}{4}, 1)$, allowing us to rely on \eqref{Burgers' Estimate on R} to estimate, along with the third term of \eqref{Burgers' est 165}, 
\begin{subequations}  
\begin{align}
&\lVert \mathcal{R} (f, \sigma(D)\eta, \eta) \rVert_{H^{2 \alpha + \frac{5}{2} - \kappa}}   \overset{\eqref{Burgers' Estimate on R}}{\lesssim} \lVert f \rVert_{H^{\frac{1}{8} - \frac{\alpha}{2}}} \lVert \sigma(D) \eta \rVert_{\mathscr{C}^{\alpha + \frac{5}{2}}} \lVert \eta \rVert_{\mathscr{C}^{\alpha}}  \overset{\eqref{Burgers' D eta gamma norm} \eqref{Burgers' Define E alpha}}{\lesssim} \lVert f \rVert_{\mathcal{D}_{\eta}^{\gamma}} \lVert \Theta \rVert_{\mathcal{E}^{\alpha}}^{2}, \label{Burgers' est 167}\\
&\lVert f^{\sharp} \circ \eta \rVert_{H^{2 \alpha + \frac{5}{2}}} \lesssim \lVert f^{\sharp} \circ \eta \rVert_{H^{\alpha + 2 \gamma - \frac{1}{4}}} \overset{\eqref{Sobolev products e}}{\lesssim} \lVert f^{\sharp} \rVert_{H^{2\gamma - \frac{1}{4}}} \lVert \eta \rVert_{\mathscr{C}^{\alpha}}  \overset{\eqref{Burgers' Define E alpha} \eqref{Burgers' D eta gamma norm}}{\lesssim} \lVert f \rVert_{\mathcal{D}_{\eta}^{\gamma}} \lVert\Theta \rVert_{\mathcal{E}^{\alpha}}. \label{Burgers' est 168}
\end{align} 
\end{subequations}
 By applying \eqref{Burgers' est 166}, \eqref{Burgers' est 167}, and \eqref{Burgers' est 168} in \eqref{Burgers' est 165} we conclude  \eqref{Burgers' est 163}.   

(2) Similar computations to the proof of part (1) immediately lead us to 
\begin{align}
& \lVert f_{n} \circ \eta_{n} - f \circ \eta \rVert_{H^{2 \alpha + \frac{5}{2} - \kappa}}  \nonumber \\
\overset{\eqref{Burgers' est 165}}{\lesssim}& \lVert (f_{n} - f) \circ \eta_{n} \rVert_{H^{2\alpha + \frac{5}{2} - \kappa}} + \lVert f(\Theta_{2}^{n} - \Theta_{2})  \rVert_{H^{2 \alpha + \frac{5}{2} - \kappa}}    + \lVert \mathcal{R} (f, \sigma(D) (\eta_{n} - \eta), \eta_{n}) \rVert_{H^{2 \alpha + \frac{5}{2} - \kappa}}  \nonumber \\
& \hspace{10mm} + \lVert \mathcal{R} (f, \sigma(D) (\eta_{n} - \eta), \eta) \rVert_{H^{2 \alpha + \frac{5}{2} - \kappa}} + \lVert f^{\sharp} \circ (\eta_{n} - \eta) \rVert_{H^{2 \alpha + \frac{5}{2} - \kappa}}   \overset{\eqref{Burgers' est 169} \eqref{Burgers' est 170}}{\to} 0 
\end{align}
as $n \nearrow \infty$ which concludes \eqref{Burgers' est 171}.  This completes the proof of Proposition \ref{Burgers' Proposition 6.1}. 
\end{proof} 

\begin{proposition}\label{Burgers' Proposition 6.2} 
Let $\alpha \in (-\frac{19}{12}, -\frac{5}{4}), \eta \in \mathscr{C}^{\alpha}$, and $\gamma \in (\frac{11}{12}, \alpha + \frac{5}{2})$. Then $\mathcal{D}_{\eta}^{\gamma}$ is dense in $L^{2}$. 
\end{proposition}

\begin{proof}[Proof of Proposition \ref{Burgers' Proposition 6.2} ]
We fix an arbitrary $g \in C^{\infty}$ and define 
\begin{equation}\label{Burgers' Define sigma a}
\sigma_{a}(k) \triangleq - \frac{1}{1+ a + \lvert k \rvert^{\frac{5}{2}}} \text{ for } a > 0
\end{equation} 
and consider a map 
\begin{equation}\label{Burgers' Define Gamma} 
\Gamma: \hspace{1mm} H^{\gamma}\mapsto H^{\gamma} \text{ defined by } \Gamma(f) \triangleq \sigma_{a}(D) (f\prec \eta) + g. 
\end{equation}
For any $k$, multi-index $r$ and $\vartheta \in [0, 1]$, we have for $\sigma(D)$ defined in \eqref{Burgers' Define K alpha},  
\begin{equation}\label{Burgers' est 172} 
\lvert D^{r} \sigma_{a}(k) \rvert \lesssim \frac{a^{\vartheta -1}}{(1+ \lvert k \rvert)^{\frac{5}{2} \vartheta + r}} \hspace{1mm} \text{ and } \hspace{1mm} \lvert D^{r} ( \sigma_{a} - \sigma) (k) \rvert \lesssim \frac{a^{\vartheta}}{1+ \lvert k \rvert^{\frac{5}{2}+ \frac{5}{2} \vartheta + r}}. 
\end{equation} 
Because $\gamma < \alpha + \frac{5}{2}$ by hypothesis, we can find $\epsilon_{1} > 0$ sufficiently small so that 
\begin{equation}\label{Burgers' Define first epsilon} 
\epsilon_{1} < \alpha + \frac{5}{2} - \gamma
\end{equation} 
and estimate via Lemma \ref{Burgers' Lemma A.1} for $\vartheta =  \frac{2(\gamma + \epsilon_{1} - \alpha)}{5} \in [0,1]$, for all $a \geq A$ with $A$ sufficiently large, 
\begin{equation}\label{Burgers' est 184} 
 \lVert \Gamma(f_{1}) - \Gamma(f_{2}) \rVert_{H^{\gamma}} \overset{\eqref{Burgers' Define first epsilon} \eqref{Burgers' Schauder}}{\lesssim} a^{\frac{ 2(\gamma + \epsilon_{1} - \alpha)}{5} -1} \lVert (f_{1} - f_{2}) \prec \eta \rVert_{H^{ \alpha - \epsilon_{1}}} \overset{\eqref{Sobolev products a} \eqref{Burgers' Define first epsilon}}{\ll} \lVert f_{1} - f_{2} \rVert_{L^{2}}, 
\end{equation} 
and thus $\Gamma$, as a contraction for all such large $a$, admits a unique fixed point $f_{a}$. 

An identical estimate in \eqref{Burgers' est 184} shows that, since $\gamma > \frac{11}{12}$ by hypothesis, the fixed point $f_{a}$ satisfies 
\begin{equation}\label{Burgers' est 185} 
\lVert f_{a} - g \rVert_{H^{\gamma}} \overset{\eqref{Burgers' Schauder}}{\lesssim} a^{\frac{ 2(\gamma + \epsilon_{1} - \alpha)}{5} - 1} \lVert f_{a} \prec \eta \rVert_{H^{\alpha - \epsilon_{1}}} \overset{\eqref{Sobolev products a}}{\lesssim}  a^{ \frac{ 2(\gamma + \epsilon_{1} - \alpha)}{5} - 1} \lVert f_{a} \rVert_{H^{\gamma}} \lVert \eta  \rVert_{\mathscr{C}^{\alpha}}.
\end{equation}
Taking $a \geq A$ for $A \gg \lVert \eta \rVert_{\mathscr{C}^{\alpha}}^{\frac{1}{1- \frac{2(\gamma + \epsilon_{1} - \alpha)}{5}}}$ in this inequality gives us $\sup_{a \geq A} \lVert f_{a} \rVert_{H^{\gamma}} \leq 2 \lVert g\rVert_{H^{\gamma}}$ and plugging this inequality back into the upper bound in \eqref{Burgers' est 185} finally shows 
\begin{equation*}
\lVert f_{a} - g \rVert_{H^{\gamma}} \lesssim a^{ \frac{ 2(\gamma + \epsilon_{1} - \alpha)}{5} - 1} \lVert g \rVert_{H^{\gamma}} \lVert \eta  \rVert_{\mathscr{C}^{\alpha}} 
\end{equation*}
which allows us to conclude that $f_{a}$ converges to $g$ in $H^{\gamma}$ as $a\nearrow \infty$ and hence in $L^{2}$. 

To show that $f_{a} - f_{a} \prec \sigma(D) \eta \in H^{2\gamma - \frac{1}{4}}$, we first rewrite using \eqref{Burgers' Define Gamma}, 
\begin{equation}\label{Burgers' est 186}
 f_{a} - f_{a} \prec \sigma(D) \eta = \sigma_{a}(D) (f_{a} \prec \eta) - f_{a} \prec \sigma_{a}(D) \eta + f_{a} \prec \left( \sigma_{a}(D) - \sigma(D) \right)\eta + g 
\end{equation} 
and conclude via the estimates of 
\begin{subequations}\label{Burgers' est 187}
\begin{align}
& \lVert \sigma_{a}(D) (f_{a} \prec \eta) - f_{a} \prec \sigma_{a}(D) \eta \rVert_{H^{2\gamma - \frac{1}{4}}}  \lesssim \lVert f_{a} \rVert_{H^{\gamma}} \lVert \eta \rVert_{\mathscr{C}^{\alpha}}, \label{Burgers' est 187a} \\
&  \lVert f_{a} \prec \left(\sigma_{a} (D) -\sigma(D)\right) \eta \rVert_{H^{2\gamma - \frac{1}{4}}} \lesssim_{a} \lVert f_{a} \rVert_{H^{\gamma}} \lVert \eta \rVert_{\mathscr{C}^{\alpha}} \label{Burgers' est 187b}
\end{align}
\end{subequations} 
where we used Lemma \ref{Burgers' Lemma A.2}, \eqref{Sobolev products c}, and \eqref{Burgers' est 172}. Considering \eqref{Burgers' est 186}-\eqref{Burgers' est 187} allows us to conclude that $f_{a} - f_{a} \prec \sigma(D) \eta \in H^{2\gamma - \frac{1}{4}}$ and therefore $f_{a} \in \mathcal{D}_{\eta}^{\gamma}$ in \eqref{Burgers' Define D eta gamma} so that $\mathcal{D}_{\eta}^{\gamma}$ is dense in $C^{\infty}$, which implies the claim and completes the proof of Proposition \ref{Burgers' Proposition 6.2}.
\end{proof}  

\begin{proposition}\label{Burgers' Proposition 6.3}
Define $\mathcal{H}$ by \eqref{Burgers' Define H}. Let $\alpha \in (-\frac{19}{12}, -\frac{5}{4}), \gamma \in (\frac{11}{12}, \alpha + \frac{5}{2})$, 
\begin{equation}\label{Burgers' Define rho}
\rho \in \left( \frac{2}{5} \left(2 \gamma - \alpha - \frac{11}{4} \right), 1 + \frac{2\alpha}{5}\right), 
\end{equation}  
and $\Theta = ( \eta, \Theta_{2}) \in \mathcal{K}^{\alpha}$. Then there exists $A = A( \lVert \Theta \rVert_{\mathcal{E}^{\alpha}})$ such that for all $a \geq A$ and $g \in H^{2 \gamma - \frac{11}{4}}$, 
\begin{equation}\label{Burgers' est 188}
(-\mathcal{H} + a) f = g 
\end{equation} 
admits a unique solution $f_{a} \in \mathcal{D}_{\eta}^{\gamma}$. Additionally, the mapping 
\begin{equation}\label{Burgers' Define mathcal G}
\mathcal{G}_{a}: \hspace{1mm} L^{2} \mapsto \mathcal{D}_{\eta}^{\gamma} \text{ for } a \geq A, \text{ defined by } \mathcal{G}_{a} g \triangleq f_{a}, 
\end{equation}  
is uniformly bounded; in fact, for all $g \in H^{-\delta}$, all $\delta \in [0, \frac{11}{4}- 2 \gamma]$, the following estimates hold: 
\begin{subequations}\label{Burgers' est 202} 
\begin{align}
&\lVert f_{a} \rVert_{H^{\gamma}} + a^{-\rho} \lVert f_{a}^{\sharp} \rVert_{H^{2\gamma - \frac{1}{4}}} \lesssim \left( a^{\frac{2(\gamma + \delta)}{5} -1} + a^{-\rho + \frac{2}{5} (2 \gamma - \frac{1}{4} + \delta) -1} \right) \lVert g \rVert_{H^{-\delta}},\label{Burgers' est 202a} \\
&\lVert \mathcal{G}_{a} g \rVert_{\mathcal{D}_{\eta}^{\gamma}} \lesssim \left(a^{\rho + \frac{2(\gamma + \delta)}{5} - 1} + a^{\frac{2}{5} (2\gamma - \frac{1}{4} + \delta) - 1} \right) \lVert g \rVert_{H^{-\delta}}. \label{Burgers' est 202b} 
\end{align}
\end{subequations}
\end{proposition}

\begin{proof}[Proof of Proposition \ref{Burgers' Proposition 6.3}] 
For any $A > 0$, we define the Banach space  
\begin{subequations}
\begin{align}
& \tilde{\mathcal{D}}_{\eta}^{\gamma, \rho, A} \triangleq \{ (f_{a}, f_{a}')_{a \geq A} \in C([A, \infty); H^{\gamma})^{2}:  \hspace{1mm} f_{a} \in \mathcal{D}_{\eta}^{\gamma}, \lVert (f, f') \rVert_{\tilde{D}_{\eta}^{\gamma, \rho, A}} < \infty \}, \label{Burgers' Define tilde D}\\
\text{where } &\lVert (f, f') \rVert_{\tilde{D}_{\eta}^{\gamma, \rho, A}} \triangleq \sup_{a \geq A} \lVert f_{a}' \rVert_{H^{\gamma}} + \sup_{a\geq A} a^{-\rho} \lVert f_{a}^{\sharp} \rVert_{H^{2\gamma - \frac{1}{4}}}+ \sup_{a \geq A} \lVert f_{a} \rVert_{H^{\gamma}}, \label{Burgers' Define tilde D norm}
\end{align}
\end{subequations}
and 
\begin{equation}\label{Burgers' Define mathcal M}
\mathcal{M}(f, f') \triangleq (M(f, f'), f), \hspace{3mm} M(f, f')_{a} \triangleq \tilde{\sigma}_{a}(D) (f_{a} \eta - g) \hspace{3mm} \forall \hspace{1mm} (f, f') \in \tilde{D}_{\eta}^{\gamma, \rho, A}, 
\end{equation} 
where 
\begin{subequations}\label{Burgers' Define tilde sigma a}
\begin{align}
&\tilde{\sigma}_{a}(D) \triangleq -\frac{1}{a+ \lvert k \rvert^{\frac{5}{2}}} \hspace{2mm} \text{ for } a > 2 \\
\text{that satisfies } &\lvert D^{r} \tilde{\sigma}_{a}(k) \rvert \lesssim \frac{a^{\vartheta -1}}{(1+ \lvert k \rvert)^{\frac{5}{2} \vartheta + r}}  \hspace{2mm} \text{ and } \hspace{2mm} \lvert D^{r} ( \tilde{\sigma}_{a} - \sigma) (k) \rvert \lesssim \frac{a^{\vartheta}}{1+ \lvert k \rvert^{\frac{5}{2}+ \frac{5}{2} \vartheta + r}}
\end{align}
\end{subequations} 
for any multi-index $r$, similarly to \eqref{Burgers' est 172}, and the product $f_{a} \eta$ is justified via \eqref{Burgers' est 165}. To find a solution to \eqref{Burgers' est 188}, it suffices to prove that $\mathcal{M}$ admits a unique fixed point in $\tilde{\mathcal{D}}_{\eta}^{\gamma, \rho, A}$. To do so, the idea is to show that if $(f, f') \in \tilde{\mathcal{D}}_{\eta}^{\gamma, \rho, A}$, then 
\begin{equation}\label{Burgers' est 194} 
M(f, f')_{a} \in H^{\gamma} \hspace{1mm} \text{ and } \hspace{1mm} M(f, f')^{\sharp} \triangleq M(f, f') - f \prec \sigma(D) \eta \in H^{2\gamma} 
\end{equation} 
so that $\mathcal{M}(f, f') \in \tilde{D}_{\eta}^{\gamma, \rho, A}$ allowing us to conclude that $\mathcal{M} (\tilde{\mathcal{D}}_{\eta}^{\gamma, \rho, A}) \subset \tilde{\mathcal{D}}_{\eta}^{\gamma, \rho, A}$. For this purpose, first, we need to improve the $\epsilon_{1} > 0$ from \eqref{Burgers' Define first epsilon}. Because $\gamma < \alpha + \frac{5}{2}$ by hypothesis and $\rho >  \frac{2}{5} (2 \gamma - \alpha -\frac{11}{4})$ by \eqref{Burgers' Define rho},  we can find $\epsilon_{2} > 0$ sufficiently small such that 
\begin{equation}\label{Burgers' Define epsilon two}
\epsilon_{2} < \min\left\{ \alpha + \frac{5}{2} - \gamma,  \frac{11}{4} + \alpha + \frac{5\rho}{2} - 2 \gamma, \frac{1}{3}\right\}.
\end{equation} 
Then, with choices of $\vartheta = \frac{2(\gamma + \epsilon_{2} - \alpha)}{5}, - \frac{2\alpha}{5}, \frac{2(\gamma + \delta)}{5} \in [0,1]$ for any $\delta \in [0, \frac{11}{4} - 2 \gamma]$, we deduce by \eqref{Burgers' Schauder}, \eqref{Sobolev products c} and \eqref{Burgers' Define tilde sigma a}, 
\begin{subequations}\label{Burgers' est 173}   
\begin{align}
& \lVert \tilde{\sigma}_{a}(D) (f_{a} \prec \eta) \rVert_{H^{\gamma}} \lesssim a^{\frac{ 2(\gamma + \epsilon_{2} - \alpha)}{5} -1} \lVert f_{a} \prec \eta \rVert_{H^{\alpha - \epsilon_{2}}} \lesssim a^{\frac{ 2(\gamma + \epsilon_{2} - \alpha)}{5} -1} \lVert f_{a} \rVert_{H^{\gamma}} \lVert \eta \rVert_{\mathscr{C}^{\alpha}}, \label{Burgers' est 173a} \\
& \lVert \tilde{\sigma}_{a}(D) ( f_{a} \circ \eta + f_{a} \succ \eta ) \rVert_{H^{\gamma}}  \lesssim  a^{- \frac{2\alpha}{5} -1} \left( \lVert f_{a} \circ \eta \rVert_{H^{\gamma + \alpha}} + \lVert f_{a}  \rVert_{H^{\gamma}} \lVert \eta \rVert_{\mathscr{C}^{\alpha}}  \right),  \label{Burgers' est 173b} \\
& \lVert \tilde{\sigma}_{a}(D) g \rVert_{H^{\gamma}} \lesssim  a^{\frac{2(\gamma + \delta)}{5}  - 1} \lVert g \rVert_{H^{-\delta}}. \label{Burgers' est 173c} 
\end{align}  
\end{subequations} 
To treat $\lVert f_{a} \circ \eta \rVert_{H^{\gamma + \alpha}}$ in \eqref{Burgers' est 173b}, we write $f_{a} \circ \eta = f_{a} \circ \eta - f_{a}^{\sharp} \circ \eta + f_{a}^{\sharp} \circ \eta$ and estimate 
\begin{align}
&\lVert f_{a} \circ \eta - f_{a}^{\sharp} \circ \eta \rVert_{H^{\gamma + \alpha}}  \lesssim \lVert f_{a}' \prec\Theta_{2} \rVert_{H^{\gamma + \alpha}} + \lVert f_{a}' \succ \Theta_{2} \rVert_{H^{\gamma + \alpha}} + \lVert f_{a}' \circ \Theta_{2} \rVert_{H^{\gamma + \alpha}} \nonumber \\
&  \hspace{18mm} + \lVert \mathcal{R} ( f_{a}', \sigma(D)\eta, \eta)  \rVert_{H^{\gamma + \alpha}}   \overset{\eqref{Sobolev products c}-\eqref{Sobolev products e} \eqref{Burgers' Estimate on R}}{\lesssim} \lVert f_{a}' \rVert_{H^{\gamma}} ( \lVert \Theta_{2} \rVert_{\mathscr{C}^{2 \alpha + \frac{5}{2}}} + \lVert \eta \rVert_{\mathscr{C}^{\alpha}}^{2}); \label{Burgers' est 191}
\end{align}
applying \eqref{Sobolev products e} in the other term $\lVert f_{a}^{\sharp} \circ \eta \rVert_{H^{\gamma + \alpha}}$, we have in sum
\begin{equation}\label{Burgers' est 193}
\lVert f_{a} \circ \eta \rVert_{H^{\gamma + \alpha}} \lesssim  \lVert f_{a}' \rVert_{H^{\gamma}} ( \lVert \Theta_{2} \rVert_{\mathscr{C}^{2 \alpha + \frac{5}{2}}} + \lVert \eta \rVert_{\mathscr{C}^{\alpha}}^{2})  + a^{\rho} \left( \frac{ \lVert f_{a}^{\sharp} \rVert_{H^{2\gamma - \frac{1}{4}}}}{a^{\rho}} \right) \lVert \eta \rVert_{\mathscr{C}^{\alpha}}. 
\end{equation} 
Applying \eqref{Burgers' est 173}, \eqref{Burgers' est 193}, \eqref{Burgers' Define tilde D norm} to \eqref{Burgers' Define mathcal M} allows us to deduce 
\begin{align}
& \lVert M(f, f')_{a} \rVert_{H^{\gamma}} \nonumber \overset{\eqref{Burgers' Define mathcal M}}{=} \lVert \tilde{\sigma}_{a}(D) (f_{a} \eta - g) \rVert_{H^{\gamma}} \nonumber \\ 
\lesssim& a^{\max \{ \frac{ 2(\gamma + \epsilon_{2} - \alpha)}{5} -1, \rho - \frac{2\alpha}{5} -1 \}} \lVert (f, f') \rVert_{\tilde{\mathcal{D}}_{\eta}^{\gamma, \rho, A}} ( \lVert \eta \rVert_{\mathscr{C}^{\alpha}} + \lVert \eta \rVert_{\mathscr{C}^{\alpha}}^{2} + \lVert \Theta_{2} \rVert_{\mathscr{C}^{2 \alpha + \frac{5}{2}}} )+a^{\frac{2(\gamma + \delta)}{5} -1} \lVert g\rVert_{H^{-\delta}},  \label{Burgers' est 198}
\end{align}
which implies $M(f, f')_{a} \in H^{\gamma}$, the first claim in \eqref{Burgers' est 194}. 

Next, to show the second claim in \eqref{Burgers' est 194}, namely that $M(f, f')^{\sharp} \in H^{2\gamma - \frac{1}{4}}$, we write from \eqref{Burgers' est 194} and \eqref{Burgers' Define mathcal M}, 
\begin{equation}\label{Burgers' est 197}
M(f, f')_{a}^{\sharp} = \sum_{k=1}^{6} \RomanIV_{k} 
\end{equation} 
where  
\begin{subequations} 
\begin{align}
&\RomanIV_{1} \triangleq \mathscr{C}_{a}(f_{a}, \eta), \hspace{9mm} \RomanIV_{2} \triangleq  \tilde{\sigma}_{a}(D) (f_{a} \circ \eta - f_{a}^{\sharp} \circ \eta),  \hspace{1mm} \RomanIV_{3} \triangleq  \tilde{\sigma}_{a}(D) (f_{a}^{\sharp} \circ \eta), \label{Burgers' Define IV1, IV2, and IV3}\\
& \RomanIV_{4} \triangleq \tilde{\sigma}_{a}(D) (f_{a}\succ \eta), \hspace{1mm} \RomanIV_{5} \triangleq - \tilde{\sigma}_{a}(D) g, \hspace{20mm} \RomanIV_{6} \triangleq f_{a} \prec \left( \tilde{\sigma}_{a}- \sigma \right)(D) \eta, \label{Burgers' Define IV3, IV4, and IV5}
\end{align}
\end{subequations} 
and
\begin{equation}\label{Burgers' Define mathcal Ca}
\mathscr{C}_{a}(f, g) \triangleq \tilde{\sigma}_{a}(D) (f \prec g) - f \prec \tilde{\sigma}_{a}(D)g
\end{equation} 
analogously to \eqref{Burgers' Define mathcal C}. With 
\begin{align*}
&\vartheta = \frac{2(\gamma + \epsilon_{2} - \alpha)}{5}, \hspace{1mm} \vartheta = \frac{2(\gamma - \frac{1}{4} - \alpha)}{5}, \\
&\vartheta = -\frac{2\alpha}{5},  \hspace{1mm}\vartheta = \frac{2(2 \gamma - \frac{1}{4} + \delta)}{5},  \hspace{1mm}\vartheta = \frac{2}{5} \left(2\gamma - \frac{11}{4} + \epsilon_{2} - \alpha\right)  \in [0,1]
\end{align*}
where $\delta \in [0, \frac{11}{4}-2\gamma]$, we estimate using \eqref{Burgers' est 195}, \eqref{Burgers' Define tilde sigma a}, \eqref{Burgers' Schauder}, \eqref{Burgers' est 191}, \eqref{Sobolev products},
\begin{subequations}\label{Burgers' est 196}  
\begin{align}
&\lVert \RomanIV_{1} \rVert_{H^{2\gamma - \frac{1}{4}}} \lesssim a^{\frac{ 2(\gamma + \epsilon_{2} - \alpha)}{5} - 1} \lVert f_{a} \rVert_{H^{\gamma}} \lVert \eta \rVert_{\mathscr{C}^{\alpha}}, \label{Burgers' est 196a} \\
& \lVert \RomanIV_{2}  \rVert_{H^{2\gamma - \frac{1}{4}}} \lesssim a^{\frac{2(\gamma - \frac{1}{4} - \alpha)}{5} -1} \lVert f_{a}' \rVert_{H^{\gamma}} \left( \lVert \Theta_{2} \rVert_{\mathscr{C}^{2\alpha + \frac{5}{2} }} + \lVert \eta \rVert_{\mathscr{C}^{\alpha}}^{2} \right), \\
&\lVert \RomanIV_{3} \rVert_{H^{2\gamma - \frac{1}{4}}}  \lesssim a^{-\frac{2\alpha}{5}} \lVert f_{a}^{\sharp} \rVert_{H^{2\gamma - \frac{1}{4}}} \lVert \eta \rVert_{\mathscr{C}^{\alpha}}, \hspace{4mm}  \lVert \RomanIV_{4} \rVert_{H^{2\gamma - \frac{1}{4}}} \lesssim a^{\frac{2(\gamma - \frac{1}{4} - \alpha)}{5} -1} \lVert f_{a} \rVert_{H^{\gamma}} \lVert \eta \rVert_{\mathscr{C}^{\alpha}}, \label{Burgers' est 196b}\\
& \lVert \RomanIV_{5} \rVert_{H^{2\gamma - \frac{1}{4}}} \lesssim a^{\frac{ 2 (2\gamma - \frac{1}{4} + \delta)}{5} - 1}\lVert g \rVert_{H^{-\delta}}, \hspace{8mm}  \lVert \RomanIV_{6} \rVert_{H^{2\gamma - \frac{1}{4}}} \lesssim a^{ \frac{2}{5} (2 \gamma - \frac{11}{4} + \epsilon_{2} - \alpha)} \lVert f_{a} \rVert_{H^{\gamma}} \lVert \eta \rVert_{\mathscr{C}^{\alpha}}. \label{Burgers' est 196c}
\end{align} 
\end{subequations} 
Applying \eqref{Burgers' est 196} to  \eqref{Burgers' est 197}, and using \eqref{Burgers' Define tilde D norm} give us
\begin{align}
& a^{-\rho} \lVert M(f, f')_{a}^{\sharp} \rVert_{H^{2\gamma - \frac{1}{4}}}\lesssim \lVert (f, f' ) \rVert_{\tilde{\mathcal{D}}_{\eta}^{\gamma, \rho, A}} \Bigg[ a^{\frac{2(\gamma - \frac{1}{4} - \alpha)}{5} - 1 - \rho} ( \lVert \Theta_{2} \rVert_{\mathscr{C}^{2 \alpha + \frac{5}{2} }} +\lVert \eta \rVert_{\mathscr{C}^{\alpha}}^{2} )  \nonumber \\
& \hspace{30mm}  + a^{ \max \{ \frac{2( 2 \gamma - \frac{11}{4} + \epsilon_{2} - \alpha)}{5} - \rho,  -\frac{2\alpha}{5}  -1 \}} \lVert \eta \rVert_{\mathscr{C}^{\alpha}} \Bigg] + a^{ \frac{2(2\gamma - \frac{1}{4} + \delta)}{5} - 1-  \rho } \lVert g \rVert_{H^{-\delta}}. \label{Burgers' est 199}
\end{align}
Consequently, applying \eqref{Burgers' est 198} and \eqref{Burgers' est 199}, using \eqref{Burgers' Define rho} and  \eqref{Burgers' Define epsilon two}, and taking $\delta = 0$ for convenience leads us to 
\begin{align}
& \lVert \mathcal{M} (f, f') \rVert_{\tilde{\mathcal{D}}_{\eta}^{\gamma, \rho, A}} \lesssim  \sup_{a \geq A} \lVert f_{a} \rVert_{H^{\gamma}}  \nonumber\\
&\hspace{3mm}+  \sup_{a \geq A} a^{- \lambda} \lVert (f, f') \rVert_{\tilde{\mathcal{D}}_{\eta}^{\gamma, \rho, A}} ( 1+ \lVert \Theta \rVert_{\mathcal{E}^{\alpha}}^{2})  + \left( A^{ \frac{2}{5} (2\gamma - \frac{1}{4}) - 1 - \rho} + A^{\frac{2\gamma}{5} - 1} \right) \lVert g \rVert_{L^{2}},  \label{Burgers' est 200} 
\end{align}
where 
\begin{equation}\label{Burgers' Define lambda}
\lambda \triangleq \min \left\{ \rho - \frac{2}{5} \left(2 \gamma - \frac{11}{4} + \epsilon_{2} - \alpha \right), 1 - \frac{2}{5} (\gamma + \epsilon_{2} - \alpha), 1 + \frac{2\alpha}{5} - \rho \right\} > 0
\end{equation} 
and therefore we conclude that $\mathcal{M}(f,f') \in \tilde{\mathcal{D}}_{\eta}^{\gamma, \rho, A}$. Similarly to \eqref{Burgers' est 200}, we can show 
\begin{align}
 \lVert \mathcal{M} (f, f') - \mathcal{M} (h, h') \rVert_{\tilde{\mathcal{D}}_{\eta}^{\gamma, \rho, A}} \lesssim&  \sup_{a \geq A} \lVert f_{a} - h_{a} \rVert_{H^{\gamma}}  \label{Burgers' est 201} \\
&+  A^{- \lambda} \lVert (f, f') - (h, h') \rVert_{\tilde{\mathcal{D}}_{\eta}^{\gamma, \rho, A}} ( 1+ \lVert \Theta \rVert_{\mathcal{E}^{\alpha}}^{2}).  \nonumber 
\end{align}
We can make use of \eqref{Burgers' est 201} and analogous computations to \eqref{Burgers' est 198} to obtain 
\begin{align}
& \lVert \mathcal{M}^{2} (f, f') - \mathcal{M}^{2} (h, h') \rVert_{\tilde{\mathcal{D}}_{\eta}^{\gamma, \rho, A}} \nonumber \\
\lesssim& A^{- \lambda} \lVert (f, f') - (h, h') \rVert_{\tilde{\mathcal{D}}_{\eta}^{\gamma, \rho, A}} \Bigg[ 1+\lVert \Theta \rVert_{\mathcal{E}^{\alpha}}^{4} \Bigg] \ll \lVert (f, f') - (h, h') \rVert_{\tilde{\mathcal{D}}_{\eta}^{\gamma, \rho, A}}   \label{Burgers' est 203}
\end{align}
for 
\begin{equation}\label{Burgers' large A}
A \gg [1 + \lVert \Theta \rVert_{\mathcal{E}^{\alpha}}^{4} ]^{\frac{1}{ \lambda}} 
\end{equation} 
and therefore the mapping $\mathcal{M}^{2}: \hspace{1mm} \tilde{\mathcal{D}}_{\eta}^{\gamma, \rho, A} \mapsto \tilde{\mathcal{D}}_{\eta}^{\gamma, \rho, A}$ is a contraction. Consequently, the fixed point theorem gives us unique $(f, f') \in \tilde{\mathcal{D}}_{\eta}^{\gamma, \rho, A}$ such that $\mathcal{M}(f,f') = (f,f')$. Finally, making use of $M(f,f') = f$, $f = f'$, and computations that led to \eqref{Burgers' est 198} and \eqref{Burgers' est 199} lead to, for a universal constant $C \geq 0$, 
\begin{align}
&\lVert f_{a} \rVert_{H^{\gamma}} + a^{-\rho} \lVert f_{a}^{\sharp} \rVert_{H^{2\gamma - \frac{1}{4}}} \nonumber\\
\leq& \frac{1}{2} [\lVert f_{a} \rVert_{H^{\gamma}} + a^{-\rho} \lVert f_{a}^{\sharp} \rVert_{H^{2\gamma - \frac{1}{4}}}] + C \left( a^{\frac{2(\gamma + \delta)}{5} -1} + a^{-\rho + \frac{2}{5} (2\gamma - \frac{1}{4} + \delta) -1} \right) \lVert g \rVert_{H^{-\delta}}\label{Burgers' est 204}
\end{align} 
for all $A \geq 1$ sufficiently large. Subtracting $\frac{1}{2} [\lVert f_{a} \rVert_{H^{\gamma}} + a^{-\rho} \lVert f_{a}^{\sharp} \rVert_{H^{2\gamma - \frac{1}{4}}}]$ from both sides leads to \eqref{Burgers' est 202} as desired because $\mathcal{G}_{a} g = f_{a}$ by \eqref{Burgers' Define mathcal G}. This completes the proof of Proposition \ref{Burgers' Proposition 6.3}.
\end{proof}  

\begin{proposition}\label{Burgers' Proposition 6.4} 
Let $\alpha \in ( -\frac{19}{12}, -\frac{5}{4})$,  $\gamma \in (\frac{11}{12}, \alpha + \frac{5}{2})$ and define $\lambda$ by \eqref{Burgers' Define lambda}. Then there exists a constant $C > 0$ such that for all $\Theta = (\eta, \Theta_{2})$, $\tilde{\Theta} = (\tilde{\eta}, \tilde{\Theta}_{2}) \in \mathcal{K}^{\alpha}$, and $a \geq C[1+\lVert \Theta \rVert_{\mathcal{E}^{\alpha}}^{4}]^{\frac{1}{\lambda}}$ from \eqref{Burgers' large A}, we have the following bounds:
\begin{align}
\lVert \mathcal{G}_{a}(\Theta) g - \mathcal{G}_{a} ( \tilde{\Theta}) g \rVert_{H^{\gamma}} \leq& \lVert (\mathcal{G}_{a}(\Theta) g - \mathcal{G}_{a} (\tilde{\Theta}) g,\mathcal{G}_{a}(\Theta) g - \mathcal{G}_{a} (\tilde{\Theta}) g) \rVert_{\tilde{\mathcal{D}}_{\eta}^{\gamma, \rho, A}}  \nonumber\\
\lesssim& \lVert g \rVert_{L^{2}} \lVert \Theta - \tilde{\Theta} \rVert_{\mathcal{E}^{\alpha}} ( 1+ \lVert \Theta \rVert_{\mathcal{E}^{\alpha}} + \lVert \tilde{\Theta} \rVert_{\mathcal{E}^{\alpha}}),  \label{Burgers' est 267}
\end{align} 
where $\mathcal{G}_{a}(\Theta), \mathcal{G}_{a}(\tilde{\Theta}): \hspace{1mm} L^{2} \mapsto \mathcal{D}_{\eta}^{\gamma}$ are the resolvent operators associated to the rough distributions $\Theta, \tilde{\Theta} \in \mathcal{K}^{\alpha}$ constructed in Proposition \ref{Burgers' Proposition 6.3}. 
\end{proposition} 

\begin{proof}[Proof of Proposition \ref{Burgers' Proposition 6.4}]
We take $a \geq A( \lVert \Theta \rVert_{\mathcal{E}^{\alpha}}) + A ( \lVert \tilde{\Theta} \rVert_{\mathcal{E}^{\alpha}})$ according to \eqref{Burgers' large A} so that 
\begin{equation}\label{Burgers' est 220}
f_{a} \triangleq \mathcal{G}_{a} (\Theta) g \hspace{1mm} \text{ and } \hspace{1mm} \tilde{f}_{a} \triangleq \mathcal{G}_{a} (\tilde{\Theta}) g 
\end{equation} 
are well-defined by Proposition \ref{Burgers' Proposition 6.3}. We can verify from \eqref{Burgers' est 188}, \eqref{Burgers' Define H}, and \eqref{Burgers' Define tilde sigma a} that 
\begin{subequations}
\begin{align}
f_{a} - \tilde{f}_{a}=& \tilde{\sigma}_{a}(D) [ f_{a} \prec \eta - \tilde{f}_{a} \prec \tilde{\eta} + f_{a} \succ \eta - \tilde{f}_{a} \succ \tilde{\eta} + f_{a} \Theta_{2} - \tilde{f}_{a} \tilde{\Theta}_{2} + f_{a}^{\sharp} \circ \eta - \tilde{f}_{a}^{\sharp} \circ \tilde{\eta} ]  \nonumber \\
& \hspace{20mm} + \tilde{\sigma}_{a}(D) [ \mathcal{R} (f_{a}, \sigma(D) \eta, \eta) - \mathcal{R} (\tilde{f}_{a}, \sigma(D) \tilde{\eta}, \tilde{\eta} ) ],  \label{Burgers' est 205a} \\
f_{a}^{\sharp} - \tilde{f}_{a}^{\sharp} =&  \tilde{\sigma}_{a}(D) [ f_{a} \succ \eta - \tilde{f}_{a} \succ \tilde{\eta}  + f_{a} \Theta_{2} - \tilde{f}_{a} \tilde{\Theta}_{2} + f_{a}^{\sharp} \circ \eta - \tilde{f}_{a}^{\sharp} \circ \eta]  \nonumber \\
&+ \tilde{\sigma}_{a}(D) [ \mathcal{R} (f_{a}, \sigma(D) \eta, \eta) - \mathcal{R} (\tilde{f}_{a}, \sigma(D) \tilde{\eta}, \tilde{\eta} ) ]  \nonumber \\
&+\mathscr{C}_{a} (f_{a}, \eta) + f_{a} \prec (\tilde{\sigma}_{a} - \sigma) (D) \eta  - \mathscr{C}_{a} (\tilde{f}_{a}, \tilde{\eta})- \tilde{f}_{a} \prec (\tilde{\sigma}_{a} - \sigma) (D) \tilde{\eta}, \label{Burgers' est 205b}  
\end{align}
\end{subequations} 
where $\mathscr{C}_{a}(f_{a}, \eta)$ was defined in \eqref{Burgers' Define mathcal Ca}. With the same $\rho$ from \eqref{Burgers' Define rho}, because $\gamma < \alpha + \frac{5}{2}$ by hypothesis, we can find the same $\epsilon_{2}$ in \eqref{Burgers' Define epsilon two} and estimate similarly to \eqref{Burgers' est 196a} with $\vartheta = \frac{2(\gamma + \epsilon_{2} - \alpha)}{5} \in [0,1]$ in \eqref{Burgers' est 195}, 
\begin{subequations}\label{Burgers' est 206}
\begin{align}
&a^{- \rho} \lVert \mathscr{C}_{a} (f_{a} - \tilde{f}_{a},\eta) \rVert_{H^{2\gamma - \frac{1}{4}}}  \lesssim  a^{-( \rho  + 1 - \frac{2}{5} (\gamma + \epsilon_{2} - \alpha) )} \lVert f_{a} - \tilde{f}_{a} \rVert_{H^{\gamma}} \lVert \eta \rVert_{\mathscr{C}^{\alpha}},  \label{Burgers' est 206a}\\
&a^{-\rho} \lVert \mathscr{C}_{a} (\tilde{f}_{a}, \eta - \tilde{\eta}) \rVert_{H^{2\gamma - \frac{1}{4}}} \lesssim a^{- (\rho + 1 - \frac{2}{5} (\gamma + \epsilon_{2} - \alpha))} \lVert \tilde{f}_{a} \rVert_{H^{\gamma}} \lVert \eta - \tilde{\eta} \rVert_{\mathscr{C}^{\alpha}}. \label{Burgers' est 206b} 
\end{align}
\end{subequations} 
Additionally, as $\rho >  \frac{2}{5} (2 \gamma - \alpha - \frac{11}{4})$ from \eqref{Burgers' Define rho}, we can find 
\begin{equation}\label{Burgers' Define epsilon three}
\epsilon_{3} \in \left(0, \rho - \frac{2}{5} \left(2 \gamma - \alpha - \frac{11}{4} \right)\right)
\end{equation} 
and estimate by \eqref{Sobolev products c} and \eqref{Burgers' Define tilde sigma a} with $\vartheta = \frac{2}{5} (2\gamma - \alpha - \frac{11}{4}) + \epsilon_{3} \in [0,1]$, 
\begin{subequations}\label{Burgers' est 207} 
\begin{align}
&a^{-\rho}\lVert ( f_{a} - \tilde{f}_{a}) \prec (\tilde{\sigma}_{a} - \sigma) (D) \eta \rVert_{H^{2\gamma - \frac{1}{4}}} \lesssim  a^{-( \rho - \frac{2}{5} (2\gamma - \alpha - \frac{11}{4}) - \epsilon_{3})} \lVert f_{a} - \tilde{f}_{a} \rVert_{H^{\gamma}} \lVert \eta \rVert_{\mathscr{C}^{\alpha}}, \label{Burgers' est 207a}  \\
& a^{-\rho} \lVert \tilde{f}_{a} \prec ( \tilde{\sigma}_{a} - \sigma) (D) (\eta - \tilde{\eta} ) \rVert_{H^{2\gamma - \frac{1}{4}}} \lesssim a^{-( \rho - \frac{2}{5} (2\gamma - \alpha - \frac{11}{4})- \epsilon_{3})} \lVert \tilde{f}_{a} \rVert_{H^{\gamma}} \lVert \eta - \tilde{\eta} \rVert_{\mathscr{C}^{\alpha}}.  \label{Burgers' est 207b}
\end{align} 
\end{subequations} 
Applying \eqref{Burgers' est 206} and \eqref{Burgers' est 207}, and making use of $\rho - \frac{2}{5} ( 2\gamma - \alpha - \frac{11}{4}) - \epsilon_{3} > 0$ due to \eqref{Burgers' Define epsilon three} lead to 
\begin{align}
a^{-\rho} \lVert f_{a}^{\sharp} - \tilde{f}_{a}^{\sharp} \rVert_{H^{2\gamma - \frac{1}{4}}} \lesssim& a^{\frac{4\gamma}{5} - \frac{11}{10}} \lVert g \rVert_{L^{2}} \lVert \Theta - \tilde{\Theta} \rVert_{\mathcal{E}^{\alpha}}  \nonumber \\
&+ a^{-(\rho - \frac{2}{5} (2\gamma - \alpha - \frac{11}{4}) - \epsilon_{3})} \lVert f_{a} - \tilde{f}_{a} \rVert_{H^{\gamma}} (1+ \lVert \Theta \rVert_{\mathcal{E}^{\alpha}})^{2}. \label{Burgers' est 222}
\end{align} 
Next, with the same $\epsilon_{2}$ from \eqref{Burgers' Define epsilon two}, due to \eqref{Burgers' Schauder} and \eqref{Sobolev products c}, we can estimate  
\begin{equation}\label{Burgers' est 223}
\lVert \tilde{\sigma}_{a}(D) [ (f_{a} - \tilde{f}_{a}) \prec \eta ] \rVert_{H^{\gamma}} \lesssim a^{\frac{2(\gamma + \epsilon_{2} - \alpha)}{5} - 1} \lVert f_{a} - \tilde{f}_{a} \rVert_{H^{\gamma}} \lVert \eta \rVert_{\mathscr{C}^{\alpha}}, 
\end{equation} 
which leads to 
\begin{align}
\lVert f_{a} - \tilde{f}_{a} \rVert_{H^{\gamma}}   \lesssim& a^{\frac{2}{5} (\gamma + \epsilon_{2} - \alpha) - 1} \left(a^{\frac{2\gamma}{5} -1} + a^{-\rho + \frac{2}{5} (2\gamma - \frac{1}{4}) - 1} \right) \lVert g \rVert_{L^{2}} \lVert \Theta - \tilde{\Theta} \rVert_{\mathcal{E}^{\alpha}} (1+ \lVert \Theta \rVert_{\mathcal{E}^{\alpha}} + \lVert \tilde{\Theta} \rVert_{\mathcal{E}^{\alpha}})  \nonumber \\
&+ a^{\frac{2}{5} (-\gamma + \frac{1}{4} - \alpha) - 1} \lVert f_{a}^{\sharp} - \tilde{f}_{a}^{\sharp} \rVert_{H^{2\gamma - \frac{1}{4}}} \lVert \eta \rVert_{\mathscr{C}^{\alpha}}  \nonumber \\
&+ a^{\frac{2}{5} (-\gamma + \frac{1}{4} - \alpha) - 1}  \left(a^{\rho + \frac{2\gamma}{5} -1} + a^{\frac{2}{5} (2\gamma - \frac{1}{4}) -1} \right)\lVert g \rVert_{L^{2}} \lVert \Theta - \tilde{\Theta} \rVert_{\mathcal{E}^{\alpha}}. \label{Burgers' est 234} 
\end{align}
Summing \eqref{Burgers' est 234} and \eqref{Burgers' est 222}, making use of $\rho - \frac{2}{5} (2\gamma - \alpha - \frac{11}{4}) - \epsilon_{3} > 0$ from \eqref{Burgers' Define epsilon three}, we obtain
\begin{align}
&\lVert f_{a} - \tilde{f}_{a} \rVert_{H^{\gamma}} + a^{-\rho} \lVert f_{a}^{\sharp} - \tilde{f}_{a}^{\sharp} \rVert_{H^{2\gamma - \frac{1}{4}}} \nonumber\\
\lesssim& a^{\frac{4\gamma}{5} - \frac{11}{10}} \lVert g \rVert_{L^{2}} \lVert \Theta- \tilde{\Theta} \rVert_{\mathcal{E}^{\alpha}}( 1+ \lVert \Theta \rVert_{\mathcal{E}^{\alpha}} +\lVert \tilde{\Theta} \rVert_{\mathcal{E}^{\alpha}}), \label{Burgers' est 235}
\end{align}
which allows us to conclude \eqref{Burgers' est 267}. This completes the proof of Proposition \ref{Burgers' Proposition 6.4}. 
\end{proof} 

In Definition \ref{Burgers' Definition 6.3}, we had $\alpha \in (-\frac{19}{12}, -\frac{5}{4})$, $\gamma \in (\frac{1}{8} -\frac{\alpha}{2}, \alpha + \frac{5}{2}]$. We now tighten this to $\alpha > -\frac{17}{12}$ and $\gamma > \frac{\alpha}{4} + \frac{21}{16}$ while keeping the same upper bounds. 

\begin{define}\label{Burgers' Definition 6.4}
Let $\alpha \in (-\frac{17}{12}, -\frac{5}{4}), \gamma \in (\frac{\alpha}{4} + \frac{21}{16}, \alpha + \frac{5}{2})$, and $\Theta = (\eta, \Theta_{2}) \in \mathcal{K}^{\alpha}$. We define a bilinear operator $B: \hspace{1mm} H^{\gamma} \times \mathcal{K}^{\alpha} \mapsto H^{2\gamma - \frac{1}{4}}$ by 
\begin{align}
B(f,\Theta) \triangleq& \sigma(D) [ \Lambda^{\frac{5}{2}} ( f \prec \sigma(D) \eta) - \Lambda^{\frac{5}{2}} f \prec \sigma(D) \eta - f \prec \Lambda^{\frac{5}{2}} \sigma(D) \eta \nonumber \\
& \hspace{20mm}  - (1-\Lambda^{\frac{5}{2}}) f \prec \sigma(D) \eta + f \succ \eta + f \Theta_{2} ]. \label{Burgers' Define B}
\end{align}  
Then we define 
\begin{equation}\label{Burgers' Define D Theta gamma}
\mathcal{D}_{\Theta}^{\gamma} \triangleq \{ f \in H^{\gamma}: \hspace{1mm} f^{\flat} \triangleq f - f \prec \sigma(D) \eta -B(f, \Theta) \in H^{\frac{5}{2}} \} 
\end{equation} 
with an inner product 
\begin{equation}\label{Burgers' Norm D Theta gamma}
\langle f, g \rangle_{\mathcal{D}_{\Theta}^{\gamma}} \triangleq \langle f, g \rangle_{H^{\gamma}} + \langle f^{\flat}, g^{\flat} \rangle_{H^{\frac{5}{2}}}. 
\end{equation} 
\end{define} 

\begin{remark}
We see the commutator 
\begin{equation*}
\Lambda^{\frac{5}{2}} ( f \prec \sigma(D) \eta) - \Lambda^{\frac{5}{2}} f \prec \sigma(D) \eta - f \prec \Lambda^{\frac{5}{2}} \sigma(D) \eta
\end{equation*}
within the definition of $B(f, \Theta)$. For completeness, let us describe the motivation behind these definitions, justification of their regularity through an \emph{a priori} estimate that will be useful subsequently. First, using \eqref{Burgers' Define D eta gamma} and \eqref{Burgers' Define K alpha}, one can verify that 
\begin{align}
\Lambda^{\frac{5}{2}} f=& \Lambda^{\frac{5}{2}} f^{\sharp} + \Lambda^{\frac{5}{2}} \left( f \prec \sigma(D) \eta \right) - \Lambda^{\frac{5}{2}} f \prec \sigma(D) \eta - f \prec \Lambda^{\frac{5}{2}} \sigma(D) \eta \nonumber\\
&\hspace{10mm} - (1- \Lambda^{\frac{5}{2}}) f \prec \sigma(D) \eta - f \prec \eta \label{Burgers' est 256} 
\end{align}
and consequently, relying on \eqref{Burgers' Define H} and \eqref{Burgers' est 165} shows 
\begin{align}
\mathcal{H} f \overset{\eqref{Burgers' Define H}}{=}&  -\Lambda^{\frac{5}{2}} f^{\sharp} - \Lambda^{\frac{5}{2}} \left( f \prec \sigma(D) \eta \right) + \Lambda^{\frac{5}{2}} f \prec \sigma(D) \eta + f \prec \Lambda^{\frac{5}{2}} \sigma(D) \eta  \nonumber \\
& \hspace{10mm} + (1- \Lambda^{\frac{5}{2}} ) f \prec \sigma(D) \eta - f \succ \eta - f \Theta_{2} - \mathcal{R} (f, \sigma(D) \eta, \eta) - f^{\sharp} \circ \eta.  \label{est 271} 
\end{align} 
Considering that $f^{\sharp} \in H^{2\gamma - \frac{1}{4}}$ from \eqref{Burgers' Define D eta gamma} if $f \in \mathcal{D}_{\eta}^{\gamma}$, we claim that $\mathcal{H}f \in H^{2\gamma - \frac{11}{4}}$. Once again, the most non-trivial term is the commutator 
\begin{equation*}
- \Lambda^{\frac{5}{2}} \left( f \prec \sigma(D) \eta \right) + \Lambda^{\frac{5}{2}} f \prec \sigma(D) \eta + f \prec \Lambda^{\frac{5}{2}} \sigma(D) \eta 
\end{equation*}
in \eqref{est 271}. We proceed with this estimate similarly as before (recall \eqref{Split C4}-\eqref{Estimate on C4} and \eqref{est 126}-\eqref{est 131}). First,  
\begin{equation}\label{est 273}
\lVert \Lambda^{\frac{5}{2}} f \prec \sigma(D) \eta \rVert_{H^{2\gamma - \frac{11}{4}}} \overset{\eqref{Sobolev products c}}{\lesssim} \lVert \Lambda^{\frac{5}{2}} f \rVert_{H^{\gamma - \frac{5}{2}}} \lVert \sigma(D) \eta \rVert_{\mathscr{C}^{\gamma - \frac{1}{4}}} \lesssim \lVert f \rVert_{H^{\gamma}} \lVert \eta \rVert_{\mathscr{C}^{\alpha}}.
\end{equation} 
Next, relying on $N_{1}$ from \eqref{est 152}, for $\epsilon \in (0, \frac{3}{2})$ to be optimized subsequently, we can estimate 
\begin{align}
& \lVert \Lambda^{\frac{5}{2}} \left( f \prec \sigma(D) \eta \right) - f \prec \Lambda^{\frac{5}{2}} \sigma(D) \eta \rVert_{H^{2\gamma - \frac{11}{4}}}^{2} \nonumber\\
=& \sum_{j\geq -1} 2^{2j(2\gamma - \frac{11}{4})} \lVert \sum_{r: \lvert r-j \rvert \leq N_{1}} \Lambda^{\frac{5}{2}} \left( S_{r-1} f \Delta_{r} \sigma(D) \eta \right) - S_{r-1} f \Delta_{r} \Lambda^{\frac{5}{2}} \sigma(D) \eta\rVert_{L^{2}}^{2}  \nonumber\\
\overset{\eqref{Kato-Ponce}}{\lesssim}& \sum_{j\geq -1} 2^{2j(2 \gamma - \frac{11}{4})} \left[ \lVert \Lambda^{\frac{5}{2}} S_{j-1} f \rVert_{L^{2}} \lVert \Delta_{j} \sigma(D) \eta \rVert_{L^{\infty}} + \lVert \nabla S_{j-1} f \rVert_{L^{\frac{6}{3-2\epsilon}}} \lVert \Delta_{j} \Lambda^{\frac{3}{2}} \sigma(D) \eta \rVert_{L^{\frac{3}{\epsilon}}} \right]^{2}  \nonumber \\
\lesssim& \lVert \eta \rVert_{\mathscr{C}^{\alpha}}^{2} \Bigg[ \sum_{j\geq -1} 2^{2j(\gamma - \frac{11}{4} - \alpha)} \left( 2^{-\lvert j \rvert (\frac{5}{2} - \gamma)} \ast 2^{j\gamma} \lVert \Delta_{j} f \rVert_{L^{2}} \right)^{2} \nonumber\\
& \hspace{20mm} + \sum_{j\geq -1} 2^{2j( \gamma - \frac{11}{4} - \alpha + \epsilon)} \left( 2^{- \lvert j \rvert (1+ \epsilon - \gamma)} \ast 2^{j\gamma} \lVert \Delta_{j} f \rVert_{L^{2}} \right)^{2} \Bigg]. \label{est 274}
\end{align}
We see that if $\epsilon \in (0, \frac{3}{2})$ additionally satisfies 
\begin{equation}\label{est 275}
\gamma - 1 < \epsilon \leq - \gamma + \frac{11}{4} + \alpha, 
\end{equation} 
then we can apply H$\ddot{o}$lder's inequality and Young's inequality for convolution on \eqref{est 274} and conclude that 
\begin{equation}\label{est 277}
 \lVert \Lambda^{\frac{5}{2}} \left( f \prec \sigma(D) \eta \right) - f \prec \Lambda^{\frac{5}{2}} \sigma(D) \eta \rVert_{H^{2\gamma - \frac{11}{4}}}^{2}  \lesssim \lVert \eta \rVert_{\mathscr{C}^{\alpha}}^{2} \lVert f \rVert_{H^{\gamma}}^{2} 
\end{equation} 
and hence deduce together with \eqref{est 273} that 
\begin{equation}\label{est 276}
\lVert  \Lambda^{\frac{5}{2}} \left( f \prec \sigma(D) \eta \right) - \Lambda^{\frac{5}{2}} f \prec \sigma(D) \eta - f \prec \Lambda^{\frac{5}{2}} \sigma(D) \eta \rVert_{H^{2\gamma - \frac{11}{4}}} \lesssim \lVert f \rVert_{H^{\gamma}} \lVert \eta \rVert_{\mathscr{C}^{\alpha}}. 
\end{equation} 
One can verify that $\epsilon = \frac{7}{8} + \frac{\alpha}{2}$ satisfies \eqref{est 275}; thus, we now conclude \eqref{est 276}. 

We estimated $- \Lambda^{\frac{5}{2}} \left( f \prec \sigma(D) \eta \right) + \Lambda^{\frac{5}{2}} f \prec \sigma(D) \eta + f \prec \Lambda^{\frac{5}{2}} \sigma(D) \eta$ in \eqref{est 271} and we can estimate the rest as follows:
\begin{subequations}\label{est 278}
\begin{align}
& \lVert (1- \Lambda^{\frac{5}{2}}) f \prec \sigma(D) \eta \rVert_{H^{2\gamma - \frac{11}{4}}} \overset{\eqref{Sobolev products c}}{\lesssim}   \lVert f \rVert_{H^{\gamma}} \lVert \eta \rVert_{\mathscr{C}^{\alpha}}, \label{est 278a} \\
&  \lVert f \succ \eta \rVert_{H^{2\gamma - \frac{11}{4}}} \overset{\eqref{Sobolev products d}}{\lesssim} \lVert f \rVert_{H^{\gamma}} \lVert \eta \rVert_{\mathscr{C}^{\alpha}},\label{est 278b} \\
&  \lVert f \Theta_{2} \rVert_{H^{2\gamma - \frac{11}{4}}}\lesssim  \lVert f \prec \Theta_{2} \rVert_{H^{2\gamma - \frac{11}{4}}}  \nonumber \\
& \hspace{15mm} + \lVert f \succ \Theta_{2} \rVert_{H^{2\gamma - \frac{11}{4}}} + \lVert f \circ \Theta_{2} \rVert_{H^{\gamma + 2 \alpha + \frac{5}{2}}} \overset{\eqref{Sobolev products c} \eqref{Sobolev products d} \eqref{Sobolev products e}}{\lesssim} \lVert f \rVert_{H^{\gamma}} \lVert \Theta_{2} \rVert_{\mathscr{C}^{2\alpha + \frac{5}{2}}}, \label{est 278c}\\ 
& \lVert \mathcal{R} (f, \sigma(D) \eta, \eta) \rVert_{H^{2\gamma - \frac{11}{4}}}  \overset{\eqref{Burgers' Estimate on R}}{\lesssim} \lVert f \rVert_{H^{\frac{1}{8} - \frac{\alpha}{2}}} \lVert \sigma(D) \eta \rVert_{\mathscr{C}^{\alpha + \frac{5}{2}}} \lVert \eta \rVert_{\mathscr{C}^{\alpha}} \lesssim \lVert f \rVert_{H^{\gamma}} \lVert \eta \rVert_{\mathscr{C}^{\alpha}}^{2},  \label{est 278d}\\
& \lVert f^{\sharp} \circ \eta \rVert_{H^{2\gamma - \frac{11}{4}}} \lesssim \lVert f^{\sharp} \circ \eta \rVert_{H^{2\gamma - \frac{1}{4} + \alpha}} 
\overset{\eqref{Sobolev products e}}{\lesssim} \lVert f^{\sharp} \rVert_{H^{2\gamma - \frac{1}{4}}} \lVert \eta \rVert_{\mathscr{C}^{\alpha}}, \label{est 278e}
\end{align}
\end{subequations} 
and conclude that $\mathcal{H} f \in H^{2\gamma - \frac{11}{4}}$. If we assume that $\mathcal{H}f = \lambda f$ for some scalar $\lambda$ and define 
\begin{equation}\label{est 279} 
f^{\flat} \triangleq \sigma(D) [ - f^{\sharp} + \lambda f + \mathcal{R} (f, \sigma(D) \eta, \eta) + f^{\sharp} \circ \eta], 
\end{equation} 
then one can show using \eqref{est 271}, \eqref{Burgers' Define K alpha} and the definition of $B(f, \Theta)$ from \eqref{Burgers' Define B} that 
\begin{equation}\label{Burgers' est 245} 
f^{\sharp} = B(f,\Theta) + f^{\flat}.
\end{equation} 
Comparing with $f^{\sharp} = f - f \prec \sigma(D) \eta$ from \eqref{Burgers' Define D eta gamma}, this justifies the definition $f^{\flat} \triangleq f - f \prec \sigma(D) \eta -B(f, \Theta)$ in \eqref{Burgers' Define D Theta gamma}. Moreover, one can directly verify that $f^{\flat} \in H^{2\gamma + \frac{9}{4} + \alpha}$ which justifies the regularity $f^{\flat} \in H^{\frac{5}{2}}$ in \eqref{Burgers' Define D Theta gamma}.  At last, combining \eqref{est 276}, \eqref{est 278a}, \eqref{est 278b}, and \eqref{est 278c} justifies the fact that $B$ is $H^{2\gamma - \frac{1}{4}}$-valued in harmony with Definition \ref{Burgers' Definition 6.4}. 
\end{remark} 

\begin{remark}
We observe that $\mathcal{D}_{\Theta}^{\gamma} \subset \mathcal{D}_{\eta}^{\gamma}$ and that any $f \in \mathcal{D}_{\Theta}^{\gamma}$ has the regularity of $H^{(\alpha + \frac{5}{2})-}$ which motivates us to define for any $\alpha \in \left( - \frac{17}{12}, -\frac{5}{4} \right)$, $\Theta = ( \eta, \Theta_{2} ) \in \mathcal{K}^{\alpha}$, and 
\begin{equation}\label{Burgers' Define new kappa}
\kappa \in \left( 0, \frac{1}{4} \right) 
\end{equation} 
sufficiently small fixed, 
\begin{equation}\label{Burgers' Define D Theta}
\mathcal{D}_{\Theta} \triangleq \{ f \in H^{\alpha + \frac{5}{2} - \kappa}: \hspace{1mm} f^{\flat} \triangleq f - f \prec \sigma(D) \eta - B(f, \Theta) \in H^{\frac{5}{2}} \}. 
\end{equation} 
\end{remark} 

\begin{proposition}\label{Burgers' Proposition 6.5}
Let $\alpha \in (-\frac{17}{12}, -\frac{5}{4})$ and $\gamma \in (1, \alpha + \frac{5}{2})$. Then, for any $a \geq 2$, we define 
\begin{align}
B_{a} (f, \Theta) \triangleq& \sigma_{a}(D) [ \Lambda^{\frac{5}{2}} \left( f \prec \sigma(D) \eta \right) - \Lambda^{\frac{5}{2}} f \prec \sigma(D) \eta - f \prec \Lambda^{\frac{5}{2}} \sigma(D) \eta \nonumber \\
& \hspace{20mm}  - (1-\Lambda^{\frac{5}{2}}) f \prec \sigma(D) \eta + f \succ \eta + f \Theta_{2} ], \label{Burgers' Define Ba}
\end{align} 
where $\sigma_{a}$ was defined in \eqref{Burgers' Define sigma a}. Then 
\begin{equation}\label{Burgers' est 240} 
\lVert B_{a} (f, \Theta) \rVert_{H^{2\gamma - \frac{1}{4}}} \lesssim  \lVert f \rVert_{H^{\gamma}} \lVert \Theta \rVert_{\mathcal{E}^{\alpha}}, \hspace{5mm} \lVert (B- B_{a}) (f, \Theta) \rVert_{H^{2 \gamma + \frac{9}{4}}} \lesssim a  \lVert f \rVert_{H^{\gamma}} \lVert \Theta \rVert_{\mathcal{E}^{\alpha}}. 
\end{equation} 
\end{proposition}

\begin{proof}[Proof of Proposition \ref{Burgers' Proposition 6.5}]
Applying the two inequalities in \eqref{Burgers' est 172} with $\vartheta = 1$ and relying on \eqref{est 276}, \eqref{est 278a}, \eqref{est 278b}, and \eqref{est 278c} give us the desired result \eqref{Burgers' est 240}. 
\end{proof}  

\begin{proposition}\label{Burgers' Proposition 6.6}
Let $\alpha \in (-\frac{17}{12}, -\frac{5}{4})$, $\Theta = ( \eta, \Theta_{2} ) \in \mathcal{K}^{\alpha}$, and $\kappa$ from \eqref{Burgers' Define new kappa} sufficiently small fixed. If $f \in \mathcal{D}_{\Theta}$ where $\mathcal{D}_{\Theta}$ is defined in \eqref{Burgers' Define D Theta}, then $\mathcal{H} f \in L^{2}$. 
\end{proposition}

\begin{proof}[Proof of Proposition \ref{Burgers' Proposition 6.6}]
The claim follows by writing 
\begin{equation}\label{Burgers' est 251}
\mathcal{H} f = -\Lambda^{\frac{5}{2}} f^{\flat} +  B(f, \Theta) - \mathcal{R} (f, \sigma(D) \eta, \eta) - \eta \circ \left(B(f, \Theta) + f^{\flat} \right)
\end{equation} 
and making use of \eqref{Burgers' Define D Theta}, \eqref{Burgers' Define B}, \eqref{Burgers' Estimate on R}, and \eqref{Sobolev products e} to verify that each term is in $L^{2}$. 
\end{proof} 

\begin{proposition}\label{Burgers' Proposition 6.7}
Let $\alpha \in (-\frac{17}{12}, -\frac{5}{4}), \gamma \in (\frac{\alpha}{4} + \frac{21}{16}, \alpha + \frac{5}{2})$, and $\Theta = (\eta, \Theta_{2}), \tilde{\Theta} = (\tilde{\eta}, \tilde{\Theta}_{2}) \in \mathcal{K}^{\alpha}$. Then, for all $f \in \mathcal{D}_{\Theta}$, there exists $g \in \mathcal{D}_{\tilde{\Theta}}$ such that 
\begin{equation}\label{Burgers' est 250} 
\lVert f-g \rVert_{H^{\gamma}} + \lVert f^{\flat} - g^{\flat} \rVert_{H^{\frac{5}{2}}} \lesssim ( \lVert f \rVert_{H^{\gamma}} + \lVert g \rVert_{H^{\gamma}} ) (1+ \lVert \tilde{\Theta} \rVert_{\mathcal{E}^{\alpha}}) \lVert \Theta - \tilde{\Theta} \rVert_{\mathcal{E}^{\alpha}}, 
\end{equation} 
where 
\begin{equation}\label{Burgers' Define g flat}
g^{\flat} \triangleq g - g \prec \sigma(D) \tilde{\eta} - B(g, \tilde{\Theta}). 
\end{equation} 
In particular, if $(\eta_{n}, c_{n})_{n\in\mathbb{N}} \subset C^{\infty} \times \mathbb{R}$ is a family such that  $(\eta_{n}, -\eta_{n} \circ \sigma(D) \eta_{n} - c_{n}) \to \Theta$ in $\mathcal{E}^{\alpha}$ as $n \nearrow 0$, then there exists a family $\{f_{n} \}_{n \in\mathbb{N}} \subset H^{\frac{5}{2} }$ such that  
\begin{equation}\label{Burgers' est 254} 
\lim_{n \nearrow \infty} \lVert f_{n} - f \rVert_{H^{\gamma}} + \lVert f_{n}^{\flat} - f^{\flat} \rVert_{H^{\frac{5}{2}}} = 0, 
\end{equation} 
where $f_{n}^{\flat} \triangleq f_{n} - f_{n}\prec \sigma(D) \eta_{n} - B(f_{n}, \Theta^{n})$ and $\Theta^{n} \triangleq (\eta_{n}, -\eta_{n} \circ \sigma(D) \eta_{n} - c_{n})$. 
\end{proposition} 

\begin{proof}[Proof of Proposition \ref{Burgers' Proposition 6.7}]
Let $f \in \mathcal{D}_{\Theta}$ and define $\Gamma: \hspace{1mm} H^{\gamma} \mapsto H^{\gamma}$ by 
\begin{equation}\label{Burgers' est 175}
\Gamma(g) \triangleq g \prec \sigma_{a}(D) \tilde{\eta} + B_{a} (g, \tilde{\Theta}) + f - [f \prec \sigma_{a}(D) \eta+ B_{a} (f, \Theta) ], 
\end{equation} 
where $B_{a}(g, \tilde{\Theta})$ and $B_{a} (f, \Theta)$ are defined according to \eqref{Burgers' Define Ba}. Then, for all $g_{1}, g_{2} \in H^{\gamma}$, we can estimate using \eqref{Burgers' est 175} and \eqref{Burgers' Define Ba} 
\begin{align}  
 \lVert \Gamma(g_{1}) - \Gamma(g_{2}) \rVert_{H^{\gamma}} \leq& \lVert (g_{1} - g_{2}) \prec \sigma_{a}(D) \tilde{\eta} \rVert_{H^{\gamma}} + \lVert \sigma_{a}(D) \left( \Lambda^{\frac{5}{2}} (g_{1} - g_{2}) \prec \sigma(D) \tilde{\eta} \right) \rVert_{H^{\gamma}}  \nonumber\\
&+ \lVert \sigma_{a}(D) \left( \Lambda^{\frac{5}{2}} \left( ( g_{1} - g_{2}) \prec \sigma(D)\tilde{\eta} \right) - (g_{1} - g_{2}) \prec \Lambda^{\frac{5}{2}} \sigma(D) \tilde{\eta} \right) \rVert_{H^{\gamma}}  \nonumber\\
&+ \lVert \sigma_{a}(D) \left( (1- \Lambda^{\frac{5}{2}}) (g_{1} - g_{2})  \prec \sigma(D) \tilde{\eta} \right) \rVert_{H^{\gamma}} + \lVert \sigma_{a}(D) \left( (g_{1} - g_{2}) \succ \tilde{\eta} \right) \rVert_{H^{\gamma}} \nonumber \\
&+ \lVert \sigma_{a}(D) \left( (g_{1} - g_{2}) \tilde{\Theta}_{2} \right) \rVert_{H^{\gamma}}. \label{Burgers' est 176}
\end{align}
Because $\gamma < \alpha + \frac{5}{2}$ by hypothesis, we can find $\epsilon_{2} > 0$ that satisfies \eqref{Burgers' Define epsilon two} and rely on \eqref{Sobolev products a} and  \eqref{Burgers' est 172} to deduce 
\begin{equation}\label{Burgers' est 177}
\lVert (g_{1} - g_{2}) \prec \sigma_{a}(D) \tilde{\eta} \rVert_{H^{\gamma}} \lesssim a^{\frac{ 2(\gamma + \epsilon_{2} - \alpha)}{5} -1} \lVert \tilde{\Theta} \rVert_{\mathcal{E}^{\alpha}} \lVert g_{1} - g_{2} \rVert_{H^{\gamma}}.
\end{equation} 
For the rest of the terms in \eqref{Burgers' est 176}, we can rely on \eqref{est 273}, \eqref{est 277}, \eqref{est 278a}, \eqref{est 278b}, and \eqref{est 278c} and deduce 
\begin{equation}\label{Burgers' est 179}
 \lVert \Gamma (g_{1}) - \Gamma(g_{2}) \rVert_{H^{\gamma}} \lesssim  a^{\frac{2(\gamma + \epsilon_{2} - \alpha)}{5} -1} \lVert g_{1} - g_{2}  \rVert_{H^{\gamma}}.
\end{equation} 
Considering \eqref{Burgers' Define epsilon two}, we deduce that $\Gamma$ is a contraction for $a \gg 1$ so that there exists a unique $g$ such that 
\begin{subequations}\label{Burgers' est 180}
\begin{align}
g \overset{\eqref{Burgers' est 175}}{=}& g \prec \sigma_{a}(D) \tilde{\eta} + B_{a} (g, \tilde{\Theta}) + f - [f \prec \sigma_{a}(D) \eta+ B_{a} (f, \Theta) ] \label{Burgers' est 180a}\\
\overset{\eqref{Burgers' Define D Theta}}{=}&g \prec \sigma_{a}(D) \tilde{\eta} + B_{a} (g, \tilde{\Theta}) + f^{\flat} + f \prec ( \sigma - \sigma_{a})(D) \eta + B(f, \Theta) - B_{a}(f, \Theta). \label{Burgers' est 180b}
\end{align}
\end{subequations} 
Next, for any $\kappa > 0$, we can estimate 
\begin{equation}\label{est 281} 
\lVert g \prec \sigma_{a}(D) \tilde{\eta} \rVert_{H^{\frac{5}{2}+ \alpha - \kappa}} \overset{\eqref{Sobolev products a}}{\lesssim} \lVert g \rVert_{L^{2}} \lVert \sigma_{a}(D) \tilde{\eta} \rVert_{\mathscr{C}^{\frac{5}{2}+ \alpha}} \lesssim \lVert g \rVert_{H^{\gamma}} \lVert \tilde{\eta} \rVert_{\mathscr{C}^{\alpha}} \lesssim 1, 
\end{equation} 
and deduce $g \in H^{\frac{5}{2} + \alpha - \kappa}$ by relying on \eqref{Burgers' est 180b}. Additionally, we can write using \eqref{Burgers' Define g flat} and \eqref{Burgers' est 180b}
\begin{align}
g^{\flat} =& f^{\flat} + g \prec (\sigma_{a} - \sigma)(D) \tilde{\eta} + f \prec (\sigma - \sigma_{a}) (D) \eta \nonumber \\
&+ B(f, \Theta) - B_{a}(f, \Theta) + B_{a} (g, \tilde{\Theta}) - B(g, \tilde{\Theta}),  \label{Burgers' est 246} 
\end{align}
and show that $g^{\flat} \in H^{\frac{5}{2}}$; particularly, the estimates \eqref{est 276}, \eqref{est 278a}, \eqref{est 278b}, and \eqref{est 278c} can be used to estimate $\lVert B(f, \Theta) - B_{a}(f, \Theta) \rVert_{H^{\frac{5}{2}}}$. This allows us to conclude that $g \in \mathcal{D}_{\tilde{\Theta}}$ by \eqref{Burgers' Define D Theta}. Next, we estimate from \eqref{Burgers' est 180a}
\begin{align}
\lVert f -g \rVert_{H^{\gamma}} \leq& \lVert (g-f) \prec \sigma_{a}(D) \tilde{\eta} \rVert_{H^{\gamma}} + \lVert f \prec \sigma_{a}(D) (\tilde{\eta} - \eta) \rVert_{H^{\gamma}}  \nonumber \\
& \hspace{10mm} + \lVert B_{a}(f-g, \Theta) \rVert_{H^{\gamma}} + \lVert B_{a} (g, \Theta - \tilde{\Theta}) \rVert_{H^{\gamma}}.  \label{Burgers' est 248} 
\end{align}
Because $\gamma < \alpha + \frac{5}{2}$ by hypothesis, we can find $\epsilon_{2} > 0$ that satisfies \eqref{Burgers' Define epsilon two} and estimate 
\begin{equation}\label{Burgers' est 247}
 \lVert (g-f) \prec \sigma_{a}(D) \tilde{\eta} \rVert_{H^{\gamma}} \overset{\eqref{Sobolev products a}}{\lesssim}\lVert g-f \rVert_{L^{2}} \lVert \sigma_{a}(D) \tilde{\eta} \rVert_{\mathscr{C}^{\gamma + \epsilon_{2}}} \overset{\eqref{Burgers' est 172}}{\lesssim}a^{\frac{ 2(\gamma + \epsilon_{2} - \alpha)}{5} -1} \lVert g -f \rVert_{H^{\gamma}}\lVert \tilde{\Theta} \rVert_{\mathcal{E}^{\alpha}}
\end{equation} 
while we can rely on \eqref{est 276}, \eqref{est 278a}, \eqref{est 278b}, and \eqref{est 278c} to deduce 
\begin{align*}
\lVert B_{a}(f-g, \Theta) \rVert_{H^{\gamma}} \lesssim a^{\frac{2}{5}(\frac{11}{4} - \gamma) - 1} \lVert f-g \rVert_{H^{\gamma}} \lVert \Theta \rVert_{\mathcal{E}^{\alpha}}
\end{align*}
so that  
\begin{align*}
\lVert f-g \rVert_{H^{\gamma}} \lesssim& [a^{\frac{2(\gamma + \epsilon_{2} - \alpha)}{5} -1} \lVert \tilde{\Theta} \rVert_{\mathcal{E}^{\alpha}} + a^{\frac{2}{5} (\frac{11}{4} - \gamma) -1} \lVert \Theta \rVert_{\mathcal{E}^{\alpha}} ] \lVert f-g \rVert_{H^{\gamma}}  \\
&+ [a^{\frac{ 2(\gamma + \epsilon_{2} - \alpha)}{5} -1} \lVert f \rVert_{H^{\gamma}} + a^{\frac{2}{5} (\frac{11}{4} - \gamma) -1} \lVert g \rVert_{H^{\gamma}} ] \lVert \Theta - \tilde{\Theta} \rVert_{\mathcal{E}^{\alpha}}.
\end{align*}
Making use of the fact that $\frac{2( \gamma + \epsilon_{2} - \alpha)}{5} - 1 < 0$ and $\frac{2}{5} \left( \frac{11}{4} - \gamma \right) - 1 < 0$ leads to, for all sufficiently large $a \gg 1$, 
\begin{equation}\label{Burgers' est 249}
 \lVert f-g \rVert_{H^{\gamma}} \lesssim [a^{\frac{ 2}{5} (\gamma + \epsilon_{2} - \alpha) -1} \lVert f \rVert_{H^{\gamma}} + a^{\frac{2}{5} (\frac{11}{4} - \gamma) -1} \lVert g \rVert_{H^{\gamma}} ] \lVert \Theta - \tilde{\Theta} \rVert_{\mathcal{E}^{\alpha}}. 
\end{equation}
At last, we can estimate $f^{\flat} - g^{\flat}$ similarly starting from \eqref{Burgers' Define D Theta} and \eqref{Burgers' Define g flat}, making use of \eqref{Burgers' est 249} and conclude \eqref{Burgers' est 250}.  This completes the proof of Proposition \ref{Burgers' Proposition 6.7}.
\end{proof}  

\begin{proposition}\label{Burgers' Proposition 6.8}
Let $\alpha \in (-\frac{17}{12}, -\frac{5}{4}), \gamma \in (1,\alpha +\frac{5}{2})$, and $\Theta = (\eta, \Theta_{2}), \tilde{\Theta} = (\tilde{\eta}, \tilde{\Theta}_{2}) \in \mathcal{K}^{\alpha}$. Define $\mathcal{H}^{\Theta} \triangleq  -\Lambda^{\frac{5}{2}} - \eta$ and $\mathcal{H}^{\tilde{\Theta}} \triangleq -\Lambda^{\frac{5}{2}} - \tilde{\eta}$ with respective domains denoted by $\mathcal{D}_{\Theta}$ and $\mathcal{D}_{\tilde{\Theta}}$ defined according to \eqref{Burgers' Define D Theta}. Then 
\begin{align}
\lVert \mathcal{H}^{\Theta} f - \mathcal{H}^{\tilde{\Theta}} g \rVert_{L^{2}} \lesssim& ( \lVert f-g \rVert_{H^{\gamma}} + \lVert f^{\flat} - g^{\flat} \rVert_{H^{\frac{5}{2}}} + \lVert \Theta - \tilde{\Theta} \rVert_{\mathcal{E}^{\alpha}})  \nonumber \\
& \times \left( 1+ \lVert \Theta \rVert_{\mathcal{E}^{\alpha}}^{2} + \lVert \tilde{\Theta} \rVert_{\mathcal{E}^{\alpha}}^{2} \right) \left( 1+ \lVert f \rVert_{H^{\gamma}} + \lVert g \rVert_{H^{\gamma}} + \lVert f^{\flat} \rVert_{H^{\frac{5}{2}}} \right).  \label{Burgers' est 252}
\end{align} 
Moreover, the operator $\mathcal{H}: \hspace{1mm} \mathcal{D}_{\Theta} \mapsto L^{2}$ is symmetric in $L^{2}$ so that 
\begin{equation}
\langle \mathcal{H}f, g \rangle_{L^{2}} = \langle f, \mathcal{H} g \rangle_{L^{2}} \hspace{3mm} \forall \hspace{1mm} f, g \in \mathcal{D}_{\Theta}. 
\end{equation} 
\end{proposition} 
\begin{proof}[Proof of Proposition \ref{Burgers' Proposition 6.8}]  
We can bound from \eqref{Burgers' est 251}, 
\begin{align}
& \lVert \mathcal{H}^{\Theta} f -\mathcal{H}^{\tilde{\Theta}}  g \rVert_{L^{2}}  \leq \lVert \Lambda^{\frac{5}{2}} (f^{\flat} - g^{\flat}) \rVert_{L^{2}} + \lVert B(f-g, \Theta) \rVert_{L^{2}} + \lVert B(g, \Theta - \tilde{\Theta}) \rVert_{L^{2}}   \nonumber \\
&\hspace{18mm} + \lVert \mathcal{R} (f-g, \sigma(D) \eta, \eta) \rVert_{L^{2}} + \lVert \mathcal{R} (g, \sigma(D) (\eta - \tilde{\eta}), \eta) \rVert_{L^{2}} + \lVert \mathcal{R} (g, \sigma(D) \tilde{\eta}, \eta - \tilde{\eta}) \rVert_{L^{2}}  \nonumber \\
& \hspace{18mm} + \lVert ( \eta - \tilde{\eta}) \circ B(f,\Theta) \rVert_{L^{2}} + \lVert \tilde{\eta} \circ B(f-g, \Theta) \rVert_{L^{2}} + \lVert \tilde{\eta} \circ B(g, \Theta- \tilde{\Theta}) \rVert_{L^{2}}  \nonumber \\
& \hspace{18mm} + \lVert (\eta - \tilde{\eta}) \circ f^{\flat} \rVert_{L^{2}} + \lVert \tilde{\eta} \circ (f^{\flat} - g^{\flat}) \rVert_{L^{2}}.  \label{Burgers' est 253}
\end{align}
We can rely on the estimates \eqref{est 276}, \eqref{est 278a}, \eqref{est 278b}, and \eqref{est 278c} to handle the terms involving $B(f-g, \Theta), B(g, \Theta - \tilde{\Theta}), B(f,\Theta), B(f-g, \Theta)$, and $B(g, \Theta - \tilde{\Theta})$, and readily deduce \eqref{Burgers' est 252}. 

Next, we let $\{ \Theta^{n} \}_{n \in \mathbb{N}} = \{ (\eta_{n}, -\eta_{n} \circ \sigma(D) \eta_{n} - c_{n}) \}_{n\in\mathbb{N}} \subset C^{\infty}$ satisfy $\Theta^{n} \to \Theta$ in $\mathcal{E}^{\alpha}$ as $n \nearrow \infty$. Then, by Proposition \ref{Burgers' Proposition 6.7}, there exists $\{f_{n}\}_{n \in\mathbb{N}} \subset H^{\frac{5}{2}}$ such that \eqref{Burgers' est 254} is satisfied. The smoothness of $\eta_{n}$ allows us to define $\mathcal{H}_{n} \triangleq  -\Lambda^{\frac{5}{2}} - \eta_{n}$ on $H^{\frac{5}{2}}$. Moreover, applying the assumption of $\Theta^{n} \to \Theta$ in $\mathcal{E}^{\alpha}$ as $n \nearrow \infty$ and \eqref{Burgers' est 254} to \eqref{Burgers' est 252} shows that $\mathcal{H}_{n} f_{n} \to \mathcal{H} f$ in $L^{2}$. Additionally, if $\{g_{n}\}_{n \in \mathbb{N}} \subset H^{\frac{5}{2}}$ satisfies $\mathcal{H}_{n}g_{n} \to \mathcal{H} g$ in $L^{2}$ as $n \nearrow \infty$ and $\lim_{n \nearrow \infty} \lVert g_{n} - g \rVert_{H^{\gamma}} + \lVert g_{n}^{\flat} - g^{\flat} \rVert_{H^{\frac{5}{2}}} = 0$ similarly to \eqref{Burgers' est 254}, then $\langle \mathcal{H}_{n} f_{n}, g_{n} \rangle_{L^{2}}  = \langle f_{n}, \mathcal{H}_{n} g_{n} \rangle_{L^{2}}$ so that 
\begin{align*}
\lvert \langle \mathcal{H}f, g \rangle_{L^{2}} - \langle f, \mathcal{H} g \rangle_{L^{2}}& \rvert \leq \lVert \mathcal{H} f - \mathcal{H} f_{n} \rVert_{L^{2}} \lVert g \rVert_{L^{2}} + \lVert \mathcal{H} f_{n} \rVert_{L^{2}} \lVert g - g_{n} \rVert_{L^{2}} \\
&+ \lVert f_{n} - f \rVert_{L^{2}} \lVert \mathcal{H}_{n} g_{n} \rVert_{L^{2}} + \lVert f \rVert_{L^{2}} \lVert \mathcal{H}_{n}g_{n} - \mathcal{H} g \rVert_{L^{2}}  \to 0 \text{ as } n \nearrow \infty. 
\end{align*}
This completes the proof of Proposition \ref{Burgers' Proposition 6.8}.
\end{proof} 

\begin{proposition}\label{Burgers' Proposition 6.9}
Let $\alpha \in (-\frac{17}{12}, -\frac{5}{4}), \gamma \in (1, \alpha + \frac{5}{2}),$ and $A = A( \lVert \Theta \rVert_{\mathcal{E}^{\alpha}})$ from Proposition \ref{Burgers' Proposition 6.3}. Then, for all $a \geq A$, $-\mathcal{H} + a: \hspace{1mm} \mathcal{D}_{\Theta} \mapsto L^{2}$ is invertible with inverse $\mathcal{G}_{a}: \hspace{1mm} L^{2} \mapsto \mathcal{D}_{\Theta}$. Additionally, $\mathcal{G}_{a}: \hspace{1mm} L^{2} \mapsto L^{2}$ is bounded, self-adjoint, and compact. 
\end{proposition} 

\begin{proof}[Proof of Proposition \ref{Burgers' Proposition 6.9}]
Let $g \in L^{2}$. By Proposition \ref{Burgers' Proposition 6.3}, there exists a unique $f_{a} \in \mathcal{D}_{\eta}^{\gamma}$ such that $f_{a} = \mathcal{G}_{a} g$; i.e., $(-\mathcal{H} + a) f_{a} = g$, and consequently due to \eqref{Burgers' Define H} and \eqref{Burgers' est 165}, 
\begin{align}
 (1 + \Lambda^{\frac{5}{2}}) & f_{a}^{\sharp} = f_{a}^{\sharp} + g - af_{a} - \Lambda^{\frac{5}{2}} \left( f_{a} \prec \sigma(D) \eta \right) + \Lambda^{\frac{5}{2}} f_{a} \prec \sigma(D) \eta+ f_{a} \prec \Lambda^{\frac{5}{2}} \sigma(D) \eta \nonumber \\
 & + (1 - \Lambda^{\frac{5}{2}}) f_{a} \prec \sigma(D) \eta - f_{a} \succ \eta   - \left( f_{a} \Theta_{2} + \mathcal{R} (f_{a}, \sigma(D) \eta, \eta ) + f_{a}^{\sharp} \circ \eta \right).  \label{Burgers' est 257}
\end{align} 
Resultantly, 
\begin{align}
f_{a}^{\flat} \triangleq f_{a}^{\sharp} - B(f_{a}, \Theta) \overset{\eqref{Burgers' est 257} \eqref{Burgers' Define B}}{=}  \sigma(D) [-f_{a}^{\sharp} - g + a f_{a} + \mathcal{R} (f_{a}, \sigma(D) \eta, \eta) + f_{a}^{\sharp} \circ \eta]   \label{Burgers' est 258}
\end{align}
where we can show that $f_{a}^{\flat} \in H^{\frac{5}{2}}$ 
\begin{equation}\label{Burgers' est 259}
\lVert f_{a}^{\flat} \rVert_{H^{\frac{5}{2}}} \overset{\eqref{Burgers' Estimate on R} \eqref{Sobolev products e}}{\lesssim}  a \lVert f_{a} \rVert_{\mathcal{D}_{\eta}^{\gamma}} (1+ \lVert \eta \rVert_{\mathscr{C}^{\alpha}}^{2}) + \lVert g \rVert_{L^{2}}  \lesssim 1, 
\end{equation} 
and consequently $f_{a} \in H^{\alpha + \frac{5}{2} - \kappa}$ so that $f_{a} \in \mathcal{D}_{\Theta}$ by \eqref{Burgers' Define D Theta}.  Moreover, we can choose $\rho = \frac{2}{5}(\gamma -\frac{1}{4})$ since it satisfies \eqref{Burgers' Define rho} and deduce from \eqref{Burgers' Norm D Theta gamma}, 
\begin{equation}\label{Burgers' est 260} 
\lVert \mathcal{G}_{a} g \rVert_{\mathcal{D}_{\Theta}^{\gamma}} \leq \lVert \mathcal{G}_{a} g \rVert_{H^{\gamma}} + a^{-\rho} \lVert ( \mathcal{G}_{a} g)^{\sharp} \rVert_{H^{2\gamma - \frac{1}{4}}} + \lVert f_{a}^{\flat} \rVert_{H^{\frac{5}{2}}}  \overset{\eqref{Burgers' est 259}\eqref{Burgers' est 202}}{\lesssim} a^{\frac{4\gamma}{5} - \frac{1}{10}} (1+\lVert \eta \rVert_{\mathscr{C}^{\alpha}}^{2} ) \lVert g\rVert_{L^{2}}.
\end{equation}   
Next, the fact that $\mathcal{G}_{a}: \hspace{1mm} L^{2} \mapsto L^{2}$ is self-adjoint follows from the symmetry of $\mathcal{H}$. Finally, writing $\mathcal{G}_{a}: \hspace{1mm} L^{2} \mapsto L^{2}$ as a composition of $\mathcal{G}_{a}:\hspace{1mm} L^{2} \mapsto H^{\gamma}$ and an embedding operator $i:  \hspace{1mm} H^{\gamma} \mapsto L^{2}$ demonstrates that $\mathcal{G}_{a}: \hspace{1mm} L^{2} \mapsto L^{2}$ is compact. This  completes the proof of Proposition \ref{Burgers' Proposition 6.9}.
\end{proof}

\appendix

\section{Proof of Theorem \ref{Theorem 2.2}}\label{Proof of Theorem 2.2}
Suppose that $T^{\max} < \infty. $ By Proposition \ref{Proposition 4.12} this implies $\limsup_{t\nearrow T^{\max}} \lVert w_{u}(t) \rVert_{L^{2}} = + \infty$.  By  \eqref{Define T0 and Ti} this implies $T_{i} < T^{\max}$ for all $i \in \mathbb{N}$. Because $T^{\max} < + \infty$,  \eqref{est 110} gives us  
\begin{equation}\label{est 141} 
T_{i+1} - T_{i} \geq \frac{1}{ \tilde{C} (N_{T^{\max}}^{\kappa}) (\ln(1+i) + 1)} \ln \left( \frac{i^{2} + 2i - C(N_{T^{\max}}^{\kappa})}{i^{2} + \tilde{C} (N_{T^{\max}}^{\kappa} )} \right)
\end{equation} 
where $\sum_{i=1}^{\infty} T_{i+1} - T_{i} < \infty$ because $T_{\max} < \infty$. On the other hand, the sum over the right hand side over $i \in \mathbb{N}$ blows up to $+ \infty$ and thus we reach a contradiction.

\section*{Acknowledgments}
The author expresses deep gratitude to Prof. Carl Mueller, Prof. Tomasso Rosati, Prof. Samy Tindel Prof. Jiahong Wu, Prof. Quoc Huang Nguyen, and Prof. In-Jee Jeong for valuable discussions.

\end{document}